\documentclass{article}
\usepackage[utf8]{inputenc}

\usepackage{hyperref}

\usepackage[document]{ragged2e}

\usepackage{amsmath}
\usepackage{amssymb}
\usepackage{amsthm}
\usepackage{mathrsfs}
\usepackage{mathtools}
\usepackage{comment}
\usepackage{enumerate}
\usepackage{xcolor}
\usepackage{slashed}

\usepackage[english]{babel}

\usepackage{cite}
\usepackage{url}
\usepackage{hyperref}

\usepackage{geometry}
\numberwithin{equation}{section}

\usepackage{amscd}
\usepackage{tikz-cd}

\newtheorem{theorem}{Theorem}[section]
\newtheorem{proposition}[theorem]{Proposition}
\newtheorem{lemma}[theorem]{Lemma}
\newtheorem{corollary}[theorem]{Corollary}

\newtheorem{remark}[theorem]{Remark}
\newtheorem{example}[theorem]{Example}
\theoremstyle{definition}
\newtheorem{definition}[theorem]{Definition}

\usepackage{subfiles} 

\usepackage[nottoc]{tocbibind}

\title{An analytic approach to Lefschetz and Morse theory on stratified pseudomanifolds}
\author{Gayana Jayasinghe}
\date{}
 
\begin{document}
\maketitle
\justify

\begin{abstract}  
We develop an analytic framework for Lefschetz fixed point theory and Morse theory for Hilbert complexes on stratified pseudomanifolds. We develop formulas for both global and local Lefschetz numbers and Morse, Poincar\'e polynomials as (polynomial) supertraces over cohomology groups of Hilbert complexes, developing techniques for relating local and global quantities using heat kernel and Witten deformation based methods.

We focus on the case where the metric is wedge and the Hilbert complex is associated to a Dirac-type operator and satisfies the Witt condition, constructing Lefschetz versions of Bismut-Cheeger $\mathcal{J}$ forms for local Lefschetz numbers of Dirac operators, with specialized formulas for twisted de Rham, Dolbeault and spin$^{\mathbb{C}}$ Dirac complexes as supertraces of geometric endomorphisms on cohomology groups of local Hilbert complexes.
We construct geometric endomorphisms to define de Rham Lefschetz numbers for some self-maps for which the pullback does not induce a bounded operator on $L^2$ forms.

A de Rham Witten instanton complex is constructed for Witt spaces with stratified Morse functions, proving Morse inequalities related to other results in the literature including Goresky and MacPherson's in intersection cohomology. We also prove a Lefschetz-Morse inequality for geometric endomorphisms on the instanton complex that is new even on smooth manifolds.

We derive $L^2$ Lefschetz-Riemann-Roch formulas, which we compare and contrast with algebraic versions of Baum-Fulton-Quart. In the complex setting, we derive Lefschetz 
formulas for spin Dirac complexes and Hirzebruch $\chi_y$ genera which we relate to signature, self-dual and anti-self-dual Lefschetz numbers, studying their properties and applications including instanton counting. We compute these invariants in various examples with different features, comparing with versions in other cohomology theories.
\end{abstract}

\tableofcontents

\section{Introduction}

Atiyah and Bott \cite{AtiyahBott1,AtiyahBott2} generalized the fixed point theorem of Lefschetz \cite{lefschetz1937fixed} to elliptic complexes of differential operators on smooth manifolds. Witten gave a groundbreaking take on the Morse inequalities in \cite{witten1982supersymmetry}, relating it to ideas in physics (see for instance \cite{witten1982constraints}). In particular he used methods inspired by physics to prove character valued formulas for indices of spin Dirac, Rarita Schwinger and other operators in \cite{witten1983fermion}, and formulated an instanton complex for the Dolbeault complex yielding holomorphic Morse inequalities in \cite{witten1984holomorphic}.
Extensions and applications of these results have been widespread in various settings, including some on singular spaces, and we shall discuss these in Subsection \ref{subsection_History_motivation}.

\textbf{This article broadly describes a framework to extend the Atiyah-Bott-Lefschetz fixed point theorem and Morse theory to stratified spaces, for elliptic complexes of differential operators.} The normal fibers to fixed point sets in the smooth setting are always Euclidean which is why the theory is much simpler in the smooth setting. In this article we study the case of isolated fixed points (the morally generic case according to Bott in \cite[\S 4]{bott1988fixed}) but non-isolated singularities, and we treat smooth and singular fixed points on a similar footing.
These results are highly interwoven and our unified study gives rise to various new formulae, including versions of Lefschetz-Morse inequalities that have not been observed previously even in the smooth setting.
The richness of these results in the smooth setting are well established and there are many applications in both mathematics and physics, some of which we generalize to singular spaces in this article. We study the holomorphic Witten instanton complexes in the follow up article \cite{jayasinghe2024holomorphicwitteninstantoncomplexes} where we also study rigidity of invariants and applications.

More specifically, in this article we generalize the Atiyah-Bott-Lefschetz fixed point theorem to elliptic complexes associated to Dirac-type operators on oriented stratified pseudomanifolds equipped with wedge metrics, which can be described roughly as \textit{asymptotically iterated conic metrics}. 
We develop a theory of Lefschetz numbers for abstract Hilbert complexes satisfying certain conditions, building technical machinery to express \textit{global} and \textit{local} (near fixed points) Lefschetz numbers as well as Morse and Poincar\'e polynomials, as (polynomial) \textit{supertraces} of endomorphisms of cohomology groups as well as relating them to heat supertraces. We do this with sufficient generality to capture many interesting complexes that are widely studied in geometry and for geometric endomorphisms associated to \textit{nice} self-maps on stratified pseudomanifolds.

For the case of the $L^2$ de Rham complex, our results are analogues of the Lefschetz fixed point theorem and the Morse inequalities of Goresky and MacPherson in intersection homology (see \cite{Goresky_1985_Lefschetz} and \cite[\S 6.12]{goresky1988stratified}), and generalize previous work in the case of isolated singularities in \cite{Bei_2012_L2atiyahbottlefschetz} and of \cite{ludwig2013analytic,Jesus2018Wittensgeneral}. While the Morse inequalities have been worked out in \cite{Jesus2018Wittens,Jesus2018Wittensgeneral,ludwig2013analytic}, our construction of the Witten instanton complex extends techniques in \cite{Zhanglectures} to the singular setting which we use to construct holomorphic Witten instanton complexes in \cite{jayasinghe2024holomorphicwitteninstantoncomplexes}.
Our definitions of local cohomology groups are analytic. We also construct a generalized endomorphism for the de Rham complex using the \textit{Hilsum-Skandalis} replacement, enabling us to prove the Lefschetz fixed point theorem for a class of self-maps $f$ for which the pullback $f^*$ is not a bounded operator on $L^2$ forms.

In the cases of de Rham, Dolbeault and spin$^{\mathbb{C}}$ Dirac complexes with coefficients in a bundle satisfying suitable flatness conditions, the Lefschetz fixed point theorem can be roughly stated as; \textbf{\textit{supertraces of induced maps on global cohomology are equal to sums of supertraces of induced maps on local cohomology at fixed points}}. Thus both the global and local Lefschetz numbers are counting something, in a generalized sense. The topological twisting by flat bundles allows us to study self-dual, anti-self-dual, spin Dirac and other complexes with these results.

The local cohomology of the twisted Dolbeault and spin$^{\mathbb{C}}$ Dirac complexes are infinite dimensional, and we prove a renormalized version of the \textit{McKean-Singer theorem} (see \cite{mckean1967curvature}) for \textit{Lefschetz heat supertraces} of Laplace-type operators on neighbourhoods of fixed points with $\overline{\partial}$-Neumann boundary condition and modifications studied by \cite{EpsteinSubellipticSpinc3_2007,epstein2006subelliptic} defining renormalized trace formulas for local Lefschetz numbers. 
Our formulation and framework is new, even in the case of smooth complex manifolds, and our formulas are new in the case of isolated singularities. The closest analogs are those in the algebraic setting such as the Lefschetz-Riemann-Roch formulas in \cite{baum1979lefschetz,Baumformula81,MaximJorgCharacteristicsingulartoric2015,MaximSaitoJorgHodgemodules2011}
and we investigate similarities and differences with these in many examples, exploring various aspects of formulas on algebraic varieties. We compute many other equivariant indices of complexes on stratified pseudomanifolds with wedge complex (and some almost complex) structures.

We now give an overview of our results, referring the reader to the text for precise definitions and details. This also serves as a more detailed description of the content in each section.

\subsection{Overview of results.}

In Section \ref{section_stratified_space_basics}, we review the basic objects of our study. In particular that there is a functorial equivalence between Thom-Mather stratified spaces and manifolds with corners with iterated fibration structures that resolve them. Thus we define objects on pseudomanifolds $\widehat{X}$ by defining them on their resolution $X$. This includes self-adjoint Dirac-type operators satisfying a \textit{Witt} condition with the \textit{VAPS domain}.
We then introduce Hilbert complexes and an abstract Lefschetz fixed point theorem in that setting in Section \ref{section_Hilbert_complexes}.
Given a Hilbert complex $\mathcal{P}=(H, P)$ on $X$ of the form
\begin{equation}
    0 \rightarrow H_0 \xrightharpoonup{P_0} H_1 \xrightharpoonup{P_1} H_2 \xrightharpoonup{P_2} ... \xrightharpoonup{P_{n-1}} H_n \rightarrow 0,
\end{equation}
where $\xrightharpoonup{}$ indicates a partially defined map, we study Lefschetz numbers of endomorphisms $T$ of the complex. We denote the cohomology groups of these complexes in degree $k$ by $\mathcal{H}^k(\mathcal{P})$ and there are induced endomorphisms in cohomology which we shall also represent by a tuple $T$ with some abuse of notation.
The \textit{\textbf{global Lefschetz number}} is defined as the supertrace of the induced endomorphism on cohomology,
\begin{equation} 
L(X,\mathcal{P},T) := \sum_{k=0}^n (-1)^k Tr(T_k|_{\mathcal{H}^k(\mathcal{P})}) 
\end{equation}
where the word \textit{super} denotes that it is an alternating sum of traces. The term supertrace evokes the relation to supersymmetry in physics (see for instance \cite{witten1982supersymmetry}), which is now understood to have significant connections to index theory and topological invariants in general.

In the development of the theory in Subsection \ref{subsection_abstract_lefschetz_supertraces}, we use an enrichment of these traces where instead of using the $\mathbb{Z}_2$ grading in the supertrace, we keep track of all degrees in what we call a \textit{\textbf{polynomial Lefschetz supertrace}}, replacing the factor of $(-1)^k$ with $b^k$ to get
\begin{equation} 
\label{blug_2987}
L(\mathcal{P}(X),T)(b) := \sum_{k=0}^n b^k Tr(T_k|_{\mathcal{H}^k(\mathcal{P}(X))}).
\end{equation}
Such decorated traces are natural in framing generalized Morse inequalities for complexes, and have been used for instance in \cite{mathai1997equivariant} to prove Witten's holomorphic Morse inequalities in polynomial form given in \cite{witten1984holomorphic} for K\"ahler actions.
This approach enables us to give a uniform treatment of dualities and other techniques which are used in common in proofs of Lefschetz fixed point theorems and Morse inequalities.
There are equivalent approaches with signed traces in \cite{bismut1986witten,bismut1987demailly}, \cite[\S 1.7]{MaMarinescu2007HolMorseBook}.
With this setup, we prove the following generalization of the Atiyah-Bott-Lefschetz fixed point theorem.

\begin{theorem}
\label{Lefschetz_supertrace_intro_version}
Let $\mathcal{P}=(H,P)$ be an elliptic complex where the associated Laplace-type operators in each degree have discrete spectrum and trace class heat kernels. Let $T$ be an endomorphism of the complex. For all $t \in \mathbb{R}^+$
\begin{equation} 
    \mathcal{L}(\mathcal{P},T)(b,t)=L(\mathcal{P},T)(b)+ (1+b) \sum_{k=0}^{n-1} b^k S_k(t)
\end{equation}
where
\begin{equation} 
    S_k(t)=\sum_{\lambda_i \in Spec(\Delta_k)} e^{-t \lambda_i}  \langle T_k v_{\lambda_{i}}, v_{\lambda_{i}} \rangle
\end{equation}
where $\{v_{\lambda_i}\}_{i \in \mathbb{N}}$ are an orthonormal basis of co-exact eigensections of $\Delta_k$. In particular
\begin{equation}
    L(\mathcal{P},T)= \mathcal{L}(\mathcal{P},T)(-1,t)
\end{equation}
and the Lefschetz heat supertrace is independent of $t$.
\end{theorem}
This is Theorem \ref{Lefschetz_supertrace}, proven using heat kernel methods for Fredholm complexes with discrete spectrum by techniques which are well understood at this point. 
In Section \ref{section_Geometric_endomorphism} 
we define geometric endomorphisms associated to self maps $\widehat{f}:\widehat{X} \rightarrow \widehat{X}$ that lift to self maps $f:X \rightarrow X$, $T_f=T=(T_0,T_1,..,T_n)$, with $T_k=\phi_k \circ f^*$, where $\phi_k: f^*E_k \rightarrow E_k$ is a bundle morphism. This is for complexes with Hilbert spaces $H_k$ which are $L^2$ bounded sections on stratified pseudomanifolds with values in the bundles $E_k$. We assume that the self maps $\widehat{f}$ are stratum preserving homeomorphisms which lift to diffeomorphisms $f$ of the resolved space. This ensures that the $T_k$ are endomorphisms of the Hilbert spaces $H_k$. The condition that $[T,P]=0$ is natural for most complexes. For the Dolbeault complex, it amounts to requiring that $f$ is a local biholomorphism, and the condition is automatically satisfied for the de Rham complex. 
In this article, we focus on self maps $\widehat{f}$ which have isolated fixed points, and in this case we say that $f$ has isolated fixed points. Indeed, the connected components of fixed point sets of such maps $f$ are metrically identified.

In Section \ref{section_formulas}, we use heat kernels on stratified pseudomanifolds constructed in \cite{Albin_2017_index} to show that the Lefschetz numbers of the Hilbert complexes on stratified pseudomanifolds can be computed as sums of \textit{local Lefschetz numbers} $L(\mathcal{P}(U_p),T_f,p)$ at fixed points $p$ which are determined by local data of the self map $\widehat{f}$ near the fixed point (see Theorem \ref{localization_of_simple_fixed_points} and the subsequent discussion).
In Theorem \ref{theorem_localization_model_metric_cone}, we show that the local Lefschetz numbers for \textit{simple} isolated fixed points depend only on the induced map $\mathfrak{T}_pf$ on the tangent cone $\mathfrak{T}_pX$ at the fixed point, with a formula in terms of a supertrace on the tangent cone which is a Lefschetz variant of the Bismut-Cheeger $\mathcal{J}$ form constructed in \cite{Albin_2017_index}.

\begin{theorem}
\label{theorem_localization_model_metric_cone_intro_version}
Let $f$ be a self map on a resolved pseudomanifold $X$ which is associated to a geometric endomorphism $T_f=\varphi \circ f$ of an elliptic complex $\mathcal{P}(X)$. For an isolated simple fixed point at $p=y_0$, which lies on a stratum $Y$, with a fundamental neighbourhood $U_p=\mathbb{D}^k \times C(Z)$ where we assume that the elliptic complex decomposes as a product, the local Lefschetz number $L(\mathcal{P}(U_p),T_f,p)$ can be expressed as 
\begin{equation}
\label{equation_finale_tangent_cone_intro}
    \Bigg( \sum_{k=0}^n \frac{(-1)^k Tr(\varphi_{k,p} \circ D_pC)}{|det(Id-D_pC)|} \Bigg) \int_{\tau=0}^{\infty} \int_{Z}  str \Big(\varphi \circ \widetilde{K}_Z(\widetilde{A}(z),\widetilde{B}(z),1,z,\tau) dvol_Z \Big) \frac{1}{2\tau} d\tau.
\end{equation}
In particular, the local Lefschetz number only depends on the induced map on the tangent cone.
\end{theorem}

The notation and details are explained in Theorem \ref{theorem_localization_model_metric_cone}. Subsection \ref{subsection_heat_kernel_localization} uses techniques developed using functional calculus on infinite model cones while in Subsection \ref{subsection_cohomological_formulae} we develop machinery on truncated cones, including Hilbert complexes on \textit{fundamental neighbourhoods} $\widehat{U_p}$ of isolated fixed points $p$, that have resolutions $U_p$ that we refer to as fundamental neighbourhoods by abuse of notation (indeed they are metrically the same). We pick \textit{good} domains for the operators of complexes restricted to $U_p$, which we denote as $\mathcal{P}(U_p)$. This corresponds to picking boundary conditions for the associated Laplace-type operator on the \textit{smooth boundary} $\partial{U_p}$, that is the preimage under the resolving map of the boundary of $\widehat{U_p}$, and we distinguish between local and non-local boundary conditions. The Lefschetz heat supertrace for the heat kernel with boundary conditions will in general have boundary contributions at the boundary in addition to contributions in the interior.
We define 
\begin{equation}
    L(\mathcal{P}(U_p),T_f)=\sum_{k=0}^n (-1)^k Tr[f^*|_{\mathcal{H}^k(\mathcal{P}(U_p))}]
\end{equation}
to be the \textit{\textbf{global Lefschetz number of the geometric endomorphism $T_f$, on the fundamental neighbourhood $U_p$}}. 

We study self-maps $f$ for which each isolated fixed point $p$ has a fundamental neighbourhood $U_p$ such that there is a product decomposition $U_p=U_1\times U_2$ where $f$ is a \textit{non-expanding} map $f_1$ on $U_1$, and a \textit{non-attracting} map $f_2$ on $U_2$ (see Proposition \ref{proposition_local_product_Lefschetz_heat}).
In this case we have product formulas for Lefschetz numbers of the form
\begin{equation}
\label{intro_product_equation_refering}
    L(\mathcal{P}_B(U_p), T_f,p)=L(\mathcal{P}_N(U_1), T_{f_1},p) \cdot (-1)^m L(\mathcal{P}^*_N(U_2), T^{P^*}_{f_2^{-1}},p)
\end{equation}
where $\mathcal{P}_N,\mathcal{P}^*_N$ are dual complexes, and $T^{P^*}_{f_2^{-1}}$ denotes the geometric endomorphism induced on the dual complex. This interplay between dual domains and adjoint geometric endomorphisms plays a major role in obtaining computable local formulas, and we refer the reader to Subsection \ref{subsubsection_Local_to_global_formulae} for more details, including the applicability of this hypothesis. Such intertwining of dual cohomology groups and dual endomorphisms was studied in \cite{Goresky_1985_Lefschetz,Goresky_1993_local_Lefschetz,Bei_2012_L2atiyahbottlefschetz}, restricted to the de Rham complex.

We define the \textit{\textbf{Lefschetz boundary contribution}}
\begin{equation}
\label{equation_Lefschetz_boundary_contribution}
    L(\partial U_p, \mathcal{P}(U_p), T_f)= \lim_{t \rightarrow 0} \sum_{k=0}^n \int_{W_p} (-1)^k tr(\phi_k \circ \mathcal{K}_{U_p}^k(t,f(q),q) dvol)
\end{equation}
where $W_p \subset U_p$ is an open neighbourhood of $\partial U_p$ that does not contain the fixed point $p$, and $\mathcal{K}_{U_p}^k(t,\cdot, \cdot)$ is the heat kernel of the Laplace-type operator $\Delta_k$ on $U_p$ corresponding to boundary conditions for $\mathcal{P}_B(U_p)$. This boundary contribution vanishes for the local \textit{generalized Neumann} boundary conditions we study for the de Rham and Dolbeault complexes, while for the Atiyah-Patodi-Singer conditions it corresponds to a generalized eta invariant. 
We have the following theorem, which is a simplified version of Proposition \ref{Local_Lefschetz_numbers_definition}.
\begin{theorem}[Local Lefschetz numbers]
\label{theorem3}
In the same setting as Theorem \ref{theorem_localization_model_metric_cone_intro_version}, let $p$ be an isolated fixed point of $f$ which has a fundamental neighbourhood $U_p$, with a choice of domain, yielding an elliptic complex $\mathcal{P}(U_p)$.
Then we have that
\begin{equation}
    L(\mathcal{P}(U_p),T_f)=L(\mathcal{P}(U_p),T_f,p)+L(\partial U_p, \mathcal{P}(U_p), T_f).
\end{equation}
\end{theorem}
For the de Rham and Dolbeault complexes we have Theorem \ref{Corollary_Polynomial_Lefschetz_fixed_point_theorem} which we summarize here.

\begin{theorem}
\label{Corollary_Polynomial_Lefschetz_fixed_point_theorem_introduction_version}
Let $\mathcal{P}(X)$ be the de Rham or Dolbeault complex (possibly with twisted coefficients) satisfying the Witt condition, on a pseudomanifold $X$. Let $f$ be a self map on $X$ with simple isolated fixed points, associated to a geometric endomorphism $T_f$ on the complex. Then
\begin{equation}
    \sum_{p \in Fix(f)}  L(\mathcal{P}_B(U_p), T_f)= L(\mathcal{P}(X), T_f).
\end{equation}
where the expression on the left hand side is the sum of the local Lefschetz numbers at the fixed points, and that on the right the global Lefschetz number.
\end{theorem}
This is a local to global formula. In the de Rham case, the generalized-Neumann boundary conditions give a Fredholm complex while for the Dolbeault complex, they give a non-Fredholm one. We introduce renormalized Lefschetz heat supertraces for the Dolbeault complex in Subsection \ref{subsubsection_renormalized_Lefschetz}, and prove a renormalized version of Theorem \ref{Lefschetz_supertrace_intro_version} for $b=-1$. The renormalization is similar \textit{in spirit} to that used in generalized eta invariants \cite{weiping1990note} and formulas in \cite{witten1984holomorphic,Baumformula81}.

In Section \ref{section_de_Rham} we focus on the $L^2$ de Rham complex of \textit{Witt spaces}, which are stratified pseudomanifolds where the middle dimensional $L^2$ cohomology of links at strata vanish.
We briefly recall Goresky and MacPherson's Lefschetz formulas in intersection homology in Subsection \ref{subsection_intersection_Lefschetz_number} for comparison. Their formulas are for placid self-maps which satisfy weaker conditions than the general self-maps we study. Since $L^2$ de Rham cohomology only depends on the metric on the regular part of the Witt space, we define \textit{Hilsum-Skandalis} geometric endomorphisms for \textit{proper} self maps on $X^{reg}$. We observe that this gives general geometric endomorphisms for self-maps on $X^{reg}$ that preserve the singular set, but we must impose more constraints for studying local Lefschetz numbers near fixed points at the singularities.
In Subsection \ref{subsection_induced_map_in_cohomology}, we show the stratified homotopy invariance of de Rham Lefschetz numbers, and how the Hilsum-Skandalis replacement can be used away from fundamental neighbourhoods of fixed points, to apply our results to a wide class of self maps.
In Subsection \ref{subsection_analytic_derivation_intersection} we derive local Lefschetz numbers corresponding to those of Goresky and MacPherson.

We then construct the Witten instanton complex for the de Rham complex corresponding to stratified Morse functions, proving that there is an \textit{instanton complex}, or a \textit{small eigenvalue complex} which is a subcomplex of the Witten deformed complex. The Witten deformed complex $\mathcal{P}_{\varepsilon}(X)$ is a co-chain complex that is isomorphic as a co-chain complex to the de Rham complex $\mathcal{P}_{\varepsilon=0}(X)$, and the instanton complex is a finite dimensional subcomplex of the latter that is quasi-isomorphic to the de Rham complex that can be described in terms of local information at the critical points of the stratified Morse function. The following is a version of Theorem \ref{theorem_small_eig_complex}.

\begin{theorem}[Witten instanton complex (de Rham)]
\label{theorem_small_eig_complex_de_Rham_intro}
Let $\widehat{X}$ be a Witt space of dimension $n$ with a wedge metric and a stratified Morse function $h$. Let $E$ be a flat vector bundle on $X$. Let $\mathcal{P}(X)$ be the de Rham complex $(L^2\Omega(X;E),d^E)$.  
For any integer $0 \leq k \leq n$, let $
\mathrm{F}_{\varepsilon, k}^{[0, c]} \subset L^2\Omega^k(X;E)$ denote the vector space generated by the eigenspaces of $\Delta_{\varepsilon,k}$ associated with eigenvalues in $[0, c]$. 

For any $c>0$, there exists $\varepsilon_0>0$ such that when $\varepsilon>\varepsilon_0$, the dimension of the eigenspaces are the same as $\sum_{a \in Cr(h)} \mathcal{H}^k(\mathcal{P}_{\varepsilon,B}(U_a))$, and together form a finite dimensional subcomplex of $\mathcal{P}_{\varepsilon}(X)$ :
\begin{equation}
    \label{small_eigenvalue_complex_intro_de_Rham}
\left(\mathrm{F}_{\varepsilon, k}^{[0, c]}, P_{\varepsilon}\right): 0 \longrightarrow \mathrm{F}_{\varepsilon, 0}^{[0, c]} \stackrel{P_{\varepsilon}}{\longrightarrow} \mathrm{F}_{\varepsilon, 1}^{[0, c]} \stackrel{P_{\varepsilon}}{\longrightarrow} \cdots \stackrel{P_{\varepsilon}}{\longrightarrow} \mathrm{F}_{\varepsilon, n}^{[0, c]} \longrightarrow 0.
\end{equation}
\end{theorem}

This categorifies a version of the Morse inequalities of Goresky and MacPherson in \cite[\S 6.12]{goresky1988stratified} (c.f. \cite{Jesus2018Wittensgeneral,ludwig2017comparison}), 
which follows as a corollary, and we present it as follows.

\begin{theorem}
\label{Theorem_strong_Morse_de_Rham_intro}
[Strong polynomial Morse inequalities for the de Rham complex]
Let $\widehat{X}$ be a Witt space of dimension $n$ with a wedge metric and a stratified Morse function $h$. Let $E$ be a flat vector bundle on $X$. Let $\mathcal{P}(X)=(L^2\Omega(X;E),d^E)$ be the de Rham complex. 
Then there exist non-negative integers $Q_0,..., Q_{n-1}$ such that
\begin{equation}
    \Big( \sum_{a \in Crit(h)}  \sum_{k=0}^n b^k dim(\mathcal{H}^{k}(\mathcal{P}_B(U_a)) \Big) - \sum_{k=0}^n b^k dim(\mathcal{H}^{k}(\mathcal{P}(X))) = (1+b) \sum_{k=0}^{n-1} Q_k b^k
\end{equation}
where $\mathcal{P}_B(U_a)$ is the product complex in Definition \ref{definition_product_decomposition_critical_point}.
\end{theorem}
The \textit{inequalities} advertised by the theorem's label are those imposed on the coefficients of the polynomial on the left hand side by the fact that $Q_j \in \mathbb{N}_{0}$.
This is Theorem \ref{Theorem_strong_Morse_de_Rham} which generalizes the classical polynomial Morse inequalities.

We prove the following result, which seems to be new even in the smooth setting.
\begin{theorem}[Lefschetz-Morse inequalities]
\label{theorem_Lefschetz_Morse_inequalities_de_Rham_intro}
Let $X$ be a resolution of a stratified pseudomanifold equipped with a resolution of a stratified Morse function $h$. Let $f$ be a self map of $X$ such that $f^*h=h$ and the fixed points of $f$ are a subset of the critical points of $h$ that lifts to a geometric endomorphism of the de Rham complex on $X$. Then we have the Lefschetz inequalities
\begin{equation}
\label{Lefschetz_inequalities_intro}
    \mathcal{L}(\mathcal{P},T_f)(b)=L(\mathcal{P},T_f)(b)+ (1+b) \sum_{k=0}^{n-1} b^k Q_k
\end{equation}
where $Q_k=0$ if $h$ is perfect at degrees $k, k+1$. 
\end{theorem}

This is a restatement of Theorem \ref{theorem_Lefschetz_Morse_inequalities_de_Rham} and we work out a simple example in Example \ref{Example_Lefschetz_Morse_1}. It is a natural generalization of taking the trace of the identity map on the Witten instanton complex which categorifies the Morse inequalities.
We study examples and applications, related to other invariants in complex geometry as well as a generalized Arnold-Floer type theorem in Section \ref{Dolbeault_section}. We study the holomorphic Witten instanton complex and applications in \cite{jayasinghe2024holomorphicwitteninstantoncomplexes} where we explore this and other instanton complexes in greater depth.

In Section \ref{Dolbeault_section}, we focus on holomorphic Lefschetz numbers, and more generally, equivariant indices of Hilbert complexes on stratified pseudomanifolds with wedge complex (and almost complex) structures. We start by developing a good understanding of the $L^2$ Dolbeault complex and its cohomology on complex cones in Subsection \ref{subsection_L2_Dolbeault_cohomology}, in particular establishing important dualities using Hodge theory.
In Subsection \ref{subsection_Local_Lefschetz_Dolbeault_main} we study local holomorphic Lefschetz numbers. In Subsection \ref{subsubsection_Baum_Fulton_Quart_method_compute} we review the formulas of \cite{baum1979lefschetz,Baumformula81} for holomorphic actions on coherent sheaves on quasi-projective varieties. 
We then consider our $L^2$ versions on fundamental neighbourhoods. 
In Subsection \ref{subsection_classes_of_singularities}, we discuss properties of our local cohomology groups for various classes of singularities studied in algebraic and complex geometry. We compare our local Lefschetz numbers with those of Baum-Fulton-Quart and discuss relations to work in \cite{baumfultonmacKtheoryRiemannRoch,AnalyticToddBeiPiazza,LottHilbertconplex} where corresponding algebraic and analytic Todd classes have been studied.

In Subsection \ref{subsection_almost_complex_spin_c_Dirac}, we prove a local to global formula (similar to Theorem \ref{Corollary_Polynomial_Lefschetz_fixed_point_theorem_introduction_version}) for spin$^\mathbb{C}$ Dirac operators when the spin$^\mathbb{C}$ structures induced on the tangent cones of fixed points have compatible complex structures.
We emphasize that many equivariant invariants in the complex setting can be generalized to this setting.

We show how in the complex setting, the holomorphic Lefschetz numbers can be used to compute Lefschetz versions of spin numbers, Hirzebruch's $\chi_y$ genera and signature numbers.
We study important dualities for these invariants in the $L^2$ setting, and compare with versions in other cohomology theories \cite{donten2018equivariant,MaximSaitoJorgHodgemodules2011,Banaglbook07}, for some of which we have not seen local Lefschetz formulas in the literature. We also study equivariant self-dual and anti self-dual indices on four dimensional spaces, which are key to understanding invariants that arise in gauge theory and certain supersymmetric quantum field theories such as in \cite{pestun2012localization,festuccia2020twisting}. 
In the smooth setting, there is a relation between the equivariant self-dual and anti self-dual indices of a torus action $(\lambda,\mu) \in (\mathbb{C}^*)^2$ at a fixed point $a$, given by
\begin{equation}
\label{equation_duality_SD_ASD_intro}
    ind_{ASD}(U_a,\lambda,\mu)=ind_{SD}(U_a,\lambda,\overline{\mu})=ind_{SD}(U_a,\overline{\lambda},\mu),
\end{equation}
and we refer to Subsection \ref{subsection_asd_sd_complexes_localization} for more details. We identify this as the source of the relation
\begin{equation}
\label{equation_duality_partition_function_SD_ASD_intro}
    Z^{\text{anti-inst}}_{\epsilon_1,\epsilon_2}(a,\overline{q})=Z^{\text{inst}}_{\epsilon_1,-\epsilon_2}(a,\overline{q})=Z^{\text{inst}}_{-\epsilon_1,\epsilon_2}(a,\overline{q})
\end{equation}
between the instanton and anti-instanton partition functions of certain supersymmetric Yang Mills theories (see equation (109) of \cite{festuccia2020twisting}), related to physical interpretations of the instanton partition functions as quantum states.
We study this relation in equation \ref{equation_duality_SD_ASD_intro} on singular spaces, showing that it continues to hold for certain spaces while failing to go through on others, with explicit examples.

In Subsection \ref{subsection_Morse_Hirzebruch_cohomological}, we study how some well known properties of spaces and invariants that follow from the existence of group actions in the smooth case, generalize to our singular setting. We see that local cohomology plays an important role in vanishing theorems and counting formulas including $L^2$ versions of the Arnold-Floer theorem in restricted settings, subtleties which are hidden in the smooth setting.

We devote Subsection \ref{subsection_Computations_for_singular_examples} to explicit computations of the equivariant invariants we defined on singular spaces, primarily in the case of toric varieties where we explore how various intricacies of these spaces lead to different features in the global and local Lefschetz numbers of these spaces, both for our $L^2$ versions as well as the algebraic versions corresponding to those of \cite{baum1979lefschetz}.

\subsection{History, motivations and related work.}
\label{subsection_History_motivation}

In 1926, Solomon Lefschetz used methods of topology to establish his famous fixed point theorem in \cite{lefschetz1926intersections}, which under suitable circumstances expresses the number of fixed points of a self map from a closed \textit{topological} manifold $X$ to itself in terms of the transformation induced by the map in the cohomology of $X$. It associates to each such self map what is now called a Lefschetz number, whose non-vanishing implies the existence of a fixed point.
An algebraic version of the theorem was first conjectured by Shimura and proven by Eichler for algebraic curves.

While Verdier proved his now famous version in \'etale cohomology \cite{verdier1967lefschetz},
Atiyah and Bott generalized the result to arbitrary elliptic complexes on smooth manifolds \cite{AtiyahBottBull66,AtiyahBott1} using analytic methods (see \cite{tu2015genesis} for historical anecdotes).
When applied to the de Rham complex, their theorem not only recovers Lefschetz's original result on \textit{smooth} manifolds but also provides more general formulas for local Lefschetz numbers at smooth fixed points.
A number of specialized local Lefschetz formulas were derived such as for Spin Dirac and Signature operators in \cite{AtiyahBott2} and were used to study interactions of number theory, group theory, geometry and topology. This heralded the origin of equivariant index theory and was followed by significant advances of \textit{localization formulas} including the Atiyah-Bott-Berline-Vergne localization theorem \cite[\S 7]{berline2003heat} that have become essential in various areas of mathematics and physics, gauge theory, enumerative geometry and integrable quantum field theory (see for instance \cite{pestun2017localization}) to name a few.

While index theory and intersection theory give formulas for various invariants important in geometry and topology, equivariant versions of these quantities are typically more easily computable in the presence of symmetries. 
In \cite[\S 5]{bott1967vector}, Bott explains how the holomorphic Lefschetz fixed point theorem can be used to prove the \textit{Bott residue formula} and describes how this observation gave the impetus to prove the more general result and explains the underlying principle as to why more general invariants can be expressed in terms of holomorphic Lefschetz numbers. This principle based on the topological twisting is widely used by experts in mathematical physics (see, e.g. \cite{witten1983fermion,pestun2012localization,pestun2017localization}).

Stratified pseudomanifolds are a natural generalization of smooth manifolds that can be loosely described as the metric completion of a smooth open manifold $\widehat{X}^{reg}$ called the regular part, with a metric that is degenerate on the completion $\widehat{X}$. The structure of Thom-Mather stratifications is a strengthening of this notion, that arises in many natural spaces such as vanishing sets of polynomials, orbits and quotient spaces of group actions, and moduli spaces.

Goresky and MacPherson introduced intersection (co)homology on compact stratified pseudomanifolds and extended intersection theory on these singular spaces \cite{goresky1980intersection}, for choices of \textit{perversity functions} that dictated how cycles of singular homology were allowed to intersect at singularities. For stratified pseudomanifolds satisfying the Witt condition, intersection homology with \textit{middle perversity} satisfies Poincar\'e duality and they generalized the Lefschetz fixed point theorem in \cite{Goresky_1985_Lefschetz}.
Their intersection Lefschetz numbers are equal to the Lefschetz numbers of Atiyah and Bott in the case of the de Rham complex on smooth compact manifolds. Crucially, they defined local Lefschetz numbers as alternating traces of induced maps on local cohomology groups, studying this further in \cite{goresky1993local}. They also generalized the Morse inequalities using similar ideas in \cite[\S 6.12]{goresky1988stratified}.

Simultaneously, Cheeger started studying $L^2$ cohomology and calculus on spaces with degenerate metrics such as conic and horn metrics \cite{cheeger1980hodge,cheeger1983spectral}, laying the foundations for the study of index theory on such spaces.
Together Cheeger, Goresky and MacPherson showed that for iterated conic metrics (which we call wedge metrics in this article) $L^2$ de Rham cohomology is isomorphic to middle perversity intersection homology \cite{cheeger1982l2}, and they conjectured that the $L^2$ cohomology of the Fubini-Study metric of a projective variety is isomorphic to intersection homology with middle perversity.
One aim was to find Hodge structures on intersection cohomology using the Fubini-Study metric. Since it is known that the $L^2$ de Rham cohomology of wedge metrics is isomorphic to the middle perversity intersection homology, this is equivalent to the statement that the $L^2$ de Rham cohomology of the Fubini Study metric is the same as that of a wedge metric. This was proven for the case of complex curves and surfaces and for the case of isolated singularities (see \cite[\S 2.4]{cruz2020examples} for history, and progress on the conjecture). However, evidence against the underlying philosophy has also been presented (see \cite{Kollarlinkssingularities}). 

The groundbreaking work of Cheeger using calculus on singular spaces was followed by further work on index theory on singular spaces including Cheeger's joint work with Bismut. We refer to the introduction of \cite{Albin_2017_index} for a review of the relevant literature. In that article, a families index theorem for Dirac-type operators on stratified pseudomanifolds with wedge metrics is proven. Many results including heat kernel constructions presented in that work are utilized in our research.

There are versions of Lefschetz fixed point theorems in $L^2$ cohomology for spaces with isolated singularities in the literature such as \cite{nazaikinskii2005elliptic,Bei_2012_L2atiyahbottlefschetz}. In the latter article, more easily computable formulas for the case of the $L^2$ de Rham complex with conic metrics were proven by appealing to the isomorphism with intersection homology and using local intersection homology groups.
The same isomorphism was used in \cite{ludwig2013analytic,ludwig2017comparison,ludwig2020extension} to prove Morse inequalities and Cheeger-M\"uller theorems on spaces with isolated conic singularities using Witten deformation. In studies of such operators the absolute and relative cohomology groups of the de Rham operator have been used earlier to describe cohomology groups.
While the Morse inequalities for $L^2$ de Rham cohomology were proven in \cite{Jesus2018Wittens,Jesus2018Wittensgeneral} for more general metrics and choices of domains, the categorification of the Morse polynomial leads to more results, which have not been studied even in the smooth setting. An instanton complex was constructed by Ludwig in \cite{ludwig2017comparison} for slightly restricted notions of Morse functions where the singular critical points are either completely attracting or completely expanding.
The stratified Morse functions we consider are equivalent to those in \cite{Jesus2018Wittensgeneral} a natural generalization of the notion of Morse functions on smooth manifolds, and we refer the reader to Definition \ref{definition_stratified_Morse_function} for a precise definition. 
We work out the details of the B\"ochner formulas, which show that properties of the gradient flow near the strata need to be studied carefully to generalize the construction of the complex in the non-essentially self-adjoint situation. We extend these techniques to the Dolbeault complex in \cite{jayasinghe2024holomorphicwitteninstantoncomplexes} where the results are new.

In \cite{Bei_2012_L2atiyahbottlefschetz}, Bei framed global Lefschetz numbers for elliptic complexes as supertraces in global cohomology groups of Hilbert complexes. Our results in Section \ref{subsection_abstract_lefschetz_supertraces} are generalizations of those of \cite{Bei_2012_L2atiyahbottlefschetz} for Hilbert complexes, and in Subsection \ref{subsection_cohomological_formulae} we study Hilbert complexes associated to isolated fixed points of self maps on geometric spaces. While Cheeger used functional calculus on the infinite model cone in his proofs \cite{cheeger1983spectral}, and to compute the contributions to the index of operators from cone points using eta invariants, we use functional calculus on the truncated tangent cone, associated to Hilbert complexes on stratified pseudomanifolds with (stratified) boundary to compute our local Lefschetz numbers, defining them as (renormalized) supertraces over the cohomology of these \textit{local complexes} in the case of the de Rham and Dolbeault complexes. 
While clearly related to the \textit{spectral asymmetry} captured by (generalized) eta invariants, this formulation puts local and global Lefschetz numbers in a somewhat equal footing and allows the use of powerful tools such as dualities to better understand richer structures of local Lefschetz numbers.
In the case where there is only a spin$^{\mathbb{C}}$ structure, where near the boundary of a fundamental neighbourhood of a fixed point it is defined by an almost complex structure, the local cohomology can no longer be expressed as an element in the kernel of an operator $P$ up to elements in the image of $P$ where $D=P+P^*$ as in the de Rham and Dolbeault cases. In fact the spin$^{\mathbb{C}}$ Dirac operator and its square only preserves odd and even degrees in this case.

This approach gives analytic replacements for the local intersection homology groups we mentioned above in the setting of $L^2$ de Rham cohomology, and we recover versions of Goresky and MacPherson's Lefschetz fixed point theorem as well as their Morse inequalities.
Goresky and MacPherson prove their result in middle perversity intersection cohomology theory in order to use Poincar\'e duality. As we discussed in the earlier subsection, we use the Hilsum-Skandalis replacement to study de Rham Lefschetz numbers on a very general class of self maps on stratified spaces, which restricted to Thom-Mather stratified spaces are in general different from the placid maps studied by Goresky and MacPherson. This is essentially because $L^2$ cohomology only depends on the metric on $X^{reg}$ and we only require \textit{good} properties of maps restricted to the regular part of $X$. Moreover we prove the Morse inequalities for stratified Morse functions, the gradient flow of which do not necessarily preserve strata.

In \cite[\S 8]{AtiyahBott1}, Atiyah and Bott state the following.
\textit{``It is well known in various homology theories in topology and algebraic geometry that a Lefschetz fixed point theory is a formal consequence of three things. K\"unneth formula (giving the cohomology of a product), Poincar\'e duality (the isomorphism between homology and cohomology), compatibility between intersection of cycles and cup product of cohomology classes."}
A similar maxim was proposed by Verdier in \cite[\S 3]{verdier1967lefschetz}, whose framework for Lefschetz trace formulas in certain categories with six functor formalisms is now widely used (see, e.g., 
\cite{Zheng2015VerdierLefschetzDeligneMumfordSixfunctor}).

In our analytic framework, the role of homology is played by the cohomology of the adjoint complex of the Hilbert complex, and the duality is between the cohomology of the complex and that of the adjoint complex. Since we restrict our study to complexes associated to self-adjoint Dirac operators, the compatibility is built in by the definition of the adjoint complex, both for local and global complexes. K\"unneth formulas continue to play a key role in global as well as local cohomology, in particular in the proof of formulas such as that in equation \ref{intro_product_equation_refering} which show how to piece together boundary conditions for attracting and expanding \textit{factors} of self maps.

These ideas generalize to stratified spaces with other degenerate metrics such as horn metrics and warped product metrics of the form studied in \cite{cruz2020examples,Jesus2018Wittensgeneral}, and this is the subject of upcoming work. 

In the case of the Dolbeault complex, the local cohomology groups are infinite dimensional and the supertraces need to be renormalized for general geometric endomorphisms. After developing the basics of the analytic theory, we were able compute the local Lefschetz numbers in many examples, which led to the discovery of close ties to the work of \cite{Baumformula81}. In the formulas in Section 3 of that article, Baum explains how the Lefschetz-Riemann-Roch formulas in \cite{baum1979lefschetz} for complex automorphisms $f$ (such that $f^n=Id$) acting on the structure sheaf of projective varieties can be realized as sums of regularized traces at the local ring of regular functions at isolated fixed points. In this light, our local cohomology groups are analytic replacements for the local ring, at least in the context of localization that we study here.

There are different versions of renormalized traces using heat kernel methods used in \cite{donnelly1986fixed,kytmanov2004holomorphic} to define global Lefschetz numbers on certain smooth non-compact complex manifolds and compact complex manifolds with boundary. We study connections between these and our formalism in Subsection \ref{subsection_Lefschetz_L2_cohomology_development}.
The local cohomology groups correspond to the null spaces of certain $\overline{\partial}$-Neumann problems for the Dolbeault Laplacian on model cones, and thus relate the holomorphic Lefschetz fixed point theorem to studies of the $\overline{\partial}$-Neumann problem (see the introductions in \cite{RuppenthalL2dbarsingular2019,RuppenthalSerre2018,AnalyticToddBeiPiazza} for related work). 

The Lefschetz-Riemann-Roch theorem of \cite{baum1979lefschetz} for coherent sheaves on schemes has been generalized to localization formulas for equivariant Chow groups (\cite{equivariant_intersection_edidin,localization_algebraic_Graham_Edidin}) and on stacks and quotient spaces.
A common theme in those results is to embed the singular space in a smooth one to find suitable sheaves and complexes to prove localization theorems. While many important stacks are stratified pseudomanifolds, the nature of stacks is that they are studied together with embeddings in smooth structures and with natural group actions, a setting which Thom-Mather stratifications were initially developed to capture.
The abundance of such localization formulas for singular algebraic varities using algebraic methods is in sharp contrast to the lack of analytic versions. There is currently a fairly good understanding of domains of Dirac-type operators on certain singular spaces, as well as associated characteristic classes and K-theory, and there are many ideas that can be borrowed from the algebraic versions in order to develop the theory on the analytic side. By using equivalences between the theories, one can use generalizations such as the Bott residue formula for singular algebraic varieties in \cite{localization_algebraic_Graham_Edidin} (see Remark \ref{remark_Bott_residue_quadric}) to compute integrals of interesting quantities.

Our methods do not need embeddings of the singular spaces into smooth spaces and in the case where such embeddings exist, the differences of our formulas as opposed to those of Baum in \cite{Baumformula81} reflect the differences of the intrinsic vs. extrinsic approaches. For instance, the local ring at a singular point on a non-normal affine variety does not include some of the holomorphic functions on the singular space, just those that extend to holomorphic functions on the smooth space it is embedded in. Baum shows that the local holomorphic Lefschetz numbers for the structure sheaf are renormalized traces over the local ring, thus giving a different result from our renormalized traces in local cohomology, since we include all $L^2$ bounded meromorphic functions that are in the VAPS domain.
We discuss such themes in Subsection \ref{subsection_classes_of_singularities} in more detail and show the difference of Lefschetz numbers in explicit examples in Subsection \ref{subsection_Computations_for_singular_examples}. 

While our results hold for spaces that are not algebraic, algebraic geometry provides many tools that are useful to investigate the various features of $L^2$ cohomology of Hilbert complexes on different types of spaces before generalizing beyond. This was very much the modus operandi of Atiyah, Bott, Hirzebruch and other pioneers of index theory who built on ideas of Grothendieck.
In Subsection \ref{subsection_Computations_for_singular_examples}, we compute holomorphic Lefschetz numbers in different spaces including non-normal algebraic varieties, in cases where they are normal pseudomanifolds and non-normal pseudomanifolds (see Definition \ref{definition_normal_pseudomanifold}).
In particular we demonstrate in an explicit example (see Example \ref{subsection_nonnormal_examples}) that our framework can be used to compute local Lefschetz numbers for the Dolbeault complex, corresponding to the Lefschetz fixed point theorem of \cite{Bei_2012_L2atiyahbottlefschetz} with the minimal and maximal domains (as opposed to the VAPS domain that we focus on) in the case of isolated singularities. Our work in \cite{jayasinghe2024holomorphicwitteninstantoncomplexes} treats such domains directly.

Using the holomorphic Lefschetz numbers we can compute Lefschetz versions of spin numbers, Hirzebruch $\chi_y$ genera, signature numbers, and self-dual and anti-self-dual indices. 
In the smooth setting, it is well known that on complex manifolds, the Lefschetz Hirzebruch $\chi_y$ invariant for the trivial bundle evaluated at $y=-1$ is the Lefschetz number of the de Rham complex, at $y=0$ is the holomorphic Lefschetz number and at $y=1$ is the Lefschetz signature, and there are many dualities of the $\chi_y$ invariant which arise as consequences of Poincar\'e duality and in the K\"ahler case with Serre duality.
We show that these properties go through in our singular setting with $L^2$ cohomology. This is in contrast to the Baum-Fulton-MacPherson version of these invariants for singular projective algebraic varieties (see Remark \ref{Remark_our_dualities_are_better}). Our proofs of such properties follow from dualities of local cohomology groups, giving a more conceptual understanding of such properties.
We briefly explain some of the connections here, postponing a deeper study to Subsection \ref{subsection_Local_Lefschetz_Dolbeault_main}. 

Analytic Todd classes corresponding to more general domains have been studied in  \cite{AnalyticToddBeiPiazza}, and we discuss how our results are Lefschetz versions of these, similar to how the results in \cite{baum1979lefschetz} are Lefschetz versions of \cite{baumfultonmacKtheoryRiemannRoch}.
Algebraic versions of Hirzebruch's $\chi_y$ genera and their Lefschetz versions when evaluated at $y=0$ give the Baum-Fulton-Quart formulas on singular varieties and have been studied for instance in \cite{donten2018equivariant}. For $y=0$, these invariants are related to Chern-Schwarz-MacPherson classes and in the case of $y=1$, to versions of the signature and $L$ classes of Goresky and MacPherson, and Thom-Milnor $L$ classes, under various assumptions on the types of singularities. We refer the reader to \cite{cappell2023equivariant,GeneralatticeCappellShaneson94,Brasselet2010,MaximJorgCharacteristicsingulartoric2015,Albin_signature,Banaglbook07}. 
Lefschetz Hirzebruch $\chi_y$ genera are related to Molien series and equivariant eta invariants \cite{donten2018equivariant,Degeratuthesis}, and are key in defining elliptic genera (see \cite{hirzebruch1992manifolds,borisov2003elliptic,donten2018equivariant}) and the dualities we prove for the $L^2$ versions are crucial for many important properties of these more general invariants. We provide the basic tools to start the study of $L^2$ versions of these invariants.

Under certain cohomological conditions on smooth manifolds, Hirzebruch's $\chi_y$ genus is related to Morse polynomials and Poincar\'e polynomials and we show how some of these relationships continue to hold in our singular setting in Subsection \ref{subsection_Morse_Hirzebruch_cohomological}.
Many of the easiest corollaries of Morse inequalities and Lefschetz fixed point formulas are for spaces where the global $L^2$ Betti numbers vanish in odd degree, and local cohomology groups (being simple in the smooth setting) do not show up in statements of such theorems. In Subsection \ref{subsection_Morse_Hirzebruch_cohomological} we investigate how such results extend to the singular setting with additional conditions on the local cohomology groups. In particular, we observe how the statement of the Arnold-Floer theorem needs to be modified to account for multiplicities of fixed points, for autonomous Hamiltonian systems, or under the assumption that odd dimensional $L^2$ Betti-numbers are zero.

Self-dual and anti-self-dual complexes and their local equivariant indices play an important role in understanding the Nekrasov partition function, Donaldson-Witten theory and Seiberg Witten theory. Nekrasov and Okounkov proved that the Seiberg Witten prepotential can be obtained as a limit of the Nekrasov partition function (see \cite{nekrasov2003seiberg,NekrasovABCDinsta,NekrasovOkounkovSW2006}).
Generalized versions of this cohomologically twisted field theory are now widely studied, for instance in \cite{pestun2012localization,festuccia2020twisting,festuccia2020transversally} and contact 5 manifolds studied in \cite{baraglia2016moduli,hekmati2023equivariant}. We show how the Lefschetz Hirzebruch $\chi_y$ invariant can be used to compute $L^2$ versions of the self-dual and anti-self-dual indices, exploring how the duality in equation \eqref{equation_duality_SD_ASD_intro} which is important for generalized instanton counting formulae, fails to hold on certain singular spaces. In particular we show that some algorithms used to compute such instanton partition functions and equivariant indices on smooth spaces fail to go through in the singular case, but that the equivariant $\chi_y$ invariants can be used to compute them.

A common theme in the articles on cohomologically twisted field theories that we cited above, is to study the local Dolbeault complex using the theory of transversally elliptic operators (see \cite{Atiyahellipticopsandcompactgroups74}), since the setting comes with actions of compact connected Lie groups on spaces. This is widely used in various settings of mathematics and physics (c.f., Remark \ref{remark_yau_nekrasov_conifold}). It is natural to think of the holomorphic Lefschetz numbers as character formulas in that setting, and certain choices of Laurent expansions of local Lefschetz numbers become important (see Proposition 6.2 of \cite{Atiyahellipticopsandcompactgroups74}).
We explore how these choices are related to choices of local cohomology with different boundary conditions (see Remark \ref{remark_dual_complexes_and_laurent_expansions}), in particular showing how the adjoint complexes play a key role in the singular setting.

This relation of Laurent expansions with choices of local cohomology groups is implicit in the work of \cite{witten1984holomorphic} in the smooth setting, and corresponds to different choices of local holomorphic Morse polynomials corresponding to action chambers of a Lie group action in \cite{wu1996equivariant}. In his celebrated article \cite{witten1982supersymmetry} Witten explored the relationship of supersymmetry and Morse theory (c.f. \cite{witten1982constraints} and ``Lessons from Raoul Bott" in \cite{Yaubookfoundersindex}), exploring the possibility of Morse inequalities for other elliptic complexes such as the spin Dirac, Rarita-Schwinger and signature complexes, working out details in \cite{witten1983fermion}. In \cite{witten1984holomorphic} he formulated his \textit{equivariant} holomorphic Morse inequalities for effective Hamiltonian circle actions on compact K\"ahler manifolds (to be distinguished from Demailly's \textit{asymptotic} holomorphic Morse inequalities \cite{demailly1991holomorphic}).
A heat kernel proof was given for this theorem, for effective K\"ahler Hamiltonian group actions of compact connected Lie groups in the first two of the series of papers  \cite{mathai1997equivariant,wu1996equivariant,wu1998equivariant,wu2003instanton} by Wu and his collaborators. 
The third paper with Zhang uses methods similar to those in \cite{Zhanglectures} which we have generalized in this article for the case of the de Rham complex (see Remark 3 \cite[\S 5.7]{Zhanglectures}).

Witten's starting point is the holomorphic Lefschetz fixed point theorem, and a close comparison of his methods and ours shows how to formulate equivariant holomorphic Morse inequalities in our singular setting, in many aspects similar to the de Rham version in Theorem \ref{Theorem_strong_Morse_de_Rham_intro} and we have proven these in  \cite{jayasinghe2024holomorphicwitteninstantoncomplexes}, using it to prove results such as rigidity of various invariants.
Our framework of analytic local cohomology groups allows easy conjectural formulations of more results connected to Witten deformation such as those for the stratified Morse-Bott, Morse-Novikov inequalities and counting formulae for the Kervaire semi-characteristic and the mod 2 Signature on certain singular spaces. In the smooth setting, the above mentioned results have been proven using methods of Witten deformation (see \cite{ShubinNovikov,Zhanglectures}) and it is easy to see that most of the techniques generalize. We will study some of these aspects in upcoming work.

These methods have also been widely used to study the equivariant index of spin$^{\mathbb{C}}$ Dirac operator on symplectic manifolds in the presence of Hamiltonian actions to study quantization and its relations with symplectic reduction (see, e.g. \cite{ZhanganalyticQR98}, and \cite{meinrenken1998symplectic,meinrenken1999singular} for some work in singular settings), hence the importance of generalizing the techniques beyond the K\"ahler case. In the smooth setting (and even orbifolds after taking lifts), the fibers of the normal bundle to any fixed point set of a group action copies of $\mathbb{R}^n$ and have K\"ahler structures, but in the singular setting this differs even for isolated fixed points.
Since our Lefschetz fixed point theorem covers the case of twisted spin$^{\mathbb{C}}$ Dirac operators it can be used to investigate such questions on stratified pseudomanifolds.

\begin{justify}
 \textbf{Acknowledgements:} 
Many thanks are due to my advisor Pierre Albin for suggesting this problem and his patient guidance and support with ideas, directions and technicalities.
A reading group on the Cheeger-M\"uller theorem organized by him, as well as conversations on Witten deformation with Kesav Krishnan proved to be useful. Another reading group organized by Jesse Huang and James Pascaleff on perverse sheaves was helpful. I am grateful for discussions with Gabriele La Nave and Hadrian Quan on symplectic and almost complex structures on singular spaces, and other aspects of this work. 
I thank Nachiketa Adhikari for many useful conversations on this work, including but not limited to algebraic, toric and symplectic geometry.
I thank interesting discussions with, and comments on this work from Katrina Morgan, Jacob Shapiro, Karthik Vasu, Mengxuan Yang and Xinran Yu.
I thank Joey Palmer and Susan Tolman for fostering my interest on connections to toric geometry.
I was partially supported by Pierre Albin's NSF grant DMS-1711325. 
\end{justify}

\section{Stratified pseudomanifolds and Dirac operators}
\label{section_stratified_space_basics}

In this section we introduce manifolds with corners with iterated fibration structures which resolve Thom Mather stratified pseudomanifolds.
Stratified pseudomanifolds of dimension $n$ are topological spaces $\widehat{X}$ which have a dense open set called the \textit{\textbf{regular set}} $\widehat{X}^{reg}$ which is a smooth open manifold of dimension $n$. The \textit{\textbf{singular set}} $\widehat{X}^{sing}:=\widehat{X} \setminus \widehat{X}^{reg}$ is a union of smooth manifolds of different dimensions, all of which have dimension less than $n-2$. 
In this article, we look at \textit{wedge metrics}, which can be described roughly as iterated conic metrics, which are bundle metrics on the tangent space of $\widehat{X}^{reg}$ which degenerate near the singularities in a specific manner which we shall describe in the first subsection.

We will then define wedge Clifford modules and wedge Dirac operators on these spaces. These are formally self-adjoint operators and we will describe how to choose self adjoint extensions for these operators. In the last subsection we will define certain local structures and fix some conventions.

We follow the exposition in the first 2 sections of \cite{Albin_2017_index} for the most part. In that article, families of Dirac-type operators on families of stratified pseudomanifolds which fiber over a base $B$ are introduced. Our setting corresponds to the case where $B$ is a point, and we simplify the notation accordingly.

\subsection{Stratified pseudomanifolds}

All topological spaces we consider will be Hausdorff, locally compact topological spaces with a countable basis for its topology. Recall that a subset $W$ of a topological space $V$ is locally closed if every point $a \in W$ has a neighborhood $\mathcal{U}$ in $V$ such that $W \cap \mathcal{U}$ is closed in $\mathcal{U}$. A collection ${S}$ of subsets of $V$ is locally finite if every point $v \in V$ has a neighborhood that intersects only finitely many sets in ${S}$.

A Thom-Mather stratified pseudomanifold is a topologically stratified space (see Definition 4.11 of \cite{Kirwan&woolf_2006_book}) with additional conditions. The following is Definition 2.1 of \cite{Albin_signature}.

\begin{definition}
    \label{Thom-Mather stratified space}    
A Thom-Mather stratified space $\widehat{X}$ is a metrizable, locally compact, second countable space which admits a locally finite decomposition into a union of locally closed strata $\mathcal{S}(\widehat{X})=\left\{Y_{\alpha}\right\}$, where each $Y_{\alpha}$ is a smooth manifold, with dimension depending on the index $\alpha$. We assume the following:
\begin{enumerate}
    \item If $Y_{\alpha}, Y_{\beta} \in \mathcal{S}(X)$ and $Y_{\alpha} \cap \overline{Y_{\beta}} \neq \emptyset$, then $Y_{\alpha} \subset \overline{Y_{\beta}}$.

    \item Each stratum $Y$ is endowed with a set of `control data' $T_{Y}, \pi_{Y}$ and $\rho_{Y}$; here $T_{Y}$ is a neighbourhood of $Y$ in $X$ which retracts onto $Y, \pi_{Y}: T_{Y} \longrightarrow Y$ is a fixed continuous retraction and $\rho_{Y}: T_{Y} \rightarrow[0,2)$ is a proper `radial function' in this tubular neighbourhood such that $\rho_{Y}^{-1}(0)=Y$. Furthermore, we require that if $\widetilde{Y} \in \mathcal{S}(X)$ and $\widetilde{Y} \cap T_{Y} \neq \emptyset$, then
\begin{equation}    
    \left(\pi_{Y}, \rho_{Y}\right): T_{Y} \cap \widetilde{Y} \longrightarrow Y \times[0,2)
\end{equation}
is a proper differentiable submersion.
\item If $W, Y, \widetilde{Y} \in \mathcal{S}(X)$, and if $p \in T_{Y} \cap T_{\widetilde{Y}} \cap W$ and $\pi_{\widetilde{Y}}(p) \in T_{Y} \cap \widetilde{Y}$, then $\pi_{Y}\left(\pi_{\widetilde{Y}}(p)\right)=\pi_{Y}(p)$ and $\rho_{Y}\left(\pi_{\widetilde{Y}}(p)\right)=\rho_{Y}(p)$.

\item If $Y, \widetilde{Y} \in \mathcal{S}(X)$, then
$$
\begin{aligned}
Y \cap \overline{\widetilde{Y}} \neq \emptyset & \Leftrightarrow T_{Y} \cap \widetilde{Y} \neq \emptyset, \\
T_{Y} \cap T_{\widetilde{Y}} \neq \emptyset & \Leftrightarrow Y \subset \overline{\widetilde{Y}}, Y=\widetilde{Y} \quad \text {or } \widetilde{Y} \subset \overline{Y} .
\end{aligned}
$$

\item For each $Y \in \mathcal{S}(X)$, the restriction $\pi_{Y}: T_{Y} \rightarrow Y$ is a locally trivial fibration with fibre the cone $C\left(Z_{Y}\right)$ over some other stratified space $Z_{Y}$ (called the \textbf{\textit{link}} over $Y)$, with atlas $\mathcal{U}_{Y}=\{(\phi, \mathcal{U})\}$ where each $\phi$ is a trivialization 
\begin{equation}
\label{equation_chartlike_map}
\pi_{Y}^{-1}(\mathcal{U}) \rightarrow \mathcal{U} \times C\left(Z_{Y}\right),    
\end{equation}
and the transition functions are stratified isomorphisms (in the sense of Definition \ref{smoothly stratified} below) of $C\left(Z_{Y}\right)$ which preserve the rays of each conic fibre as well as the radial variable $\rho_{Y}$ itself, hence are suspensions of isomorphisms of each link $Z_{Y}$ which vary smoothly with the variable $y \in \mathcal{U}$.

If in addition we let $\widehat{X}_{j}$ be the union of all strata of dimensions less than or equal to $j$, and require that 

\item $\widehat{X}_{n-1}=\widehat{X}_{n-2}$ and $\widehat{X} \backslash \widehat{X}_{n-2}$ is dense in $\widehat{X}$, 
then we say that $\widehat{X}$ is a stratified pseudomanifold.
\end{enumerate}
\end{definition}

We remark that this ensures that $\widehat{X}_j \backslash \widehat{X}_{j-1}$ is a smooth manifold of dimension $j$, and the connected components of this are called the strata of \textit{\textbf{depth}} $j$. 
This definition shows that it is natural to study the topology of these spaces using iterated conic metrics, which we shall call wedge metrics and introduce more formally in this section.
The following definition is based on the first proposition of \cite[\S 4.1]{goresky1980intersection}.
\begin{definition}
\label{definition_normal_pseudomanifold}
    Given a stratified pseudomanifold $\widehat{X}$ of dimension $n$, if the link at each $x \in \widehat{X}_{n-2}$ is connected, it is called a \textit{\textbf{normal pseudomanifold}}.
\end{definition}
Goresky and MacPherson explore this definition in \cite[\S 4]{goresky1980intersection} to which we refer the reader. The metric completion of $\widehat{X}^{reg}$ with respect to a wedge metric corresponds to the topological normalization of non-normal pseudomanifolds. This topological normalization is unique for a given pseudomanifold.
There is a functorial equivalence between Thom-Mather stratified spaces and manifolds with corners and iterated fibration structures (see Proposition 2.5 of \cite{Albin_signature}, Theorem 6.3 of \cite{Albin_hodge_theory_cheeger_spaces}). 


\subsubsection{Manifolds with corners with iterated fibration structures}

In section 1 of \cite{Albin_2017_index}, there is a detailed description of iterated fibration structures on manifolds with corners. We explain some of these structures from said article which we will use in technical arguments and refer the reader to the source for more details. In \cite{kottke2022products}, these are referred to as manifolds with fibered corners that are interior maximal.

An n-dimensional manifold with corners $X$ is an $n$-dimensional topological manifold with boundary, with a smooth atlas modeled on $(\mathbb{R}^+)^n$ whose boundary hypersurfaces are embedded. We denote the set of boundary hypersurfaces of $X$ by $\mathcal{M}_1(X)$. A collective boundary hypersurface refers to a finite union of non-intersecting boundary hypersurfaces.

\begin{definition}
\label{iterated_fibration_structure}
An \textit{\textbf{ iterated fibration structure}} on a manifold with corners $X$ consists of a collection of fiber bundles
\begin{center}
    $Z_Y -\mathfrak{B}_Y \xrightarrow{\phi_Y} Y$
\end{center}
where $\mathfrak{B}_Y$ is a collective boundary hypersurface of $X$ with base and fiber manifolds with corners such that:

\begin{enumerate}
    \item Each boundary hypersurface of $X$ occurs in exactly one collective boundary hypersurface $\mathfrak{B}_Y$.
    \item If $\mathfrak{B}_Y$ and $\mathfrak{B}_{\widetilde{Y}}$ intersect, then dim $Y \neq$ dim $\widetilde{Y}$, and we write $Y < \widetilde{Y}$ if dim $Y < \text{dim} \widetilde{Y}$.
    \item If $Y < \widetilde{Y}$, then $\widetilde{Y}$ has a collective boundary hypersurface $\mathfrak{B}_{Y\widetilde{Y}}$ participating in a fiber bundle $\phi_{Y\widetilde{Y}} : \mathfrak{B}_{Y\widetilde{Y}} \rightarrow Y$ such that the diagram 
    
\[\begin{tikzcd}
	{\mathfrak{B}_Y \cap \mathfrak{B}_{\widetilde{Y}}} && {\mathfrak{B}_{Y\widetilde{Y}} \subseteq \widetilde{Y}} \\
	& Y
	\arrow["{\phi_{\widetilde{Y}}}", from=1-1, to=1-3]
	\arrow["{\phi_Y}"', from=1-1, to=2-2]
	\arrow["{\phi_{Y\widetilde{Y}}}", from=1-3, to=2-2]
\end{tikzcd}\]
commutes.
\end{enumerate}
\end{definition}

The base can be assumed to be connected but the fibers are in general disconnected.
As we mentioned above, there is an equivalence between Thom-Mather stratified spaces and manifolds with corners with iterated fibration structures. 

If we view a cone over a link $Z$ as the quotient space of $[0,1]_x \times Z$ under the identification where the points of the link at $\{x=0\} \times Z$ are identified, then the quotient map is a blow-down map.
More generally, there is an inductive desingularization procedure which replaces the Thom-Mather stratified space with a manifold with corners with iterated fibration structures.
This corresponds to a \textit{\textbf{blow-down}} map $\beta : X \rightarrow \widehat{X}$, which satisfies properties given in Proposition 2.5 of \cite{Albin_signature}.
See Remark 3.3 of \cite{Albin_2017_index} for an instructive toy example.

Under this equivalence, the bases of the boundary fibrations correspond to the different strata, which we shall denote by
\begin{equation}
    \mathcal{S}(X) = \{Y: Y \textit{ is the base of a boundary fibration of X} \}.
\end{equation}
The bases and fibers of the boundary fiber bundles are manifolds with corners with iterated fibration structures (see for instance Lemma 3.4 of \cite{albin2010resolution}). The condition dim $Z_Y > 0$ for all $Y$ corresponds to the category of pseudomanifolds within the larger category of stratified spaces. The partial order on $\mathcal{S}(X)$ gives us a notion of depth
\begin{center}
$ depth_X(Y) = max \{ n \in \mathbb{N}_0 : \exists Y_i \in \mathcal{S}(X)$ s.t. $Y=Y_0 <Y_1 < ... <Y_n \}$.
\end{center}
The depth of $X$ is then the maximum of the integers $depth_X(Y)$ over $Y \in \mathcal{S}(X)$. This corresponds to the depth of the associated stratified space from Definition  \ref{Thom-Mather stratified space}.

We now introduce some auxiliary structures associated to manifolds with corners with iterated fibration structures.
If $H$ is a boundary hypersurface of $X$, then because it is assumed to be embedded, there is a smooth non-negative function $\rho_H$ such that 
$\rho_{H}^{-1}(0)=H$ and $d\rho_H$  does not vanish at any point on $H$. We call any such function a \textbf{\textit{boundary defining function for $H$}}.
For each $Y \in \mathcal{S}(X)$, we denote a \textbf{\textit{collective boundary defining function}} by
\begin{center}
    $\rho_Y = \prod_{H \in \mathfrak{B}_Y} \rho_H$,
    and by
    $\rho_X = \prod_{H \in \mathcal{M}_1(X)} \rho_H$
\end{center}
 a \textbf{\textit{total boundary defining function}}, where $\mathcal{M}_1(X)$ denotes the set of boundary hypersurfaces of $X$.

\begin{remark}
\label{remark_normal_boundary_hypersurfaces}
    In the resolution, the collective boundary hypersurfaces keep track of information that distinguish between a non-normal pseudomanifold and its normalization (see Definition \ref{definition_normal_pseudomanifold}). For instance, in the case of the torus with a pinched meridian $\widehat{X}$ (see page 152 of \cite{goresky1980intersection}, and the example on page 55 of\cite{Kirwan&woolf_2006_book}), the link of the isolated singularity is the disjoint union of two circles. The topological normalization is a topological sphere where the pre-image of the singular point corresponding to the blow-down map consists of two points, each of which has as its link a single circle. The boundary hypersurfaces are also circles.
\end{remark}

When describing the natural analogues of objects in differential geometry on singular spaces, the iterated fibration structure comes into play. For example,
\begin{equation}
\label{equation_smooth_functions_on_stratified_spaces}
\mathcal{C}_{\Phi}^{\infty}(X)= \{ f \in \mathcal{C}^{\infty}(X) : f \big|_{\mathfrak{B}_Y} \in 
\phi_{Y}^*\mathcal{C}^{\infty}(Y) \text{ for all } Y \in \mathcal{S}(X) \}
\end{equation}
corresponds to the smooth functions on X that are continuous on the underlying
stratified space. 

\subsubsection{Wedge metrics and related structures}
\label{subsubsection_wedge_metrics_related_structures}

On Thom-Mather stratified pseudomanifolds we can define metrics which are \textit{locally conic}. For instance, consider the model space $\widehat{X}=\mathbb{R}^k \times {Z}^+$, where ${Z}^+$ is a cone over a smooth manifold $Z$. The resolved manifold with corners that corresponds to the blowup of this space is $X=\mathbb{R}^k \times [0,\infty)_x \times Z$. We have the model wedge metric
\begin{equation}
\label{equation_definition_homogeneous_metric_cone}
    g_{w} = g_{\mathbb{R}^k} + dx^2 + x^2 g_Z,
\end{equation}
which is a product metric on the product space $X^{reg}$ and which we call a  \textbf{\textit{product type wedge metric}}, degenerates as it approaches the stratum at $x=0$. These metrics are degenerate as bundle metrics on the tangent bundle but we can introduce a rescaled bundle on which they are non-degenerate. 

Formally, we proceed as follows. Let $X$ be a manifold with corners and iterated fibration structure. Consider the ‘wedge one-forms’
\begin{center}
    $\mathcal{V}^{*}_{w}= \{ \omega \in \mathcal{C}^{\infty}(X; T^*X) : \text{ for each } Y \in \mathcal{S}(X), \text{  }i^*_{\mathfrak{B}_Y} \omega (V) =0 \text{ for all } V \in \text{ker } D\phi_Y \}$.
\end{center}
We can identify $\mathcal{V}^{*}_{w}$ with the space of sections of a vector bundle which we call $\prescript{{w}}{}{T}^*X$, the wedge cotangent bundle of X, together with a map
\begin{equation}
    i_{w} : \prescript{{w}}{}{T}^*X \rightarrow T^*X
\end{equation}
that is an isomorphism over $X^{reg}$ such that,
\begin{center}
    $(i_{w})_*\mathcal{C}^{\infty}(X;\prescript{{w}}{}{T}^*X) =  \mathcal{V}^{*}_{w} \subseteq \mathcal{C}^{\infty}(X; T^*X)$.
\end{center}
In local coordinates near the collective boundary hypersurface, the wedge cotangent bundle is spanned by
\begin{center}
    $ dx$, $xdz$, $dy$
\end{center}
where $x$ is a boundary defining function for $\mathfrak{B}_Y$, $dz$ represents covectors along the fibers and $dy$ covectors along the base.
The dual bundle to the wedge cotangent bundle is known as the wedge tangent bundle, $\prescript{{w}}{}{T}X$. It is locally spanned by
\begin{center}
    $\partial_x$, $\frac{1}{x}\partial_z$, $\partial_y$
\end{center}
A \textit{\textbf{wedge metric}} is simply a bundle metric on the wedge tangent bundle.

The notion of a \textit{\textbf{wedge differential operator}} 
$P$ of order $k$ acting on sections of a vector bundle $E$, taking them to sections of a vector bundle $F$ is described on page 11 of \cite{Albin_2017_index}. We first define the \textit{\textbf{edge vector fields on X}} by
\begin{center}
    $\mathcal{V}_e = \{ V \in  \mathcal{C}^{\infty}(X;TX): V \big|_{\mathfrak{B}_Y}$ is tangent to the fibers of $\phi_Y$ for all $Y \in \mathcal{S}(X) \}$.
\end{center}
There is a rescaled vector bundle that is called the \textit{\textbf{edge tangent bundle}} $\prescript{e}{}TX$ together with a natural vector bundle map $i_e : \prescript{e}{}TX \rightarrow TX$ that is an isomorphism over the interior and satisfies
\begin{center}
    $(i_e)_* \mathcal{C}^{\infty}(X;\prescript{e}{}TX)=\mathcal{V}_e$.
\end{center}
In local coordinates near a point in $\mathfrak{B}_Y, (x, y, z)$,
a local frame for $\prescript{e}{}TX$ is given by 
\begin{center}
    $ x\partial_x, x\partial_y, \partial_z$
\end{center}
Note that the vector fields $x\partial_x$ and $x\partial_y$ are degenerate as sections of $TX$, but not as sections of $\prescript{e}{}TX$. 

The universal enveloping algebra of $\mathcal{V}_e$ is the ring $\text{Diff}^*_e(X)$ of edge differential operators. That is, these are the differential operators on $X$ that can be expressed locally as finite sums of products of elements of $\mathcal{V}_e$. They have the usual notion of degree and extension to sections of vector bundles, as well as an edge symbol map defined on the edge cotangent bundle (see \cite{Mazzeo_Edge_Elliptic_1,Albin_signature,Albin_hodge_theory_cheeger_spaces}).
Similarly, $\text{Diff}^*_e(X; E, F)$ denotes the edge differential operators acting on sections of a vector bundle $E$ and taking them to a sections on a vector bundle $F$.
The edge symbol (which can be defined on the space of edge pseudo-differential operators in general) is used to define the notion of ellipticity used in this article. We follow \cite{Albin_2017_index} and define the map
\begin{equation}
    \overline{\sigma_k} : \text{Diff}^k_e(X;E,F) \rightarrow \rho_{RC}^{-k}\mathcal{C}^{\infty}(RC(\prescript{e}{}T^*X), \pi^*hom(E,F))
\end{equation}
which is the usual symbol map where $RC(\prescript{e}{}T^*X)$ denotes the radial compactification of the edge cotangent bundle and $\pi: \prescript{e}{}T^*X \rightarrow X$ is the projection map. We denote by $\rho_{RC}$ a boundary defining function for the
boundary at radial infinity. Multiplying $\overline{\sigma_k}$ by $\rho_{RC}^k$ 
the resulting map $\sigma_k$ is called the \textbf{\textit{edge symbol}} (see Section 3.3 of \cite{Albin_2017_index}).
This fits into the following short exact sequence
\begin{equation}
\label{edge symbol sequence}
    0 \rightarrow \text{Diff}^{k-1}_e(X;E,F) \rightarrow \text{Diff}^{k}_e(X;E,F) \xrightarrow{\sigma} \mathcal{C}^{\infty}(^e\mathbb{S}^*X,\pi^*hom(E,F)) \rightarrow 0.
\end{equation}
If the edge symbol of a differential operator is invertible away from the zero section we call such an operator  \textbf{\textit{edge elliptic}}.
We define \textit{\textbf{wedge differential operators}} by 
\begin{equation*}
    \text{Diff}^k_w(X;E,F) = \rho^{-k}_X \text{Diff}^k_e(X;E,F)
\end{equation*}
following, e.g., \cite{Albin_2017_index}, where $\rho_X$ is a total boundary defining function for $X$. 
If a wedge differential operator $D_k$ of order $k$ can be written as $\rho_X^{-k}A_k$ where $A_k$ is an edge elliptic differential operator of order $k$, then $D_k$ is said to be a  \textbf{\textit{wedge elliptic}} operator.

\begin{remark}
\label{simplification_of setting}
In \cite{Albin_2017_index}, the families index theorem is proven for a smooth family of manifolds with corners and iterated fibration structures, 
\begin{center}
    $X - M \xrightarrow{\psi} B$
\end{center}
where $X, M$, and $B$ are manifolds with corners. We only study the case where $B$ is a point which means we can simplify the language greatly. 
In Definition 1.3 of \cite{Albin_2017_index}, the set of boundary hypersurfaces transverse to $\psi$ happens to be $\mathcal{S}(X)$ in our case since the transversality condition is vacuous. For every collective boundary hypersurface, we have only the fibration given in \ref{iterated_fibration_structure}. 
\end{remark}


A totally geodesic wedge metric is defined inductively. On depth zero spaces, which are just smooth manifolds, a wedge metric is a Riemannian metric. Assuming we have defined totally geodesic wedge metrics at spaces of depth less than $k$, let us assume $X$ has depth $k$.
A wedge metric $g_w$ on $X$ is a \textit{\textbf{totally geodesic wedge metric}} if, for every $Y \in \mathcal{S}(X)$ of depth $k$ there is a collar neighbourhood $\mathscr{C}(\mathfrak{B}_Y) \cong [0,1)_x \times \mathfrak{B}_Y$ of $\mathfrak{B}_Y$ in $X$, a metric $g_{w, pt}$ of the form
\begin{equation}
    g_{w, pt} = dx^2 +x^2 g_{\mathfrak{B}_Y/Y} +\phi^*_Yg_Y
\end{equation}
where $g_Y$ is a totally geodesic wedge metric on Y, $g_{\mathfrak{B}_Y/Y}+\phi^*_Yg_Y$ is a submersion metric for $\mathfrak{B}_Y \xrightarrow{\phi_Y} Y$ and $g_{\mathfrak{B}_Y/Y}$ restricts to each fiber of $\phi_Y$ to be a totally geodesic wedge metric on $Z_Y$ and
\begin{equation}
\label{equation_structure of the metric}
    g_{w} - g_{w, pt} \in x^2 \mathcal{C}^{\infty} (\mathscr{C}(\mathfrak{B}_Y); S^2(^{w}T^*X))).
\end{equation}
If at each step  $g_{w} = g_{w, pt}$, we say $g_w$ is a \textit{\textbf{rigid}} or \textit{\textbf{product-type wedge metric}}. If at every step, $g_{w} - g_{w, pt} = \mathcal{O}(x)$ as a symmetric two-tensor on the wedge tangent bundle, we say $g_w$ is an \textit{\textbf{exact wedge metric}}. 
We work with exact wedge metrics in this article.

\subsection{Wedge Dirac-type operators}
We now introduce wedge Dirac-type operators briefly. We refer the reader to Section 1.3 of \cite{Albin_2017_index} for more details. In that article families of wedge Dirac-type operators are studied where there is a fibration of manifolds with corners on a base $B$ in the notation of that article. We study the case where $B$ is just a point, and we shall simplify the notation used in that article based on this fact.
\begin{definition}
\label{wedge_Clifford_module}
Let $g_{w}$ be a totally geodesic wedge metric on $X$. A \textbf{\textit{wedge Clifford module}} consists of
\begin{enumerate}
    \item a complex vector bundle $E \rightarrow X$
    \item a Hermitian bundle metric $g_E$
    \item a connection $\nabla^E$ on $E$, compatible with $g_E$
    \item  a bundle homomorphism from the ‘wedge Clifford algebra’ into the endomorphism bundle of $E$,
    \begin{equation*}
        {cl} : \mathbb{C}l_{w}(X)=\mathbb{C} \otimes Cl(\prescript{{w}}{}{T}^*X, g_{X}) \rightarrow End(E)
    \end{equation*}
    compatible with the metric and connection in that, for all $\theta \in \mathcal{C}^{\infty}(X;\prescript{{w}}{}{T}^*X)$,
    \begin{itemize}
        \item $g_E({cl}(\theta)\cdot, \cdot) = -g_E(\cdot, {cl}(\theta) \cdot)$
        \item $\nabla^E_{W} {cl}(\theta)= {cl}(\theta) \nabla^E_W + {cl}(\nabla^{X}_W \theta)$ as endomorphisms of $E$, for all $W \in TX$.
    \end{itemize}
\end{enumerate}

This information determines a smooth \textit{\textbf{wedge Dirac-type operator}} by
\begin{equation}
\label{equation_wedge_Dirac_operator_11}
D_{X}:\mathcal{C}^{\infty}_c(\mathring{X};E) \xrightarrow{\nabla^E} \mathcal{C}^{\infty}_c(\mathring{X};T^*X \otimes E) \xrightarrow{{cl}} \mathcal{C}^{\infty}_c(\mathring{X};E) 
\end{equation}
where we have used that $T^*X$ and $\prescript{{w}}{}{T}^*X$ are canonically isomorphic over $\widehat{X}^{reg}=\mathring{X}$, the interior of $X$.
\end{definition}

If $\widehat{X}^{reg}$ is even dimensional we demand that $E$ admits a $\mathbb{Z}_2$ grading $E=E^+ \oplus E^-$, compatible in that it is orthogonal with respect to $g_E$, parallel with respect to $\nabla^E$, and odd with respect to ${cl}$. 
We can equivalently express this in terms of an involution on $E$.
\begin{definition}
\label{Grading operator}
A $\mathbb{Z}_2$ \textbf{\textit{grading operator}} on a Dirac bundle $E$ is an involution
\begin{equation*}
    \gamma \in C^{\infty}(\mathring{X}; End(E))
\end{equation*}
satisfying
\begin{equation*}
    \gamma^2=Id, \hspace{3mm} \gamma {cl}(v)= -{cl}(v) \gamma, \hspace{3mm} \nabla_w^E\gamma = 0, \hspace{3mm} \gamma^*=\gamma,
\end{equation*}
which gives a splitting of the bundle $E$ into the direct sum of the spaces $E^{\pm}=\{v \in E : \gamma(v)=\pm v \}$.
\end{definition}

We can construct such a grading/chirality operator on even dimensional spaces as follows.
Let $e_j$, $1 \leq j \leq m$ be an oriented orthonormal frame for $E$.
Then the operator 
\begin{equation}
\label{equation_even_dimensional_spaces_chirality_operator}
    \gamma= i^p cl(e_1)...cl(e_m) \in C^{\infty}(\mathring{X}; End(E))
\end{equation}
where $p=m/2$ is a grading operator. This is a well known construction for Dirac bundles associated to Clifford modules on smooth manifolds (see \cite[Lemma 3.17]{berline2003heat}) and it is easy to see it extends to the wedge case.

In local coordinates, we can express the Dirac operator constructed in equation \eqref{equation_wedge_Dirac_operator_11} as,
\begin{equation}
    D_{X}= \sum_{i=0}^n {cl}(\theta^i)\nabla^E_{(\theta^i)^{\#}}
\end{equation}
where $\theta^i$ runs over a $g_{w}$-orthonormal frame of $^{w}T^*X$. If we restrict to a collar neighborhood of $\mathfrak{B}_Y, Y\in \mathcal{S}(X)$, this takes the form
\begin{equation}
\label{temp_lah_di_dah}
cl(dx)\nabla^E_{\partial_x} + {cl}(xdz^i)\nabla^E_{\frac{1}{x}\partial_{z_i}} +{cl}(dy^j)\nabla^E_{\partial_{y_j}}\\ 
= {cl}(dx)\nabla^E_{\partial_x} + {cl}(xdz^i)\frac{1}{x}\nabla^E_{\partial_{z_i}} +{cl}(dy^j)\nabla^E_{\partial_{y_j}}
\end{equation}
up to a differential operator in $\text{Diff}^1_e(X;E)$. Here $x$ is a boundary defining function for $\mathfrak{B}_Y$, and we recognize \eqref{temp_lah_di_dah} as a wedge differential operator of order one. It is easy to check that the symbol of this operator is given by Clifford multiplication and as a result it is wedge elliptic. It is easy to check that given a grading operator $\gamma$ as above, we have that $D_X \gamma = -\gamma D_X$.

We are interested in this operator acting on the natural $L^2$ space associated to the wedge metric $g_w$ and the Hermitian metric $g_E$, which we denote by $L^2(X;E)$. 

\begin{remark}
\label{Remark_conjugation_multiweights_notation_changes}
The Hilbert space that we denote as $L^2(X;E)$ is denoted as $L^2_w(X;E)$ in \cite[\S 1.3]{Albin_2017_index} and the notation $L^2(X;E)$ is reserved for the Hilbert space corresponding to a non-degenerate density in that article.
which we notate as $L^2_e(X;E):=\rho_X^{\mathfrak{b}}L^2(X;E)$ for the choice of \textit{multiweight}
\begin{equation}
\label{equation_conjugation_multiweight}
    \mathfrak{b}(H) = \frac{1}{2}(\dim \mathfrak{B}_Y /Y) \text{  for all  } H \subset \mathfrak{B}_Y \text{  and  } Y \in \mathcal{S}(X), 
\end{equation}
corresponding to $1/2$ the dimension of the links at the boundary hypersurfaces.
Furthermore, the authors of that article define the unitarily equivalent family of operators $\eth_{X}$ by
\begin{equation}
\label{conjugate_relation_for_Dirac}
    \eth_{X} = \rho^{\mathfrak{b}}_XD_{X}\rho^{-\mathfrak{b}}_X = D_{X} - \sum_{Y \in {S}(X)} \frac{\dim \mathfrak{B}_Y /Y}{2\rho_Y}{cl}(d\rho_Y). 
\end{equation}
Then $\eth_{X}$ is also a wedge differential operator of order one, and some results are phrased in terms of this operator. We will phrase those results after unraveling the conjugation where we need them.

\end{remark}

We briefly introduce some objects from \cite[\S 2.1]{Albin_2017_index} which we use to describe the domains of the Dirac operators that we study, referring the reader to that article for a more detailed exposition.
If $D_X$ is the Dirac-type operator corresponding to a wedge Clifford module $(E, g_E, \nabla^E, {cl})$ on $X$, 
to each $Y \in \mathcal{S}(X)$, we can associate a family of \textit{$\phi_Y$-vertical operators} $\rho_Y D_X \big|_{\mathfrak{B}_Y}$. 
This in turn defines family of vertical Dirac-type operators $D_{\mathfrak{B}_Y/Y}$ which we call the \textit{boundary family of $D_X$ at $Y$}, defined by the equation
\begin{equation}
\label{equation_773435}
    \rho_Y D_X \big|_{\mathfrak{B}_Y} = D_{\mathfrak{B}_Y/Y} + \frac{l}{2} {cl}(dx)
\end{equation}
where $l=\dim \mathfrak{B}_Y/Y$ is the dimension of the link.
In particular, the restriction of the operator $D_{\mathfrak{B}_Y/Y}$ 
to the fiber $Z_y=\phi_Y^{-1}(y)$ over $y \in Y$ is denoted by
\begin{equation}
\label{operator_at_strata}
    D_{\mathfrak{B}_Y/Y} \big|_{Z_y} = D_{Z_y},
\end{equation}
and we call this the \textit{vertical Dirac-type operator on the link at $y \in Y$}.
This family of operators is key to understanding the choices of self adjoint domains for Dirac-type operators.

\subsubsection{Witt assumption and vertical APS domain}

This subsection builds on \cite[\S 2.2]{Albin_2017_index}.
As an unbounded operator on $L^2(X; E)$, the Dirac-type operators constructed above are formally self adjoint, but in general are not essentially self adjoint. There are different choices of closed extensions from smooth sections of compact support on $X^{reg}$ to closed operators on $L^2(X; E)$. We emphasize that we use a different notation for the $L^2$ spaces from that in \cite{Albin_2017_index}, as clarified in Remark \ref{Remark_conjugation_multiweights_notation_changes}.

There are two canonical domains which are the minimal domain,
\begin{equation}
    \mathcal{D}_{min}(D_X)= \{ u \in L^2(X;E) : \exists (u_n) \subseteq {C}^{\infty}_c(\mathring{X};E) \text{ s.t. }
    u_n \rightarrow u \text{ and } (D_Xu_n) \text{ is } L^2-\text{Cauchy} \},
\end{equation}
and the maximal domain,
\begin{equation*}
    \mathcal{D}_{max}(D_X)= \{ u \in L^2(X;E) : (D_Xu) \in L^2(X;E) \},
\end{equation*}
wherein $D_Xu$ is computed distributionally.
These domains satisfy the inclusions
\begin{equation*}
    \rho_XH^1_e(X;E) \subseteq \mathcal{D}_{min}(D_X) \subseteq \mathcal{D}_{max}(D_X) \subseteq H^1_e(X;E)
\end{equation*}
where
\begin{equation*}
    H^1_e(X;E) = \{ u \in L^2(X;E) : Vu \in L^2(X;E) \text{ for all } V \in  \mathcal{C}^{\infty}(X;\prescript{e}{}TX) \}
\end{equation*}
is the edge Sobolev space introduced in \cite{Mazzeo_Edge_Elliptic_1}.
The main domain of interest to us is the following which lies between the minimal and maximal domains.

\begin{definition}
\label{VAPS_Domain}
The \textbf{\textit{vertical APS (VAPS) domain}} of the operator ${D_X}$ is defined to be the graph closure of ${\rho^{1/2}_X H^1_e(X;E) \cap \mathcal{D}_{max} (D_X) }$. That is
\begin{multline}
\label{VAPS_domain}
\mathcal{D}_{VAPS}(D_X) = \{ u \in L^2(X;E) : \exists (u_n) \subset \rho^{1/2}_X H^1_e(X;E) \cap \mathcal{D}_{max} (D_X)\\ 
\text{such that } u_n \rightarrow u \; \text{and} \;  (D_X u_n) \; \text{is} \; L^2- \text{Cauchy} \}
\end{multline}
\end{definition}
The following proposition gives a characterization of these domains using a lower regularity space. We can take this to be an equivalent definition for the VAPS domain.
\begin{proposition}
\label{equivalent_definition_VAPS_domain}
If $D_X$ is a wedge Dirac-type operator on a stratified pseudomanifold $X$, then 
\begin{equation*}
    \mathcal{D}_{VAPS}(D_X) = \text{graph closure of } \{\rho^{1/2}_X L^2(X;E) \cap \mathcal{D}_{max} (D_X) \}
\end{equation*}
\end{proposition}

\begin{proof}
We will show that $\rho^{1/2}_X L^2(X;E)\cap \mathcal{D}_{max} (D_X) =\rho^{1/2}_X H^1_e(X;E) \cap \mathcal{D}_{max} (D_X)$. Note that since $\rho_X D_X$ is an elliptic 1st order edge operator, we have $H^1_e(X;E)=\mathcal{D}_{max} (\rho_X D_X)$, and so given $v \in \rho^{1/2}_X L^2(X;E)\cap \mathcal{D}_{max} (D_X)$, we need to show that $\rho_X D_X(\rho_X^{-1/2}v) \in L^2(X;E)$. 

Since $X$ is compact, $\sigma (D_X)(d \rho_X)$ is a smooth bounded map on $L^2(X;E)$, where $\sigma (D_X)=\rho_X^{-1} \sigma_1 (\rho_X D_X)$ wherein $\sigma_1$ is the edge symbol given in \eqref{edge symbol sequence}. If $v \in \rho^{1/2}_X L^2(X;E) \cap \mathcal{D}_{max} (D_X)$ then 
\begin{equation*}
   \rho_X D_X(\rho_X^{-1/2}v)=\rho_X \left[D_X,\rho_X^{-1/2}\right] v + \rho_X^{1/2} D_Xv
\end{equation*}
The last term on the right hand side is in $L^2(X;E)$ since $v \in \mathcal{D}_{max} (D_X)$. The first term can be written as
\begin{equation*}
   \rho_X \left[D_X,\rho_X^{-1/2}\right] v= \rho_X \sigma (D_X)\left(-\frac{d \rho_X}{2\rho_X^{3/2}}\right)v=\frac{-1}{2}\rho_X^{-1/2} \sigma (D_X)(d \rho_X)v
\end{equation*}
which is in $L^2(X;E)$ since $\sigma (D_X)(d \rho_X)$ is a smooth bounded map. This completes the proof.
\end{proof}

\begin{definition}
\label{Witt_assumption}
The operator $(D_X, \mathcal{D}_{VAPS})$ is said to satisfy the \textbf{\textit{geometric Witt condition}} if 
\begin{equation}
    Y \in \mathcal{S}(X), y \in Y \implies \text{Spec}(D_{Z_y}) \cap (-\frac{1}{2},\frac{1}{2}) = \emptyset
\end{equation}
If instead, we only require
\begin{equation}
    Y \in \mathcal{S}(X), y \in Y \implies \text{Spec}(D_{Z_y}) \cap \{ 0 \} = \emptyset
\end{equation}
then we say that $(D_X, \mathcal{D}_{VAPS})$ satisfies the \textit{\textbf{Witt condition}}.
\end{definition}

\begin{remark}
\label{measures_and_critical_interval}
Here the domain of the operators $D_{Z_y}$ from \eqref{operator_at_strata} at the links are taken to be the VAPS domains. One thinks of the domains on the fibers as induced from the domain on the space, but this definition can be stated referring only to the spectrum of $\text{Spec}(D_{Z_y}, \mathcal{D}_{VAPS}(D_{Z_y}))$.
\end{remark}

Remark 2.5 of \cite{Albin_2017_index} justifies the nomenclature `vertical APS domain'. The different local ideal boundary conditions for $D_X$ involve the spectrum of $D_{Z_y}$ in the interval (-1/2,1/2) and the VAPS domain corresponds to projecting off of the negative half of this interval, analogously to the Atiyah-Patodi-Singer boundary conditions.

In Remark 4.10 of the same article, it is shown that the metrics on the links $Z$ can be rescaled to make the operators essentially self adjoint, but special structures (such as K\"ahler structures) may be lost. 

We now present a key theorem from \cite{Albin_2017_index} which tells us that we have a trace-class heat kernel for the square of the Dirac-type operators with the VAPS domain. This is Theorem 1 of \cite{Albin_2017_index}, rephrased in the notation we have introduced.
\begin{theorem}
\label{Trace_class_theorem_from_Pierre}
Let $X$ be a manifold with corners and iterated fibration structure endowed with a totally geodesic wedge metric $g_w$. Let $(E,g_E, \nabla^E, {cl})$ be a wedge Clifford module with associated Dirac-type operator $D_X$ and satisfying the Witt assumption \ref{Witt_assumption}. Then $D_X$ with its vertical APS domain
\begin{equation*}
    \mathcal{D}_{VAPS}(D_X) = \text{graph closure of } \{u \in \rho^{1/2}_X H^1_e(X;E):  D_Xu \in L^2(X;E) \}
\end{equation*}
is a closed operator on $L^2(X;E)$, that is self-adjoint and Fredholm with compact resolvent. The heat kernel of $D_X^2$ with the induced domain is trace-class.
\end{theorem}

We shall study the heat kernel in more detail in Section \ref{section_formulas},

\subsection{Local neighbourhoods of singular points}
\label{subsection_local_structures}

In this article we study self maps with isolated fixed points and we will be focusing a lot on local neighbourhoods and local structures.
Given any point $p$ of a stratified pseudomanifold $\widehat{X}$, we can find a neighbourhood $\widehat{U_p}$ which has a homeomorphism as in equation \eqref{equation_chartlike_map}
\begin{equation}
\label{fundamental neighbourhood}
   \widehat{\phi}:\widehat{U_p} \longrightarrow  \widehat{\phi}(\widehat{U_p}) \subset \mathbb{R}^k_{y} \times \widehat{Z}^+_{z}
\end{equation}
where the image is bounded.
We call such a neighbourhood $\widehat{U_p}$ a \textit{\textbf{fundamental neighbourhood of $p$}}. Here $\widehat{Z}_{z}$ is another stratified space, and $\widehat{Z}^+$ is the infinite cone over this link. We can choose $\widehat{U_p}$ and $\phi$ such that $\phi(\widehat{U_p})=\mathbb{D}^{k} \times C(\widehat{Z})$, where by $C(\widehat{Z})$ we denote the truncated cone $[0,1]_x \times \widehat{Z}_z / _\sim $ where the points at $\{x=0\}$ are identified. The restriction of an exact wedge metric to such a neighbourhood is as we discussed in Subsection \ref{subsubsection_wedge_metrics_related_structures}.

By the equivalence of Thom-Mather stratified spaces and manifolds with corners with iterated fibration structures, there exists a lift of this map $\phi: U_p \rightarrow \phi(U_p)$, where $U_p$ is a manifold with corners with iterated fibration structures, such that $\widehat{\phi} \circ \beta=\beta \circ \phi$ where $\beta$ is the blow down map, and where $\phi$ is a diffeomorphism of manifolds with corners. We will refer to $U_p$ as a fundamental neighbourhood of $p$ as well, denoting the difference as needed by our notation. Indeed, these are metrically the same spaces.

We refer to $\widehat{\phi}$ as a \textit{\textbf{diffeomorphism of the stratified pseudomanifold $\widehat{U_p}$}}. It is clear that this definition extends naturally to diffeomorphisms of Thom-Mather stratified spaces for which such lifts exist.

If $p$ is a singular point of $\widehat{X}$ contained on the stratum $Y$, and $\widehat{Z}$ is the link of $p$ in $\widehat{X}$, then the tangent cone of $\widehat{X}$ at $p$ is $T_pY \times C(\widehat{Z})$. Given a wedge metric on $\widehat{X}$, there is a canonical metric on the tangent cone of a point obtained by freezing coefficients and extending homogeneously.

We will denote the resolved \textbf{\textit{tangent cone}} at $p$ by $\mathfrak{T}_pX$, the metric product of the unit ball in $T_pY$ and the truncated cone $C_{[0,1]}(Z)=[0,1]\times Z$ which is topologically a manifold with boundary, but is metrically singular. Indeed $\mathfrak{T}_pX$ is a Riemannian space with the restriction of the homogeneous induced metric.

Given a fundamental neighbourhood $\widehat{U_p}$ with boundary $\partial \widehat{U_p}$, the resolution $U_p$ has a boundary component given by the (possibly resolved) boundary $\partial \widehat{U_p}$ as well as the resolved singularities. However, the resolved singularities are still metrically singular and we will refer to them as the singularities of $U_p$, while referring to the (possibly resolved) boundary $\partial \widehat{U_p}$ as the \textit{smooth boundary} of $U_p$, and we will denote it by $\partial U_p$. Indeed, the smooth boundary has an open dense set which is smooth.


\section{Hilbert Complexes}
\label{section_Hilbert_complexes}

In this section we will review some basic theory of Hilbert complexes that we shall use, including specializations to elliptic complexes and two-term Dirac complexes.

\subsection{Abstract Hilbert complexes}
\label{subsection_abstract_hilbert_complexes}
We begin by defining Hilbert complexes following \cite{bru1992hilbert}. We adopt the notation $A \xrightharpoonup{} B$ to denote partial functions between $A$ and $B$, i.e., a function defined from some subset of $A$ into $B$. We use this notation when the differentials we study are only defined from their domain to their range.

\begin{definition}
A \textbf{\textit{Hilbert complex}} is a complex, $\mathcal{P}=(H,P)$, of the form:
\begin{equation}
    0 \rightarrow H_0 \xrightharpoonup{P_0} H_1 \xrightharpoonup{P_1} H_2 \xrightharpoonup{P_2} ... \xrightharpoonup{P_{n-1}} H_n \rightarrow 0.
\end{equation}
Here each $H_k$ is a separable Hilbert space and each map $P_k$ is a closed first order operator which is called the differential, such that:
\begin{itemize}
    \item the domain of $P_k$, $\mathcal{D}(P_k)$, is dense in $H_k$,
    \item the range of $P_k$ satisfies $ran(P_k) \subset \mathcal{D}(P_{k+1})$,
    \item $P_{k+1} \circ P_k = 0$ for all $k$.
\end{itemize}
\end{definition}

We shall sometimes notate the complex as $\mathcal{P}(X)$ when the Hilbert spaces are sections of a vector bundle on the resolution of a stratified pseudomanifold $X$, and we say that the \textit{\textbf{Hilbert complex $\mathcal{P}(X)$ is associated to the space $X$}}.
The \textit{\textbf{cohomology groups}} of a Hilbert complex are defined to be $\mathcal{H}^k(H,P):= ker(P_k)/ran(P_{k-1})$. We shall often use the notation $\mathcal{H}^k(\mathcal{P})$, or even just $\mathcal{H}^k$, where the complex used is clear from the context and $\mathcal{H}^k(\mathcal{P}(X))$ when the space needs to be specified (including spaces with boundary when they come up later on). If these groups are finite dimensional in each degree, we say that it is a \textit{\textbf{Fredholm complex}}.
In this article, unless otherwise specified, we work with Hilbert complexes associated to wedge operators as introduced in the previous section.

\begin{definition}[(wedge) elliptic complex]
Let $\mathcal{P}(X)=(H, P)$ be a Hilbert complex where each $P_j$ is a  first order wedge operator acting on Hilbert spaces $H_k=L^2(X;E_k)$ of sections of a vector bundle $E_k$ on $X$, with leading wedge symbol (see the discussion following equation \eqref{edge symbol sequence}) $p_k= \sigma_1 (P_k)$. In particular, $p_k \circ p_{k-1}=0$. We say $\mathcal{P}$ is a \textbf{\textit{(wedge) elliptic complex}} if $ker(p_k(\xi))= ran(p_{k-1}(\xi))$ for $\xi \neq 0$.
\end{definition}

We will refer to these as elliptic complexes when it is clear from context that we are working with wedge operators. Elliptic complexes on smooth manifolds are also wedge elliptic complexes. 

For every Hilbert complex $\mathcal{P}$ there is a \textit{\textbf{dual Hilbert complex}} $\mathcal{P}^*$, given by
\begin{equation}
\label{dual_complex}
    0 \rightarrow H_n \xrightharpoonup{P_{n-1}^*} H_{n-1} \xrightharpoonup{P_{n-2}^*} H_{n-2} \xrightharpoonup{P_{n-3}^*} .... \xrightharpoonup{P_{0}^*} H_0 \rightarrow 0
\end{equation}
where the differentials are $P_k^*: Dom(P^*_k) \subset H_{k+1} \rightarrow H_k$, the Hilbert space adjoints of the differentials of $\mathcal{P}$. That is, the Hilbert space in degree $k$ of the dual complex $\mathcal{P}^*$ is the Hilbert space in degree $n-k$ of the complex $\mathcal{P}$, and the operator in degree $k$ of $\mathcal{P}^*$ is the adjoint of the operator in degree $(n-1-k)$ of $\mathcal{P}$. The corresponding cohomology groups of $\mathcal{P}^*(H,P^*)$ are $\mathcal{H}^k(H, (P)^*) := ker(P^*_{n-k-1})/ran(P^*_{n-k})$.
For instance, in the case of the de Rham complex $(L^2\Omega^k(X),d_{\max})$, the dual complex $\mathcal{P}^*$ is the complex $\mathcal{Q}=(L^2\Omega^{n-k}(X),\delta_{\min})$, since the operators $d$ and $\delta$ are formal adjoints of each other. 

Similarly, the dual of $\mathcal{P}=(L^2\Omega^{p,q}(X),\overline{\partial}_{\max})$ for a stratified pseudomanifold with a wedge complex structure can be identified with $\mathcal{Q}=(L^2\Omega^{n-p,n-q}(X),\overline{\partial}^*_{\min})$. The Hilbert spaces of these complexes are different, as are the operators and adjoints but it is well known that the Hodge star gives an isomorphism between these complexes. This is an instance of Serre duality which we shall explore more in Subsection \ref{subsubsection_Serre_duality}.

We introduce a special type of elliptic complex as follows.
\begin{definition}
\label{equation_two_term_complex}
A \textit{\textbf{(wedge) Dirac complex}} $\mathcal{R}=(H, D)$ is a two term wedge elliptic complex of the form
\begin{equation}
    0 \rightarrow H^{+} \xrightharpoonup{D^+} H^- \rightarrow 0
\end{equation}
where $D^+$ is a Dirac-type operator.
\end{definition}
Given a wedge elliptic complex, we can associate a two term complex as in Definition \ref{equation_two_term_complex} where $H^+ =\bigoplus_{k=even} H_k$, and $H^- = \bigoplus_{k=odd} H_k$. We shall refer to such a two term complex as a \textit{\textbf{(wedge) Dirac complex $\mathcal{D_P}$ associated to the elliptic complex $\mathcal{P}$}}. Here we take the domain for the operator $D=P+P^*$ to be
\begin{equation}
    \label{Domain_Dirac_first}
    \mathcal{D}(D)=\mathcal{D}(P) \cap \mathcal{D}(P^*).
\end{equation}

For an elliptic complex $\mathcal{P}=(H,P)$, there is an associated \textbf{\textit{Laplace-type operator}} $\Delta_k = P_{k}^*P_k+P_{k-1}P_{k-1}^*$ in each degree, which is a self adjoint operator with domain
\begin{equation}
\label{Laplacian_P_type}
\mathcal{D}(\Delta_k) = \{ v \in \mathcal{D}(P_k) \cap \mathcal{D}(P_{k-1}^*) : P_k v \in \mathcal{D}(P_k^*), P^*_{k-1} v \in \mathcal{D}(P_{k-1}) \},
\end{equation}
and with nullspace
\begin{equation}
   \widehat{\mathcal{H}}^k(\mathcal{P}):= ker(\Delta_k) = ker(P_k) \cap ker(P_{k-1}^*).
\end{equation}
The Kodaira decomposition which we present below in Proposition \ref{Kodaira_decomposition} identifies this with the cohomology of the complex $\mathcal{H}^k(\mathcal{P})$.
We observe that this Laplace-type operator can be written as the square of the associated Dirac-type operator $D=(P+P^*)$, restricted to each degree to obtain $\Delta_k$, and that the domain can be written equivalently as
\begin{equation}
\label{Laplacian_D_type}
\mathcal{D}(\Delta_k) = \{ v \in \mathcal{D}(D) : D v \in \mathcal{D}(D) \}.
\end{equation}
The null space is isomorphic to the cohomology for Fredholm complexes.

\begin{definition}
The \textit{\textbf{reduced cohomology groups}} of a complex are defined as
\begin{center}
    $\overline{\mathcal{H}}^k(H, P) := Ker(P_k)/ \overline{(ran(P_{k-1}))}$
\end{center}
where the overline indicates the closure of the range of the operator $P_{k-1}$.
\end{definition}

We state the following generalization of the Hodge decomposition for the de Rham complex, known as the \textbf{\textit{weak Kodaira decomposition}}.
\begin{proposition}
\label{Kodaira_decomposition}
For any Hilbert complex $\mathcal{P}=(H, P)$ we have
\begin{center}
    $H_k=\widehat{\mathcal{H}}^k(H, P) \bigoplus \overline{ran(P_{k-1})} \bigoplus \overline{ran(P_k^*)}$
\end{center}
\end{proposition}
This is Lemma 2.1 of \cite{bru1992hilbert}. 

\begin{proposition}
\label{Fredholm_is_closed_range}
If the cohomology of an elliptic complex $\mathcal{P}=(H_*, P_*)$ is finite dimensional then, for all $k$, $ran(P_{k-1})$ is closed and therefore $\mathcal{H}^k(\mathcal{P}) \cong \widehat{\mathcal{H}}^k(\mathcal{P})$. 
\end{proposition}
This is corollary 2.5 of \cite{bru1992hilbert}. The next result justifies the use of the term \textit{Fredholm complex}.
\begin{proposition}
A Hilbert complex $(H_k, P_k)$, $k=0,...,n$ is a Fredholm complex if and only if, for each $k$, the Laplace-type operator $\Delta_k$ with the domain defined in \eqref{Laplacian_P_type} is a Fredholm operator.
\end{proposition}
This is Lemma 1 on page 203 of \cite{schulze1986elliptic}. 
Due to these results, we can identify the space of \textit{harmonic elements}, or the elements of the Hilbert space which are in the null space of the Laplace-type operator, with the cohomology of the complex in the corresponding degree. We shall use the same terminology for non-Fredholm complexes which we study as well. 

For Fredholm complexes, the null space of the Laplacian is isomorphic to the cohomology of the complex since the operator has closed range.

\begin{proposition}
\label{Kernel_equals_cohomology}
A Hilbert complex $\mathcal{P}=(H, P)$, is a Fredholm complex if and only if its dual complex, $(\mathcal{P}^*)$ is Fredholm.
If it is Fredholm, then 
\begin{equation}
    \mathcal{H}^k(\mathcal{P}) \cong \overline{\mathcal{H}}^k(\mathcal{P}) \cong \overline{\mathcal{H}}^{n-k}(\mathcal{P}^*) \cong \mathcal{H}^{n-k}(\mathcal{P}^*).
\end{equation}
\end{proposition}

In particular, for operators with closed range, the reduced cohomology groups are the same as the cohomology groups and are isomorphic to the null space of the Laplace-type operator, in which case the above decomposition is called the \textbf{\textit{(strong) Kodaira decomposition}}.

\begin{remark}
\label{remark_only_normal_pseudomanifolds}
    Since the stratum $\widehat{X}_{n-2}$ has measure $0$ with respect to a conic metric, the $L^2$ functions on $X^{reg}$ are the same for the space and its normalization and the Hilbert complexes on $\widehat{X}$ and its normalization can be canonically identified. \textbf{From now on we will study topologically normal pseudomanifolds unless otherwise specified}. We will indicate how to compute Lefschetz numbers in the non-normal case in examples in Subsection \ref{subsection_nonnormal_examples}.
\end{remark}

\subsection{Lefschetz heat supertraces}
\label{subsection_abstract_lefschetz_supertraces}

In this subsection we prove results for Lefschetz numbers of wedge elliptic complexes and their endomorphisms.
\begin{definition}
\label{geometric_endo}
An \textbf{\textit{endomorphism}} $T$  \textbf{\textit{of a wedge elliptic complex $\mathcal{P}=(H, P)$}} is given by an n-tuple of maps $T=(T_0, T_1,...,T_n)$, where $T_k:H_k \rightarrow H_k$ are bounded maps of Hilbert spaces, that satisfy the following properties.
\begin{enumerate}
    \item $T_k(\mathcal{D}(P_k)) \subseteq \mathcal{D}(P_k)$
    \item $P_k \circ T_k= T_{k+1} \circ P_k$ on $\mathcal{D}(P_k)$
\end{enumerate}
Each endomorphism $T_k$ has an adjoint $T_k^*$. If each $T_k^*$ preserves the domain $\mathcal{D}(P_k^*)$, then we call $T^*=(T^*_0, T^*_1,...,T^*_n)$ the \textit{\textbf{adjoint endomorphism}} of the dual complex.
\end{definition}

The commutation condition (condition 2) ensures that the endomorphisms will induce a map on the kernel of the Dirac-type operators that preserves the grading. In the next section, we will study endomorphisms of complexes in greater detail.

\begin{definition}
\label{L2_Lefschetz}
Let $X$ be a pseudomanifold with a Fredholm complex $\mathcal{P}=(H,P)$ with an endomorphism $T$, we define the associated \textit{\textbf{Lefschetz polynomial}} to be
\begin{equation} 
L(\mathcal{P},T)(b):= \sum_{k=0}^n b^k tr(T_k|_{\mathcal{H}^k(\mathcal{P})}) \in \mathbb{C}[b]
\end{equation}
and the associated \textit{\textbf{Lefschetz number}} to be
\begin{equation} 
\begin{split}
L(\mathcal{P},T) &:= L(\mathcal{P},T)(-1) \\
 & = Tr(T|_{\mathcal{H}^{+}({\mathcal{D_P}})})- Tr(T|_{\mathcal{H}^{-}({\mathcal{D_P}})})
\end{split}
\end{equation}
where $T_k|_{\mathcal{H}^k(\mathcal{P})} :\mathcal{H}^k(\mathcal{P}) \rightarrow \mathcal{H}^k(\mathcal{P})$ is the map induced by the endomorphism $T$ on the $k$-th degree cohomology group $\mathcal{H}^k(\mathcal{P})$ and $T|_{\mathcal{H}^{\pm}({\mathcal{D_P}})}$ is the map induced on the cohomology of the associated Dirac complex $\mathcal{D_P}$.
\end{definition}
The following is a generalization of the Definition \ref{L2_Lefschetz} which shows up in the Morse inequalities which we shall study later.
\begin{definition}
\label{definition_polynomial_Lefschetz_supertrace}
Let $\mathcal{P}=(H,P)$ be a wedge elliptic complex where the associated Laplace-type operators in each degree have discrete spectrum and trace class heat kernels. Let $T$ be an endomorphism of the complex. For all $t \in \mathbb{R}^+$, we define the \textbf{\textit{polynomial Lefschetz heat supertrace}} as
\begin{equation}
    \mathcal{L}(\mathcal{P},T)(b,t)=\mathcal{L}(\mathcal{P}(X),T)(b,t):=\sum_{k=0}^n  Tr( b^k T_k e^{-t \Delta_k})
\end{equation}
and we call $\mathcal{L}(\mathcal{P},T)(-1,t)$ the \textit{\textbf{Lefschetz heat supertrace}} associated to the complex. Here, we use the notation $\mathcal{L}(\mathcal{P}(X),T)(b,t)$ when the complex $\mathcal{P}=\mathcal{P}(X)$ is associated to a pseudomanifold $X$.
\end{definition}

The Kodaira decomposition of the complex $\mathcal{P}$ is
\begin{center}
    $H_k=\mathcal{H}^k(\mathcal{P}) \bigoplus ran(P_{k-1}) \bigoplus ran(P_{k}^*)$.
\end{center}
We shall call the elements in $ran(P_{k-1})$ the \textit{\textbf{exact elements}}, and the elements in $ran(P_{k}^*)$ the \textit{\textbf{co-exact elements}}, generalizing the nomenclature from the case of the de Rham operator. We call them exact sections and co-exact sections when the elements of the Hilbert space are $L^2$ sections of a vector bundle. In light of the identity $P_{k-1}^* P_{k}^*=0$, the Laplace-type operator $\Delta_k$ restricted to the co-exact elements is just $P_{k}^* P_k$, while on exact elements it is simply $P_{k-1} P_{k-1}^*$. 
Given an orthonormal basis of elements of $H_k$ which are eigensections of $\Delta_k$, for all $k$, this next Lemma shows that we can find two canonically isomorphic bases for the exact and co-exact eigensections. This relation is sometimes called supersymmetry.

\begin{lemma}
\label{Lemma_super_symmetry}
For a wedge elliptic complex $\mathcal{P}$ associated to wedge Dirac operators on a stratified pseudomanifold $X$, given $\{v_{\lambda_{i}}\}_{i \in \mathbb{N}}$, an orthonormal basis of eigenvectors for the co-exact elements of degree $k$, the set $\{\frac{1}{\sqrt{\lambda_{i}}}P_k v_{\lambda_{i}}\}_{i \in \mathbb{N}}$ is an orthonormal basis of eigenvectors (of the operator $\Delta_{k+1}$) for the exact elements of degree $k+1$.
\end{lemma}

\begin{proof}
Since $\Delta_k v_{\lambda_{i}} = (P_{k}^* P_k) v_{\lambda_{i}}=  \lambda_{i} v_{\lambda_{i}}$, to show that $P_k v_{\lambda_{i}}$ is an eigenvector, we first observe that
\begin{equation*}
\begin{split}
\Delta_{k+1} (P_k v_{\lambda_{i}}) & =( P_{k+1}^*P_{k+1}+P_{k}P_{k}^*) (P_k v_{\lambda_{i}})\\
                                 & = P_{k}P_{k}^* (P_k v_{\lambda_{i}})\\
                                 & = P_{k}(P_{k}^* P_k v_{\lambda_{i}})\\
                                 & = P_{k}(\Delta_k v_{\lambda_{i}})\\
                                 & = \lambda_{i} P_k v_{\lambda_{i}}
\end{split}
\end{equation*}
where the second and the fourth inequalities follow from the identities $P_{k+1}P_k=0$ and $P_kP_{k-1}=0$ respectively. 
The computation 
\begin{equation*}
\langle P_k v_{\lambda_{i}}, P_k v_{\mu_{i}} \rangle = \langle v_{\lambda_{i}}, (P_{k}^*P_k) v_{\mu_{i}} \rangle = \langle v_{\lambda_{i}}, \Delta_k v_{\mu_{i}} \rangle = \langle v_{\lambda_{i}}, \mu_{i} v_{\mu_{i}} \rangle
\end{equation*}
where the inner products are those on $H_{k+1}$
shows that $\frac{1}{\sqrt{\lambda_{i}}}P_kv_{\lambda_{i}}$ are orthonormal, since the last expression is equal to $\mu_i=\lambda_i$ if the eigenvectors are the same, and zero otherwise.
\end{proof}

\begin{theorem}
\label{Lefschetz_supertrace}
Let $\mathcal{P}=(H,P)$ be an elliptic complex where the associated Laplace-type operators in each degree have discrete spectrum and trace class heat kernels. Let $T$ be an endomorphism of the complex. For all $t \in \mathbb{R}^+$
\begin{equation} 
\label{equation_with_the_b}
    \mathcal{L}(\mathcal{P},T)(b,t)=L(\mathcal{P},T)(b)+ (1+b) \sum_{k=0}^{n-1} b^k S_k(t)
\end{equation}
where
\begin{equation} 
    S_k(t)=\sum_{\lambda_i \in Spec(\Delta_k)} e^{-t \lambda_i}  \langle T_k v_{\lambda_{i}}, v_{\lambda_{i}} \rangle
\end{equation}
where $\{v_{\lambda_i}\}_{i \in \mathbb{N}}$ are an orthonormal basis of co-exact eigensections of $\Delta_k$. In particular
\begin{equation}
\label{Heat_formula_all_t_P}
    L(\mathcal{P},T)= \mathcal{L}(\mathcal{P},T)(-1,t),
\end{equation}
and the Lefschetz heat supertrace is independent of $t$.
\end{theorem}

\begin{remark}
We emphasize the fact that this theorem only uses the structure of an abstract Hilbert complex with discrete spectrum and a trace class heat kernel, and abstract endomorphisms of a complex, using an old argument going back to at least \cite{mckean1967curvature} which is sometimes also refered to as supersymmetric cancellations. We will generalize this to a renormalized version in the non-Fredholm case in Theorem \ref{theorem_renormalized_McKean_Singer}.
\end{remark}

\begin{proof}
We can expand the polynomial Lefschetz heat supertrace using the Kodaira decomposition and an orthonormal basis of eigenvectors constructed as in the previous lemma to get

\begin{multline}
\label{equation_Mahinda_Hora}
\sum_{k=0}^n b^k Tr(T_k e^{-t \Delta_k}) =  \sum_{k=0}^n b^k tr(T_k|_{\mathcal{H}^k(\mathcal{P})})\\ 
+\sum_{k=0}^{n-1} b^k \sum_{\lambda_i \in Spec(\Delta_k)} e^{-t \lambda_i}  \langle T_k v_{\lambda_{i}}, v_{\lambda_{i}} \rangle \\ 
+\sum_{k=0}^{n-1}  b^{k+1} \sum_{\lambda_i \in Spec(\Delta_k)} e^{-t \lambda_i}  \langle T_{k+1} P_k \frac{v_{\lambda_{i}}}{\sqrt{\lambda_i}}, P_k \frac{v_{\lambda_{i}}}{\sqrt{\lambda_i}} \rangle
\end{multline}
where we have used the fact that there are no exact elements in degree 0, and that there are no co-exact elements in degree $n$. To show that we have a $1+b$ factor and the expression for $S_k$, we use the supersymmetry given by the previous lemma as follows:
\begin{equation}
\label{argument_Carrickfergus}
\begin{split}
\frac{1}{\lambda_i}\langle T_{k+1} P_k v_{\lambda_{i}}, P_k v_{\lambda_{i}} \rangle &= \frac{1}{\lambda_i}\langle P_k T_k v_{\lambda_{i}}, P_k v_{\lambda_{i}} \rangle\\
=\frac{1}{\lambda_i}\langle T_k v_{\lambda_{i}}, P_{k}^* P_k v_{\lambda_{i}} \rangle
&=\langle T_k v_{\lambda_{i}}, \frac{1}{\lambda_i} \Delta_k v_{\lambda_{i}} \rangle\\
&=\langle T_k v_{\lambda_{i}}, v_{\lambda_{i}} \rangle.
\end{split}
\end{equation}
It is now easy to see that we have the expression for $S_k(t)$ and that equation \eqref{equation_with_the_b} holds for $t>0$. It is clear that for $b=-1$, this yields
\begin{equation*}
\label{equation_Namal_Hora}
    \sum_{k=0}^n (-1)^k Tr(T_k e^{-t \Delta_k}) = \sum_{i=0}^n (-1)^k tr(T_k|_{\mathcal{H}^k(\mathcal{P})})
\end{equation*}
which is independent of $t$, proving the theorem.
\end{proof}

The following proposition relates the Lefschetz number of an endomorphism of a complex to the Lefschetz number on the adjoint endomorphism of the dual complex. This is a generalization of equation (83) of \cite{Bei_2012_L2atiyahbottlefschetz} coupled with 
the standard duality given by $b \mapsto 1/b$ and Poincar\'e duality in Morse theory (see the discussion following equation 1 of \cite{witten1984holomorphic}) and captures the duality of the Lefschetz polynomials. 

\begin{proposition}[Duality]
\label{proposition_Lefschetz_on_adjoint}
Let $\mathcal{P}=(H,P)$ be an elliptic complex of maximal non-trivial degree $n$ and let $T$ be an endomorphism. Let $\mathcal{P^*}$ be the dual complex and let $T^*$ denote the endomorphism induced on the dual complex. Then
\begin{equation}
\label{Kalman_filter}
   b^n \mathcal{L}(\mathcal{P}^*,T^*)(b^{-1},t)=\mathcal{L}(\mathcal{P},T)(b,t).
\end{equation}
In particular, we have the equality
\begin{equation}
    L(\mathcal{P},T)= (-1)^n L(\mathcal{P^*},T^*).
\end{equation}
\end{proposition}

\begin{proof}
We can express $\mathcal{S}_k(t)$ in the expression \eqref{equation_with_the_b} for the right hand side of equation \eqref{Kalman_filter} in terms of the exact sections of the operator $P$ using Lemma \ref{Lemma_super_symmetry}. This is similar to the argument used to reduce the expression in equation \ref{equation_Mahinda_Hora} to one in terms of the co-exact eigenvectors. Using the identity
\begin{equation}
    e^{-t \lambda_i}  \langle T_k v_{\lambda_{i}}, v_{\lambda_{i}} \rangle=e^{-t \lambda_i} \frac{1}{\lambda_i} \langle  P_{k}v_{\lambda_{i}}, T_{k+1}^* P_kv_{\lambda_{i}} \rangle
\end{equation}
which follows from equation \eqref{argument_Carrickfergus} and the definition of the adjoint, the summation in equation \eqref{equation_with_the_b} can be written as
\begin{equation} 
    (1+b) \sum_{k=0}^{n-1} b^k S_k(t)=(b^{-1}+1) \sum_{k=1}^{n} b^k S_k(t)=(b^{-1}+1) \sum_{k=1}^{n} b^k \widetilde{S}_k(t)
\end{equation}
where 
\begin{equation} 
    \widetilde{S}_k(t)=\sum_{\lambda_i \in Spec(\Delta_k) \setminus \{0 \} } e^{-t \lambda_i} \frac{1}{\lambda_i} \langle  P_kv_{\lambda_{i}}, T_{k+1}^* P_kv_{\lambda_{i}} \rangle.
\end{equation}
The definition of the dual complex shows us that the co-exact (exact) elements of the complex $\mathcal{P}$ of degree $k$ are the same as the exact (co-exact) elements of the dual complex $\mathcal{P}^*$ of degree $n-k$. This allows us to write
\begin{equation}
\label{hypasist_1}
   \mathcal{L}(\mathcal{P},T)(b,t)=L(\mathcal{P},T)(b)+ (b^{-1}+1) \sum_{k=1}^{n} b^k \widetilde{S}_k(t)
\end{equation}
using equation \eqref{equation_with_the_b}.
Moreover the cohomologies are isomorphic (in dual degrees), as expressed in Proposition \ref{Kernel_equals_cohomology}. This shows that \begin{equation*}
    L(\mathcal{P},T)(b)= (b^n)L(\mathcal{P}^*,T^*)(b^{-1})
\end{equation*}
Using this in equation \eqref{hypasist_1}, we see that it is precisely equal to the expansion of $b^n \mathcal{L}(\mathcal{P}^*,T^*)(b^{-1},t)$ using Definition \ref{definition_polynomial_Lefschetz_supertrace} for the dual complex.
The duality for Lefschetz numbers is obtained by setting $b=-1$.
\end{proof}

This will be particularly important when we study local Lefschetz polynomials in Section \ref{section_formulas}. The important point here is that the Laplace-type operator is the same for the elliptic complex and its dual complex, and instead of using the endomorphism in the Lefschetz heat trace, one can use the adjoint of the endomorphism.

\subsection{Hilbert complexes on stratified pseudomanifolds}
\label{subsection_Hilbert_complexes_stratified_pseudomanifolds}

In this article we study wedge elliptic complexes on pseudomanifolds with associated Dirac complexes, where the Dirac-type operators are those on stratified pseudomanifolds. If the Dirac-type operator satisfies the Witt/geometric Witt condition of Definition \ref{Witt_assumption}, we say that the \textbf{\textit{complex satisfies the Witt/geometric Witt condition}}, in which case the operators are essentially self-adjoint. If not we will choose the VAPS domain.

Given an elliptic complex $\mathcal{P}(X)=(L^2(X;E),P)$ on a pseudomanifold $\widehat{X}$, which is associated to a Dirac operator $D=P+P^*$ acting on sections in $L^2(X;E)$, we define the VAPS domain $\mathcal{D}_{VAPS}(P)$ as the graph closure of $P$ with respect to the VAPS domain for the Dirac operator $\mathcal{D}_{VAPS}(D)$. That is
\begin{multline}
\label{VAPS_domain_for_complex}
\mathcal{D}_{VAPS}(P) = \{ u \in L^2(X;E) : \exists (u_n) \subset \mathcal{D}_{VAPS}(D)\\ 
\text{such that } u_n \rightarrow u \; \text{and} \;  (P u_n) \; \text{is} \; L^2- \text{Cauchy} \}
\end{multline}
and the domain for the adjoint is the graph closure of $P^*$ with respect to the VAPS domain for the Dirac operator $\mathcal{D}_{VAPS}(D)$. Taking the domain for the Dirac-type operator as in equation \eqref{Domain_Dirac_first}, one can see that the corresponding domain for the Dirac-type operator $\mathcal{D}_{VAPS}(P+P^*)=\mathcal{D}_{VAPS}(P) \cap \mathcal{D}_{VAPS}(P^*)$ is the same as $\mathcal{D}_{VAPS}(D)$ given in Proposition \ref{equivalent_definition_VAPS_domain}.
This follows from an argument similar to that in Proposition 5.4 of \cite{Albin_hodge_theory_cheeger_spaces}.

\begin{remark}
We will always work with the VAPS domains $\mathcal{D}_{VAPS}(P)$ for elliptic complexes on pseudomanifolds, unless otherwise specified. We will explore some examples with different domains in Section \ref{Dolbeault_section} to indicate how the choices of domains see different aspects of the geometry and topology of the underlying spaces.
\end{remark}

For any wedge Dirac-type operator on a stratified pseudomanifold with a $\mathbb{Z}_2$ grading operator as in Definition \ref{Grading operator} yielding a splitting $E=E^+ \oplus E^-$, there is a corresponding splitting of the Hilbert space $L^2(X,E)=L^2(X,E^+) \oplus L^2(X,E^-)$. 
Since any such grading operator $\gamma$ preserves $\rho_X^{1/2}L^2(X;E)$ and $\mathcal{D}_{\max}(D)$, it preserves the VAPS domain, 
and we can decompose the VAPS domain as 
\begin{equation}
\label{Dirac_domain_bluh}
    \mathcal{D}_{VAPS}(D)=\mathcal{D}_{VAPS}(D)^+ \oplus \mathcal{D}_{VAPS}(D)^-
\end{equation}
where $\mathcal{D}_{VAPS}(D)^\pm$ are the sections with values in $E^\pm$ respectively, and where $D^\pm$ are the restrictions of $D$ to $\mathcal{D}_{VAPS}(D)^\pm$ respectively. 
This grading descends to the kernel of the operator $D$. 

\subsection{The Dolbeault Dirac and spin$^\mathbb{C}$ Dirac complexes}
\label{subsection_Dolbeault_SpinC_introduction}

We briefly discuss the two complexes in the title of this subsection since we will derive specialized results for them.

When the structure group of resolved stratified space $X$ with a wedge metric can be lifted to a spin$^\mathbb{C}$ structure, we can define a natural wedge Clifford algebra which gives rise to the spin$^\mathbb{C}$ Dirac operator on $X^{reg}$ and we refer to \cite{duistermaat2013heat} for details. 
It is well known that for an almost complex manifold of dimension $2n$ equipped with an almost hermitian metric, there is a corresponding spin$^{\mathbb{C}}$ structure and spin$^{\mathbb{C}}$ Dirac operator, and the complex spinor bundle $\mathcal{S}$ is canonically identified with $\oplus_q \Lambda^{0,q}X$, where $\oplus_{p,q} \Lambda^{p,q}X$ is the splitting of the exterior algebra of $T^*X$ in $(p,q)$ types (see, e.g., \cite{epstein2006subelliptic}).

Locally one can write the operator as 
\begin{equation}
    D=\sqrt{2} (\overline{\partial}+\overline{\partial}^*)+\mathcal{M}^{eo}
\end{equation}
where $\mathcal{M}^{eo}$ is a zero-th order operator that intertwines the odd and even degree spinors, vanishing where the metric is locally K\"ahler, in which case $D$ is equal to the Dolbeault Dirac operator $(\overline{\partial}+\overline{\partial}^*)$ (up to constants determined by various conventions).

In the non-K\"ahler case, the square of the spin$^{\mathbb{C}}$ Dirac operator does not preserve the form degree (see \cite[\S 3.1]{duistermaat2013heat}). 
On a general complex manifold, both the spin$^{\mathbb{C}}$ Dirac operator and the Dolbeault Dirac operator have the same principal symbol they have the same Fredholm index and we can compute the Lefschetz numbers for the Dolbeault complex using the spin$^{\mathbb{C}}$ Dirac complex.
For almost complex structures that are non-integrable, we will study Lefschetz numbers for the spin$^{\mathbb{C}}$ Dirac complex.


\section{Geometric endomorphisms}
\label{section_Geometric_endomorphism}

In this section we introduce endomorphisms of elliptic and Dirac complexes which are associated to self maps of geometric spaces which we call geometric endomorphisms. 
We defined endomorphisms of complexes in Definition \ref{geometric_endo}.

\begin{definition}
\label{definition_geometric_endomorphism_proper}
An endomorphism $T$ of an elliptic complex $\mathcal{P}(X)=(L^2(X;E_k),P)$ on a stratified pseudomanifold $\widehat{X}$ is a \textbf{\textit{geometric endomorphism}} if there is a smooth self map $f:X \rightarrow X$ and smooth bundle morphisms $\varphi_k : f^*E_k \rightarrow E_k$ such that 
\begin{equation}
    T_{k} S=\varphi_k (f^* S)
\end{equation}
for sections $S \in \Gamma(E_k)$. We shall denote the geometric endomorphism associated to a map $f$ by $T_f$.
\end{definition}

This definition is a natural generalization of the definition for smooth manifolds given on page 377 of \cite{AtiyahBott1}. For instance given a self-map $f$ of a smooth manifold, the pullback on forms is of the form $\phi_k \circ f^*$, where $\phi_k=\Lambda^{k}df$.

In the words of Atiyah and Bott, \textit{there may not be a natural construction for $\varphi_k$, and their existence must be postulated. In many geometrically interesting examples, however, there is a natural choice for the $\varphi_k$.}

In the setting of smooth manifolds, a smooth self map on the space induces a geometric endomorphism in this way on the complex of smooth forms. However when considering $L^2$ cohomology, looking at different extensions of domains of operators on stratified pseudomanifolds, the pullback by a self map no longer induces a bounded map on the domain unless some additional conditions are met by the self map. 
We will work with the assumption that the lift of the self map to the resolved space is a local diffeomorphism (except in Section \ref{section_de_Rham}, where we will show how this can be relaxed for the de Rham complex).
This assumption is used in the literature when studying Lefschetz numbers on spaces with isolated conic singularities \cite{Bei_2012_L2atiyahbottlefschetz}, and we shall use many ideas in that article.

In Section \ref{section_de_Rham}, we construct a modified geometric endomorphism which relaxes the conditions demanded on the self map and we shall give detailed proofs of why said modified geometric endomorphism is indeed an endomorphism on the de Rham complex. 

We begin by reviewing some of the more commonly studied self maps on stratified pseudomanifolds and the properties of their induced maps on cohomology.

\subsection{Self maps on stratified pseudomanifolds}

In order to construct suitable endomorphisms associated to self maps on a stratified pseudomanifold, we need that the self maps \textit{respect} the stratification. We introduce notions of such maps which appear commonly in the literature. 

Let $\widehat{X}, \widehat{Y}$ be two compact stratified spaces, $X, Y$ be the respective resolved manifolds with corners and $\widehat{X}^{reg}, \widehat{Y}^{reg}$ be the regular parts of the respective spaces. Following \cite{Kirwan&woolf_2006_book} Definition 4.8.1, and \cite{Albin_signature} we introduce the following definitions. 
\begin{definition}

\label{stratum_preserving}
A map ${\widehat{f} : \widehat{X} \longrightarrow \widehat{Y}}$  between stratified spaces is called \textit{\textbf{stratum preserving}} if 
\begin{equation}
{S \in \mathcal{S}_{\widehat{Y}} \implies \widehat{f}^{-1}(S)} \text{ is a union of strata of $\widehat{X}$}.
\end{equation}
A stratum preserving map is called \textit{\textbf{placid}} if for each stratum $S \in \mathcal{S}_{\widehat{Y}}$, 
\begin{equation*}
    \text{codim} \widehat{f}^{-1}(S) \geq \text{codim}(S)
\end{equation*}
and is called \textbf{\textit{codimension preserving}} if 
\begin{equation*}
    \text{codim} \widehat{f}^{-1}(S) = \text{codim}(S).
\end{equation*}
\end{definition}

\begin{definition}
A \textit{\textbf{stratum-preserving homotopy equivalence}} between topological pseudomanifolds $\widehat{X}, \widehat{Y}$ with chosen stratifications $\widehat{X}^j, \widehat{Y}^j$ is a pair of codimension preserving maps $\widehat{f} : \widehat{X} \longrightarrow \widehat{Y},\widehat{g} : \widehat{Y} \longrightarrow \widehat{X}$, such that $\widehat{f} \circ \widehat{g}$ and $\widehat{g} \circ \widehat{f}$ are homotopic to the identity via codimension preserving homotopies
\begin{equation*}
    \widehat{h}: \widehat{X} \times [0,1] \rightarrow \widehat{X} \text{  and  } \widehat{k}: \widehat{Y} \times [0,1] \rightarrow \widehat{Y}.
\end{equation*}
\end{definition}
Friedman shows that stratum preserving homotopy equivalent maps induce the same map at the level of intersection homology (see Chapter 4 of \cite{friedman2020singular}). 
Recall that we use the notation
\begin{center}
    $\rho_X = \prod_{H \in \mathcal{M}_1(X)} \rho_H$
\end{center}
for a total boundary defining function. A \textbf{multiweight} for $X$ is a map
\begin{center}
    $\mathfrak{s} : \mathcal{M}_1(X) \rightarrow \mathbb{R} \cup \{ \infty \} $
\end{center}
and we denote the corresponding product of boundary defining functions by 
\begin{center}
    $\rho_X^{\mathfrak{s}} = \prod_{H \in \mathcal{M}_1(X)} \rho_H^{\mathfrak{s}(H)}$
\end{center}
We write $\mathfrak{s} \leq \mathfrak{s'}$ if $\mathfrak{s}(H) \leq \mathfrak{s'}(H)$ for all $H \in \mathcal{M}_1(X)$.
A smooth map between manifolds with corners $f : X \rightarrow Y$ is a \textbf{b-map} if, for each $H \in \mathcal{M}_1(X)$, and some choice of boundary defining functions, we have
\begin{equation}
    f^*\rho_H = \prod_{H \in \mathcal{M}_1(X)} \rho_G^{e_f(H,G)}
\end{equation}
where $e_f(H,G)$ is a non-negative integer. These are called `interior b-maps’ in \cite{melrose1992calculus} because they map the interior of the domain into the interior of the target.
The map 
\begin{center}
    $e_f : \mathcal{M}_1(X) \times \mathcal{M}_1(Y) \rightarrow \mathbb{N}_{0}$
\end{center}
is known as the \textbf{exponent matrix} of the b-map $f$ and we write
\begin{center}
    $ker(e_f)= \{ H \in \mathcal{M}_1(X) : e_f(H,G)=0$ for all $G \in \mathcal{M}_1(Y) \}$
\end{center}
If the exponent matrix of a $b$ map has only entries $0,1$, we say that it is a \textit{\textbf{simple $b$-map}}.

\begin{definition}
\label{smoothly stratified}
A \textit{\textbf{smoothly stratified self-map}} $\widehat{f}$ of a stratified space $\widehat{X}$, is a continuous map $\widehat{f} : \widehat{X} \rightarrow \widehat{X}$ sending the regular part $X^{reg}$ of the space to itself, and for which there exists a lift $f : X \rightarrow X$, a b-map of manifolds with corners, such that $\widehat{f} \circ \beta = \beta \circ f$ (where $\beta: X \rightarrow \widehat{X}$ is the blow-down map corresponding to the resolution) which is compatible with the iterated fibration structures. 
\end{definition}

\begin{remark}[Nomenclature conventions]
If a smoothly stratified map is the lift under $\beta$ of a \textit{stratum preserving, placid} or \textit{codimension preserving} map, we refer to the lifted b-map by the same adjectives. The lift of a \textit{stratum-preserving homotopy equivalence} will similarly be refered to as a smoothly stratum-preserving homotopy equivalence.

\end{remark}

\begin{remark}
In \cite{Albin_signature}, codimension preserving is also referred to as strongly codimension preserving. Definition 4.8.7 of \cite{Kirwan&woolf_2006_book} is a version of the definition of \textit{stratified maps} which does not use the resolution of the space.
\end{remark}

In this article we work with self maps $f$ that are simple b-maps. This is similar to the condition imposed on maps $f$ on manifolds with cylindrical ends in \cite{seyfarth1991lefschetz} where they prove a Lefschetz fixed point formula. A simple computation shows that pull-back by a simple $b$-map which is a local diffeomorphism induces a bounded map in $L^2$, and we shall see that there are many interesting examples which satisfy this condition.

\begin{remark}[convention]
\label{unit_exponent}
From now on, unless otherwise stated, we will assume that the self maps we study are simple b-maps and are stratum preserving. 
\end{remark}

Manifolds with fibered corners and their morphisms are studied in detail in \cite{kottke2022products}. In particular, the manifolds with corners with iterated fibration structures that we study are \textit{interior maximal} manifolds with fibered corners. Morphisms in the category of manifolds with fibered corners are interior b-maps $f:X \rightarrow Y$ that are simple b-maps, \textit{b-normal, ordered, fibered and consistent}. The authors of the article note that the consistency is automatically satisfied in the interior maximal case, and following the definitions it is easy to see that so is the fibered condition. We briefly describe the other terms and their definitions, referring to \cite{kottke2022products} for more details.

Given a b-map $f: X \rightarrow Y$, the differential extends by continuity to a bundle map $\prescript{b}{}{df}_*: \prescript{b}{}{NX} \rightarrow \prescript{b}{}{NY}$ over $f$ called the b-differential, which restricts over each $E \in \mathcal{M}(X)$ to a bundle map
\begin{equation}
    \prescript{b}{}{df}_*: \prescript{b}{}{NE} \rightarrow \prescript{b}{}{NF}, \quad F=f_{\sharp}(E) \subset \mathcal{M}(Y).
\end{equation}
A b-map is called b-normal if the map above is surjective on fibers for all boundary hypersurfaces $E \in \mathcal{M}_1(X)$. The set $\mathcal{M}_{1,0}(X):=\mathcal{M}_1(X) \cup \mathcal{M}_0(X)=\mathcal{M}_0(X) \cup \{X\}$ is called the set of principal faces and an interior b-map is called ordered if $f_{\sharp}:\mathcal{M}_{1,0}(X) \rightarrow \mathcal{M}_{1,0}(Y)$ is a map of ordered sets (this corresponds to the placid condition, as one case see from Definition 3.1 of \cite{kottke2022products}). The discussion preceding Theorem 5.1 in the introduction of that article shows that the graph of any \textit{morphism} $f:X \rightarrow X$ is a p-submanifold. 

\begin{remark}
\label{remark_Fibered_morphism}
Except for the Hilsum-Skandalis geometric endomorphism that we introduce in Section \ref{section_de_Rham}, we will study geometric endomorphisms associated to self maps $\widehat{f}: \widehat{X} \rightarrow \widehat{X}$ that are codimension preserving local homeomorphisms of Thom-Mather stratified pseudomanifolds, each of which lifts to a unique (see \cite[\S 2.4]{Albin_signature}) self map $f:X \rightarrow X$ of the resolved space.
It is easy to observe that these are morphisms in the sense described above, and therefore the graph of $f$ is a p-submanifold.
\end{remark}

\subsection{Mapping properties of geometric endomorphisms}
\begin{proposition}
\label{Geometric_Endomorphism_Diffeo}
Let $\widehat{X}$ be a stratified pseudomanifold with an elliptic complex $\mathcal{P}=(H=L^2(X;E),P)$, and let $f: X \rightarrow X$ be a simple b-map and a local diffeomorphism. Let $\varphi_k: f^*E_k \rightarrow E_k$ 
be smooth bundle homomorphisms such that $T=T_f=\varphi \circ f^*$ commutes with the operator $P$ 
on $C^{\infty}_c(\widehat{X}^{reg};E_*)$. Then $T_k= \varphi_k \circ f^*$
is a geometric endomorphism of the elliptic complex.
\end{proposition}

\begin{proof}
What we need to prove is that $T_f$ extends to a bounded map of $L^2(X;E_*)$ which preserves $\mathcal{D}_{VAPS}(P)$.
Our proof follows the outline of \cite{Bei_2012_L2atiyahbottlefschetz}, and we give straightforward generalizations of Lemma 2 and Proposition 10 of that article.

\begin{lemma}
\label{lemmata_mapping_properties}
In the same setting as in the proposition, the geometric endomorphism satisfies the following properties:
\begin{enumerate}

    \item For each k, $T_k$ extends as a bounded operator from ${L^2(X;  E_k)}$ to itself (we denote this extension by $T_k$ as well).
    
    \item For each k, let ${T_k^* : L^2(X;  E_k) \rightarrow L^2(X; E_k) }$ be the adjoint of $T_k$.\\
    Then ${T_k^* :  C^{\infty}_c(\widehat{X}^{reg};  E_k) \rightarrow  C^{\infty}_c(\widehat{X}^{reg};  E_k) }$. 
\end{enumerate}
\end{lemma}

\begin{proof}
The first property follows from the fact that $f$ is a local diffeomorphism on the resolved pseudomanifold. 

The metric on the bundle $E_k$ pulls back to a metric on the pullback bundle $f^*E_k$, and we can define the adjoint map $\varphi_k^* :E_k \rightarrow f^*E_k$ as the adjoint of $\varphi_k$ in each fiber, with respect to the pointwise inner product. 

The pullback of the volume form $f^*dvol_g$ can be written as $h \cdot dvol_g$ for a smooth non-vanishing function $h$, which is positive if $f$ is orientation preserving and negative if orientation reversing. 
We can define an operator 
\begin{equation}
\label{equation_Bei_adjoint}
    S_k(u)(\xi) = \varepsilon h(\xi) (\varphi_k^* \circ (f^{-1})^*) (u)(\xi)
\end{equation}
for all $u \in L^2(X; E_k)$, where we take $\varepsilon$ to be 1 if $f$ is orientation preserving and $-1$ if it is reversing.
If $f$ is not a diffeomorphism and only a local diffeomorphism then since $\widehat{X}$ is compact, there is a covering space $\pi: X_2 \rightarrow \widehat{X}^{reg}$ with finite covering degree $c$ such that the elliptic complex lifts to $X_2$, and 
\begin{equation}
    \langle \phi_k \circ f^* v , u \rangle_{L^2(X,E_k)} = \frac{1}{c} \langle \phi_k \circ \widetilde{f}^*v , \pi^*u \rangle_{L^2(X_2, \pi^*E_k)} 
\end{equation}
for all $u, v \in L^2(X; E_k)$, where the $\phi_k$ on the right hand side is the bundle morphism on the lift, and $\widetilde{f}: X^{reg} \rightarrow X_2$ is the unique diffeomorphism which satisfies $\widetilde{f} \circ \pi=f$. We can therefore assume for the rest of the proof that $f$ is a diffeomorphism.
By construction we have the equality
\begin{equation}
    \langle T_k v , u \rangle_{L^2(X,E_k)} = \langle v , S_k u \rangle_{L^2(X,E_k)} 
\end{equation}
for all $u, v \in L^2(X; E_k)$, which shows that over $C^{\infty}_c(X^{reg},  E^k)$, $T_k^*$ coincides with $S_k$ from which the second property follows immediately.
\end{proof}

Now we will show that $T_k$ gives rise to geometric endomorphisms on the complexes of minimal and maximal extensions, generalizing Proposition 10 of \cite{Bei_2012_L2atiyahbottlefschetz}.

By definition, for $s \in \mathcal{D}_{min}(P_k)$, there exists a sequence $\{ s_j \}$ with $s_j \in C^{\infty}_c(\widehat{X}^{reg},  E_k)$ such that $s_j \rightarrow s$ in $L^2(X,E_k)$, and $P_ks_j \rightarrow P_ks$ in $L^2(X,E_{k+1})$. It follows that $T_k(s) \in \mathcal{D}_{min}(P_k)$ since $T_k s_j \rightarrow T_k s $ in $L^2(X,E_{k})$, and $P_k(T_ks_j)=T_{k+1}P_k(s_j) \rightarrow T_{k+1}P_ks$ in $L^2(X,E_{k+1})$. This also shows that $P_k T_k s =T_{k+1} P_k s$, i.e., that $T$ is a geometric endomorphism for the complex with the minimal domain of $P$.

Notice that by property 2 of the lemma we just proved, the arguments used for the minimal domain show that $P^*$ commutes with $T^*$ on the minimal domain of $P^*$. Since the maximal domain is defined as the adjoint of the minimal domain, this shows that $P$ commutes with $T$ on the maximal domain of $P$.

Since $P$ commutes with $T$ on its maximal domain, it commutes with the operator on any domain that is preserved by the geometric endomorphism. In Subsection \ref{subsection_Hilbert_complexes_stratified_pseudomanifolds} we defined the VAPS domain for complexes satisfying the Witt condition.

Since we assume that the map $f$ is a simple b-map, elements of the weighted $L^2$ space $\rho^{1/2}_X L^2(X;E)$ are mapped to itself. The same argument that showed that the minimal domain was preserved now shows that the VAPS domain is preserved, since it can be defined as a graph closure as seen by Proposition \ref{equivalent_definition_VAPS_domain}. This shows that the geometric endomorphism preserves the VAPS domain, proving the result.
\end{proof}

In Definition \ref{geometric_endo} we defined the adjoint endomorphism $T^*$ of an endomorphism $T$, and it is easy to see that it commutes with the operator $P^*$ on the adjoint domain. We observe that the second property of Lemma \ref{lemmata_mapping_properties} shows that the adjoint endomorphism of a geometric endomorphism can be realized as a \textit{geometric} endomorphism of the dual complex. In the case where $f$ is a diffeomorphism on the resolved pseudomanifold, and given a complex $\mathcal{P}$ and its adjoint complex $\mathcal{P}^*$ (with the VAPS conditions), for all $u \in L^2(X;E)$, for all $\xi \in X$, 
\begin{equation}
\label{equation_dual_complex_adjoint_geometric_endo_by_inverse}
    T^{P^*}_{f^{-1}}(u)(\xi)=\varphi^{\dag} \circ (f^{-1})^*:=(T^{P}_f)^*=\varepsilon h(\xi) (\varphi^* \circ (f^{-1})^*) (u)(\xi)
\end{equation}
where $h$ and $\varepsilon$ are as explained in the proof above. This defines the bundle morphism $\varphi^{\dag}$ such that given a self map $g$, $T^{P^*}_g= \varphi^{\dag} \circ g^*$ is a geometric endomorphism for the dual complex, in particular commuting with $P^*$.
Let us study the adjoint geometric endomorphism for the de Rham complex. 

\begin{example}[Case of the de Rham complex]
\label{remark_notation_bundles_vs_density_bundles_2}
For the de Rham complex $\mathcal{P}=(L^2\Omega(X),d)$, given a diffeomorphism of resolved stratified pseudomanifolds $f:X \rightarrow X$,
\begin{equation*}
\int_X f^*u \wedge \star v =\int_X (f^{-1})^*(f^*u \wedge \star v) =\int_X u \wedge (f^{-1})^*(\star v)
\end{equation*}
shows that the adjoint $T_f^*$
is $\star^{-1} (f^{-1})^* \star$. It is clear that this operator commutes with $\delta$, which is $\star^{-1} d \star$ up to a sign. 
Hence, for a self map $g: X \rightarrow X$, we have a geometric endomorphism on the complex $\mathcal{P}^*=(\Omega,\delta)$ given by $T^{\delta}_g:=\star^{-1} (g)^* \star$, which can be identified with $T_{g^{-1}}^*$ (the adjoint geometric endomorphism of $T_{g^{-1}}$ on the complex $\mathcal{P}$), since the underlying Hilbert spaces are the same.
\end{example}

\begin{remark}
One reason that we expound on the adjoint operator and the dualities is due to its importance when studying the space with different boundary conditions for the complex and the dual complex, and the role played by duality in deriving local Lefschetz formulas.
\end{remark}


\section{Lefschetz numbers for stratified pseudomanifolds}
\label{section_formulas}

In this section we focus on deriving formulae for local Lefschetz numbers at fixed points of self maps $f:X \rightarrow X$ of resolved stratified pseudomanifolds that give rise to geometric endomorphisms of wedge elliptic complexes.

In the first subsection, we will explain why Lefschetz numbers can be computed as sums of local Lefschetz numbers at fixed points using heat kernel localization. We will show that for wedge elliptic complexes associated to Dirac-type operators, the Lefschetz numbers depend only on the self map induced at the tangent cone $\mathfrak{T}_pf$. This subsection uses results developed in \cite{Albin_2017_index} using functional calculus on infinite cones to describe model heat kernels. 

In the next subsection we study elliptic complexes at fundamental neighbourhoods of fixed points, where we pick domains with boundary conditions. When the choices of domain for the elliptic complexes yield a trace class heat kernel, we can use Theorem \ref{Lefschetz_supertrace} but this needs to be replaced for non-trace class heat kernels. By using functional calculus on truncated tangent cones with boundary conditions, we develop renormalized traces which can be used to define the local Lefschetz numbers at a fixed point for the de Rham and Dolbeault complexes.

\subsection{Heat Kernel Localization}
\label{subsection_heat_kernel_localization}

One important phenomenon used to study local Lefschetz numbers is \textit{heat kernel localization}. We will first review the heat kernel construction for stratified pseudomanifolds in \cite{Albin_2017_index}, and then use this to show that the local Lefschetz numbers only depend on the map induced at the tangent cone.

\subsubsection{Heat kernel on stratified pseudomanifolds}

It is well known that the heat kernel on a smooth manifold is a distributional section on $X \times X \times \mathbb{R}^+$, where the singularities are at the diagonal of $X\times X$ at $\{t=0\}$. 
For our purposes, the key is to understand the Lefschetz supertrace in the limit as $t$ goes to $0$, and that can be achieved by first understanding the limit of the heat kernel as $t$ goes to $0$. In the smooth case, Melrose introduced a \textit{parabolic blow up} of the heat space at the singular set using the scaling behaviour of the operator in the limit, which allows one to make sense of the limits important for index theory \cite[\S7]{melroseAPS}. More complicated constructions have been done to understand these limits in various singular spaces. 

In the breakthrough work of Atiyah and Bott, the formulas for local Lefschetz numbers were computed using \textit{nice asymptotics} of elliptic parametrices. Heat kernel asymptotics have since been used in the smooth setting and a nice exposition of this method is given in Theorem 10.12 of \cite{roe1999elliptic}. One of the key ideas used is that the heat kernel has exponential decay estimates in $t$, away from the diagonal of the heat space, as $t$ goes to $0$. This is used to get estimates for the Lefschetz heat supertraces.

In \cite{Albin_2017_index}, the authors show that the heat kernel lifts to a suitably resolved wedge heat space. 
We now briefly introduce the edge double space corresponding to a stratified pseudomanifold.
Given a manifold with corners and an iterated fibration structure $X$, edge pseudodifferential operators can be described via their integral kernels on a replacement of $X^{2}$ that takes the iterated fibration structure into account.

A submanifold $W \subseteq X$ of a manifold with corners is a p-submanifold if every point in $W$ has a neighborhood $\mathcal{U}$ in $X$ such that
$$
\begin{gathered}
X \cap \mathcal{U}=X^{\prime} \times X^{\prime \prime}, \text { where } \partial X^{\prime \prime}=\emptyset, \\
W \cap \mathcal{U}=X^{\prime} \times\left\{p^{\prime \prime}\right\} \text { for some } p^{\prime \prime} \in X^{\prime \prime} .
\end{gathered}
$$
The reader is referred to \cite[Appendix B]{EpsteinMelroseMendozaResolvent1991} for more details.
The radial blow-up of a manifold with corners $X$ along a p-submanifold $W$ is the manifold with corners $[X ; W]$ obtained by replacing $W$ with the inward pointing part of its spherical normal bundle, see e.g., \cite[\S4.2]{melroseAPS}, \cite[\S2.2]{mazzeomelroseEta}, \cite[Chapter 5]{melrose1996differential}.

Recall that there is a partial order on $\mathcal{S}(X)$, $Y <Y^{\prime}$ if and only if $\mathfrak{B}_{Y} \cap \mathfrak{B}_{Y^{\prime}} \neq \emptyset$ and $\operatorname{dim} Y<$ $\operatorname{dim} Y^{\prime}$. The edge double space associated to $X$ is obtained from $X^{2}$ by blowing-up certain p-submanifolds. For each $Y \in \mathcal{S}(X)$ we denote the the fiber diagonal of $\phi_{Y}$ in $X^{2}$ by

$$
\mathfrak{B}_{Y} \times_{\phi_{Y}} \mathfrak{B}_{Y}=\left\{\left(\zeta, \zeta^{\prime}\right) \in\left(\mathfrak{B}_{Y}\right)^{2}: \phi_{Y}(\zeta)=\phi_{Y}\left(\zeta^{\prime}\right)\right\}.
$$

We now review the construction of the wedge heat space. 
Starting with the space $X^{2} \times \mathbb{R}_{t}^{+}$ we can blow-up $\{t=0\}$ parabolically so that $\tau=\sqrt{t}$ is a smooth function. As in \cite{Albin_2017_index}, will not include this blow-up explicitly but simply change the notation to $X^{2} \times \mathbb{R}_{\tau}^{+}$.

\begin{definition}
\label{definition_heat_space_1}
Let $X$ be a manifold with corners and an iterated fibration structure and $\left\{Y_{1}, \ldots, Y_{\ell}\right\}$ a non-decreasing listing of $\mathcal{S}(X)$. The \textit{\textbf{wedge heat space}} of $X$ is
$$
H X_{\mathrm{w}}=\left[X^{2} \times \mathbb{R}_{\tau}^{+} ; \mathfrak{B}_{Y_{1}} \times_{\phi_{Y_{1}}} \mathfrak{B}_{Y_{1}} \times\{0\} ;...; \mathfrak{B}_{Y_{\ell}} \times_{\phi_{Y_{\ell}}} \mathfrak{B}_{Y_{\ell}} \times\{0\} ; \operatorname{diag}_{X} \times\{0\}\right].
$$
\end{definition}
This is called the intermediate wedge heat space in \cite{Albin_2017_index}, and it is noted there that this suffices to describe the heat kernel of the Laplace-type operators we study on X as conormal distributions with bounds.
Implicit in the definition of $H X_{\mathrm{w}}$ is the fact that for $\widetilde{Y} \in \mathcal{S}(X)$ of depth $k$, the interior lift of $\mathfrak{B}_{\widetilde{Y}} \times_{\phi_{\widetilde{Y}}}\mathfrak{B}_{\widetilde{Y}}\times \{0\}$ to the space in which the $\mathfrak{B}_{Y} \times_{\phi_{Y}}\mathfrak{B}_{Y}\times \{0\}$ have been blown up for all $Y<\widetilde{Y}$ is a p-submanifold. 

Local coordinates near $\mathfrak{B}_{Y}$, say $x, y, z$ where $x$ is a bdf for $\mathfrak{B}_{Y}, y$ are coordinates along $Y$ and $z$ coordinates along the fiber $Z$ of $\phi_{Y}$, induce coordinates $x, y, z, x^{\prime}, y^{\prime}, z^{\prime}$ near $\mathfrak{B}_{Y} \times \mathfrak{B}_{Y}$, in which
$$
\mathfrak{B}_{Y} \times_{\phi_{Y}} \mathfrak{B}_{Y} \times \{0\} =\left\{x=x^{\prime}=0, y=y^{\prime}, \tau=0 \right\}
$$
and so this is a p-submanifold whenever $Y$ is a closed manifold (e.g., for $Y_{1}$).
If $Y<\widetilde{Y}$, so that $\widetilde{Y}$ has a collective boundary hypersurface $\mathfrak{B}_{Y \widetilde{Y}}$ as in Definition \ref{iterated_fibration_structure}, let us label the fibers of these fiber bundles,
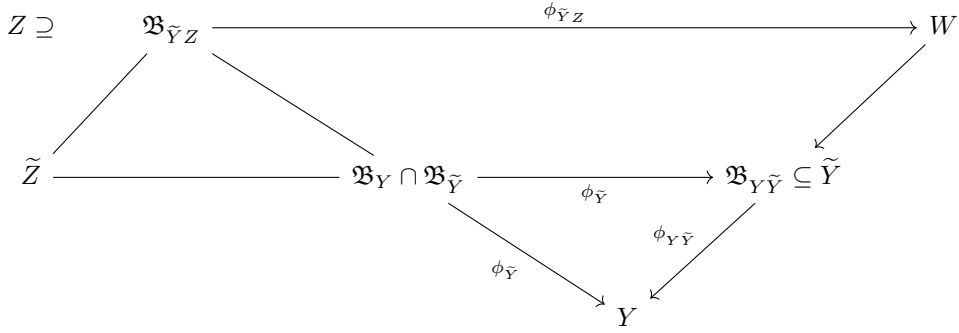
\[\begin{tikzcd}
	{Z \supseteq } & {{\mathfrak{B}_{\widetilde{Y}Z}}} &&&&&& W \\
	\\
	{\widetilde{Z}} &&& {\mathfrak{B}_Y \cap \mathfrak{B}_{\widetilde{Y}}} &&& {\mathfrak{B}_{Y\widetilde{Y}} \subseteq \widetilde{Y}} \\
	\\
	&&&&& Y
	\arrow["{\phi_{\widetilde{Y}Z}}", from=1-2, to=1-8]
	\arrow[no head, from=1-2, to=3-4]
	\arrow[from=1-8, to=3-7]
	\arrow["{\phi_{\widetilde{Y}}}"', from=3-4, to=3-7]
	\arrow["{\phi_{\widetilde{Y}}}"', from=3-4, to=5-6]
	\arrow["{\phi_{Y\widetilde{Y}}}"',from=3-7, to=5-6]
	\arrow[no head, from=3-4, to=3-1]
	\arrow[no head, from=3-1, to=1-2]
\end{tikzcd}\]
and choose coordinates near $\mathfrak{B}_{Y} \cap \mathfrak{B}_{\widetilde{Y}}$ of the form $x, y, w, r, \widetilde{z}$,
in which $x$ is a bdf for $\mathfrak{B}_Y$ and $r$ is a bdf for $\mathfrak{B}_{\widetilde{Y}}, y$ are coordinates along $Y$, $w$ coordinates along $W$ and $\widetilde{z}$ coordinates along $\widetilde{Z}$, so that $(x, y, w)$ are coordinates along $\widetilde{Y}$ and $(w, r, \widetilde{z})$ are coordinates along $Z$. In the induced coordinates $x, y, w, r, \widetilde{z}, x^{\prime}, y^{\prime}, w^{\prime}, r^{\prime}, \widetilde{z}^{\prime}$, we have
$$
\mathfrak{B}_{Y} \times_{\phi_{Y}} \mathfrak{B}_{Y} \times \{ 0 \}=\left\{x=x^{\prime}=\tau=0, y=y^{\prime} \right\}, \mathfrak{B}_{\widetilde{Y}} \times_{\phi_{\widetilde{Y}}} \mathfrak{B}_{\widetilde{Y}} \times \{0\}=\left\{r=r^{\prime}=\tau=0,(x, y, w)=\left(x^{\prime}, y^{\prime}, w^{\prime}\right) \right\}
$$
which shows that the latter is not a p-submanifold. After blowing-up the former, projective coordinates with respect to $x^{\prime}$ are given by
\begin{equation}
    \label{More_projective_coordinates}
s=\frac{x}{x^{\prime}}, \quad u=\frac{y-y^{\prime}}{x^{\prime}}, \quad w, \quad r, \quad \widetilde{z}, \quad x^{\prime}, \quad y^{\prime}, \quad w^{\prime}, \quad r^{\prime}, \quad \widetilde{z}^{\prime}, \quad T=\frac{\tau}{x^{\prime}}
\end{equation}
and the interior lift of $\mathfrak{B}_{\widetilde{Y}} \times_{\phi} \mathfrak{B}_{\widetilde{Y}} \times\{\tau=0\}$ is given by
$$
\left\{T=0, r=r^{\prime}=0, s=1, u=0, w=w^{\prime}\right\},
$$
which is a p-submanifold.
We denote the blow-down map by
$$
\beta_{(H)}: H X_{\mathrm{w}} \longrightarrow X^{2} \times \mathbb{R}_{\tau}^{+}
$$
and its composition with the projections onto the left or right factor of $X$ by $\beta_{(H), L}, \beta_{(H), R}$ respectively, which fit in the commutative diagram
\[\begin{tikzcd}
	&& {H X_{\mathrm{w}}} \\
	\\
	X \times \mathbb{R}_{\tau}^{+} && {X^{2} \times \mathbb{R}_{\tau}^{+}} && X \times \mathbb{R}_{\tau}^{+}
	\arrow["{\beta_{(H)}}", from=1-3, to=3-3]
	\arrow["{\beta_{(H), R}}", from=1-3, to=3-5]
	\arrow["{\beta_{(H), L}}"', from=1-3, to=3-1]
	\arrow["{\pi_L}"', from=3-3, to=3-1]
	\arrow["{\pi_R}", from=3-3, to=3-5].
\end{tikzcd}\]
There are boundary hypersurfaces
$$
X^{2} \times\{0\} \leftrightarrow \mathfrak{B}_{00,1}^{(H)}, \quad \operatorname{diag}_{X} \times\{0\} \leftrightarrow \mathfrak{B}_{d d, 1}^{(H)}
$$
and collective boundary hypersurfaces, one for each $Y \in \mathcal{S}(X)$,
$$
\begin{aligned}
& \mathfrak{B}_{Y} \times X \times \mathbb{R}^{+} \leftrightarrow \mathfrak{B}_{10,0}^{(H)}(Y), \quad X \times \mathfrak{B}_{Y} \times \mathbb{R}^{+} \leftrightarrow \mathfrak{B}_{01,0}^{(H)}(Y) \\
& \mathfrak{B}_{Y} \times_{\phi_{Y}} \mathfrak{B}_{Y} \times\{0\} \leftrightarrow \mathfrak{B}_{\phi \phi, 1}^{(H)}(Y).
\end{aligned}
$$
The collection 
$$\mathrm{ff}\left(H X_{\mathrm{w}}\right)=\bigcup_{Y \in \mathcal{S}(X)} \mathfrak{B}_{\phi \phi, 1}^{(H)}(Y)$$
is known as the front face of the wedge heat space.
The faces created by the blow ups are fibre bundles whose fibers are suspended versions of wedge heat spaces. In particular, one can find polyhomogenous expansions at these boundary hypersurfaces separately and is a wedge heat space where the `time' $[0, \infty)_{\tau}$ is compactified along with other normal directions. Proposition 3.5 of \cite{Albin_2017_index} explains the structure of the front faces of the wedge heat space in more detail.

In Section 4 of \cite{Albin_2017_index}, the heat kernel of the square of the Dirac-type operator is constructed and we summarize the key points which we use from that section. In Proposition 4.5 of \cite{Albin_2017_index}, the heat kernel of the model operator at the tangent cone at a point $y \in Y$ is shown to be a conormal distribution with polyhonomogenous expansions at the faces $\mathfrak{B}^{(H)}_{dd,1} \cap \mathfrak{B}^{(H)}_{\phi \phi,1}$.
This is constructed in the projective coordinates for a \textit{vertical heat operator} $\frac{1}{2}\sigma \partial_{\sigma} +\sigma^2 (\mathcal{N}_y(\eth_X))^2$, where $\mathcal{N}_y(\eth_X)$ is the \textit{normal operator} at the tangent cone of the square of a conjugated Dirac-type operator, which on the cone $Z^+$ is simply the Dirac operator for the product type heat kernel.

Theorem 4.13 of that article shows that the heat kernel vanishes away from the diagonal on the wedge heat space as $t$ goes to $0$. Moreover, it vanishes to infinite order in $t$ away from all the boundary hypersurfaces except for $\mathfrak{B}^{(H)}_{dd,1}$ and $\mathfrak{B}^{(H)}_{\phi \phi,1}(Y)$ for all boundary faces $Y \in \mathcal{S}(X)$. 

Moreover, it is shown that the heat kernel at a boundary hypersurface $\mathfrak{B}^{(H)}_{\phi \phi,1}(Y)$ of the front face has leading term given by the model heat kernel (as a right density)
\begin{equation}
\label{model_heat_kernel_product_type_11}
    \mathcal{N}_{\mathfrak{B}^{(H)}_{\phi \phi,1}(Y)}(e^{-t \eth^2_X})=e^{-\sigma^2\Delta_{TY}} e^{-\sigma^2 D^2_{C(\phi_Y)/Y}}
\end{equation}
where $Z^+ \textendash C(\phi_Y)/Y \xrightarrow{\phi^+_Y} Y$ ($C(\phi_Y)$ denotes the mapping cylinder of $\phi_Y$). Here $\Delta_{TY}$ is a Laplace-type operator on the tangent space of $Y$, and $D^2_{C(\phi_Y)/Y}$ corresponds to a Dirac-type operator on each fiber $Z^+$ of the mapping cylinder. We will study the heat kernel of this model operator in more detail below.
Theorem 4.13 of that article also shows that the heat kernel vanishes away from the diagonal to order $\mathcal{O}(t^\infty)$.

Since we are considering isolated fixed points, in the limit as $t$ goes to $0$ we need only study the heat kernel on the strata $Y$, in a lift of a neighbourhood of the fixed point onto the wedge heat space.
At the tangent cone of a point $y \in Y$, we can decompose the heat kernel into the two factors in \eqref{model_heat_kernel_product_type_11}.

\subsubsection{Heat kernel on model cone}
\label{subsubsection_heat_kernels_intro}

We now describe the structure of the heat kernel on the model cone following the discussion in Section 4.2 of \cite{Albin_2017_index}. Since the heat kernel of a product is the product of the heat kernels, we can disregard the $\mathbb{R}^h$ factor of the model wedge and focus on the exact Riemannian cone $Z^+ = \mathbb{R}^+_x \times Z$, $g_{Z^+} =dx^2 +x^2g_Z$.
We denote $E$ pulled back to $Z^+$ by $E$ again (abusing notation) and the corresponding Dirac-type operator by 
\begin{equation}
\label{separation_of_variables_cone}
    \eth_{Z^+} = {cl}(dx) \partial_x +\frac{1}{x}D_Z
\end{equation}
From the discussion that follows in \cite{Albin_2017_index}, we see that 
\begin{center}
    $\mathcal{D}_{VAPS}(\eth^2_{Z^+})= \{ u \in \mathcal{D}_{VAPS}(\eth_{Z^+}) : \eth_{Z^+}u \in \mathcal{D}_{VAPS}(\eth_{Z^+}) \}$
\end{center}
is a self adjoint domain. 
Let $e^{-t\eth^2_{Z^+}}$ be the heat kernel of $(\eth_{Z^+},\mathcal{D}_{VAPS}(\eth^2_{Z^+}))$ considered as a density,
\begin{center}
    $e^{-t\eth^2_{Z^+}}= \widetilde{\mathcal{K}} \mu_R$ 
\end{center}
where $\mu_R$ is the volume measure of the product metric $g_{Z^+} =dx^2 +g_Z$ in the notation of \cite{Albin_2017_index}. 
We can instead write this as $\mathcal{\widetilde{K}} \mu$ where $\mu$ is the volume measure of the wedge metric, and $\mathcal{\widetilde{K}}$ is the heat kernel as a conormal section on the wedge heat space.

\begin{remark}
\label{remark_choice_of_measure_conjugation}
Here we use the volume measure in a fundamental neighbourhood (the degenerate measure), $x^{l}dxdzdy$, whereas, in \cite{Albin_2017_index} the measure is taken to be $dxdzdy$ and there is a factor of $x^{l}$ in the heat kernel.
Our choice of measure is that used in \cite{Bei_2012_L2atiyahbottlefschetz,cheeger1980hodge}.

\end{remark}
From the spectral theorem, we know that $\mathcal{\widetilde{K}}$ is a distribution on the space $\mathbb{R}^+_t \times (Z^+)^2$, such that:
\begin{itemize}
    \item $\text{lim}_{t \rightarrow 0} e^{-t\eth^2_{Z^+}} = \text{Id}$ 
    \item For every $t > 0$, $\mathcal{\widetilde{K}}(t,x,z,x',z')$ is smooth in all of its variables in the interior of $(Z^+)^2 \times \mathbb{R}^+_t$. 
    \item For each t, $e^{-t\eth^2_{Z^+}}$ is a self adjoint operator, and hence $\mathcal{\widetilde{K}}(t,x,z,x',z')=\mathcal{\widetilde{K}}(t, x',z',x,z)$ 
\end{itemize}

The heat kernel on the non-compact tangent cone can be expressed as follows after conjugating the expression in \cite[\S 4.2]{Albin_2017_index}. For an instance of such a conjugation see Proposition 9 of \cite{Bei_2012_L2atiyahbottlefschetz}. The heat kernel on the cone with a homogeneous wedge metric is of the form
\begin{equation}
\label{heat_kernel_exact_cone}
    \sum_{\lambda \in Spec(D_{Z_y})} \frac{(xx')^{\frac{1-l}{2}}}{2 t} I_{\nu_{APS}(\lambda)} \Big( \frac{xx'}{2t} \Big) exp \Big(-\frac{(x^2)+(x')^2}{4t} \Big) \Phi_{\lambda}(z,z')
\end{equation}
where $l$ is the dimension of the link $Z$. Here we use the notation,
\begin{equation}
    \nu_{APS}(\lambda) =\begin{cases}
      - |\lambda -\frac{1}{2}|, & \text{if}\ \lambda \in Spec(D_{Z_y}) \cap (0,\frac{1}{2})  \\
      |\lambda -\frac{1}{2}|, & \text{if }\ \lambda \in Spec(D_{Z_y}) \setminus (0,\frac{1}{2}),
\end{cases}
\end{equation}
where $I_{\nu_{APS}(\lambda)}$ is the modified Bessel function of the first kind, and $\Phi_{\lambda}(z,z')$ is the projection onto the $\lambda$-eigenspace of ${cl}(dx)D_{Z_y}$. This series converges uniformly on compact subsets of $(Z^+)^2 \times \mathbb{R}^+_t$ (see the comments following Example 3.1 of \cite{cheeger1983spectral}).

One of the main properties that is clear from this form and the homogeneity of the cone is the following scaling invariance of the heat kernel. We set
\begin{equation}
\label{Scaling_relation}
S_c: \mathbb{R}^+_t \times (Z^+)^2 \rightarrow \mathbb{R}^+_t \times (Z^+)^2, \text{   }
S_c(t,x,z,x',z')=(c^2t, cx,z,cx',z')
\end{equation}
We will use the same symbol to denote the corresponding scalings on $(Z^+)^2$ and $Z^+$. As $E$ is pulled back from $Z$, it makes sense to pull-back a section of $E$ over $Z^+$ along $S_c$ and it is easy to see that $S_c^*: {C}_c^{\infty}(\mathring{(Z^+)} ; E) \rightarrow {C}_c^{\infty}(\mathring{(Z^+)} ; E)$ extends to a bounded map on $L^2(\mathring{(Z^+)} ; E)$.
It's easy to see (conjugating the version of the operator in \cite[\S 4.2]{Albin_2017_index}, or alternately as in Lemma 1 of \cite{Bei_2012_L2atiyahbottlefschetz}) that 
\begin{equation}
\label{scaling_operator}
    S_c^* \mathcal{\widetilde{K}} = c^{1+l} \mathcal{\widetilde{K}}
\end{equation}
In particular note that $S_{x^{-1}}^* \mathcal{\widetilde{K}} = x^{-(1+l)} \mathcal{\widetilde{K}}$.

This dilation invariance is exploited in the construction of the heat kernel, and the next construction in \cite{Albin_2017_index} is a heat space for the model operator which resolves the singularity in the limit as $t$ goes to $0$. First $\{t=0\}$ is blown up parabolically so that $\tau=\sqrt{t}$ is a generator of the smooth structure. Second, the cone-tip at time $0$ is blown up to obtain the space
\begin{equation}
\label{model_heat_space_at_cone_tip}
    H_1Z^+ =[\mathbb{R}^+_s\times Z \times \mathbb{R}^+_s\times Z \times \mathbb{R}_{\tau}^+; \{s=s'=\tau=0\}] \cong Z \times Z' \times \mathbb{S}^2_+ \times \mathbb{R}^+_R,
\end{equation}
where $\mathbb{S}^2_+=\{(\omega_s,\omega_{s'},\omega_{\tau}) \in [0,\infty)^3 : \omega_s^2 + \omega_{s'}^2 +\omega_{\tau}^2=1 \}$. The blow-down map is
\begin{equation}
    H_1Z^+ \xlongrightarrow{\beta} Z^+ \times Z^+ \times \mathbb{R}^+_t
\end{equation}
\begin{equation}
\label{lifted_scaling_relation}
    (z,z', (\omega_s,\omega_{s'},\omega_{\tau}),R) \mapsto ((R\omega_s,z),(R\omega_{s'},z'), R^2\tau^2)
\end{equation}
and it is observed that the dilation \eqref{scaling_operator} lifts to be simply $R \mapsto cR$, and $\{R=0\}$ is denoted by $\mathfrak{B}_R(H_1Z^+)$.
Instead of using polar coordinates, the projective coordinates
\begin{equation}
\label{equation_key_projective_coordinates}
    r=\frac{s}{s'}, z, s',z', \sigma=\frac{\tau}{s'},
\end{equation}
which are valid away from $\omega_{s'}=0$ are introduced, in which $s'$ is a boundary defining function for $\mathfrak{B}_R(H_1Z^+)$ and $S_c$ is $s' \mapsto cs'$.

At the tangent cone of a point $y \in Y$, we can decompose the leading term of the expansion of the heat kernel at the boundary hypersurface $\mathfrak{B}^{H}_{\phi\phi,1}$ into the two factors in \eqref{model_heat_kernel_product_type_11}, and on the conic factor the model heat kernel is that on the heat space \eqref{model_heat_space_at_cone_tip}, where $\sigma$ can be identified with the projective coordinate $\sigma=\tau/s'$ introduced in \eqref{equation_key_projective_coordinates}. The heat kernel on the smooth factor can be similarly understood, considering it to be the smooth cone over a sphere of one dimension less, or using more conventional methods as used in Theorem 10.12 of \cite{roe1999elliptic} since it is simply the Euclidean heat kernel.
We present the following formal statement, which follows from Theorem 4.13 of \cite{Albin_2017_index}.

\begin{proposition}
\label{proposition_summarizing_heat_kernel_stuff_Pierre_Jesse}
Let $D$ be a Dirac-type wedge operator acting on sections of a Clifford bundle $E$ on the resolution $X$ of a stratified pseudomanifold $\widehat{X}$ with the VAPS domain. Then its heat kernel is a conormal distributional section of $Hom(E)$ on the wedge heat space $HX_{w}$, polyhomogenous at the face  $\mathfrak{B}^{(H)}_{dd,1}$ and each $\mathfrak{B}^{(H)}_{\phi \phi,1}(Y)$, where the leading term at the latter is given by
\begin{equation*}
e^{-\sigma^2{\Delta_{TY}}} e^{-\sigma^2 D^2_{C(\phi_Y)/Y}}
\end{equation*} 
where $D_{C(\phi_Y)/Y}$ at $y \in Y$ is given by $D_{Z^+_y}=\eth_{Z^+}$ from equation \eqref{separation_of_variables_cone}. 
Furthermore, the heat kernel vanishes to infinite order in $t$ as $t$ goes to $0$, away from points on the wedge double space which are in the inverse image of the blow down map of the diagonal in $X^2$.
\end{proposition}

The model heat kernel at the front face at $\mathfrak{B}^{(H)}_{\phi \phi,1}(Y)$ in equation \eqref{model_heat_kernel_product_type_11} is a right density where the volume forms $dvol_{\mathbb{D}^{k}} dvol_Z$ at the front face are included in the density. We can equivalently represent this model heat kernel as
\begin{equation}
    \widetilde{K}_{TY}(\sigma,y,y')dvol_Y \widetilde{K}_Z(r,z,s',z',\sigma) dvol_Z s^lds
\end{equation}
where we have used the measure on the cone as in Remark \ref{remark_choice_of_measure_conjugation}.

In the case of isolated conic singularities, the dilation invariance of the heat kernel on the model cone was used to show that the Lefschetz numbers of ``scalable operators" can be expressed as the zeros of a $\zeta$ function in \cite{Bei_2012_L2atiyahbottlefschetz}, generalizing the result of Atiyah and Bott to cones with product type wedge metrics in fundamental neighbourhoods of isolated fixed points as in Subsection \ref{subsection_local_structures}. We will show that the local Lefschetz number at a \textit{simple} isolated fixed point only depends on the induced self map on the tangent cone, and since the heat kernels on the tangent cone are dilation invariant, we extend the expressions in \cite{Bei_2012_L2atiyahbottlefschetz} to the general case. 

\begin{remark}
    We use a standard abuse of notation in the literature going back to \cite{AtiyahBott1}, denoting by the left hand side of the expression
    \begin{equation}
    T_f \circ \mathcal{K}_{e^{-\Delta t}}(t,x,y):=
    \varphi \circ \mathcal{K}_{e^{-\Delta t}}(t,f(x),y)=\mathcal{K}_{T_f \circ e^{-\Delta t}}(t,x,y)
\end{equation}
the quantity on the right hand side. 
It is understood by convention that the pullback is on the left factor, and that the bundle morphism acts on the left factor (note that the bundle morphism is tensorial).    
\end{remark}


\subsubsection{Heat kernel localization}

Given an elliptic complex $\mathcal{P}$ of wedge elliptic operators on a stratified pseudomanifold $\widehat{X}$ and a geometric endomorphism $T_f=T$, since $L(\mathcal{P}(X), T)$ is independent of $t$, we can take the limit as $t \rightarrow 0$, yielding
\begin{equation}
\label{L definition limit}
    L(\mathcal{P}(X), T)= \lim_{t \rightarrow 0} \sum_{k=0}^n (-1)^k Tr(T_k e^{-t \Delta_k}). 
\end{equation}

\begin{definition}
\label{simple_fixed_points_definition}
If $p$ is an isolated fixed point of a smoothly stratified placid self map of a pseudomanifold 
$\widehat{f}:\widehat{X} \rightarrow \widehat{X}$ which lies on $\widehat{X}^{reg}$, and if $det(Id-Df_p) \neq 0$, where $Df$ is the induced map on the tangent space at $p$, we call $p$ a \textit{\textbf{simple fixed point}}. 
\end{definition}
Notice that the definition of simple is equivalent to saying that the only fixed point of $Df$ on the tangent space at the fixed point is the origin.
If we assume that $f$ has only isolated fixed points, we can prove the following.
\begin{theorem}
\label{localization_of_simple_fixed_points}
Let $\widehat{X}$ be an n dimensional pseudomanifold with a wedge metric g and $\mathcal{P}$, an elliptic complex associated to Dirac-type operators and let $T_f$ 
be a geometric endomorphism associated to a self map $f$ of $X$. If $f$ has only isolated fixed points, then
\begin{equation}
\label{localization}
    L(\mathcal{P}(X),T_f)= \lim_{t \rightarrow 0} \{ \sum_{p \in Fix(f)} \sum_{k=0}^n (-1)^k \int_{U_p} tr(T_k \circ \mathcal{K}^k_t) dvol_g \}
\end{equation}
where $Fix(f)$ is the set of fixed points of $f$ on $\widehat{X}$, and ${U_p}$ is a fundamental neighbourhood of $p$ as in \eqref{fundamental neighbourhood}. Here $\mathcal{K}^k_t$ is the heat kernel of the Laplace-type operator $\Delta_k$ on the resolution $X$.
\end{theorem}

\begin{proof}
The proof is similar to that of Theorem 5 of \cite{Bei_2012_L2atiyahbottlefschetz} which treats the case for spaces with singular strata of dimension 0 (isolated conic singularities). 
For the general case, since we assume that $f$ has only isolated fixed points, we can choose disjoint fundamental neighbourhoods $\widehat{U_p}$ for each fixed point $p$.
Let $V= X - \cup_{p \in Fix(f)} {U_p}$. The equation \eqref{L definition limit} with this decomposition yields 

\begin{equation*}
    L(\mathcal{P}(X), T_f)= \lim_{t \rightarrow 0} \sum_{p \in Fix(f)} \sum_{k=0}^n (-1)^k \int_{U_p} tr(T_k \circ \mathcal{K}^k_t)dvol_g\\ + \lim_{t \rightarrow 0} \sum_{k=0}^n (-1)^k \int_{V} tr(T_k \circ \mathcal{K}^k_t)dvol_g 
\end{equation*}

Then Proposition \eqref{proposition_summarizing_heat_kernel_stuff_Pierre_Jesse} shows that the second term is 0, proving the result.
\end{proof}

This suggests that we can break the Lefschetz number of maps with isolated fixed points into local Lefschetz numbers at each isolated fixed point.

\begin{definition}
\label{definition_local_Lefschetz_numbers}
In the notation above, we define the \textbf{\textit{local Lefschetz number}} at $p \in X$ to be
\begin{center}
    $L(\mathcal{P}(X), T_f, p) = \lim_{t \rightarrow 0} \sum_{k=0}^n (-1)^k \int_{U_p} tr(T_k \circ \mathcal{K}^k_t) dvol_g$
\end{center}
\end{definition}

Now the aim is to find explicit formulas for the Lefschetz number in each stratum. The fixed points at regular strata will have the classical localization formula of Atiyah and Bott. We follow their notation for the induced linear maps at the vector bundles on the fixed point $q$, denoting it as
\begin{equation}
    \varphi_{k,p} : (E_k)_{f(p)=p} \longrightarrow (E_k)_p.
\end{equation}
This means that the trace of $\varphi_{k,p} \circ f^*$ is well defined on the vector space $(E_k)_p$.

\begin{theorem}
\label{Nonsingular_Contribution}
Under the assumptions of Theorem \ref{localization_of_simple_fixed_points}, consider a smooth isolated simple fixed point $p \in X^{reg}$ and let $D_pf$ be the differential of $f$ at $p$.
We have that 
\begin{equation}
    L(\mathcal{P}(X),T_f, p) = \sum_{k=0}^n \frac{(-1)^k Tr(\varphi_{k,p} \circ D_pf)}{|det(Id-D_pf)|}
\end{equation}
\end{theorem}

\begin{proof}
It is clear that this is a local expression and must be the same expression as that on a compact smooth manifold, given in Theorem A of \cite{AtiyahBott1}, cf. Theorem 10.12 of \cite{roe1999elliptic}.
\end{proof}

For Dirac complexes, we have the formula 
\begin{equation}
    L(\mathcal{D}(X), T_f, p) = \frac{Tr(\varphi_{+,p} \circ D_pf)-Tr(\varphi_{-,p} \circ D_pf)}{|det(Id-D_pf)|}
\end{equation}
which is equation (8.38) of \cite{AtiyahBott2}.

In the proof given in \cite{roe1999elliptic}, the main ingredients are the asymptotics of the metric, the geometric endomorphism (the bundle morphism and the map self map), as well as the structure of the heat kernel in a neighbourhood of the fixed point, which allows one to compute using Gaussian integrals.

The translation invariance of the heat kernel on Euclidean space is key and allows one to write the heat kernel with the spacial variables appearing only in an exponential factor ($e^{-|x-y|^2/t}$). In the conic case, we only have the dilation invariance, but we can still show that the Lefschetz number only depends on the induced map at the tangent cone. We prove this in the next part of this section. 

\subsubsection{Local Lefschetz numbers on the tangent cone of fixed points}
\label{subsubsection_baruk_khazad}

We will now present some normal forms for functions near fixed points as well as some further simplifications of local Lefschetz numbers when the metric on the fundamental domain is a product type metric on the Euclidean and conic factors. We will show that for an isolated simple fixed point we can assume that the map we are studying is the induced map on the tangent cone.

We first observe that since $f$ is a smooth map of stratified spaces and is in particular continuous, if $p$ is an isolated fixed point then there are fundamental neighbourhoods $\widehat{U_1}, \widehat{U_2}$ around $p$ such that 
\begin{equation}
\label{containment}
    \overline{{U_1}} \subset f^{-1} ({U_2})
\end{equation}
where $U_1,U_2$ are the resolved fundamental neighbourhoods (see Subsection \ref{subsection_local_structures}).
We know from the discussion in Subsection \ref{subsection_local_structures} that this means there are homeomorphisms of the pseudomanifolds which lift to diffeomorphisms
\begin{equation}
\label{fundamental_coordinates}
    \psi_i:U_i \longrightarrow \mathbb{R}^k_{y_i} \times C_{x_i}(Z_{z_i})
\end{equation}
between the resolutions of the respective spaces. For an isolated fixed point $p$ we will take the larger of these two neighbourhoods to be $U$ and the corresponding map to be $\psi$.
We shall call this a \textbf{\textit{fundamental homeomorphism}} for the fundamental neighbourhood at a fixed point $p$, where $Z$ is the link at $p$ and $x$ is the radial variable of the cone. 

The resolution of this neighbourhood corresponds to $\mathbb{R}^k \times [0,1]_x \times Z$, where $Z$ is the resolution of the link $\widehat{Z}$ if it is singular. We will refer to the resolved neighbourhood and the resolved coordinates by the same name.
If we consider placid (in particular stratified) self maps $f$ restricted to the (resolved) fundamental neighbourhoods (which we refer to as $f$ by abuse of notation), this allows us to decompose $f$ locally as three maps, $A, B$ and $C$ as follows:
\begin{gather} 
\label{Decomposition_of_function}
F =\psi \circ f \circ {\psi}^{-1} : \mathbb{R}^k_{y} \times [0,1]_{x} \times Z_{z}
    \longrightarrow \mathbb{R}^k_{y} \times [0,1]_{x} \times Z_{z}\\ 
\hspace{30mm} (y,x,z) \longrightarrow (C(y,x,z), xA(y,x,z), B(y,x,z))
\end{gather}
where $A:U_1 \rightarrow [0,1]_{x}$, $B:U_1 \rightarrow Z_{z}$ and $C: U_1 \rightarrow \mathbb{R}^k_{y}$. We say that the map $F$ is the \textbf{\textit{local model form}} of the map $f$ on the fundamental neighbourhood.

The following generalizes the notions of attracting/contracting and repelling fixed points given in Definition 16 of \cite{Bei_2012_L2atiyahbottlefschetz}.

\begin{definition}
\label{attracting/repelling_fixed_point}
Let $p$ be an isolated fixed point of a map $f$ at a singular strata such that the decomposition of $f$ in a fundamental neighbourhood centered at the fixed point $p$ (with no other fixed points in the neighbourhood) is as in \eqref{Decomposition_of_function}. If one can find fundamental neighbourhoods $U_1,U_2$ as in \eqref{containment}, such that $U_2 \subseteq U_1$, then the fixed point is called \textbf{\textit{non-expanding}} and if the containment $U_2 \subset U_1$ is proper, it is called attracting.
If one can find fundamental neighbourhoods as in \eqref{containment}, such that $U_1 \subseteq U_2$, then the fixed point is called \textbf{\textit{non-attracting}} and if the containment $U_1 \subset U_2$ is proper, it is called expanding.
\end{definition}

\begin{remark}[notational history]
These are versions of the following conditions used by Goresky and MacPherson to prove the analog of Theorem \ref{Shylock} for intersection homology in Section 10 of \cite{Goresky_1985_Lefschetz}. In that article, the term contracting is used instead of attracting. In \cite{Bei_2012_L2atiyahbottlefschetz}, the terms attractive and repulsive are used for the conditions.
\end{remark}

We now prove a product decomposition of Lefschetz numbers that we shall use to simplify the formulas at the tangent cone. 

\begin{proposition}[Lefschetz K\"unneth formula]
\label{Lefschetz_product_formula}
Let $A:M\rightarrow M$ and $B: N \rightarrow N$ be two self maps of resolved spaces $M$ and $N$, with elliptic complexes $\mathcal{P}_M= (H_M,P_M)$ and $\mathcal{P}_N=(H_N,P_N)$ respectively. Then $M\times N$ with the product wedge metric carries a complex $\mathcal{P}_{M}\times \mathcal{P}_{N}=(H_M \bigoplus H_N, P_M+P_N)$. Let $T_A$, $T_B$ and $T_{A\times B}$ be geometric endomorphisms associated to $A$, $B$ and $A\times B$ respectively. Then 
\begin{equation}
    \mathcal{L}(\mathcal{P}_{M}\times \mathcal{P}_{N},T_{A\times B})(b,t) = \mathcal{L}(\mathcal{P}_{M},T_A)(b,t) \cdot \mathcal{L}(\mathcal{P}_{N},T_{B})(b,t)
\end{equation}
when the heat kernels are trace class.
In particular,
\begin{equation}
\label{pint_tile_213}
    L(\mathcal{P}_{M}\times \mathcal{P}_{N},T_{A\times B}) = L(\mathcal{P}_{M},T_A) \cdot L(\mathcal{P}_{N},T_{B})
\end{equation}
\end{proposition}

\begin{proof}
The polynomial Lefschetz heat supertrace, is expressed using heat kernels in Definition \eqref{definition_polynomial_Lefschetz_supertrace}. 
Since the heat kernel on the product is the product of heat kernels, we can write the heat kernel on the product manifold as 
\begin{equation}
    \mathcal{K}_{M\times N}^{k}= \bigoplus_{k=i+j}\mathcal{K}_{M}^{i} \otimes \mathcal{K}_{N}^{j}.
\end{equation}
Since the trace over a tensor product is the product of the traces, and the grading keeps track of the degree of the monomials in $b$, the result follows.

This includes the result for $t \rightarrow \infty$, when the contributions from the eigensections for positive eigenvalues to the traces vanish exponentially and the formula reduces to one in cohomology which is a \textit{polynomial K\"unneth formula}. For $b=-1$, this yields equation \eqref{pint_tile_213}.
\end{proof}

\begin{remark}   
\label{remark_decomposes_as_a_product}
If an elliptic complex on such a product space admits a decomposition as in the statement of the Proposition, we say that \textbf{\textit{the elliptic complex decomposes as a product}}. In particular, on the tangent cone with a product type wedge metric, the normal operator of a Dirac operator always decomposes as a product.
\end{remark}

\begin{remark}
Whenever we take traces to define polynomial Lefschetz heat supertraces or Lefschetz heat supertraces in this Subsection we will assume that the heat kernel is trace-class. However in Theorem \ref{theorem_renormalized_McKean_Singer} we shall introduce renormalized versions of these traces when the heat kernels are not trace class for which the result above can be generalized as well as the other Lefschetz heat traces considered in this section.
\end{remark}

Consider a self map $f$ associated to a geometric endomorphism $T_f$ of an elliptic complex $\mathcal{P}$, with a fixed point $p$ that has a fundamental neighbourhood $U_p=\mathbb{D}^k \times C(Z)$ on which the metric is of product type. We assume that the elliptic complex decomposes as a product into $\mathcal{P}(\mathbb{D}^k) \times \mathcal{P}(C_x(Z))$ in this neighbourhood, and that $f$ has a local model map $F$ in this fundamental neighbourhood as given in \eqref{Decomposition_of_function} such that \textit{$C(y,x,z)$ is independent of $x$ and $z$, and that $A$ and $B$ are independent of $y$}, then Proposition \ref{Lefschetz_product_formula} shows that we can decompose the polynomial local Lefschetz heat supertraces as
\begin{equation}
\label{product_cones_links_22}
    \mathcal{L}(\mathcal{P}(U_p),T_f,p)(b,t)= \mathcal{L}(\mathcal{P}(\mathbb{D}^k), T_{C},p)(b,t) \cdot \mathcal{L}(\mathcal{P}(C_x(Z)), T_{(Ax,B)},p)(b,t)
\end{equation}
and the local Lefschetz numbers as
\begin{equation}
\label{equation_product_formula_77}
    L(\mathcal{P}(U_p),T_f,p)= L(\mathcal{P}(\mathbb{D}^k), T_{C},p) \cdot L(\mathcal{P}(C_x(Z)), T_{(Ax,B)},p)
\end{equation}
where the geometric endomorphisms are the induced ones on the factors. For the moment, it is clear from Proposition \ref{Lefschetz_product_formula} that all one needs for this is that the heat kernel of the Laplacian of $\mathcal{P}(U_p)$ decomposes into the product of the heat kernels of the Laplacians of $\mathcal{P}(\mathbb{D}^k)$ and $\mathcal{P}(C_x(Z))$.

What we show next is that we can take the fundamental neighbourhood to be the truncated tangent cone on which the metric is of product type, and by Proposition \ref{proposition_summarizing_heat_kernel_stuff_Pierre_Jesse} we can assume that the heat kernels factor as stated above. Furthermore, we will see that the geometric endomorphism on the tangent cone will factor as in Proposition \ref{Lefschetz_product_formula}.

Moreover, we will show that the local Lefschetz number only depends on the induced map at the tangent cone of a space at $p$, which we will denote $\mathfrak{T}_pf$. Later in Subsection \ref{subsubsection_local_lefschetz_numbers_fundamental_neighbourhoods} we shall study in detail the polynomial local Lefschetz heat supertraces. 

Let $f$ be a placid self map of a stratified pseudomanifold $X$. Assume that there is an isolated fixed point on a singular stratum where the restriction to a fundamental neighbourhood $U_1$ of $f$ has local model form $F$ as in equation \ref{Decomposition_of_function}. Assume that $f$ is a homeomorphism that lifts to a diffeomorphism of the resolved space. Say that the fixed point is at $p=\{y=0, x=0\}$ in the image of the fundamental homeomorphism \eqref{fundamental_coordinates}. In the resolved fundamental neighbourhood, we can see that $C(y,x,z)=C(y,0,z)+\mathcal{O}(x)$, since $C$ is a smoothly stratified map. Moreover, $C(y,0,z)$ is actually independent of $z$, since the map at $x=0$ is the resolution of the smoothly stratified map restricted to the stratum $Y$ where the link has metrically degenerated to a point. We define the functions
\begin{equation}
\label{equation_maps_building_tangent_map}
    \widetilde{A}(z):=A(0,0,z): Z \rightarrow [0,1], \hspace{5mm} \widetilde{B}(z)=B(0,0,z): Z \rightarrow Z 
\end{equation}
and we observe that by the smoothness of $f$ we have that $xA(x,0,z)=x\widetilde{A}(z)+\mathcal{O}(x^2)$.
We also define
\begin{equation}
    \widetilde{C}(y) = DC(y,0,z) : \mathbb{R}^k \rightarrow \mathbb{R}^k  
\end{equation}
where $DC$ is the differential of the map $C$ on the tangent space of the $\mathbb{R}^k$ factor at the fixed point, which we identify with $\mathbb{R}^k$. This is independent of $z$ because $f$ is a smoothly stratified map as in Definition \ref{smoothly stratified}.
Recall that we defined the tangent cone $\mathfrak{T}_pX$ at a point $p \in X$ in Subsection \ref{subsection_local_structures}.
We call 
\begin{equation}
\label{equation_map_induced_on_tangent_cone}
    \mathfrak{T}_pf(y,x,z)=( \widetilde{C}(y), x \widetilde{A}(z),  \widetilde{B}(z)): \mathfrak{T}_pX \rightarrow \mathfrak{T}_pX
\end{equation}
the \textbf{\textit{induced map on the tangent cone at $p$}}, where $\mathfrak{T}_pX=\mathbb{D} \times C(Z)$. While the map $\widetilde{C}$ is linear on $\mathbb{R}^k$, the map $(x \widetilde{A}(z), \widetilde{B}(z))$ on the cone $C(Z)$ is a homogeneous map in $x$ for positive $x$.

\begin{remark}  
\label{remark_tangent_cone_join_construction}
Notice that this map is homogeneous in both the (radial variable) of the cone as well as on the Euclidean factor (where it is even linear). If we think of the tangent cone as the product of two cones, $C_{r'}(\mathbb{S}^{k-1}_{\omega})\times C_x(Z_z)$ where the factor $C_{r'}(\mathbb{S}^{k-1}_{\omega})=\mathbb{D}^k$ is smooth, we can alternatively write the induced map on the tangent cone as
\begin{equation}
\label{equation_map_induced_on_tangent_cone_22}
    \mathfrak{T}_pf(r',\omega,x,z)=(\widetilde{C}_1(\omega) r', \widetilde{C}_2(\omega), x \widetilde{A}(z),  \widetilde{B}(z))
\end{equation}
and it is homogeneous in $r_1=\sqrt{(r')^2+x^2}$. This is in fact the radial variable of the cone $C_{r_1}(J_j)$, over the join of the link $Z$ and the sphere $\mathbb{S}^{k-1}$. We can write the map as
\begin{equation}
\label{equation_map_induced_on_tangent_cone_24}
    \mathfrak{T}_pf(r_1,j)=(\widetilde{E}(j)r_1, \widetilde{F}(j)),
\end{equation}
similar to \eqref{equation_map_induced_on_tangent_cone} in the case where $Y$ is trivial.
\end{remark}

We will now give a geometric definition of a simple map.
\begin{definition}
\label{definition_simple_map_on_tangent_cone}
An isolated fixed point $p$ of a self map $f$ on a Witt space $X$ is called simple if the induced map $\mathfrak{T}_pf$ on the tangent cone at $p$ has only one fixed point on the tangent cone $\mathfrak{T}_pX$.
\end{definition}
It is easy to see that the notions of simple used by Atiyah and Bott (Definition \ref{simple_fixed_points_definition}) and that in Definition 15 of \cite{Bei_2012_L2atiyahbottlefschetz} are generalized by this definition. 
We can now define infinitesimal versions of the properties in Definition \ref{attracting/repelling_fixed_point}. 

\begin{definition}
\label{infinitesimally_attracting/repelling_fixed_point}
Let $p$ be the isolated fixed point of a self map $f$. We say that the map is \textit{\textbf{infinitesimally attracting/non-expanding/expanding/non-attracting}} if the induced map on the tangent cone is attracting/non-expanding/expanding/non-attracting.

The condition of infinitesimally non-expanding can be equivalently defined as $||\widetilde{C}||\leq 1$, $||\widetilde{A}|| \leq 1$, for the map induced on the tangent cone as in equation \eqref{equation_map_induced_on_tangent_cone}, and the other conditions can be phrased similarly.

\end{definition}

It is clear that infinitesimally attracting/ repelling/ non-expanding/ non-attracting imply that the fixed point is attracting/ repelling/ non-expanding/ non-attracting respectively.
The next theorem shows that the local Lefschetz number at a simple fixed point only depends on the induced map on the tangent cone at the fixed point.

\begin{theorem}
\label{theorem_localization_model_metric_cone}
Let $f$ be a self map on a pseudomanifold which is associated to a geometric endomorphism $T_f=\varphi \circ f$ of an elliptic complex $\mathcal{P}$. For an isolated simple fixed point at $p=y_0$ which lies on a stratum $Y$ and has a fundamental neighbourhood $U_p=\mathbb{D}^k \times C(Z)$, we assume that the elliptic complex decomposes as a product. Then the local Lefschetz number $L(\mathcal{P}(U_p),T_f,p)$ can be expressed as 
\begin{equation}
\label{equation_finale_tangent_cone}
    \Bigg( \sum_{k=0}^n \frac{(-1)^k Tr(\varphi_{k,p} \circ D_pC)}{|det(Id-D_pC)|} \Bigg) \int_{\tau=0}^{\infty} \int_{Z}  str \Big(\varphi \circ \widetilde{K}_Z(\widetilde{A}(z),\widetilde{B}(z),1,z,\tau) \Big) dvol_Z  \frac{1}{2\tau} d\tau
\end{equation}
where we use the notation introduced in the discussion above. In particular, the local Lefschetz number only depends on the induced map on the tangent cone.
\end{theorem}

\begin{remark}
\label{Remark_Bei_Bismut_Cheeger_Zhang_no_renormalization_needed}

The integral in \eqref{equation_finale_tangent_cone} is a \textbf{Lefschetz version} of the \textbf{Bismut-Cheeger $\mathcal{J}$-form} given in \cite[\S 6]{Albin_2017_index}.
There a renormalized integral is used for the initial definition but it is shown that no renormalization is needed. For the generalized eta invariants in \cite[\S 3]{weiping1990note} it is shown that the renormalization is not necessary in the contribution to the equivariant index. Similarly in Theorem 7 of \cite{Bei_2012_L2atiyahbottlefschetz}, a renormalized integral is used to express local Lefschetz numbers but it is then shown that no renormalization is needed, and we use this to show that the same holds true for local Lefschetz numbers, i.e., that the integrals are convergent.
\end{remark}

\begin{proof}
In Proposition \ref{proposition_summarizing_heat_kernel_stuff_Pierre_Jesse} we explained that the heat kernel lifts to a conormal distribution on the wedge heat space, which at a singular isolated fixed point $p$ on a stratum $Y$ has a polyhomogenous expansion at the boundary hypersurface $\mathfrak{B}^{(H)}_{\phi \phi,1}(Y)$.

The discussion above shows that the map $f$ on $U_p$ is equal to the map induced on the tangent cone $\mathfrak{T}_pf$ up to terms vanishing to higher orders of the radial distance in the conic and Euclidean factors, and lifts to a simple b-map on the resolution. 
We denote the blow down map of the wedge heat space to $X^2 \times \mathbb{R}_{\tau}^+$ as $\beta_H$ where $X$ is the resolved space.

The intersection of the interior lift of the graph of $f$ in $X \times X$ to the wedge heat space and the boundary hypersurface $\mathfrak{B}^{(H)}_{\phi \phi,1}(Y)$ of the front face is the intersection of the lift of the graph of the model map $\mathfrak{T}_pf$ with the front face. This can be seen by studying the graph of the map $f$ in local coordinates near the 
the inward pointing spherical normal bundle of $\mathfrak{B}_{Y} \times_{\phi_{Y}} \mathfrak{B}_{Y} \times \{0\}$. In the projective coordinates for $\mathfrak{B}^{(H)}_{\phi \phi,1}(Y)$ for $Y$ containing $y_0$ given in \eqref{equation_key_projective_coordinates}, the lift of the map can be written as 
$$(r, z, s, z, \sigma=\tau/s) \mapsto (\widetilde{A}(z)r, \widetilde{B}(z), s, z,\sigma=\tau/s)$$
up to terms that vanish to higher order in the boundary defining function for $\mathfrak{B}^{(H)}_{\phi \phi,1}(Y)$ when restricted to the diagonal. Since $r=s$ on the diagonal, this amounts to vanishing to higher orders of $r=s$.
In particular, this shows that the graph of the self map $f$ lifts to a $p$-submanifold of the wedge heat space, at least in the pre-image under $\beta_H$ of neighbourhoods of fixed points at the diagonal on $X \times X \times \mathbb{R}_{\tau}^+$ for all $\tau$. Away from the diagonal, we can see that this is a $p$-submanifold by Remark \ref{remark_Fibered_morphism}.

We use Melrose's \textit{\textbf{pullback and pushforward theorems}} in this setting. These are explained in detail in, e.g., Section 3.1 of \cite{Albin_2017_index} and Appendix B of \cite{EpsteinMelroseMendozaResolvent1991}.
The heat kernel has a polyhomogenous expansion at the face $\mathfrak{B}^{(H)}_{\phi \phi,1}(Y)$.
By our discussion above, the interior lift of the graph of $f$ is a $p$-submanifold on the wedge heat space and pulling back the heat kernel $\mathcal{K}$ along the inclusion of the graph of $f$, we get a distribution that has a polyhomogeneous expansion at the front face by the pullback theorem.

Moreover the leading term in the polyhomogenous expansion is given by $\widetilde{\mathcal{K}}(\tau,\mathfrak{T}_pf(y,s,z),(y,s,z))$ where $\widetilde{\mathcal{K}}$ is the model heat kernel given in Proposition \ref{proposition_summarizing_heat_kernel_stuff_Pierre_Jesse}. 
Using the notation of the discussion that follows said proposition, we can write this as 
\begin{equation}
    \widetilde{K}_{TY}(\sigma,\widetilde{C}(y),y) dvol_Y \widetilde{K}_Z(\widetilde{A}(z), \widetilde{B}(z),s,z, \sigma ) dvol_Z.
\end{equation}

The index sets of the polyhomogenous expansions of the heat kernel as a right density have been worked out in Theorem 4.4 of \cite{Albin_2017_index} and since $f$ is a simple b-map, the pullback by $f$ and the geometric endomorphism preserve the index sets on the relevant faces of the wedge heat space discussed above.
The local Lefschetz number is given by the limit of the integral
\begin{equation}
    \label{roe_type_expansion_24}
    \lim_{t \rightarrow 0} \int_{U_p} str \Big( \varphi \circ \mathcal{K}(t,f(y,x,z),(y,x,z)) \Big) dvol_{g_w}.
\end{equation}
Theorem \ref{localization_of_simple_fixed_points} shows that the Lefschetz heat supertrace vanishes away from where the graph of $f$ intersects the diagonal of $X \times X$. The interior lift of the graph of $f$ does not intersect the front face away from points which are in the pre-image of $\beta_H$ of fixed points at the diagonal of $X^2 \times \mathbb{R}^+_{\tau}$. 

We show that the integral, which is a trace of a wedge heat operator on the wedge heat space, can be realized as a pushforward of a b-fibration following Section 5.2 of \cite{Albin_2017_index}. There it is shown that the interior lift of the diagonal of $X^2$ in $X^2 \times \mathbb{R}^+_{\tau}$ can be identified with
\begin{equation}
    \text{diag}_w^{(H)}(X)=[X \times \mathbb{R}^+_{\tau}; \mathfrak{B}_{Y_1}\times \{0\};...;\mathfrak{B}_{Y_{l'}}\times \{0\} ]
\end{equation}
where $\mathcal{S}(X)=\{Y_1,...,Y_{l'}\}$ are listed in a non-decreasing order. 
We denote the blow down map from $\text{diag}_w^{(H)}(X)$ to its image in $X^2 \times \mathbb{R}^+_{\tau}$ by 
$$\beta_{(\Delta)}: \text{diag}_w^{(H)}(X) \rightarrow X \times \mathbb{R}^+_{\tau},$$ and for $Y \in \mathcal{S}(X)$ we denote the boundary hypersurfaces of $\text{diag}_w^{(H)}(X)$ above
$\mathfrak{B}_Y \times\{0\}$ by $\mathfrak{B}_{1, 1}^{(\Delta)}(Y)
$. We denote composition of $\beta_{(\Delta)}$ with the projection onto $\mathbb{R}^+_{\tau}$ as 
\begin{equation}
    \beta_{(\Delta),\tau}: \text{diag}_w^{(H)}(X) \rightarrow \mathbb{R}^+_{\tau}
\end{equation}
which is a b-fibration.
Now we observe that the integral in the expression \eqref{roe_type_expansion_24} corresponds to the pushforward of the integrand along this b-fibration, and so Theorem 5.6 of \cite{Albin_2017_index} can be applied to our integral. Since the index sets have been worked out in \cite{Albin_2017_index}, arguing similarly to Corollary 5.7 of \cite{Albin_2017_index}, we see that the pushforward gives a polyhomogenous expansion in $\tau$ such that the limit exists at $\tau=0$.
The discussion above shows that the leading order term of the polyhomogenous expansion is given by
\begin{equation}
    \lim_{\tau \rightarrow 0} \int_{\mathbb{R}^k \times Z^+} str \Big( \varphi \circ \widetilde{K}_{TY}(\sigma,(\widetilde{C}(y),y) dvol_Y \widetilde{K}_Z(\widetilde{A}(z), \widetilde{B}(z),s,z,\sigma) dvol_Z s^lds\Big).
\end{equation}
Since we assume that the elliptic complex decomposes as a product (see Remark \ref{remark_decomposes_as_a_product}), we can use the geometric endomorphisms on each factor, and split this integral into the product of the factor
\begin{equation}
    \lim_{\tau \rightarrow 0} \int_{\mathbb{R}^k} str \Big( \varphi \circ \widetilde{K}_{TY}(\tau,(\widetilde{C}(y),y) dvol_Y \Big)=\sum_{k=0}^n \frac{(-1)^k Tr(\varphi_{k,p} \circ D_pC)}{|det(Id-D_pC)|}
\end{equation}
where the last equality follows from Theorem \ref{Nonsingular_Contribution}, and the factor
\begin{equation}
\label{Lefschetz_version_Bismut_Cheeger_J_form_prequel}
    \lim_{\tau \rightarrow 0} \int_{Z^+} str \Big( \varphi \circ \widetilde{K}_Z(\widetilde{A}(z), \widetilde{B}(z),s,z,\sigma) \Big) dvol_Z s^lds.
\end{equation}

Observe that this factor depends only on the heat kernel of the product type metric at the tangent cone, and the induced map on the tangent cone. 
Now the result follows from Theorem 7 of \cite{Bei_2012_L2atiyahbottlefschetz}, where the expression \eqref{Lefschetz_version_Bismut_Cheeger_J_form_prequel} is shown to be equal to the second factor in \eqref{equation_finale_tangent_cone} when the heat kernel is \text{scalable}. 
In particular this is the case for heat kernels of product type metrics. As we discussed in Remark \ref{Remark_Bei_Bismut_Cheeger_Zhang_no_renormalization_needed}, Bei shows in \cite{Bei_2012_L2atiyahbottlefschetz} that the second factor in \eqref{equation_finale_tangent_cone} can be defined as the evaluation at $s=0$ of a holomorphic function in $s$ for all $s\in \mathbb{C}$ (which is referred to as a modified zeta function). Since it is holomorphic in all of $\mathbb{C}$, we see that it is not necessary to renormalize the integral, proving the theorem.

\end{proof}

\begin{remark}
\label{polynomial_version_works_with_relative_heat_kernels}
It is easy to observe that the arguments in the proof can be used to show that the contribution from the heat trace in each degree to the local Lefschetz numbers (as opposed to the supertrace) only depend on the induced map on the tangent cone. This shows that for simple isolated fixed points, even the polynomial supertraces depend only on the induced maps on the tangent cone. 
\end{remark}

Using Proposition \ref{Lefschetz_product_formula}, for non-isolated singularities of the type that we study where the model space is simply the product of the tangent space and the cone over the link $Z$, we have the following corollary.

\begin{corollary}
\label{Theorem_local_Lefschetz_tangent_cone}
Let $f$ be a self map on a pseudomanifold which is associated to a geometric endomorphism $T_f$ on the elliptic complex $\mathcal{P}$. For an isolated fixed point $p$ at a singular stratum, with a fundamental neighbourhood $U_p$, the local Lefschetz number is
\begin{equation}
    L(\mathcal{P}(U_p),T_f,p)= L(\mathcal{P}(\mathbb{D}^k),T_{\widetilde{C}},p) \cdot L(\mathcal{P}(C_x(Z)),T_{(\widetilde{A}x,\widetilde{B})},p).
\end{equation}
where the geometric endomorphisms are the ones corresponding to the elliptic complexes on the tangent cone, where the complexes on the factors are those of the product decomposition as in equation \eqref{equation_product_formula_77}.
\end{corollary}

\begin{proof}
The previous theorem shows that the local Lefschetz number depends only on the map induced on the tangent cone. The corollary follows from equation \eqref{equation_product_formula_77} and the discussion there.
\end{proof}

\subsection{Cohomological formulae for local Lefschetz numbers}
\label{subsection_cohomological_formulae}

In this subsection we will explore new techniques for deriving formulae for local Lefschetz numbers. In the work of Goresky and MacPherson, the local Lefschetz numbers in intersection homology were given a homological interpretation which was then used to compute these numbers on stratified spaces. 
We will now develop analytic definitions of local cohomology for the de Rham and Dolbeault complexes that can be used to compute local Lefschetz numbers. We develop these in a generality that covers twisted versions of these complexes as well as Witten deformed versions.

First we will construct an approximate heat kernel which will be a key technical tool. Then we will show that the Lefschetz heat supertrace on a fundamental neighbourhood of an isolated fixed point can be expressed differently by substituting the global heat kernel on the stratified space restricted to a fundamental neighbourhood, with the heat kernel on the fundamental neighbourhood (as a pseudomanifold with boundary) with certain boundary conditions, possibly up to boundary contributions.

Then we shall present a general formula, which in summary says that local Lefschetz formulae for elliptic complexes are global Lefschetz formulae for fundamental neighbourhoods considered as stratified spaces with boundary possibly up to certain boundary contributions, and discuss how this can be used for different choices of domains with boundary conditions.

For the domains and boundary conditions that we impose on the Dolbeault complex restricted to a fundamental neighbourhood, the Laplace-type operator will not be Fredholm and we will present a  renormalized McKean-Singer theorem, modifying Theorem \ref{Lefschetz_supertrace}. 

\subsubsection{Approximate heat kernel on Witt spaces with boundary}
\label{subsection_approximate_heat_kernel}

We will show that in Definition \ref{definition_local_Lefschetz_numbers} for local Lefschetz numbers at an isolated fixed point, we can replace the heat kernel $\mathcal{K}_X$ on the global manifold with the heat kernel on a fundamental neighbourhood $U_p$ for the Laplace-type operator on a domain with boundary conditions. This follows from a relative Lefschetz heat trace estimate.
While Theorem \ref{theorem_localization_model_metric_cone} can be used for self maps with simple isolated fixed points, the treatment that follows will address the non-simple case as well.

We set it up as follows. Let us assume that $f : X \rightarrow X$ has an isolated fixed point, $p_0$. Let $V_1$ be a fundamental neighbourhood as in \eqref{fundamental neighbourhood} of $p_0$. Let $V_2, V_3=U$ be slightly larger neighbourhoods of $V_1$ such that $f(V_1) \subseteq V_2 \subseteq V_3$. Consider the heat kernel for $\Delta=D^2=PP^*+P^*P$ on $U$ with a self-adjoint domain, which we denote as $\mathcal{K}_{U}(t,p,q)$. 

We require that the domain chosen for the Laplace type operator on $U$ is compatible with the VAPS domain for the Laplace-type operator on $X$ in the sense that the set of sections in the domain on $U$ restricted to $V_1$ are the same as the set of sections in the VAPS domain for the operator on $X$ restricted to $V_1$.
The domains we study will satisfy this condition (see Remark \ref{remark_boundary_conditions_for_singular_links}).

We denote the distance from $p_0$ to be $x$, and construct the sets $V_i$ for $i=1,2,3$ as $V_i= \{ x \leq c_i \}$, for appropriate constants $c_i$. This can be achieved by making $c_1$ small enough and using the continuity of the map $f$.

Let the heat kernel of $\Delta$ on $X$ be $\mathcal{K}_X(t,p,q)$. We construct an approximate heat kernel by generalizing the construction in Section 4 of \cite{aldana2013asymptotics}. First choose four positive constants $c_2 < e_1 < e_2 < e_3 < e_4 < c_3$. Let $\phi_1,\phi_2,\psi_1$ be smooth functions on $X$ satisfying

\begin{align}
    \phi_1 =
    \begin{cases}
      1 & \text{if $x>e_2$              }\\
      0 & \text{if $x<e_1$              }  
    \end{cases}       &&
    \phi_2 =
    \begin{cases}
      1 & \text{if $x<e_4$              }\\
      0 & \text{if $x>c_3$              }
    \end{cases}       &&
    \psi_1 =
    \begin{cases}
      1 & \text{if $x>e_4$              }\\
      0 & \text{if $x<e_3$              }
    \end{cases}       
\end{align}
and let $\psi_2=1-\psi_1$. Then 
\begin{equation}
\label{equation_parametrix}
    Q(t,p,q)=\phi_1(p)\mathcal{K}_{X}(t,p,q)\psi_1(q)+\phi_2(p)\mathcal{K}_{U}(t,p,q)\psi_2(q)
\end{equation}
is an approximation of the heat kernel $\mathcal{K}_X(t,p,q)$ for short times, as we shall show using the following proposition.

\begin{proposition}
\label{Gluing_heat_kernels}
In the setting introduced above,
\begin{equation}
\label{Samith}
    \lim_{t \searrow 0} \int_X \big| Q(t,f(p),p)-\mathcal{K}_X(t,f(p),p) \big| dvol_X =0
\end{equation}
\end{proposition}

\begin{proof}
We begin by proving the following Lemma.

\begin{lemma}
\label{lemma_exponential_parametrix_estimates}
For $Q$ corresponding to a Laplace-type operator $D^2$, we have that
\begin{equation}
\bigg| \bigg( \frac{\partial}{\partial t} + D^2 \bigg) Q(t,p,q) \bigg| =\mathcal{O}(t^{\infty}), \text{ for } 0 < t \leq 1
\end{equation}
that is, the expression on the left vanishes super-polynomially in $t$ as $t$ goes to $0$, and restricted to $p=q$ is identically $0$
away from $\{supp(\nabla \phi_1 ) \times supp(\psi_1) \cup supp(\nabla \phi_2 ) \times supp(\psi_2) \} \subset X \times X$, for $t>0$.
\end{lemma}

\begin{proof}
From the definition of $Q$, and the properties of the heat kernels, we have the inequality
\begin{equation}
\label{abracadabra_29}
\bigg| \bigg( \frac{\partial}{\partial t} + D^2 \bigg) Q(t,p,q) \bigg| \lesssim S_1 + S_2
\end{equation}
where the terms on the right hand side are
\begin{equation*}
     S_1:=|( \langle \nabla \phi_1, \nabla_p \mathcal{K}_{X} \rangle + (D^2 \phi_1)\mathcal{K}_{X}) \psi_1(q)|
\end{equation*}
\begin{equation*}
     S_2:=|( \langle \nabla \phi_2, \nabla_p \mathcal{K}_{U} \rangle + (D^2 \phi_2)\mathcal{K}_{U}) \psi_2(q)|
\end{equation*}
where we have used the fact that the norms of $D\phi_i$ and $\nabla \phi_i$ are comparable. 
Now we consider the first term which satisfies the inequality
\begin{equation}
    S_1 \leq (|\nabla \phi_1(p))| |\nabla_p \mathcal{K}_X(t,p,q)|+ |D^2 \phi_1(p))| |\mathcal{K}_X(t,p,q)|) \chi_{supp(\psi_1(q))}
\end{equation}
where $\chi_{supp(\psi_1(q))}$ denotes the indicator function of the set $supp(\psi_1(q))$. By the construction of the functions $\phi_1,\psi_2$, we see that $supp(\nabla \phi_1) \subseteq \{ e_1 \leq x \leq e_2 \}$ and $supp(\psi_1) \subseteq \{e_3 \geq x\}$.

For positive $t$, the function $\nabla \phi_1(p) \mathcal{K}_X(t,p,q) \chi_{supp(\psi_1(q))}$ (which is supported away from the diagonal at $t=0$) vanishes super-polynomially in $t$, which follows from standard decay results for the heat kernels (as in Proposition \ref{proposition_summarizing_heat_kernel_stuff_Pierre_Jesse})
and its derivatives away from the diagonal.

\begin{remark}  
Restricted to compact sets on $X^{reg}$, one can prove exponential vanishing estimates, similar to the form of the estimate for $S_1$ in Lemma 4.1 of \cite{aldana2013asymptotics}, which uses a similar argument for estimating the term $S_2$ as well. This is worked out in that article for the heat kernel of a cusp end, whereas we work with the heat kernel on stratified spaces.
\end{remark}

For $S_2$, we have that the supports of $\phi_2,\psi_2$ vanish near the boundary of $V_3=U$. Since heat kernels have local parametrix constructions (as used in the proofs in \cite{Albin_2017_index}) which yield super-polynomial vanishing estimates away from the boundary, we see that $S_2$ also has super-polynomial vanishing by arguments similar to those for $S_1$.

The right hand side of inequality \eqref{abracadabra_29} vanishes away from where the derivatives of the cutoff functions $\phi_i$ and the cutoff functions $\psi_i$ are supported, for $i=1,2$. This proves the lemma.
\end{proof}

\begin{remark}  
One main idea in the above lemma is what is known as \textit{Kac's principle of not feeling the boundary} (see for instance \cite{li2016heat} for a treatment in the smooth setting), and while there are exponential vanishing estimates in the smooth setting, super-polynomial estimates suffice for our purposes.
\end{remark}

We can now use this lemma to prove the proposition which is an adaptation of Lemma 4.3 of \cite{aldana2013asymptotics} to our setting. Duhamel's principle shows that
\begin{equation}
Q(t,p,q)-\mathcal{K}_X(t,p,q) = \int_0^t \int_X \mathcal{K}_X(s,p,w) \bigg( \frac{\partial}{\partial t} + D^2 \bigg) Q(t-s,w,q) dvol_X(w) ds.
\end{equation}
Now, we can take the pullback by $f$ on the left factor and restrict to the diagonal to obtain 
\begin{equation}
Q(t,f(p),p)-\mathcal{K}_X(t,f(p),p) = \int_0^t \int_{U \setminus V_2} \mathcal{K}_X(s,f(p),w) \bigg( \frac{\partial}{\partial t} + D^2 \bigg) Q(t-s,w,p) dvol_X(w) ds
\end{equation}
where we have used that the integrand is only supported in $(U \setminus V_2)$ on the diagonal for all $t$, which follows from the lemma above.
There are no fixed points of $f$ on $U \setminus V_1$, and so on $U \setminus V_1$, we have that 
\begin{equation*}
    \int_{U \setminus V_1} \big| Q(t,f(p),p)-\mathcal{K}_X(t,f(p),p) \big| dvol_X(p)
\end{equation*}
vanishes super-polynomially in $t$ as $t$ goes to $0$ by Proposition \ref{proposition_summarizing_heat_kernel_stuff_Pierre_Jesse}.
Since the integrand is identically zero for $p \in X \setminus U$ (by the definition of $Q)$, we only need to consider the integral 
\begin{multline}
\label{guchcha_1}
\int_{V_1} \big| Q(t,f(p),p)-\mathcal{K}_X(t,f(p),p) \big| dvol_X(p) \\ 
\leq \int_{V_1} \Big|\int_0^t \int_{U \setminus V_2} \mathcal{K}_X(s,f(p),w) \bigg( \frac{\partial}{\partial t} + D^2 \bigg) Q(t-s,w,p) dvol_X(w) ds \Big| dvol_X(p).
\end{multline}
Since $f(V_1) \subseteq V_2$ (in particular $f(V_1)\cap \{U \setminus V_2 \} = \emptyset$), we have that $\mathcal{K}_X(s,f(p),w)$ vanishes superpolynomially in $s$ on the support of the integrals and by Lemma \ref{lemma_exponential_parametrix_estimates} the other factor in the integrand vanishes superpolynomially in $(t-s)$. 

We observed in the proof of Proposition \ref{theorem_localization_model_metric_cone} that the heat kernels composed with the geometric endomorphisms are distributions on the wedge heat space with polyhomogenous expansions at the boundary face at the diagonal. We define $\widetilde{\mathcal{K}}(s,p,w):=\mathcal{K}_X(s,f(p),w)$, which lifts to a distribution on the wedge heat space. This distribution has bounds at the other faces, and $\widetilde{\mathcal{K}}$ is a wedge heat operator in the sense of \cite[Appendix B]{Albin_2017_index}.
Similarly 
\begin{equation}
    \bigg( \frac{\partial}{\partial t} + D^2 \bigg) Q(t-s,w,p)
\end{equation}
also defines a wedge heat operator and the integral that corresponds to the composition of wedge heat operators, which is treated in detail in \cite[Appendix B]{Albin_2017_index} shows that the expression on the right hand side of \eqref{guchcha_1} vanishes superpolynomially in $t$, proving the result.

\end{proof}

\begin{remark}
\label{Remark_only_need_morphism_in_neighbourhoods}
We observe that in the proof above, we used that the pullback and pushforward theorems are only applied to $\mathcal{K}_X(s,f(p),w)$ restricted to subsets of $W=U \times U \times \mathbb{R}^+$. In particular, we only require that the graph of $f$ is a p-submanifold (which was used in the arguments of Proposition \ref{theorem_localization_model_metric_cone}), restricted to the pre-image of the set $W$ under the blow-down map from the wedge heat space.
In particular this is the case if the map is of the model form at the tangent space, as in \eqref{Decomposition_of_function} where it is a local homeomorphism, or is of the form of the induced map on the tangent cone in \eqref{equation_map_induced_on_tangent_cone}.

Even though the bundle morphisms involved in geometric endomorphisms are not considered in the above proposition, it is easy to see that one only needs geometric endomorphisms associated to self maps that are in model form near the fixed points.
We use this observation to get formulas for more general self-maps in the de Rham case.
\end{remark}

\begin{remark}
\label{remark_not_taking_tangent_cone_neighbourhoods_even_though_possible}
While in the case of wedge metrics, for isolated simple fixed points we can take the fundamental neighbourhoods to be the truncated tangent cone using Corollary \ref{Theorem_local_Lefschetz_tangent_cone}, we state Proposition \ref{Gluing_heat_kernels} above for general fundamental neighbourhoods because we can use this for model neighbourhoods where the metric is not of product type and for self maps with non-simple fixed points.
\end{remark}

\subsubsection{Local Lefschetz numbers on fundamental neighbourhoods}
\label{subsubsection_local_lefschetz_numbers_fundamental_neighbourhoods}

In this subsection, we present formulas for Lefschetz numbers for fundamental neighbourhoods of isolated fixed points, where we consider the fundamental neighbourhoods as stratified pseudomanifolds with boundary. This will be used to compute local Lefschetz numbers at singular strata.

There are studies of Lefschetz numbers on manifolds with boundary in a number of papers, including \cite{brenner1990atiyah,brenner1981atiyah} for elliptic complexes which give explicit formulas for the de Rham complex. This was followed by \cite{kytmanov2004holomorphic} for the holomorphic Lefschetz number with non-Fredholm boundary conditions. Zhang shows how to compute local Lefschetz numbers for the spin Dirac operator in \cite{weiping1990note}, where the boundary is replaced by an isolated conic singularity to get a formula in terms of a generalized eta invariant. When one considers Dirac-type operators with non-local boundary conditions (as in \cite{donnelly1978eta}), there are contributions to the Lefschetz numbers from the boundary. We develop a framework which can capture all these instances and generalize it to the singular setting.

Here we present a general formula for local Lefschetz numbers that we shall then instantiate for different complexes in later sections. While the more general case of Lefschetz numbers of pseudomanifolds with boundary with fixed points at the boundary is interesting in its own right and much of our developments here can be applied to those cases, our primary focus is on developing tools for the purpose of computing local Lefschetz numbers in the interior of a space.

Goresky and MacPherson proved that for intersection homology, and for other sheaf theoretic cohomology theories, the local Lefschetz numbers can be expressed as the supertrace of the induced map on local homology/cohomology groups \cite{Goresky_1985_Lefschetz,Goresky_1993_local_Lefschetz}. The latter can be easily rephrased as the global Lefschetz number of a manifold with boundary, which contains only one isolated fixed point.
Using Proposition \ref{Gluing_heat_kernels} we can now take a similar viewpoint for elliptic complexes. The natural candidate for local cohomology will then be the cohomology of the elliptic complex on the stratified pseudomanifold with boundary, with a suitable domain. We can then use the results for abstract Hilbert complexes to see that a modified version of Theorem \ref{Lefschetz_supertrace} holds.

Let $\widehat{U_p}$ be a fundamental neighbourhood of an isolated fixed point of a self map $f$ which induces a geometric endomorphism on an elliptic complex $\mathcal{P}_B(U_p)=(H,P)$, where $B$ denotes the choice of domain for the complex which dictates the boundary conditions on the corresponding Laplace-type operator. Then 
\begin{equation}
     L(\mathcal{P}_B(U_p),T_f)= \sum_{k=0}^n (-1)^k Tr[f^*|_{\mathcal{H}^k(\mathcal{P}_B(U_p))}]
\end{equation}
is called the \textit{\textbf{(global) Lefschetz number of the geometric endomorphism $T_f$, on the elliptic complex $\mathcal{P}_B(U_p)$ on the fundamental neighbourhood $U_p$}}.
Here $\mathcal{H}^k(\mathcal{P}_B(U_p))$ is the cohomology of the Hilbert complex obtained by restricting the sections of the Hilbert complex associated to $\mathcal{P}$ on $X$, to the fundamental neighbourhood $U_p$, for the choice of domain denoted by $B$.

Proposition \ref{Gluing_heat_kernels} shows that we can compute local Lefschetz numbers in the following way.
If $\widehat{U_p}$ is a fundamental neighbourhood of a non-expanding fixed point $p \in X$ of a map $f$ which induces a geometric endomorphism $T_f$ on an elliptic complex $\mathcal{P}$, then
\begin{equation}
\label{whats_to_do}    L(\mathcal{P}_B(U_p),T_f)=L(\mathcal{P}_B(U_p),T_f,p)+L(\partial U_p, \mathcal{P}_B(U_p), T_f).
\end{equation}
Here \textit{\textbf{$L(\mathcal{P}_B(U_p),T_f,p)$ is the local Lefschetz number}} at the fixed point $p$, which in light of Proposition \ref{Gluing_heat_kernels} is the same as $L(\mathcal{P}(X),T_f,p)$.
The second term $L(\partial U_p, \mathcal{P}_B(U_p), T_f)$ is the \textit{\textbf{Lefschetz boundary contribution}} which we defined in equation \eqref{equation_Lefschetz_boundary_contribution},
\begin{equation}
\label{equation_Lefschetz_boundary_contribution_innit}
    L(\partial U_p, \mathcal{P}_B(U_p), T_f)= \lim_{t \rightarrow 0} \sum_{k=0}^n \int_{W_p} (-1)^k tr(\phi_k \circ \mathcal{K}_{U_p}^k(t,f(q),q) dvol)
\end{equation}
where $W_p \subset U_p$ is an open neighbourhood of the smooth boundary $\partial U_p$ that does not contain the fixed point $p$. This appears because the heat kernel for the associated Laplace-type operator does not necessarily have off diagonal decay at the boundary for all choices of self adjoint boundary conditions. These boundary contributions contribute to 
\begin{equation}
   L(\mathcal{P}_B(U_p),T_f)= \lim_{t \rightarrow 0} \sum_{k=0}^n \int_{\overline{U}_p} (-1)^k tr(\varphi_k \circ \mathcal{K}_{U_p}^k(t,f(q),q) dvol)
\end{equation}
which is the Lefschetz heat supertrace for the fundamental neighbourhood, generalizing \eqref{Heat_formula_all_t_P}. 

The Lefschetz boundary contributions give non-trivial contributions for non-local boundary conditions such as the APS boundary conditions (we will explore this in an upcoming article) 
and will vanish for local boundary conditions for which the associated the heat kernel vanishes away from the diagonal in the small time limit.
We summarize the content of this discussion in the following proposition that gives an equivalent definition for the local Lefschetz numbers presented in Definition \ref{definition_local_Lefschetz_numbers}.

\begin{proposition}
\label{Local_Lefschetz_numbers_definition}
Given a self map $f$ of a fundamental neighbourhood $\widehat{U_p}$ with a fixed point at $p$, associated to a geometric endomorphism $T_f$ on the elliptic complex $\mathcal{P}_B(U_p)=(H,P)$ with $B$ denoting the boundary conditions for the complex, the \textbf{\textit{local Lefschetz numbers}} can be expressed as
\begin{equation*}
\begin{split}
L(\mathcal{P}_B(U_p),T_f,p) & = \lim_{t \rightarrow 0} \sum_{k=0}^n (-1)^k \int_{\overline{U_p}} tr(T_k \circ e^{-t \Delta_k})-L(\partial U_p, \mathcal{P}_B(U_p), T_f)\\
 & = \sum_{k=0}^n (-1)^k tr(T_k|_{\mathcal{H}^k(\mathcal{P}_B(U_p)})-L(\partial U_p, \mathcal{P}_B(U_p), T_f)
\end{split}
\end{equation*}
where the heat kernel $e^{-t \Delta_k}$ is that on the fundamental neighbourhood $U_p$ corresponding to $\Delta_k$ with the induced domain for the Laplacian.
\end{proposition}

For Dirac complexes $\mathcal{D}(U_p)=\mathcal{D}_B(U_p)$, we can write this as
\begin{equation}
\label{APS_Index_type_lefschetz}
\begin{split}
L(\mathcal{D}(U_p),T_f,p) & = \lim_{t \rightarrow 0} \int_{\overline{U_p}} str(T_k \circ e^{-t \Delta_k})-L(\partial U_p, \mathcal{D}(U_p), T_f)\\
 & = Tr(T|_{{\mathcal{H}(\mathcal{D}(U_p))}^+})- Tr(T|_{{\mathcal{H}(\mathcal{D}(U_p))}^-})-L(\partial U_p, \mathcal{D}(U_p), T_f)
\end{split}
\end{equation}

\begin{remark}
The expression \eqref{APS_Index_type_lefschetz} is the Lefschetz version of the expression for the index in Theorem 3.10 of \cite{atiyah1975spectral} up to a rearrangement of the terms. We will explore this further when we study APS conditions in a different article.
\end{remark}

Theorem \ref{theorem3} is a simplified version of the above result.
If one picks local boundary conditions for which the boundary contribution vanishes, we have a formalism where the local Lefschetz numbers are sums of supertraces of local cohomology groups, similar to the global Lefschetz numbers on a stratified space. In this setting we can even study the polynomial version as follows.

Let $f$ be a self map of a pseudomanifold with an isolated fixed point $p$. Let $\widehat{U_p}$ be a fundamental neighbourhood of $p$. Let $\mathcal{P}_B(U_p)$ be an elliptic complex on $U_p$ with the choice of domain $B$, for which the associated Laplace-type operator has discrete spectrum and is Fredholm such that there are no Lefschetz boundary contributions. This will be the case for the local boundary conditions that we shall use for the de Rham and Dolbeault complexes. Let $T=T_f$ be a geometric endomorphism for the complex. We call 
\begin{equation*}
\begin{split}
L(\mathcal{P}_B(U_p),T_f)(b)= \sum_{k=0}^n b^k Tr[f^*|_{\mathcal{H}^k(\mathcal{P}_B(U_p))}]
\end{split}
\end{equation*}
the \textbf{\textit{local Lefschetz polynomials}} for the fixed point $p$ and we call 
\begin{equation}
    \mathcal{L}(\mathcal{P}_B(U_p),T_f)(b,t)  = \sum_{k=0}^n b^k \int_{U_p} tr(T_k \circ e^{-t \Delta_k}).
\end{equation}
the \textbf{\textit{local Lefschetz heat polynomial supertraces}} for the fixed point $p$.
In this setting, Theorem \ref{Lefschetz_supertrace} shows that we have 
\begin{equation*}
\mathcal{L}(U_p,\mathcal{P}_B(U_p),T_f)(b,t) =L(U_p,\mathcal{P}_B(U_p),T_f)(b) + (1+b) \sum_{k=0}^{n-1} b^k S_k(t)
\end{equation*}
where
\begin{equation} 
    S_k(t)=\sum_{\lambda_i \in Spec(\Delta_k) \setminus \{0 \} } e^{-t \lambda_i}  \langle T_k v_{\lambda_{i}}, v_{\lambda_{i}} \rangle
\end{equation}
where $v_{\lambda_i}$ form an orthonormal basis of co-exact eigensections of $\Delta_k$ on $U_p$.

\subsubsection{Domains and boundary conditions for fundamental neighbourhoods}
\label{subsubsubsection_Neumann_boundary_condition}

In this subsection, we discuss natural choices of domain for fundamental neighbourhoods for certain elliptic complexes such as the de Rham and Dolbeault complex. We will use these conditions in the next two sections with twisted coefficients and with Witten deformation.
The choices of domain for the elliptic complexes will result in natural boundary conditions for the Laplace-type operators which we shall call \textit{\textbf{generalized Neumann boundary conditions}}. These are local boundary conditions, which generalize the usual Neumann boundary conditions for the de Rham complex on manifolds with boundary (see Chapter 5, Section 9 of \cite{taylor1996partial}), as well as the $\overline{\partial}$-Neumann boundary conditions for the Dolbeault complex on complex manifolds with boundary (see Chapter 12 of \cite{taylor2013partial}).

On smooth manifolds with boundary, and even on $\mathbb{D}^2$, already these two boundary value problems show different properties. In the de Rham case, the operator is Fredholm, while in the $\overline{\partial}$-Neumann case, the cohomology is infinite dimensional. 
Even in the simple case of the $\overline{\partial}$-Neumann problem on the disc, there is essential spectrum in the form of eigenspaces of infinite multiplicity (including the eigenvalue 0 and the first non-zero eigenvalue, as explained in the end of page 729 of \cite{fu2007spectrum}).

We will first study the choices of domain for smooth manifolds with boundary and then discuss how this generalizes to fundamental neighbourhoods with singularities.

Given an elliptic complex $\mathcal{P}(X)=(L^2(X;E),P_X)$ on a smooth Riemannian manifold $X$, when we restrict it to a neighbourhood $\widehat{U}$ with smooth boundary $\partial U$, we have two canonical ways of defining elliptic complexes $\mathcal{P}_B(U)=(L^2(U;E),P_U)$ in a neighbourhood, and these correspond to the choices of domain for the operators $P$ of the complex. We will denote $P_U$ by $P$ when it is understood by context that we are studying the operator of the complex on $U$.
We \textbf{\textit{define the complex $\mathcal{P}_N(U)$ to be the Hilbert complex with the maximal domain}} 
\begin{equation}
    \mathcal{D}_{max}(P_U)=\{ u \in L^2(U;E_i) : P_i u \in L^2(U;E_{i+1}) \}
\end{equation}
where $P_i u$ is defined in the distributional sense.
Recall that for an operator $P \in \text{Diff}^1(U;E,F)$, we have the Green-Stokes formula (see, e.g., Proposition 9.1 of \cite{taylor1996partial})
\begin{equation}
\label{equation_Greens_identity}
    \langle Ps, \tau \rangle_F - \langle s, P^* \tau \rangle_E = \int_{\partial U} g_F( i\sigma_1(P)(dr) s, \tau) dVol_{\partial U}
\end{equation}
where $\sigma_1(P)(dr)$ is the principal symbol of the operator $P$ evaluated at the differential of a boundary defining function $r$, where we have chosen $r$ such that the outward pointing unit normal vector to the boundary is given by $\partial_r$, and the corresponding covector is $dr$.

The Hilbert space adjoint of $(P,\mathcal{D}_{max}(P))$ is then $P^*$ endowed with its minimal domain, 
the elements of which satisfy the boundary condition
\begin{equation}
\label{primitive_boundary_condition_Green_Stokes}
    \sigma(P^*)(dr)u|_{\partial U}=0.
\end{equation}
This is an easy consequence of the Green-Stokes formula (for the $\overline{\partial}$ operator this is explained in Lemma 4.2.1 of \cite{chen2001partial}). It is easy to verify that the minimal domain of $P^*$ consists of sections in the maximal domain of $P^*$ that satisfy this boundary condition distributionally.

This fixes the domain for operators in the \textit{\textbf{dual Hilbert complex of $\mathcal{P}_N(U)$}}, and we will \textit{\textbf{denote it as $(\mathcal{P}_N(U))^*$}}.
This induces a domain for the Dirac-type operator $D=P+P^*$
\begin{equation}
\label{equation_generalized_Neumann_Dirac_conditions_ultima}
    \{s \in \mathcal{D}_{\max}(D) \text{  such that  }  \sigma_1(P^*)(dr) s |_{\partial U}=0  \},
\end{equation}
by the prescription in equation \eqref{Domain_Dirac_first}.
The induced domain for the Laplace-type operator is then that given in \eqref{Laplacian_P_type}, that is,
\begin{equation}
\label{equation_generalized_Neumann_Robin_type}
    s \in \mathcal{D}_{\min}(P^*), \quad s \in \mathcal{D}_{\max}(P), \quad Ps \in \mathcal{D}_{\min}(P^*), \quad P^*s \in \mathcal{D}_{\max}(P),
\end{equation}
which we call the \textbf{\textit{generalized Neumann boundary conditions}} for the Laplace-type operator of the complex $\mathcal{P}_N(U)$. The first and third of these conditions imply 
\begin{equation}
\label{equation_generalized_Neumann_zeroth_first_order_conditions}
    \sigma_1(P^*)(dr) s |_{\partial U}=0, \quad \sigma_1(P^*)(dr) Ps |_{\partial U}=0,
\end{equation}
respectively, which we call the \textit{\textbf{zeroth order and first order conditions}} of the generalized Neumann boundary conditions, in that order.

\begin{remark}
\label{Remark_hidden_condition}
It follows from $(P^*)^2=0$, that $s \in \mathcal{D}_{\min}(P^*)$ implies $P^*s \in \mathcal{D}_{\min}(P^*)$. We note that hence any $s \in \mathcal{D}_{min}(P^*)$ satisfies $\sigma_1(P^*)(dr) P^*s |_{\partial U}=0$.
\end{remark}

Similarly if we choose the \textit{\textbf{maximal domain for the complex $\mathcal{P}^*=(L^2(X;E_i),P^*)$ restricted to $U$, which we denote as $\mathcal{P}^*_N(U)$}}, then the elements in the corresponding domain for the Laplacian will satisfy the generalized Neumann boundary conditions for $\mathcal{P}^*$,
\begin{equation}
    s \in \mathcal{D}_{\min}(P), \quad s \in \mathcal{D}_{\max}(P^*), \quad P^*s \in \mathcal{D}_{\min}(P), \quad Ps \in \mathcal{D}_{\max}(P^*).
\end{equation}
which we shall call the \textit{\textbf{generalized Dirichlet boundary conditions}} for the Laplace-type operator of the complex $\mathcal{P}$. We denote the \textit{\textbf{dual complex as $\mathcal{P}_D(U):=\{\mathcal{P}^*_N(U)\}^*$}}. This is the same as the complex $\mathcal{P}$ restricted to $U$ with the minimal domain for $P$.

It is easy to check that for the de Rham complex and the dual complex, the corresponding boundary conditions are intertwined by the Hodge star operator (see Section 9, Chapter 5 of \cite{taylor1996partial}).
The cohomology of $\mathcal{P}_N(U)$ for the de Rham complex corresponds to the absolute cohomology of the manifold with boundary $U$, and is isomorphic to the null space of the Laplace-type operator with the generalized Neumann boundary conditions. 

We now consider the case when $U$ is singular, initially restricting to the case of fundamental neighbourhoods with isolated singularities $C(Z)$ for some smooth link $Z$. In this case, we impose VAPS conditions at the singular point, and can impose boundary conditions as above at the boundary.
We will denote the operator $P$ acting on sections over $U$ as $P_U$.

We define the complex $\mathcal{P}_{N}(U)=(L^2(U;E),P)$ with the domain
\begin{equation}
    \mathcal{D}_{N}(P):= \{ v \in \mathcal{D}_{VAPS}(P_U) : Pv \in L^2(U;E) \},
\end{equation}
where $\mathcal{D}_{VAPS}(P_U)$ are sections which 
satisfy the VAPS conditions. We call this the maximal domain after imposing VAPS conditions.

\begin{remark}[Boundary conditions for singular links]
\label{remark_boundary_conditions_for_singular_links}

In the case of a general fundamental neighbourhood $U_p \subset X$ 
define the maximal domain after imposing VAPS conditions as follows. Consider a total boundary defining function $\rho_X$ on the resolved manifold with corners with iterated fibration structures (but importantly the boundary defining function of $\partial U$ does not appear in $\rho_X$). 
Then we define the domain
\begin{equation}
    \mathcal{D}_{N}(P):= \text{graph closure of } \{ \mathcal{D}_{\max}(P_U) \cap \rho_X^{1/2}L^2(U_p;E) \},
\end{equation}
The operator $P^*$ of the complex $(\mathcal{P}_{N})^*(U_p)=(L^2(X;E),P^*)$ has the adjoint domain, and the domains $\mathcal{D}_{N}(P^*)$ for the operator $P^*$, and its adjoint can be defined similarly.

\end{remark}

\begin{remark}[Product boundary conditions]
\label{remark_introducing_product_boundary_conditions}
By Corollary \ref{Theorem_local_Lefschetz_tangent_cone} for simple isolated fixed points, we can take fundamental neighbourhoods to be a cone with product type metric. If we have an isolated conic singularity with a product type metric, we can take the boundary defining function for a fundamental neighbourhood to be $1-x$, where $x$ is the distance from the singularity.
We can take this to be the boundary defining function for the conic factor and take a similar choice for the $\mathbb{D}^k$ factor. Since the metric on the tangent cone and the model wedge is product type we can impose boundary conditions on each factor of $\mathbb{D}^k$ and $C(Z)$, and take the product complex, and we call these \textit{\textbf{product boundary conditions}}.
\end{remark}

We summarize this as follows. Given a complex $\mathcal{P}=(L^2(X;E),P)$, and a fundamental neighbourhood $U$, we have the following choices of domains. We will denote the operator $P$ acting on sections over $U$ as $P_U$.
    The complex $\mathcal{P}_N(U)=(L^2(U;E),P_U)$ is equipped with the maximal domain for $P_U$ after imposing VAPS conditions. This is the graph closure of $\{ \mathcal{D}_{\max}(P_U) \cap \rho_X^{1/2}L^2(U;E) \}$.
    We denote it by $N$ since this corresponds to the generalized Neumann boundary conditions for the complex $\mathcal{P}$ restricted to $U$. The dual of the above complex is $(\mathcal{P}_N(U))^*$.    
    The complex $\mathcal{P}^*_N(U)=(L^2(U;E),P_U^*)$ has domain for $P_U^*$ is chosen to be the maximal domain after imposing VAPS conditions (i.e. the graph closure of $\{ \mathcal{D}_{\max}(P_U^*) \cap \rho_X^{1/2}L^2(U;E) \}$) 
    and the dual of which is $(\mathcal{P}^*_N(U))^*=\mathcal{P}_D(U)=(L^2(U;E),P_U)$ 
    denoted by $D$ since this corresponds to the generalized Dirichlet boundary conditions for the complex $\mathcal{P}$ restricted to $U$.

It follows that the Laplace-type operators for $\mathcal{P}_N(U)$ and $(\mathcal{P}_N(U))^*$ have the same domain (see \eqref{Laplacian_P_type}) whose sections satisfy the generalized Neumann conditions for $P$, while $\mathcal{P}^*_N(U)$ and $\mathcal{P}_D(U)$ have Laplace-type operators with the same domain. The latter domain has the generalized Dirichlet conditions of the complex $\mathcal{P}(X)$ restricted to $U$, which are the generalized Neumann conditions for the complex $\mathcal{P}^*(X)$ restricted to $U$. However, the gradings of the elements of the Laplace-type operators are reversed for the dual complexes as is apparent from the discussion following \eqref{dual_complex}. This is a crucial point in computing the Lefschetz numbers in light of duality results.
In particular, given a complex $\mathcal{P}(X)$, we know from Proposition \ref{proposition_Lefschetz_on_adjoint}, that we can choose to write the local Lefschetz numbers and the local Lefschetz polynomials for an invertible map $f$ using either the complex or the dual complex. If the map $f$ has a non-expanding fixed point $p$, then there is a fundamental neighbourhood $U_p$ such that $f|_{U_p}$ is a self map of the fundamental neighbourhood, and we can use the local Lefschetz number of the complex $\mathcal{P}_N(U)$. If the map is also strictly attracting, then the support of the pullback of any section on $U$ is supported away from $\partial U$ and the geometric endomorphism preserves the boundary conditions. If the self map is strictly attracting and locally invertible, $(f|_{U_p})^{-1}$ would be expanding and will not be a self map of the fundamental neighbourhood and we cannot compute its global Lefschetz number for the manifold with boundary $U_p$.
If it is non-attracting, and the map is locally invertible, then there is a fundamental neighbourhood such that $f^{-1}$ restricted to it is a self map, and we can use the complex $\mathcal{P}^*_N(U)$ to compute the global Lefschetz numbers for the manifold with boundary. If in addition the map is strictly expanding we cannot use the complex $\mathcal{P}_N(U)$ to compute the local Lefschetz numbers of $\mathcal{P}(X)$. 

The maximal domain of an operator on a manifold with boundary is preserved by a geometric endomorphism of a smooth map $f$ since the pullback will preserve it, and we can apply the results in Section \ref{section_Hilbert_complexes} for these \textit{local} Hilbert complexes. The restricted map $f$ will induce a geometric endomorphism $T_f$ of the complex $\mathcal{P}_N(U)$. In the case of geometric endomorphisms acting on $\mathcal{P}_D$, one has to check whether the minimal domain is preserved, which includes checking that the boundary conditions are preserved. This can be checked easily for group actions and other `nice' maps but we shall avoid this by working with the maximal domains and the generalized Neumann conditions. We summarize this in the following remark.

\begin{remark}
\label{which_boundary_conditions}
Given a complex $\mathcal{P}$ and a self map $f$ associated to a geometric endomorphism $T_f$, with a fixed point $p$, then if $p$ is non-expanding there is a fundamental neighbourhood $U_p$ such that $f(U_p) \subseteq U_p$ and $L(\mathcal{P}_N(U_p),T_f)$ is defined.
If it is non-attracting and the map $f$ is locally invertible at $p$ then there is a fundamental neighbourhood $U_p$ such that $f^{-1}(U_p) \subseteq U_p$ and $L(\mathcal{P}^*_N(U_p),T_{f^{-1}}^{P^*})$ is defined, where by $T_{f^{-1}}^{P^*}$ we denote the geometric endomorphism of the dual complex $\mathcal{P}^*_N(U_p)$ associated to $f^{-1}$.
\end{remark}

We now study how to define suitable domains for twisted spin$^{\mathbb{C}}$ Dirac complexes when there are almost complex structures on the tangent cone.
We build on observations of Epstein in \cite{epstein2006subelliptic,EpsteinSubellipticSpinc2_2007,EpsteinSubellipticSpinc3_2007} where he generalizes the $\overline{\partial}$-Neumann boundary conditions to the case where there is a spin$^{\mathbb{C}}$ structure which is defined in a neighbourhood of the boundary by an almost complex structure.
We refer to \cite{epstein2006subelliptic} for a summary of his 4 main articles on spin$^{\mathbb{C}}$ Dirac operators, referring to the third article \cite{EpsteinSubellipticSpinc3_2007} for more details.

As discussed in Subsection \ref{subsection_Dolbeault_SpinC_introduction}, the complex spinor bundle of a resolution of a stratified space $X$ with an almost complex $\mathcal{S}$ is canonically identified with $\oplus_q \Lambda^{0,q}X$ or more generally with $\oplus_q \Lambda^{p,q}X$ after twisting with $E=\Lambda^{p,0}X$. One can impose the boundary condition $\iota_{{\overline{\partial} r}^\#} u|_{r=0}=0$ where $r$ is a boundary defining function.

Now suppose $X$ is a resolution of a pseudomanifold with a spin$^{\mathbb{C}}$ structure.
Given a compatible almost complex structure in a neighbourhood of a singular point, there is an induced almost complex structure $J_p$ at the tangent cone at the fixed point (realized by freezing coefficients at $p$ and extending homogeneously) which is homogeneous with respect to dilations on the infinite cone.

Let $\widehat{X}$ be a stratified pseudomanifold of dimension $2n$, and let $\mathcal{R}(X)=(L^2(X;\mathcal{S} \otimes E), D)$ be the twisted spin$^{\mathbb{C}}$ Dirac complex (see Definition \ref{equation_two_term_complex}) for a twisted spin$^{\mathbb{C}}$ Dirac operator. Let $f$ be a self map preserving the spin$^{\mathbb{C}}$ structure, with isolated fixed points, and inducing a geometric endomorphism $T_f$ on the complex.

Let $U=U_a$ be the truncated tangent cone of a non-expanding fixed point $a$ of $f$, where the spin$^{\mathbb{C}}$ structure is defined by an almost complex structure near the boundary. Then we can identify $L^2(U;\mathcal{S} \otimes E)$ with $L^2\Omega^{0,\cdot}(U;E)$ as discussed above. 
We define the complex $\mathcal{R}_N(U_a)=(L^2\Omega^{0,\cdot}(U;E), D_U)$ with domain 
\begin{equation}
\label{equation_domain_for_spin_c_123}
    \mathcal{D}_{N}(D):= \text{graph closure of } \{ \mathcal{D}_{\max}(D_U) \cap \rho_X^{1/2}L^2(U_p;E) : \iota_{{\overline{\partial} x}^\#} u|_{x=1}=0  \}.
\end{equation}
It is easy to see that in the case where the almost complex structure is integrable, this matches up with the domain for the Dolbeault-Dirac type operator (see equation \ref{equation_generalized_Neumann_Dirac_conditions_ultima}).
For more general fixed points with a product decomposition as in Proposition \ref{proposition_local_product_Lefschetz_heat}, we can use the product complex $\mathcal{R}_B(U_a)$, where sections in the domain of the Dirac operator of $\mathcal{R}^*_N(U_a)$ satisfy the boundary condition $(\overline{\partial} x) \wedge u|_{x=1}=0$. However most geometric endomorphisms of spin$^{\mathbb{C}}$ Dirac complexes that are studied widely arise from isometries, in which case 
$\mathcal{R}_B(U_a)=\mathcal{R}_N(U_a)$.

\subsubsection{Functional calculus on truncated cones with boundary conditions}
\label{subsubsection_spectrum_cone}

Now we build machinery to imitate the proof of the McKean-Singer theorem for the truncated tangent cone $U_p=C(M)$, with boundary given by the link at $x=1$. We first do this for a self map $f$ with a simple isolated fixed point $p$ at the cone point, for an elliptic complex $\mathcal{P}_B(U_p)$, where we cover the cases of de Rham and Dolbeault complexes with twisted coefficients, with the domains $B=N, D$ or the dual complexes as explained in the previous subsection, then proceed to prove this for the spin$^{\mathbb{C}}$ Dirac complex.

As we observed in the earlier section, the $\overline{\partial}$ operator with the maximal domain is not a Fredholm complex and the corresponding Laplace-type operator with the $\overline{\partial}$ boundary conditions is not a Fredholm operator.
Non-Fredholm complexes do not have trace class heat kernels and we must show that the Lefschetz heat supertraces still make sense. The key to this is to understand the spectrum of the Laplace-type operator.

We will show that there is an orthonormal basis of eigensections of the Laplace-type operator in each degree. Moreover, we will show that there is a filtration of the eigenspaces of the Laplace-type operator on the cone, given by the eigenvalues of a Laplace-type operator on the link. We will use this to evaluate the Lefschetz heat supertrace in a renormalized sense.
In \cite{cheeger1983spectral}, Cheeger used functional calculus on infinite cones to understand the model heat kernel and we borrow techniques, ideas and results from that article.

\begin{proposition}
\label{Proposition_spectral_properties}
Let $U_p=C_x(M)$ be the truncated metric cone with link $M$, with $x \leq 1$, with an elliptic complex $\mathcal{P}_N(U_p)$, which is either a de Rham or Dolbeault complex restricted to $U_p$, or the corresponding dual complex $\mathcal{P}^*_N(U_p)$.

Let $\Delta=(P+P^*)^2$ be the associated Laplace-type operator on $C(M)$ and let $\Delta_M$ be the associated Laplace-type operator on $M$. 
Then there exists an orthonormal basis of eigensections $\psi_{q,i,j}$ of $\Delta$ such that each eigensection factors into a product of a function of the radial variable $x$ and an eigensection of $\Delta_M$, where $q$ is the degree of the section in the elliptic complex, $i$ indexes the eigenvalues $\mu_i$ of the eigensections with respect to the operator $\Delta_M$, and $j \in \mathbb{N}$.

Moreover, these are eigensections of the operator $\Delta^q + \Delta_M^q$ on the same domain as $\Delta^q$. 
The operator $\Delta^q + \Delta_M^q$ has discrete spectrum, 
with eigenvalues $\lambda_{(q,i,j)}^2 + \mu_i$, where $\lambda_{(q,i,j)}^2$ are eigenvalues of the operator $\Delta^q$. Moreover we have that for any fixed $(q,i)$ the sum $\sum_{j} \lambda^{-4}_{(q,i,j)}$ is finite, and that for each co-exact eigensection $\psi_{q,i,j}$ with positive eigenvalue $\lambda_{(q,i,j)}$ for $\Delta$, $\frac{1}{\lambda_{(q,i,j)}} P\psi_{q,i,j}$ is an exact eigensection with the same eigenvalues for $\Delta^{q}$ and $\Delta^{q}_M$.
\end{proposition}

We observe that this result strengthens Lemma \ref{Lemma_super_symmetry} restricted to the tangent cone.

\begin{proof}
\textit{\textbf{Outline:}}
The proof has two main steps. In the first main step we show how to compute the eigensections of the Laplace type operators of the complexes on the infinite tangent cone over $M$ following \cite[\S 3]{cheeger1983spectral}. In that article Cheeger studied the eigensections of the Hodge Laplacian on an exact (i.e., infinite) cone, obtaining a description of the spectrum that was rich enough for many applications in geometry and index theory, in particular for constructing the model heat kernel in \eqref{heat_kernel_exact_cone}.

The eigensections of the Hodge Laplacian with positive eigenvalues are of 6 types which he calls types $1, 2, 3, 4, E$ and $O$. Cheeger observes that these come in what we will call \textit{\textbf{supersymmetric pairs}}, by which we mean that given an eigensection $\psi_{q, i, j}$ which is a form of degree $q$ with eigenvalue $\lambda^2_{q,i,j}$ for $\Delta^{q}$ and eigenvalue $\mu_i$ for $\Delta_M$ and of one of the types $1, 4, E$, then $\frac{1}{\lambda_{(q,i,j)}} P\psi_{q,i,j}$ is an exact eigensection with the same eigenvalues for $\Delta^{q}$ and $\Delta^{q}_M$ which is of type $2, 3, O$. The Hodge star operator also interchanges types $1$ and $3$, $2$ and $4$, $E$ and 
$O$.
Similarly Cheeger works out the harmonic eigenforms in equations $(3.24-3.27)$ of \cite{cheeger1983spectral}.

We will go over the methods that Cheeger uses and show that the eigensections for the Laplace-type operator corresponding to the Dolbeault complex on the exact cone can be worked out similarly (in the K\"ahler case it is well known that the two Laplace-type operators differ by a constant factor). Roughly, we use a separation of variables ansatz, assuming that the eigensections are sums of products of eigensections of $\Delta_M$ and eigensections of Bessel type Sturm-Liouville problems on the infinite cone.

In Cheeger's work, he uses the functional calculus on the infinite cone with absolutely continuous spectrum for the Laplace-type operator. The boundary conditions we use ``quantize" the spectrum, giving discrete spectrum involving only a subset of the eigenfunctions on the infinite cone derived by Cheeger.

In the second main step we will show that the generalized Neumann boundary conditions and the VAPS condition for the complexes correspond to singular Sturm-Liouville boundary conditions for the equation obtained by separation of variables for functions on the $(0,1]$ factor of the cone. Then Sturm-Liouville theory shows that there is a complete, orthogonal basis of eigensections corresponding to simple eigenvalues $\lambda_{j}^2$ indexed by $j$ for $L^2[0,1]$ with these boundary conditions. Since the Laplacian on the compact link has domain with VAPS conditions, we have that there is a complete countable orthogonal basis of eigensections of $\Delta^q_M$ indexed by $j$. Then as in our ansatz, the eigensections $\psi_{q,i,j}$ of $\Delta^q$ on $C(M)$ are sums of products of the eigensections on $[0,1]$ and the link $M$.
These eigensections give the entire spectrum for the self adjoint Laplace-type operators $\Delta^q$ (see for instance Lemma 1.2.2 of \cite{davies1996spectral}). 

This will show that there is a filtration of each eigenspace of $\Delta^q$, the Laplace-type operator of the elliptic complex at degree $q$, given by the eigenspaces of $\Delta_M^q$. Each eigenvalue $\mu_i$ of $\Delta_M^q$ can be associated to countably many simple eigenvalues $\{\lambda^2_{q,i,j}\}_j$, where the eigenvalues grow to $\infty$ for a fixed $\mu_i$ by Sturm-Liouville theory.
Therefore the eigenvalues of $\Delta^q + \Delta_M^q$ are $\lambda_{q,\mu,j}^2 + \mu$, and the operator has discrete spectrum.

The work in Section 2.4.2 of \cite{al2008sturm} shows that $\sum_j \lambda_{q,i,j}^{-4}$ is finite (the reciprocals of the eigenvalues $\lambda_j^2$ are $\l^2$ summable) for any given $q,i$. 
Finally, normalizing the eigensections as in Lemma \ref{Lemma_super_symmetry}, and using the correspondence of supersymmetric pairs for each complex, we see that the last statement of the proposition holds.

\textit{\textbf{Step 1: Eigensections on the infinite cone.}}
We will show that for both the de Rham and Dolbeault complexes there is a decomposition of the sections on the Hilbert complex into 6 main types, which following Cheeger \cite[\S 3]{cheeger1983spectral} we call types $1, 2, 3, 4, E$ and $O$. 

\textit{\textbf{Step 1.1: Hodge Laplacian.}}
We will first study the case of the de Rham complex on the truncated cone. The de Rham complex has a natural splitting near the boundary, given by 
\begin{equation}
\label{deRham_form_decomposition}
    \Omega^k(C(M))=\Omega^k(M) \oplus dx \wedge \Omega^{k-1} (M)
\end{equation}
with respect to which the operator $d+\delta$ is given by
\begin{equation}
 d+\delta=
  \begin{bmatrix}
   d_M+\frac{1}{x^2}\delta_M &
   -\frac{1}{x}(l+2-2k) -\partial_x\\
   \partial_x &
   -d_M-\frac{1}{x^2}\delta_M 
   \end{bmatrix}
\end{equation}
and this can be used to compute the Hodge Laplacian. Since the forms in $\Omega^k(M)$ satisfy the strong Kodaira decomposition, we use this to decompose the forms further into 6 different types following \cite[\S 3]{cheeger1983spectral}.

Let $\phi^k(x,m)$ be a $k$ form such that for each $x_0 \leq 1$, $\phi^k(x_0,m)$ restricted to the link $M$ at $x=x_0$ is coexact. We denote
$\alpha(k)=1/2(1+2k-l)$, where $l$ is the dimension of $M$.
Then the $k$ forms 
\begin{equation}
\label{Cheegers_1_type}
    x^{\alpha(k)}\phi^k,
\end{equation}
\begin{equation}
\label{Cheegers_2_type}
    x^{\alpha(k-1)}d_M\phi^{k-1}+dx \wedge (x^{\alpha(k-1)}\phi^{k-1})',
\end{equation}
\begin{equation}
\label{Cheegers_3_type}
    x^{2\alpha(k-1)+1}(x^{-\alpha(k-1)}d_M\phi^{k-1})'+x^{\alpha(k-1)-1}dx \wedge \delta_M d_M \phi^k,
\end{equation}
\begin{equation}
\label{Cheegers_4_type}
    x^{\alpha(k-2)+1}dx \wedge d_M \phi^{k-2}
\end{equation}
\begin{equation}
\label{Cheegers_E_type}
    x^{\alpha(k)}h^k
\end{equation}
\begin{equation}
\label{Cheegers_O_type}
    dx \wedge (x^{\alpha(k-1)}h^{k-1})'
\end{equation}
are called forms of the type $1, 2, 3, 4, E$ and $O$ respectively, where the primed notation indicates differentiation with respect to $x$. In the latter two types, $h^k(x,m)$ are $k$ forms which are in the null space of $\Delta_M$ restricted to the link $M$ at $x=x_0$ for each $x_0 \leq 1$.

Cheeger observes that eigenforms of types $1, 3$ are coexact while forms of types $2, 4$ are exact. Moreover, these are \textit{\textbf{supersymmetric pairs}} in the sense that $P(=d)$ carries types $1, 3$ to types $2, 4$ while $P^*(=\delta)$ does the reverse. Similarly, $P$ takes type $O$ forms to those of type $E$ and $P^*$ does the reverse.
Cheeger then shows that the eigenvalues of the eigenforms of these types completely describe the spectrum of the Hodge Laplacian on the infinite cone. 
We will briefly describe why this is the case in Remark \ref{Remark_why_SUSY_gives_ONB}.

\textit{\textbf{Step 1.1.1: Detailed study on forms of types $1,4$}}.
We denote the Hodge-Laplacian restricted to forms of types $1$ and $4$ as $\Delta_1$ and $\Delta_4$ and study this in detail.
On type $1$ forms, which are in the $\Omega^k(M)$ summand of the above decomposition for the de Rham complex, the Hodge Laplacian is 
\begin{equation} 
\label{equation_Model_laplacian_type_1}
    \Delta_1 = -\partial_x^2 - \frac{l-2k}{x} \partial_x +\frac{1}{x^2} \Delta_M
\end{equation}
and on type $4$ forms which are in the $dx \wedge \Omega^{k-1} (M)$ summand (up to shifts in degree), it is 
\begin{equation} 
\label{equation_Model_laplacian_type_4}
    \Delta_4 = -\partial_x^2 - \frac{l+2-2k}{x} \partial_x +\frac{1}{x^2} \Delta_M+\frac{l+2-2k}{x^2}.
\end{equation}
If we instead look at the rescaled Laplacian acting on sections of the rescaled sub-bundle $dx \wedge x^{-A} \Omega^{k-1}(M)$ where $A=l+2-2k$, we get
\begin{equation}
    \Delta'_4 = -\partial_x^2 + \frac{l+2-2k}{x} \partial_x +\frac{1}{x^2} \Delta_M.
\end{equation}
In order to treat these two cases uniformly, we study the Laplace-type operator
\begin{equation} 
\label{equation_Model_laplacian}
    \Delta = -\partial_x^2 - \frac{A}{x} \partial_x +\frac{1}{x^2} \Delta_M
\end{equation}
where $A$ is a constant depending on the degree $k$ of the form and $l$.
We will use the standard ansatz for separation of variables, and look for Laplace eigensections (with positive eigenvalues $\lambda^2$) on $C(M)$ in the form $\psi_{\lambda^2}=f(x)x^ag_\mu(m)$ 
where $f$ is a function on the interval $[0,1]$ and $g_\mu$ is an eigensection of $\Delta_M$ on $C(M)$, with eigenvalue $\mu$. 

We first observe that for $\mu>0$, $x^ag_\mu(m)$ is in the null space of $\Delta$ for a unique choice of positive real value $a$. Indeed expanding $\Delta x^ag_\mu(m)=0$, we get
\begin{equation}
\label{equation_harmonic_power}
    {a(a-1)}+{aA}-{\mu}=0,
\end{equation}
and it is easy to see that if $\mu>0$, then there is exactly one positive root $((1-A)+\sqrt{(A-1)^2+4\mu})/2$ for $a$. We note in particular that as $\mu$ grows, this root grows of order $\sqrt{\mu}$. We choose this positive root $a$ since for that value the ansatz yields a standard Sturm-Liouville problem for the function $f$. Indeed solving $\Delta f(x)x^ag_\mu(m) =\lambda^2 f(x)x^ag_\mu(m)$ reduces to solving the equation
\begin{equation}
\label{equation_Gota}
   L[f]=(\partial_x^2 f + \frac{(2a+A)}{x}\partial_x f + \lambda^2 f)=0.
\end{equation}
If we set $f(x)=(\lambda x)^{\nu}F(\lambda x)$, one can check that 
\begin{equation}
    L[(\lambda x)^{\nu}F(\lambda x)]=\lambda^2R^{\nu}\Big( \partial_R^2 F+\frac{2\nu+2a+A}{R}\partial_R F+ \frac{(\nu^2-\nu)+\nu(2a+A)+R^2}{R^2}F(R) \Big)
\end{equation}
where $R= \lambda x$. If we set $\nu= \frac{(-2a-A+1)}{2}$, it is easy to check this reduces equation \eqref{equation_Gota} to the well known classical form of Bessel's equation
\begin{equation}
    R^2 \partial_R^2(F(R)) +R \partial_R F(R) +(R^2-\nu^2) F(R)=0.
\end{equation}
One can check that in the smooth case, this corresponds to the change of variables computations on page 117 of \cite{folland1972tangential}.
For non-integer values of $\nu$, eigenfunctions satisfying the ansatz are spanned by
\begin{equation}
\label{tempura}
    cx^{\frac{-A+1}{2}} J_\nu(\lambda x) g_\mu(m), \hspace{5mm} cx^{\frac{-A+1}{2}} Y_\nu(\lambda x) g_\mu(m)
\end{equation}
where, if they are in $L^2$ (with respect to the conic volume measure), we can choose $c$ so that the $L^2$ norm of each eigensection is 1. Here $J_\nu$ and $Y_{\nu}$ are Bessel functions of the first and second kind. For non-integer values of $\nu$, $Y_\nu$ is a linear combination of $J_{\nu}, J_{-\nu}$ and it suffices to study the eigensections with Bessel functions of the first kind with positive and negative values of $\nu$ since they form a basis for the eigenvalue problem before imposing boundary conditions or integrability. This is the approach followed by Cheeger and he obtains the first kind of eigensection in \eqref{tempura} in equations (3.15) and (3.18) of \cite{cheeger1983spectral}. 
For integer $\nu$ (half integer in Cheeger's conventions for $\nu$), the Bessel functions $J_{\nu}, J_{-\nu}$ are not linearly independent and logarithmic solutions (as in \cite{cheeger1983spectral}) replace the negative ${\nu}$ Bessel functions (c.f. the discussion in Section 5.3 of \cite{al2008sturm}).
However the eigensections corresponding to negative $\nu$ (or logarithmic solutions in the integer case) are not $L^2$ bounded, even when multiplied by a cutoff function supported away from $\infty$ (see page 588 of \cite{cheeger1983spectral}).
Thus we recover Cheeger's eigensections of types $1, 4$ which we denote as
\begin{equation}
    \psi_{1,k-1,\mu}=x^{\alpha(k-1)}J_{\nu}(\lambda x) \phi_{\mu}^{k-1} \quad \psi_{4,k,\mu}=(x^{\alpha(k-1)-1}J_{\nu}(\lambda x)) dx \wedge d_M \phi_{\mu}^{k-1}
\end{equation}
respectively.

\textit{\textbf{Step 1.1.2: General case of the Hodge Laplacian.}}
The Hodge Laplacian restricted to forms of types $2, 3$ have off-diagonal terms corresponding to the decomposition \eqref{deRham_form_decomposition}. Let us split the de Rham operator as $d=P_1+P_2$ where $P_1=dx \wedge \partial_{x}$ and $P_2=d_M$. Similarly we write $\delta=L_1+L_2$ where $L_2=\frac{1}{x^2}\delta_M$ determines the splitting.
A computation similar to that above can verify that the eigensections are of the form
\begin{multline}  
\label{Cheegers_eigenfunctions_of_type_2}
    x^{\alpha(k-1)}J_{\nu}(\lambda x) d\phi_{\mu}^{k-1}+(x^{\alpha(k-1)}J_{\nu}(\lambda x))' dx \wedge \phi_{\mu}^{k-1}=P_2(\psi_{1,k-1,\mu})+P_1(\psi_{1,k-1,\mu})=P(\psi_{1,k-1,\mu})
\end{multline}
for forms of type $2$ and
\begin{multline}  
\label{Cheeger_3_type_Dolbeault}
    x^{2\alpha(k-1)+1}(x^{-\alpha(k-1)}J_{\nu}(\lambda x))' d\phi_{\mu}^{k-1}+(x^{\alpha(k-1)-1}J_{\nu}(\lambda x)) dx \wedge \delta_M d_M \phi_{\mu}^{k-1}\\=L_1(\psi_{4,k,\mu})+L_2(\psi_{4,k,\mu})=P^*(\psi_{4,k,\mu})
\end{multline}
for forms of type $3$ where $\phi^k_{\mu}$ are co-exact eigenforms of $\Delta_M$ of degree $k$, and the values of $\nu$ are given \cite[\S 3]{cheeger1983spectral}. Cheeger obtains these as well as the eigenforms for types $E, O$ (see page 587 of \cite{cheeger1983spectral}).
For the de Rham complex, Cheeger explains that the supersymmetric pairs of the eigensections of types $1, 4$ are those of types $2, 3$ and types $E$ and $O$ are supersymmetric pairs as well. Thus one can compute the eigensections of types $1,4,E,O$ and use the supersymmetric pairing to get the eigensections of types $2,4$.
\begin{remark}
\label{Remark_why_SUSY_gives_ONB}
We briefly explain why these eigensections give an orthonormal basis of eigensections on $M$ when restricted to $\{x=1\}$. Given $\phi_k$, a coexact eigensection of degree $k$, we observe (given explicitly above) that it appears in eigenforms of type $1$ without a $dx$ factor, and in eigenforms of types $2,3$ with a $dx$ factor. Similarly, the exact eigensection $d\phi_k$ appears in forms of type $4$ with a $dx$ factor, and without a $dx$ factor in types $2,3$. These four types therefore give a linearly independent basis for $\phi_k$, $d\phi_k$, in both summands of the decomposition \eqref{deRham_form_decomposition}.
For instance, exact tangential forms for the de Rham (Dolbeault) complex can be written as linear combinations of forms of types $2,3$, as is clear from \eqref{Cheegers_2_type}, \eqref{Cheegers_2_type}.
The harmonic forms appear with and without $dx$ factors in types $O, E$. 
\end{remark}

\textit{\textbf{Step 1.2: Dolbeault Laplacian.}}
For the Dolbeault complex, the splitting in equation \eqref{deRham_form_decomposition} is replaced by 
\begin{equation}
\label{Dolbeault_form_decomposition_1}
    \Omega^{0,q}(C(M);E)=\Omega^{0,q}(M;E) \oplus (dx-\sqrt{-1}xJ(dx)) \wedge \Omega^{0,q-1} (M;E)
\end{equation}
where the forms in $\Omega^{0,q}(M;E)$ satisfy $\iota_{(xJ(dx))^{\sharp}} u=0$. Note that $dx-\sqrt{-1}xJ(dx)=\overline{\partial}x$ plays an analogous role to that of $dx$ in the case of the de Rham complex. We have that $\sigma(\overline{\partial}^*)(dx)u=\iota_{(dx-\sqrt{-1}xJ(dx))^{\sharp}} u$.

We define $P_1=(dx-\sqrt{-1}xJ(dx)) \wedge (\partial_x + \sqrt{-1}\frac{1}{x}\partial_\xi)$, which determines the decomposition $P=P_1+P_2$ of the operator $P=\overline{\partial}$ and we can write
\begin{equation}
    \Delta=P_1^*P_1+P_1P_1^*+P_2^*P_2+P_2P_2^* 
\end{equation}
where $\Delta_{M}'= x^2(P_2^*P_2+P_2P_2^*)$ is an operator on sections on $M$, for fixed values of $x$.
Let $\phi^q(x,m)$ be a $q$ form of the Dolbeault complex such that for each $x_0 \leq 1$, $\phi^q(x_0,m)$ restricted to the link $M$ at $x=x_0$ is can be written as 
$P^*_2 \psi^{q+1}$ for some form $\psi^{q+1}$. Indeed any form $\phi^{q}$ on $C(M)$ that is an eigensection of $\Delta_M$ with non-zero eigenvalue $\mu$ (in particular satisying $\partial_x \psi^{q+1}=0$) can be written as $P_2 \psi_2 + P_2^* \psi_1$ using the supersymmetry of the transversally elliptic operator $\Delta_{M}'$. This corresponds to a transverse Dolbeault complex on the space $M/ \mathcal{F}_{\xi}$ where $\xi=J(\partial_x)$ is the Reeb vector field and $\mathcal{F}_{\xi}$ is the foliation generated by $\xi$.

Then the forms of types $2$ are those of the form 
\begin{equation}
    x^{\alpha(q-1)} P_2\phi^{q-1} + P_1(x^{\alpha(q-1)}\phi^{q-1}),
\end{equation}
while those of type $4$ are those of the form
\begin{equation}
    x^{\alpha(q-2)+1} (dx-\sqrt{-1}xJ(dx)) \wedge P_2\phi^{q-2}
\end{equation}
where the forms $\phi$ are co-exact (with respect to $P=\overline{{\partial}}$).
Forms of types $2, 4$ are exact with respect to $P=\overline{\partial}$, and the forms of types $1$ and type $3$ are generated by applying $P^*$ to the forms of types $2$ and $4$, respectively.

\begin{remark}
We observe that forms of type $1$ are have the structure $x^{\alpha(q)}\phi^q$ which shows that $\overline{\partial}$ maps forms of type $1$ to those of type $2$ when $\partial_x \phi^q =0$. To see that $\overline{\partial}^*$ maps those of type $2$ to type $1$ when $\partial_x \phi^q =0$, we observe that
\begin{equation}
    P_1^*x(^{\alpha(q-1)} P_2\phi^{q-1})=0, \hspace{2mm} P_2^*x(^{\alpha(q-1)} P_2\phi^{q-1})= P_2^* P_2(x^{\alpha(q-1)}\phi^{q-1}), \hspace{2mm} P^*_2P_1(x^{\alpha(q-1)}\phi^{q-1})=0
\end{equation}
where the last equality follows since by the assumption that $\phi^{q-1}$ can be realized as $P^*_2 \psi^{q+1}$ for some form $\psi^{q+1}$ and since $(P^*_2)^2=0$. We then note that the Laplace type operator on such forms $\phi^q$ is given by the sum of $P^*_1P_1(x^{\alpha(q-1)}\phi^{q-1})$ and $P_2^* P_2(x^{\alpha(q-1)}\phi^{q-1})$. 
A similar computation shows that the forms of types $4,3$ are supersymmetric pairs.
\end{remark}

The forms of types $E, O$ are
\begin{equation}
    x^{\alpha(q)}h^q
\end{equation}
and 
\begin{equation}
    P_1 (x^{\alpha(q-1)}h^{q-1})
\end{equation}
respectively where $h^{q-1}$ are harmonic with respect to the operator $\Delta_M$. With this we put the decomposition of forms into the same formalism that we used for the Hodge Laplacian, and it is easy to see that the eigensections can be constructed analgously, yielding supersymmetric pairs $1,4$ with $2,3$ and $E$ with $O$.

\begin{remark}[Harmonic sections]
\label{Remark_growth_of_harmonic_sections}
In pages 587-588 of \cite{cheeger1983spectral}, Cheeger computes explicitly the harmonic forms of $\Delta$ on the infinite cone. Aside from those that are harmonic with respect to $\Delta_M$ (which are finite since $\Delta_M$ is a Fredholm operator on $M$), these are of the four types $1,2,3,4$. For each eigensection of $\Delta_M$ on $M$ that is co-closed, there are 2 harmonic sections of each of these four types (corresponding to $a_j^{\pm}$ in Cheeger's notation). Cheeger observes that some of these may coalesce when the forms are $\Delta_M$ harmonic.

The harmonic forms of type $1$ are $x^{a_j^{\pm}(k)}\phi_j^k$ where $\phi_j^k$ are co-closed, $a_j^{\pm}(k)=\alpha(k) \pm \nu_j(k)$ and Cheeger observes that these forms are not in $L^2$ on the infinite cone. However it is easy to see that the harmonic forms corresponding to $a_j^{+}(k)$ are precisely those of the form $x^{a_i}g_{\mu_i}$ that we considered in our ansatz from which we computed that $a_i$ must satisfy \eqref{equation_harmonic_power} for $x^{a_i}g_{\mu_i}$  to be harmonic. For $L^2$ boundedness on the truncated cone, for large enough $\mu_i$ we need the positive root of the quadratic equation $((1-A)+\sqrt{(A-1)^2+4\mu_i})/2$ for $a_i$ (which corresponds to $a_j^{+}(k)$ in Cheeger's notation). Weyl's law for the eigenvalues of $\Delta_M$ on $M$ shows that $\mu_i$ goes to $\infty$ as $i$ goes to $\infty$.
It is easy to see that as $i$ goes to $\infty$, $a_i$ grows of order $\sqrt{\mu_i}$.

Similarly, by analyzing the harmonic forms of the other three types studied by Cheeger, we can check that the harmonic forms have factors of $x^a_i$ for eigenvalues $\mu_i$ of $\Delta_M$ whose exponents grow of order $\sqrt{\mu_i}$ as $i$ goes to $\infty$.
\end{remark}

\textit{\textbf{Step 2: Singular Sturm-Liouville boundary conditions.}}
We now show that for the ODE's obtained by restricting to an eigenspace of $\Delta_M$, the generalized boundary conditions  at $x=1$ and the VAPS conditions at $x=0$ will reduce to singular Sturm-Liouville boundary conditions of the type studied in \cite[\S 2.4.3]{al2008sturm}.
The differential equations of Sturm-Liouville type that we consider are of the form
\begin{equation}
\label{equation_Sturm_Liouville_singular_equation}
    (pu')'+ru+\lambda \rho u=0
\end{equation}
on an interval $(0,1]$, where $p$ is smooth and vanishes only at $x=0$, $r$ is continuous, and $\rho$ is positive and continuous on $(0,1]$. At $x=1$, we demand the boundary condition
\begin{equation}
\label{equation_Sturm_Liouville_boundary_condition_main}
    \gamma_1 u(1) + \gamma_2 u'(1)=0
\end{equation}
for some constants $\gamma_1, \gamma_2$, and we demand that the function $u(x)$ has a well defined limit at $x=0$.

Equation \ref{equation_Gota}, up to a multiplication of the operator $L$ by $p(x)=\rho(x)=x^{2a+A}$ is of the form \eqref{equation_Sturm_Liouville_singular_equation}.
Indeed it is shown that the classical Bessel equation is of this form in \cite[\S 2.4.3]{al2008sturm}, and 
as discussed in the outline each such Sturm-Liouville problem will give eigenspaces that are one dimensional, mutually orthogonal and together span $L^2[0,1]$ with respect to the measure $x^ldx$ where $l$ is the dimension of the link in our setting.

The discussion following \eqref{tempura} shows that the $L^2$ condition means that only the eigensections with Bessel functions of the first kind with positive values of $\nu$ will be $L^2$ bounded on $[0,1]$. However all these eigensections have a well defined limit as $x$ goes to $0$ and thus are in the VAPS domain. It is easy to see that the singular Sturm-Liouville boundary condition at $x=0$ for a Bessel problem will only admit solutions that are linear combinations of eigensections with Bessel functions $J_{\nu}$ of the first kind with positive values of $\nu$.
Thus we will focus on the boundary conditions at $x=1$. Since our boundary defining function is $r=1-x$, we can take $-dx$ instead of $dr$ in the boundary conditions. We begin with the de Rham complex.

\textit{\textbf{Step 2.1: Boundary conditions for the Hodge Laplacian.}} 
The $0$-th order boundary condition of the generalized Neumann boundary conditions is
\begin{equation}
    \sigma(P^*)(dr)u|_{\partial U}=0, 
\end{equation}
 which is equivalent to $\iota_{\partial x} u|_{\{x=1\}}=0$, and
is trivially satisfied on forms of type $1$ since they do not have a factor of the form $dx$.  
On type $4$ forms which do have the components $dx$, this is a Dirichlet type condition $\psi_{\lambda^2}(1,m)=0$ which is satisfied when $\lambda$ is a zero of the Bessel function $J_\nu$. This is therefore a Dirichlet type condition and is of the form in \eqref{equation_Sturm_Liouville_boundary_condition_main} where $\gamma_1=1$, $\gamma_2=0$.

We now consider the $1$-st order boundary condition
\begin{equation}
    \sigma(P^*)(dr)Pu|_{\partial U}=0.
\end{equation}
On forms of type $4$, the first order boundary condition is satisfied when $\lambda$ is a zero of the Bessel function, which was the condition for the zero-th order boundary condition as well. Thus they are compatible and not over-determined. On forms of type $1$, the zero-th order condition was satisfied trivially, but one can see that the 1st order boundary condition reduces to the form
\begin{equation}
\label{boundary_conditions_for_type_1}
    \partial_x \psi_{\lambda}(x,m)=\partial_x (x^{\alpha(k)}J_{\nu}(\lambda x))|_{x=1}=0
\end{equation}
which is a Robin type boundary condition of the form in \eqref{equation_Sturm_Liouville_boundary_condition_main}.

The Sturm-Liouville boundary conditions and the ODE for type $E$ are the same as those for type $1$ and those for type $O$ are the same as those for type $4$. This can also be seen using Cheeger's remarks on page 590 of \cite{cheeger1983spectral}.

Let us consider the boundary conditions at $x=1$ for the eigensections of type $2$ given in equation \eqref{Cheegers_eigenfunctions_of_type_2},
\begin{equation}    
     x^{\alpha(k-1)}J_{\nu}(\lambda x) d\phi_{\mu}^{k-1}+(x^{\alpha(k-1)}J_{\nu}(\lambda x))' dx \wedge \phi_{\mu}^{k-1}=P_2(\psi_{1,k-1,\mu})+P_1(\psi_{1,k-1,\mu})=P(\psi_{1,k-1,\mu}).
\end{equation} 
The zeroth order boundary condition is trivial on the first term, $P_2(\psi_{1,k-1,\mu})$. From the zeroth order condition on the second term we see that we need $(x^{\alpha(k-1)}J_{\nu}(\lambda x))'|_{x=1}=0$. It is easy to see that the first order boundary condition amounts to the same condition and that it is of the form of \eqref{equation_Sturm_Liouville_boundary_condition_main}.
We observe that the eigensections \eqref{Cheegers_eigenfunctions_of_type_2} of type $2$ are exactly those obtained by taking $P=d$ of the eigensections of type $1$. In fact the boundary conditions of the eigensections of type $1$ in equation \eqref{boundary_conditions_for_type_1} are exactly the same boundary conditions that we obtained above for the type $2$ eigensections. This shows that the eigensections of types $1$ and $2$ satisfying the boundary conditions are supersymmetric pairs. Thus for fixed $\phi^{k}_\mu$ eigensections of $\Delta_M$, given an orthonormal basis of eigensections for the Sturm-Liouville operator on the interval for type $1$ sections, there is a corresponding orthogonal basis of type $2$ eigensections for the corresponding Sturm-Liouville operator on the interval, given by the supersymmetric pairing.
Thus, as explained in the outline above, the separation of variables ansatz gives an orthonormal basis of eigensections of the Laplace type operator $\Delta$ for $L^2$ bounded forms in the domain.
The case of eigensections of type $3$ can be handled similarly by observing that they are obtained by applying $P^*=\delta$ to the eigensections of type $4$.

We observed above that the eigensections of types $1, 3$ go to those of types $2, 4$ under $d$ and conversely by $\delta$, even once the boundary conditions are imposed.
The eigensections of types $E, O$ on $C(M)$ that satisfy the boundary conditions are supersymmetric pairs as well and can be verified more easily since the forms $h^k$ are harmonic.

\textit{\textbf{Step 2.2: Boundary conditions for the Dolbeault Laplacian}}

This is similar to the de Rham case. Let us consider the boundary conditions on the forms of types $2, 4$. We first consider the $0$-th order boundary condition at $x=1$. For forms of type $4$, this corresponds to a Dirichlet type boundary condition, similar to the case of the de Rham complex for forms of type $4$. For forms of type $2$, the zero-th order condition corresponds to the vanishing of $P_1(x^{\alpha(q-1)}\phi^{q-1})$ at $x=1$. 

The first order boundary condition on forms of type $2$ corresponds to the vanishing of $P_1(x^{\alpha(q-1)}P_2\phi^{q-1})$ at $x=1$. Observe that $P_2$ annihilates radial functions and $P_1$ commutes with $P_2$. Thus by the separation of variables ansatz, the boundary condition is is equivalent to the vanishing of $P_1(x^{\alpha(q-1)}\phi^{q-1})$ at $x=1$. This is a Robin condition, similar to the case of the de Rham complex for forms of type $4$.
It is easy to see that for all forms of the 6 types, we have Sturm-Liouville problems similar to the case of the de Rham complex, and are supersymmetric pairs.
This generalizes the arguments in the smooth case, studied in Section V of \cite{folland1972tangential} where the supersymmetric pairs are implicit in the computations.
\end{proof}

\begin{remark}[Witten deformed versions]
In \cite[\S 4]{ludwig2017index}, the eigensections for the Witten deformed de Rham complex on cones are computed using similar techniques, augmented with the deformation. Those computations are on the infinite cone. Using a separation of variables ansatz for the eigenfunction equation of the Witten deformed Laplace type operator reduces it to the direct sum of infinitely many Laguerre type Sturm-Liouville equations (as opposed to the Bessel type ones we studied above), 
and the spectrum is quantized by the $L^2$ boundedness. We will appeal to this in the proof of Proposition \ref{Proposition_model_spectral_gap_modified_general}.
\end{remark}

The following result is a generalization of Proposition \ref{Proposition_spectral_properties} to truncated metric cones with only almost complex structures, and initially seems weaker when the almost complex structure is integrable since it only takes into account the odd and even degree supersymmetry. However since it is much easier to see that in the integrable case, the Dolbeault-type operator sends $q$ forms to $q+1$ forms, the result can be used to recover the results in Proposition \ref{Proposition_spectral_properties} introducing the exact and co-exact supersymmetric pairs using the arguments in Lemma \ref{Lemma_super_symmetry}.

Given the spin$^{\mathbb{C}}$ Dirac complex $\mathcal{R}(X)=(L^2(X;\mathcal{S} \otimes E), D)$, the restricted complex $\mathcal{R}_N(U_a)=(L^2(U;\mathcal{S} \otimes E), D_U)$ is identified with $\mathcal{R}_N(U_a)=(L^2\Omega^{0,\cdot}(U;E), D_U)$ as discussed in Subsection \ref{subsubsection_local_lefschetz_numbers_fundamental_neighbourhoods}.

\begin{proposition}
\label{Proposition_spectral_properties_almost_complex}
Let $U_p=C_x(M)$ be the truncated metric cone with link $M$, with $x \leq 1$, with a twisted spin$^{\mathbb{C}}$ Dirac complex $\mathcal{R}_B(U_a)=(L^2\Omega^{0,\cdot}(U;E), D_U)$. Let $\Delta=(D_U)^2$ be the associated Laplace-type operator on $C(M)$ and let $\Delta_M$ be the associated Laplace-type operator on $M$. 
Then there exists an orthonormal basis of eigensections $\psi_{q,i,j}$ of $\Delta$ such that each eigensection factors into a product of a function of the radial variable $x$ and an eigensection of $\Delta_M$, where $q$ is the odd-even degree of the section in the complex, $i$ indexes the eigenvalues $\mu_i$ of the eigensections with respect to the operator $\Delta_M$, and $j \in \mathbb{N}$.

Moreover, these are eigensections of the operator $\Delta + \Delta_M$ on the same domain as $\Delta$. 
The operator $\Delta + \Delta_M$ has discrete spectrum, 
with eigenvalues $\lambda_{(i,j)}^2 + \mu_i$, where $\lambda_{(i,j)}^2$ are eigenvalues of the operator $\Delta$. Moreover we have that for any fixed $i$ the sum $\sum_{j} \lambda^{-4}_{(i,j)}$ is finite, and that for each even degree eigensection $\psi_{i,j}$ with positive eigenvalue $\lambda_{(i,j)}$ for $\Delta$, $\frac{1}{\lambda_{(i,j)}} D\psi_{i,j}$ is an odd degree eigensection with the same eigenvalues for $\Delta$ and $\Delta_M$.
\end{proposition}

\begin{proof}
For simplicity we will do the case of $\mathcal{R}_N(U_a)$. The more general case of $\mathcal{R}_B(U_a)$ can be handled similarly.
The decomposition of \eqref{Dolbeault_form_decomposition_1}
\begin{equation}
\label{Dolbeault_form_decomposition_22}
    \Omega^{0,q}(C(M);E)=\Omega^{0,q}(M;E) \oplus (dx-\sqrt{-1}xJ(dx)) \wedge \Omega^{0,q-1} (M;E)
\end{equation}
near the boundary at $\{x=1\}$ for the Dolbeault complex can be imposed on the spinor bundle using the identification $\mathcal{S} \otimes E=\oplus_q \Lambda^{0,q}X \otimes E$ discussed in Subsection \ref{subsubsubsection_Neumann_boundary_condition}. We refer to projections of sections to the first/ second summand as the \textit{tangential/ normal} parts of spinors respectively.

When imposing the classical $\overline{\partial}$-Neumann condition for the Dolbeault-Dirac operator, the normal parts vanish at the boundary. This is generalized by the domain \eqref{equation_domain_for_spin_c_123} (c.f. equation (12) of \cite{EpsteinSubellipticSpinc3_2007}) for the spin$^\mathbb{C}$ Dirac complex. 

Given a one form $\eta=\eta^{1,0}+\eta^{0,1}$ where $\eta^{1,0}$ is the holomorphic part and $\eta^{0,1}$ is the anti-holomorphic part (defined via the almost complex structure) respectively, the Clifford action corresponding to the Clifford module for the spin$^\mathbb{C}$ Dirac complex, identified with the sections of the spinor bundle can be written as
$cl(\eta)=\sqrt{2} (\eta^{0,1} \wedge-\iota_{X^{0,1}})$ where $X^{0,1}$ is the anti-holomorphic vector field dual to $\eta^{1,0}$. In particular, $cl(dx)=\sqrt{2}(\overline{\partial}x \wedge -\iota_{{\overline{\partial} x}^\#})$.

The discussion in \cite[\S 4.2]{Albin_2017_index} shows that the Dirac type operator has a decomposition as in equation \eqref{equation_Model_laplacian} where the constants $A$ depend on the rescaling of the sub-bundles and associated sections as discussed in the previous proposition. We will adapt the rescaling in \cite{Albin_2017_index}, and denote eigensections of $\Delta=D^2$ as
\begin{equation}
\label{equation_ansatz_Pierre_Jesse}
    \sqrt{x}J_{\nu_i}(\lambda x) \phi_i(m)
\end{equation}
where $\phi_i$ are eigensections of the operator $(cl(dx)D_Z -1/2)^2$ with eigenvalues $\nu^2_i$. Here the supersymmetry is generated by the operator $D$ of the two term complex, where each even degree eigensection with eigenvalue $\lambda$ has a corresponding odd degree eigensection with the same eigenvalue, and these sections are intertwined by $D$ (up to a normalization constant) by an argument similar to that in Lemma \ref{Lemma_super_symmetry}. Thus we call $D$ the supersymmetry generator, and each pair of intertwined eigensections are called a supersymmetric pair.

Given an eigensection $\psi_{\lambda}$ of $\Delta$  with eigenvalue $\lambda$, we have the corresponding \textit{supersymmetric partner} $D\psi_{\lambda}=\widetilde{\psi}_{\lambda}$ where $\widetilde{\psi}_{\lambda}$ is an eigensection of $\Delta$ with eigenvalue $\lambda$.
We have decompositions of these eigensections into normal and tangential parts
\begin{equation}
    \psi_{\lambda}=\psi^N_{\lambda}+\psi^T_{\lambda}, \hspace{5mm} \widetilde{\psi}_{\lambda}=\widetilde{\psi}^N_{\lambda}+\widetilde{\psi}^T_{\lambda}
\end{equation}
where the superscripts $N,T$ denote the normal and tangential parts respectively. 
The $0$-th order condition corresponds to $\psi^T_{\lambda}|_{x=1}=0$ and $\widetilde{\psi}^T_{\lambda}|_{x=1}=0$.

We denote $\psi_{\lambda}=\sqrt{x}J_{\nu_j}(\lambda x) \phi_j$ where $\phi_j=\phi^N_j+\phi_j^T$, and thus $\psi^N_{\lambda}=\sqrt{x}J_{\nu_j}(\lambda x) \phi^N_j$. The boundary condition in \eqref{equation_domain_for_spin_c_123} is satisfied if and only if $J_{\nu_i}(\lambda)=0$, a Dirichlet condition for the Bessel type Sturm-Liouville problem on the interval $[0,1]$ which yields the discretization of the spectrum that we desire.

Writing the supersymmetric partner as $D\psi_{\lambda}=\widetilde{\psi}_{\lambda}=\sqrt{x}J_{\nu_k}(\lambda x) \phi_k(m)$ using \eqref{equation_ansatz_Pierre_Jesse}, we note that the $0$-th order condition in \eqref{equation_domain_for_spin_c_123} for $D\psi_{\lambda}$ (which corresponds to $J_{\nu_k}(\lambda)=0$) is also precisely the $1$-st order condition for $\psi_{\lambda}$ corresponding to the domain for the Laplacian given by \ref{Laplacian_D_type}.

Since the $1$-st order boundary condition for a section $s$ of the Laplace-type operator is simply the $0$-th order condition for the section $Ds$, we see that these supersymmetric pairs satisfy the boundary conditions for the domain of the Laplace-type operator. Now we have shown that for a fixed eigensection of $\Delta_M$, the corresponding eigenvalue problem of $\Delta$ reduced to a Bessel problem with discrete spectrum. Now it is easy to see that the arguments in Proposition \ref{Proposition_spectral_properties} carry over, proving the result.
\end{proof}

\begin{remark}
The ansatz we used in Proposition \ref{Proposition_spectral_properties} is compatible with the ansatz in \eqref{equation_ansatz_Pierre_Jesse}.
This is due to special properties of Bessel functions which Cheeger uses in equations (3.21-3.23) of \cite{cheeger1983spectral} to explain invariance properties for eigensections in his ansatz that we used in Proposition \ref{Proposition_spectral_properties}.

The proof of Proposition \ref{Proposition_spectral_properties_almost_complex} is much simpler for the case of the Dolbeault-Dirac complex when the almost complex structure is integrable and K\"ahler in which case both said proposition and Proposition \ref{Proposition_spectral_properties} hold.
When the almost complex structure is integrable but not K\"ahler, the Dolbeault-Dirac operator and the spin$^{\mathbb{C}}$ Dirac operators are different, and the supersymmetry generators for the Dolbeault-Dirac operators  are different.
\end{remark}

Since the operators $\Delta, \Delta_M$ commute and have discrete spectrum, we can use functional calculus to write 
\begin{equation}
\label{equation_renormalizing_preparation_heat_kernel}
e^{-(s)(\Delta^q)+s(\Delta^q_M)}=\Big( \sum_{i} \sum_{j} e^{-s \lambda_j^2} e^{-s \mu_i}  \psi_{q,i,j} \otimes \psi^*_{q,i,j} \Big)  + \Big( \sum_i \sum_{c} e^{-s \mu_i} h_{q,i,c} \otimes  h^*_{q,i, c}  \Big)
\end{equation}
where $h_{q,i,c}$ make up an orthonormal basis of the harmonic sections of $\Delta^q$ satisfying the boundary conditions while $\psi_{q,i,j}$ are the eigensections introduced in the previous proposition, analogous to those in equation \ref{tempura}, which together form the complete orthonormal basis of eigensections.

\subsubsection{Renormalized local Lefschetz numbers.}
\label{subsubsection_renormalized_Lefschetz}

We will now prove the renormalized McKean-Singer theorem on the tangent cone that we have been building up to. 

\begin{definition}[Weighted trace functions]
\label{definition_renormalized_trace}
Let $U_p=C_x(M)$ be the truncated tangent cone at $p \in X$, with the Dolbeault complex $\mathcal{P}=\mathcal{P}_N(U_p)$, where we have picked the maximal domain after imposing VAPS conditions (see Remark \ref{remark_boundary_conditions_for_singular_links}). Let $f=\mathfrak{T}_p(f_1)$ be the induced self map on the tangent cone where $f_1:X \rightarrow X$ has an isolated simple fixed point at the singularity at $p=\{x=0\}$. Let $T=T_f$ be the associated to a geometric endomorphism on the complex on $C_x(M)$. We define the \textit{\textbf{weighted trace function}} by
\begin{equation} 
\begin{split}
Tr(s)\Big(\varphi_q \circ f^*|_{{\mathcal{H}}^{q}(\mathcal{P}_N(U_p))} \Big):= Tr \Big(e^{-s\sqrt{\Delta_M}}T_f|_{{\mathcal{H}}^{q}(\mathcal{P}_N(U_p))} \Big)
\end{split}
\end{equation}
where $\Delta_M$ is the non-negative Laplace-type operator on the link $M$ appearing in equation \eqref{equation_Model_laplacian}.
From Proposition \ref{Proposition_spectral_properties} we know that there is an orthonormal basis of harmonic sections of the corresponding Laplace-type operator $\Delta^q$ satisfying the boundary conditions, $h_{q,i,c}$, where the index $c$ varies in a finite set $\kappa_{q,i}$ for any $i$ such that $\mu_i>0$ by Remark \ref{Remark_growth_of_harmonic_sections}, where  $\mu_i$ is the eigenvalue of the eigensection $h_{q,i,c}$ with respect to the operator $\Delta_M$. 
While Remark \ref{Remark_growth_of_harmonic_sections} shows that there can be as many as $8$ harmonic sections corresponding to each $i$, the $L^2$ boundedness, VAPS conditions and the boundary conditions will reduce the harmonic sections that appear in the domains for the generalized Neumann condition. 
In any case for $\mu_i=0$, since the Laplace type operators $\Delta_M$ are Fredholm on $M$, again the index set $\kappa_{q,i}$ will be finite.
We therefore have that 
\begin{equation}
 Tr \Big(e^{-s\sqrt{\Delta_M}}T_f|_{{\mathcal{H}}^{q}(\mathcal{P}_N(U_p))} \Big)= \sum_{i} \sum_{c} e^{-s \sqrt{\mu_i}} \langle \varphi_q \circ f^*h_{q,i,c}, h_{q,i, c} \rangle
\end{equation}
and we will show that this converges for large $s$.

We showed in Remark \ref{Remark_growth_of_harmonic_sections} that for $\mu_i>0$ the
harmonic sections in this basis are of the form $h_{q,i,c}=x^{a_i}g_{\mu_i}(m)$ for large $\mu_i$, where $a_i$ are constants that grow of order $(\mu_i)^{1/2}$. 
We can pick these harmonic sections to be of unit norm on the truncated tangent cone. Then for self maps $f$ on the tangent cone that are infinitesimally attracting ($|f^*(x)| = A(m)x$ where $A(m)<1$), $\langle \varphi_q \circ f^*h_{q,i,c}, h_{q,i, c} \rangle$ for such a harmonic section has magnitude bounded by $A(m)^{a_i}$ times $||B||$, the operator norm of the map induced on the $L^2$ sections on the link by the pullback by $B:Z \rightarrow Z$ (see \ref{equation_maps_building_tangent_map}) composed with $\varphi_q$. This makes the series summable for $s \geq 0$. Therefore the weighted trace function is holomorphic on a right half plane including the origin and we can evaluate it at $s=0$. If the self map at the tangent cone is non-expanding, then the weighted trace is holomorphic on a right half plane for $s$ greater than some constant (determined by $\sum_i A(m)^{a_i} e^{-s\sqrt{\mu_i}}$). We can seek to meromorphically continue the weighted trace function such that there is a meromorphic extension which is regular at $s=0$ and evaluate it there. We will see that the ring structure of holomorphic functions will allow us to do this in Section \ref{Dolbeault_section}.
\textbf{It is easy to see that the construction above goes through for the spin$^{\mathbb{C}}$ Dirac complex, where the heat kernels and traces are for the odd and even degree spinors.}
\end{definition}

In the setting of Definition \ref{definition_renormalized_trace} we define the \textbf{\textit{Polynomial Lefschetz heat supertrace function}} generalizing Definition \ref{definition_polynomial_Lefschetz_supertrace} as
\begin{equation}
    \mathcal{L}(\mathcal{P}_N(U_p),T_f,p)(t,b,s):=\sum_{q=0}^n b^q Tr(T_q e^{-(t+s) \Delta} e^{-s\sqrt{\Delta_M}})
\end{equation}
for all $s>0$ and all $t>0$, when $\Delta^{(q)}+({\Delta_M^{(q)}})^{1/2}$ has a trace class heat kernel. For spin$^\mathbb{C}$ Dirac complexes, we have
\begin{equation}
    \mathcal{L}(\mathcal{R}_N(U_p),T_f,p)(t,b,s):=\sum_{q=0}^1 b^q Tr(T_q e^{-(t+s) \Delta} e^{-s\sqrt{\Delta_M}}).
\end{equation}
The following result shows that the \textit{\textbf{Lefschetz heat supertrace functions}} 
\begin{equation}
    \mathcal{L}(\mathcal{P}_N(U_p),T_f,p)(t,s):=\mathcal{L}(\mathcal{P}_N(U_p),T_f,p)(t,-1,s)
\end{equation}
obtained by setting $b=-1$ and similarly $\mathcal{L}(\mathcal{R}_N(U_p),T_f,p)(t,s)$ can be defined even when the heat kernel is not trace class.

\begin{theorem}[Renormalized McKean-Singer theorem]
\label{theorem_renormalized_McKean_Singer}
In the setting of Definition \ref{definition_renormalized_trace},
the Lefschetz heat supertrace function satisfies, for all $t,s> 0$
\begin{equation} 
\begin{split}
\mathcal{L}(\mathcal{R}_N(U_p), T_f, p)(t,s) = L(\mathcal{R}_N(U_p), T_f, p)(s) 
\end{split}
\end{equation}
where
\begin{equation} 
\begin{split}
L(\mathcal{R}_N(U_p), T_f, p)(s):= \sum_{q=0}^1 (-1)^q Tr(s) \Big(\varphi_q \circ f^*: {\mathcal{H}}^{q}(\mathcal{R}_N(U_p)) \rightarrow {\mathcal{H}}^{q}(\mathcal{R}_N(U_p)) \Big)
\end{split}
\end{equation}

\end{theorem}

\begin{proof}
From 
Proposition \ref{Proposition_spectral_properties_almost_complex}, we know that 
\begin{equation}
e^{-(t+s) \Delta} e^{-s \sqrt{\Delta_M}}=\Big( \sum_{i} \sum_{j} e^{-(s+t) \lambda_j^2} e^{-s \sqrt{\mu_i}}  \psi_{q,i,j} \otimes \psi^*_{q,i,j} \Big)  + \Big( \sum_{c} e^{-s \sqrt{\mu_i}} h_{q,i,c} \otimes  h^*_{q,i, c}  \Big)
\end{equation}
as in equation \eqref{equation_renormalizing_preparation_heat_kernel} in the notation introduced in Proposition \ref{Proposition_spectral_properties}.
We now compose this operator with the geometric endomorphism $T_f$ and express the supertrace as
\begin{multline}
\label{somename_blah1314}
Tr[ \sum_{q=0}^1 (-1)^q (\varphi_q \circ f^*) \circ e^{-(t+s) \Delta^{(q)}} e^{-s (\Delta_M^{(q)})^{1/2}}] \\ = \lim_{N \rightarrow \infty} \sum_{q=0}^1 (-1)^q
\Big( \sum_{i=1}^N \sum_{j=1}^{\infty} e^{-(t+s) \lambda_j^2} e^{-s \sqrt{\mu_i}} \langle \varphi_q \circ f^*\psi_{q,i,j},  \psi_{q,i,j} \rangle  + \sum_{c} e^{-s \sqrt{\mu_i}} \langle \varphi_q \circ f^*h_{q,i,c},  h_{q,i, c} \rangle \Big)
\end{multline}
where we have to justify that the limit exists. Here $\psi_{q,i,j}$ are the eigensections corresponding to the positive eigenvalues $\lambda^2_j$ ($\lambda_{q,i,j}$) from Proposition \ref{Proposition_spectral_properties_almost_complex} above and we have used functional calculus for the operator $\Delta^q+(\Delta_M^q)^{1/2}$ which has discrete spectrum, and the fact that $\Delta^q$ and $\Delta_M^q$ commute.

In Proposition \ref{Proposition_spectral_properties} and Proposition \ref{Proposition_spectral_properties_almost_complex}, we observed that for any given $(q,i)$, $\sum_j \lambda_j^{-4}$ is bounded.
Using Theorem 2.5 of \cite{karamataAmbrosio2018} (Karamata's Tauberian theorem) for $\gamma=4$ in the notation of that article, we see that having a bound $C_i/\Gamma(5)$ for $\sum_j \lambda_j^{-4}$ is equivalent to a bound $(C_i/s^{4})$ for $\sum_{j} e^{-s \lambda_j^2}$ as $s$ goes to $0$. Moreover Proposition \ref{Proposition_spectral_properties} shows that we have supersymmetric pairs for any fixed $q,i$, and we can replicate the arguments in the proof of Theorem \ref{Lefschetz_supertrace} to show that for any finite $N$, the expression in equation \ref{somename_blah1314} is equal to
\begin{equation}
    \sum_{q=0}^1 (-1)^q \sum_{c} e^{-s \sqrt{\mu_i}} \langle \varphi_q \circ f^*h_{q,i,c},  h_{q,i, c} \rangle.
\end{equation}
which is an alternating sum of the weighted trace functions in each degree and converges by the discussion above for large $s$. This proves the result.
 
\end{proof}


This result reduces to Theorem \eqref{Lefschetz_supertrace} if the boundary conditions yield a Fredholm complex with trace class heat kernel, without any renormalization needed. Indeed, this is the case for the de Rham complex (or its dual complex) in the smooth setting where we are dealing with the Hodge Laplacian on the disc with generalized Neumann conditions, which is well known to be elliptic in the sense of Lopatinsky-Shapiro (see \cite{brenner1990atiyah}). These results are crucial in Section \ref{Dolbeault_section}, where we will use ring structures of harmonic sections to meromorphically continue the weighted trace functions to all of $\mathbb{C}$, the evaluation of which at $s=0$ gives \textit{\textbf{renormalized trace formulas}} for local Lefschetz numbers.

For non-attracting fixed points, we can simply use the above result with the dual complex. More generally, for fundamental neighbourhoods which are metric products $U_1 \times U_2$, we can use Proposition \ref{Lefschetz_product_formula} with the weighted trace functions to compute local Lefschetz numbers. Next we explore this.

\subsubsection{Local to global formulae}
\label{subsubsection_Local_to_global_formulae}

Atiyah and Bott's formulas allow one to easily compute local Lefschetz numbers on smooth manifolds since the local geometry at the tangent cone is the same at all smooth points. However the local Lefschetz numbers at singular points see aspects of different singular structures that are captured by local cohomolgy groups in the case of the de Rham and spin$^\mathbb{C}$ Dirac complexes.

Many Lefschetz fixed point theorems proven by methods of homological algebra \cite{verdier1967lefschetz,Goresky_1985_Lefschetz,Baumformula81} are based on the idea that the local Lefschetz numbers are some form of trace on \textit{local cohomology groups}
We can now phrase the $L^2$ Atiyah-Bott-Lefschetz fixed point theorem for the spin$^\mathbb{C}$ Dirac and de Rham complexes in this way as well.
We first demonstrate how the results we have proven can be used to compute local Lefschetz numbers.

\begin{definition}
\label{definition_product_decompositions_formalized}
Let $f$ be a self map on a pseudomanifold $X$ with an isolated fixed point at $p$. Let $\mathcal{P}$ be an elliptic complex on $X$ where $f$ is associated to a geometric endomorphism $T=T_f$. Let there be a fundamental neighbourhood $U_p=U_1 \times U_2$ of the fixed point which has a product metric, where $f$ restricts to maps $f_1|_{U_1}$ and $f_2|_{U_2}$ on each factor and where $f_2$ is locally invertible. Moreover, let $f$ be non-expanding on $U_1$ and non-attracting on $U_2$. If there is such a decomposition, we say that \textit{\textbf{$U_p$ admits a product decomposition for the self map $f$}}.
If a fixed point of a self map admits a fundamental neighbourhood such that it admits a product decomposition, then we say that the \textit{\textbf{fixed point admits a product decomposition}}.
In this case we define the \textit{\textbf{product complex}} $\mathcal{P}_B(U_p):=\mathcal{P}_N(U_1) \times \mathcal{P}_N^*(U_2)$.
\end{definition}

Note that the definition of the product complex depends on the choice of decomposition, if $U_p$ admits any for a self map $f$. We assume that any isolated fixed point $p$ has a fundamental neighbourhood of the form $U_p= U_1\times U_2$, where $U_1=\mathbb{D}^{k_1} \times C(Z_1)$ and $U_2=\mathbb{D}^{k_2} \times C(Z_2)$, where the map is non-expanding on $U_1$ and non-attracting on $U_2$. Then Proposition \ref{proposition_local_product_Lefschetz_heat} can be used to compute the local Lefschetz numbers. 

Indeed, for simple isolated fixed points Theorem \ref{theorem_localization_model_metric_cone} shows that we can consider the fundamental neighbourhoods to be the truncated tangent cone, where the metric is of product type and for strictly non-expanding ($U_p=U_1$) or strictly non-attracting ($U_p=U_2$) fixed points, we have such fundamental neighbourhoods. Moreover in the smooth setting, the tangent cone is the tangent space and the linearized map on the tangent space can be split in this way based on the eigenvalues of $df$ at the fixed point.

For maps which are isometries on $U_p$ (isometries are both non-attracting and non-expanding), there may be multiple choices of decompositions. For instance consider the map $(z_1,z_2) \mapsto (e^{i\theta_1}z_1,e^{i\theta_2}z_2)$ on $\mathbb{C}^2$, and consider the unit polydisc neighbourhood to be $U_p$. We can take $U_1=\mathbb{D}^2_{z_1}$, and $U_2=\mathbb{D}^2_{z_2}$ which gives a decomposition. Alternately, we can take all of $U_p$ to be either $U_1$ or $U_2$. The following proposition shows how to compute local Lefschetz numbers given a product decomposition. 

\begin{proposition}
\label{proposition_local_product_Lefschetz_heat}
Let $f$ be a self map on a pseudomanifold $X$ with an isolated fixed point at $p$. Let $\mathcal{P}$ be an elliptic complex on $X$ where $f$ is associated to a geometric endomorphism $T=T_f$. Consider a fundamental neighbourhood $U_p$ such that it admits a product decomposition as $U_1\times U_2$ for the self map.
Let $\mathcal{P}_B(U_p)=\mathcal{P}_N(U_1) \times \mathcal{P}_N^*(U_2)$ be the corresponding product complex. We assume that the geometric endomorphism $T_f$ decomposes into a product of geometric endomorphisms on the two factors as $T_{f_1}$ and $T_{f_2}$, corresponding to the decomposition of the map $f$ into factors $f_1$ and $f_2$. 
For the local Lefschetz polynomials we have
\begin{equation}
    L(\mathcal{P}_B(U_p), T_f,p)(b,s)=L(\mathcal{P}_N(U_1), T_{f_1},p)(b,s) \cdot b^m L(\mathcal{P}^*_N(U_2),T^{P^*}_{f_2^{-1}},p)(b^{-1},s)
\end{equation} 
where $T^{P^*}_{f_2^{-1}}$ is the geometric endomorphism on the dual complex induced by $f_2^{-1}$ on $\mathcal{P}^*_N(U_2)$ (see Remark \ref{which_boundary_conditions}). Here $m$ is the maximal degree of the complex $\mathcal{P}_N^*(U_2)$. The Lefschetz numbers decompose as
\begin{equation}
\label{equation_want_to_refer_289}
    L(\mathcal{P}_B(U_p), T_f,p)=L(\mathcal{P}_N(U_1), T_{f_1},p) \cdot (-1)^m L(\mathcal{P}^*_N(U_2), T^{P^*}_{f_2^{-1}},p)
\end{equation}
where for geometric endomorphisms which are both non-attracting and non-expanding, $L(\mathcal{P}_B(U_p), T_f,p)$ is defined to be the evaluation at $s=0$ of the meromorphic continuation of $L(\mathcal{P}_B(U_p), T_f,p)(b=-1,s)$ when it exists.

\end{proposition}

We observe that given a product decomposition, the geometric endomorphisms $T_{f_1}=\phi \circ f_1^*, T_{f_2}=\phi \circ f_2^*$ and their adjoints exist for natural complexes such as the de Rham and spin$^\mathbb{C}$ Dirac complexes, but must be postulated in general, similar to how the existence of geometric endomorphisms must be postulated as we discussed following Definition \ref{definition_geometric_endomorphism_proper}.

\begin{proof}
For the product, we use Proposition \ref{Lefschetz_product_formula} in the local neighbourhood for the two complexes. In the Fredholm setting, this yielded equation \ref{product_cones_links_22} for the polynomial supertrace functions, and Corollary \ref{Theorem_local_Lefschetz_tangent_cone} for the local Lefschetz numbers.
On the non-attracting factor $U_2$, $f^{-1}(U_2) \subseteq U_2$ which shows we must use the adjoint endomorphism on the dual complex (see Remark \ref{which_boundary_conditions}). The relationship between the Lefschetz polynomials for the complex and the dual complex is given by Proposition \ref{proposition_Lefschetz_on_adjoint}, which goes through with the weighted trace functions. This proves the result.
\end{proof}

\begin{remark}
We observe that the idea in the above proposition of using the adjoint domains when computing Lefschetz numbers for attracting vs expanding maps generalizes to other choices of domains for Hilbert complexes, for instance the minimal and maximal domains which are adjoint to each other. We will demonstrate this in some explicit computations following Example \ref{example_cusp_singularity}.
\end{remark}

We remark that for simple isolated fixed points $p$ of $f$, if $U_p$ is the (truncated) tangent cone $\mathfrak{T}_pC$, then the condition that $\mathfrak{T}_pf$ is non-attracting/non-expanding is equivalent to $f$ being infinitesimally non-attracting/non expanding at $p$. In practice we will therefore use decompositions of the tangent cone as above.
We also observe that the proposition above shows that the local Lefschetz numbers are independent of the choice of product decomposition.
\begin{remark}
Henceforth we assume that all the fixed points of the maps we consider admit product decompositions.
\end{remark}

\begin{remark}
\label{Remark_on_local_boundary_condition_boundary_vanishing}
The local boundary conditions that we use for the de Rham and Dolbeault operators are widely studied in the smooth case, and it is well known that the heat kernel has off diagonal vanishing estimates in the limit as $t$ goes to $0$ up to the boundary (see \cite{brenner1981atiyah,perez2012heat}). The domains introduced by Epstein for spin$^\mathbb{C}$ Dirac complexes are less well studied but are similar to the case of the Dolbeault complex, yielding a self-adjoint extension of the Dirac operator. For simple isolated fixed points, with the realization of the spectrum as eigenvalues that we have from Proposition \ref{Proposition_spectral_properties}, we can express the heat kernel in terms of the the eigensections and functional calculus (Cheeger computes the heat kernel in this way for the infinite cone in \cite{cheeger1983spectral} deriving an expression of the form in equation \ref{heat_kernel_exact_cone}). In particular this includes the case of discs for which the vanishing estimates follow from \cite{brenner1981atiyah,brenner1990atiyah,perez2012heat}.
Similarly one can show that the heat kernel has off diagonal vanishing estimates up to the boundary in the limit as $t$ goes to $0$.

This phenomena is called the \textit{\textbf{Kac's principle of not feeling the boundary}} and has been identified to follow from general considerations (see, e.g. \cite{li2016heat}).
\end{remark}

Now we can prove the following theorem which can be stated roughly as \textit{the sum of local supertraces in cohomology equals the global supertrace in cohomology}. 

\begin{theorem}
\label{Corollary_Polynomial_Lefschetz_fixed_point_theorem}
Let $\widehat{X}$ be a pseudomanifold with the de Rham or spin$^\mathbb{C}$ Dirac complex (possibly with twisted coefficients) $\mathcal{P}$, and let $f$ be a self map on $\widehat{X}$ with simple isolated fixed points, associated to a geometric endomorphism $T_f$ on the complex. Then we have
\begin{equation}
\label{equation_11tgedgZAHSHJ}
    \sum_{p \in Fix(f)}  L(\mathcal{P}_B(U_p), T_f)= L(\mathcal{P}(X), T_f),
\end{equation}
where $L(\mathcal{P}_B(U_p), T_f)=L(\mathcal{P}_B(U_p), T_f,p)$ is as introduced in Proposition \ref{proposition_local_product_Lefschetz_heat}.
\end{theorem}

\begin{proof}
We know from Theorem \ref{localization_of_simple_fixed_points} that in the limit as $t$ goes to $0$, the local Lefschetz numbers depend only on the map and the heat kernel in a neighbouhood of the fixed point. Proposition \ref{Gluing_heat_kernels} shows that in the limit as $t$ goes to $0$, the heat kernel of the Laplace-type operator on $X$ used in the definition of the local Lefschetz numbers can be replaced by the heat kernel of the Laplace-type operator corresponding to the local complex.
This shows that we have
\begin{equation}
\label{Ding_Liren_21_preparation}
    \Big( \sum_{p \in Fix(f)}  \mathcal{L}(\mathcal{P}_B(U_p), T_f)(t,-1,s=0) \Big) - \mathcal{L}(\mathcal{P}(X), T_f)(t,-1) = E(t)
\end{equation}
where the error $E(t)$ vanishes as $t$ go to $0$, where Theorem \ref{theorem_renormalized_McKean_Singer} is used to show that $\mathcal{L}(\mathcal{P}_B(U_p), T_f)(t,-1,0)$ is the local Lefschetz number at the fixed point in the case of the spin$^\mathbb{C}$ Dirac complex. 

Since we have off-diagonal vanishing of the heat kernels with the local boundary conditions, the result follows from Proposition \ref{Local_Lefschetz_numbers_definition} and the vanishing of boundary contributions for the local boundary conditions which shows that $L(\mathcal{P}_B(U_p), T_f)=L(\mathcal{P}_B(U_p), T_f,p)$.
Proposition \ref{proposition_local_product_Lefschetz_heat} shows that the complex we choose yields equation \eqref{equation_11tgedgZAHSHJ}.

\end{proof}

\begin{remark}
\label{Remark_keep_track_spin_c_Dirac_2}
Here the fact that the boundary conditions are local is crucial for the boundary contributions to vanish. In \cite{atiyah1975spectral}, the eta invariant at the boundary arises in the formulae for the index of Fredholm operators because there is no off diagonal vanishing.
\end{remark}


\section{The de Rham complex}
\label{section_de_Rham}

In this section we study Lefschetz numbers in the case of the de Rham complex for Witt spaces. A \textit{\textbf{Witt space}} in this section is a stratified pseduomanifold where each link has vanishing middle intersection cohomology in the middle degree $l/2$, where $l$ is the dimension of the link. In particular this is always satisfied if all links are odd dimensional.

It is well known that the $L^2$ de Rham cohomology of a Witt space endowed with a wedge metric, is isomorphic to the middle perversity intersection homology of Goresky and MacPherson (see, e.g., the discussion following Theorem 1 of \cite{Albin_2017_index}). The Lefschetz fixed point theorem in the latter homology theory setting was proven in \cite{Goresky_1985_Lefschetz} for Witt spaces. 
Moreover Cheeger showed (see \cite{cheeger1982l2}) that for any wedge metric, the minimal and maximal extensions for the de Rham complex coincide. Indeed, one way to see this is to note that the $L^2$ spaces of forms, as topological vector spaces, are the same for quasi-isometric wedge metrics and so $\mathcal{D}_{\max}(d)$ and $\mathcal{D}_{\min}(d)$ remain unchanged if we replace the metric with a wedge metric satisfying the geometric Witt condition. But then we have $\mathcal{D}_{\min}(\eth)=\mathcal{D}_{\max}(\eth)$, which implies $\mathcal{D}_{\max}(d)=\mathcal{D}_{\min}(d)$.

We shall begin by giving a brief introduction to these results of Goresky and MacPherson. The maps for which they define Lefschetz numbers are placid maps, more general than those which lift to diffeomorphisms in the resolved spaces and in Subsection \ref{subsection_HS_replacement} we will consider geometric endomorphisms for the de Rham complex that will relax the conditions imposed on self maps. This construction is based on the replacement for the pullback that was constructed by Hilsum and Skandalis \cite{Hilsum_1992_invariance}, which was extended to the setting of Witt spaces in \cite{Albin_signature}. 

Since $L^2$ cohomology on a stratified pseudomanifold is completely determined by the metric on $\widehat{X}^{reg}$, we can define the Hilsum-Skandalis geometric endomorphism for self-maps which have good properties on $\widehat{X}^{reg}$. All one needs is that the \textbf{self map on $\widehat{X}^{reg}=\mathring{X}$ is smooth, and is proper with respect to a complete \textit{edge} metric on $\mathring{X}$}.
While this is satisfied by all placid self-maps, we note that the extension to the singular strata is not important when considering maps induced on the global cohomology. However, in order to obtain suitable local Lefschetz numbers, we will need to assume more structure for the self maps restricted to fundamental neighbourhoods of fixed points at the singular strata.

In Subsection \ref{subsection_induced_map_in_cohomology}, we prove that stratum preserving homotopy equivalent maps induce the same endomorphism at the level of $L^2$ cohomology for the Hodge de Rham complex. We will also show that if a map already induces a geometric endomorphism by pullback, then the Hilsum-Skandalis replacement will induce the same map on $L^2$ cohomology.
We also do hybrid constructions which allows us to study local Lefschetz numbers of self maps which are codimension preserving near fixed points of the self map at singular strata.

An analytic proof of the formula for the local Lefschetz numbers derived by Goresky and MacPherson in the case of Witt spaces with isolated conic singularities was presented in \cite{Bei_2012_L2atiyahbottlefschetz} for self maps which lift to diffeomorphisms, by appealing to the isomorphism between de Rham cohomology and intersection homology. This isomorphism has been used in extensions of the Cheeger-M\"uller theorem on spaces with isolated conic singularities (see \cite{ludwig2017comparison,ludwig2020extension}).
The local cohomology groups used in Goresky and MacPherson are isomorphic to the ones we defined in the previous section with (generalized Neumann) boundary conditions for the de Rham complex.

Goresky and MacPherson show how the local Lefschetz number of a non-repelling fixed point is computed using the absolute homology of a fundamental neighbourhood while that of a repelling fixed point corresponds to the relative homology of a fundamental neighbourhood. 
This is a topological version of the duality that we used in Proposition \ref{proposition_Lefschetz_on_adjoint} and Proposition \ref{proposition_local_product_Lefschetz_heat}.
In Subsection \ref{subsection_witten_deformation_morse_de_Rham} we will give an analytic analogue of Goresky and MacPhersons Morse inequalities in intersection homology.

One reason for an extensive study of these local cohomology groups is that we can use the ideas and techniques developed here for more general operators, in particular for the Dolbeault operator in the next section. We will study further applications and computations of our results in this section in subsections \ref{subsection_Morse_Hirzebruch_cohomological} and \ref{subsection_Computations_for_singular_examples}.

\begin{remark}
The work in this section generalizes \textit{mutatis mutandis} to the case of the de Rham complex with coefficients in a flat vector bundle, which is important in applications such as the proof of the Cheeger-M\"uller theorem. We write everything in the case of the trivial bundle for brevity.
\end{remark}

\subsection{The Intersection Lefschetz number}
\label{subsection_intersection_Lefschetz_number}

Goresky and MacPherson initiated the study of intersection homology on stratified pseudomanifolds and went on to study stratified Morse theory and the topology of singular spaces. Intersection homology is defined for a choice of perversity associated to the stratified space. The $L^2$ cohomology of the maximal domain is isomorphic to the intersection homology with ``lower middle" perversity while the $L^2$ cohomology of the minimal domain is isomorphic to the intersection homology with ``upper middle" perversity (see Theorem 6.1 in \cite{cheeger1980hodge}). For Witt spaces the intersection homology corresponding to the upper middle and lower middle perversities coincide.

In \cite{Goresky_1985_Lefschetz}, the authors proved a Lefschetz fixed point theorem for intersection homology which we review here.
\begin{definition}
Let $\widehat{X}$ be a stratified Witt space, and $f$ be a placid self map on it. This induces a self homomorphism on the intersection homology of the space with middle perversity (upper and lower middle perversities coincide for Witt spaces)
\begin{equation*}
    f_{*_k} : IH^{\overline{m}}_k(\widehat{X}) \longrightarrow IH^{\overline{m}}_k(\widehat{X})
\end{equation*}
and its \textit{\textbf{intersection Lefschetz number}} is
\begin{equation*}
    IL(f) = \sum_{k=0}^{n} (-1)^k Trace(f_{*_k})
\end{equation*}
\end{definition}
It is clear from the definition that $IL(f)$ corresponds to $L( \mathcal{P}(X), T_f)$ for the de Rham complex $\mathcal{P}(X)=(L^2\Omega(X),d)$ with the VAPS domain defined in Subsection \ref{subsection_Hilbert_complexes_stratified_pseudomanifolds}. 

The authors of \cite{Goresky_1985_Lefschetz} go on to show that there exist fundamental classes $[G(f)]$ and $[\Delta]$ in $IH^{\overline{m}}_n(\widehat{X} \times \widehat{X})$, for the graph $G(f) \subset \widehat{X} \times \widehat{X}$ and the diagonal $\Delta \subset \widehat{X} \times \widehat{X} $, and then that the intersection Lefschetz number is equal to the intersection number of these classes. The definitions of these two intersecting classes are subtle, as indicated in the paper. This is similar to how the heat supertrace localizes to where the diagonal of the double space intersects the graph of the self map $f$ on the double space in the smooth case but in the singular case we need to resolve these sub spaces. The following is Theorem 1 of \cite{Goresky_1985_Lefschetz}.
\begin{theorem}
\label{Intersection Lefschetz}
\begin{equation*}
    IL(f)= [G(f)].[\Delta]
\end{equation*}
\end{theorem}

From the theorem, it follows that non-vanishing $IL(f)$ implies that $f$ has fixed points.

If the fixed points are isolated, they obtain formulas for the local contribution to the Lefschetz number. The following is Theorem 3 of \cite{Goresky_1985_Lefschetz}.
\begin{theorem}
If p is an isolated fixed point of f
which is attracting, then there is an induced map on the local intersection homology with (upper) middle perversity,
\begin{equation*}
\overline{f}_{*_k} : IH^{\overline{m}}_k(\widehat{X}, \widehat{X}-p) \longrightarrow IH^{\overline{m}}_k(\widehat{X},\widehat{X}-p)
\end{equation*}
and the contribution to IL(f) is given by
\begin{equation*}
\sum_{k=0}^{n} (-1)^k Trace(\overline{f}_{*_k})
\end{equation*}
where $n$ is the dimension of the pseudomanifold $\widehat{X}$.
\end{theorem}

In the following, we briefly review the results of Section 10 of \cite{Goresky_1985_Lefschetz} to which we refer the interested reader. All the intersection homology groups are in middle perversity (upper and middle perversities being equal on Witt spaces).
If $f$ is a placid map and $p$ is an isolated fixed point, then one can find fundamental neighbourhoods $\widehat{V}_1, \widehat{V}_2$ around $p$, such that $\widehat{V}_1 \subset f^{-1}(\widehat{V}_2 \setminus \partial \widehat{V}_2)$. The map $f:\widehat{V}_1 \rightarrow \widehat{V}_2$ induces homomorphisms,
\begin{equation*}
(f^p_*)_k : IH^{\overline{m}}_k(\widehat{V}_1) \longrightarrow IH^{\overline{m}}_k(\widehat{V}_2) \quad (f^*_p)_{n-k} : IH^{\overline{m}}_{n-k}(\widehat{V}_2,\partial \widehat{V}_2) \longrightarrow IH^{\overline{m}}_{n-k}(\widehat{V}_1,\partial \widehat{V}_1)
\end{equation*}
 where the latter is the adjoint homomorphism of the former, on relative intersection homology.
Goresky and MacPherson define their \textbf{\textit{local trace}} to be
\begin{equation*}
    Tr_p(f)= \sum_{i=0}^{n} (-1)^i Trace((f^p_*)_k)
\end{equation*}
which gives the local Lefschetz number for an attracting fixed point. They then extend it to non-expanding fixed points. In the case where $f$ is expanding, if $f$ is a diffeomorphism restricted to $U_1$, then $(f^{-1})_*$ corresponds to the adjoint homomorphism $f^*$ in relative homology and the local intersection Lefschetz number can be written as
\begin{equation}
\label{intersection_relative_number}
    Tr_p(f) = \sum_{i=0}^{n} (-1)^i Trace((f^*_p)_k).
\end{equation}
where the trace is now over the relative homology. For Witt spaces Poincar\'e duality holds and we have that
\begin{equation*}
    IH^{\overline{m}}_{k}(\widehat{V}_1)=IH^{\overline{m}}_{n-k}(U_1,\partial \widehat{V}_1)
\end{equation*}
and if the link at the point $p$ is $Z$, then
\begin{equation}
\label{equation_Intersection_Homology_of_cone}
IH^{\overline{m}}_{k}(\widehat{V}_1)=
\begin{cases}
      0, & \text{for}\ k \geq [\frac{n+1}{2}] \\
      IH^{\overline{m}}_{k}(Z), & \text{for}\ k \leq [\frac{n-1}{2}]
\end{cases}
\end{equation}
which can be used to simplify the expression for the local Lefschetz number.

\subsection{Hilsum-Skandalis replacement}
\label{subsection_HS_replacement}

We have had brief glimpses of the geometric endomorphism induced by a self map and its adjoint in the de Rham case in the discussion following Definition \ref{definition_geometric_endomorphism_proper} and in Example \ref{remark_notation_bundles_vs_density_bundles_2}. This was under the condition that the self map $f$ was a diffeomorphism on the resolved space $X$.

Goresky and MacPherson proved their result for arbitrary placid self maps on Witt spaces, which include branched coverings and flat maps on algebraic varieties (see \cite[\S 1]{Goresky_1985_Lefschetz}).
We now present a construction of a geometric endomorphism for stratum preserving maps, which we can use to define global Lefschetz numbers.

Consider a disc bundle $\pi_X:\mathbb{D}_X \rightarrow \widehat{X}^{reg}=\mathring{X}$ over the regular part of a stratified pseudomanifold 
$\widehat{X}$ and the associated pullback bundle $f^*\mathbb{D}_X$, by the smooth resolution of a stratum preserving map $f: X^{reg} \rightarrow X^{reg}$. Denote by $\pi^f_X:f^*\mathbb{D}_X \rightarrow \mathring{X}$, the pullback bundle which fits in the commutative diagram
\begin{equation}
\begin{CD}
f^*\mathbb{D}_X @>\mathbb{D}(f)>> \mathbb{D}_X\\
@VVV @VVV\\
 \mathring{X} @>f>>  \mathring{X}.
\end{CD}
\end{equation}

Consider a (smooth) map $e: \mathbb{D}_X \rightarrow  \mathring{X}$ such that $p=e \circ \mathbb{D}(f): f^*\mathbb{D}_X \rightarrow  \mathring{X}$ is a submersion, and a choice of Thom form $\mathcal{T}$ for $\pi_X$. Following \cite{Albin_signature}, one can then define the Hilsum-Skandalis replacement of the pullback of $f$ as follows.

\begin{definition}
\label{HS_replacement}
The \textbf{\textit{Hilsum-Skandalis replacement}} for the pullback of forms by $f$ is given by
\begin{gather*}
    HS(f)= HS_{\mathcal{T},f^*\mathbb{D}_X,\mathbb{D}_X,e}(f): C_{c}^{\infty}(\mathring{X},\Lambda^*) \rightarrow C_c^{\infty}(\mathring{X},\Lambda^*)\\
    u \longrightarrow (\pi_X)_*(\mathcal{T} \wedge p^*u)
\end{gather*}
where $C_{c}^{\infty}(\mathring{X},\Lambda^*)$ denotes smooth forms with compact support on $\mathring{X}$. 
We will define the Hilsum-Skandalis replacement for self-maps that are proper on $\mathring{X}$ with respect to an edge metric, and we will refer to such self maps as \textit{\textbf{proper}} self maps. Equivalently, these are maps which preserve the singular set $\widehat{X}^{sing}$.

To construct a suitable map $e: \mathbb{D}_X \rightarrow  \mathring{X}$, first choose a complete edge metric $g_{e}$ on $\mathring{X}$. A complete edge metric is what is called an iterated edge metric in \cite{Albin_signature}, as opposed to an incomplete edge metric which is what we call wedge metrics in this article. We denote by $\mathbb{D}_X$ the disc sub-bundle $\pi_X: \mathbb{D}_X \rightarrow \mathring{X}$ of the edge tangent bundle $^{\text{e}} TX$, where we pick the radius of the discs to be less than half of the injectivity radius $r_L$ (we know there is a positive injectivity radius $r_L$ for complete edge metrics on these spaces, for instance by results in \cite{ammann2007pseudodifferential}).

The exponential map of the metric $g_e$ extends uniquely to a smooth edge submersion ${exp}$ on the disc sub-bundle ${\mathbb{D}_X}$ (see Lemma 3.3 in \cite{albin2013novikov}). This induces a surjective map $exp_*: ^{\text{e}} T \mathbb{D}_X \rightarrow   ^{\text{e}} TX$. Then the map $p=\exp \circ \mathbb{D}(f)$ is a submersion onto $\mathring{X}$. The pullback by a submersion maps $L^2$ sections to $L^2$ sections, and integrating along the fibers (which are compact) after taking the wedge product with a Thom form gives the $L^2$ bounded operator, $HS(f)$. This was observed in Section 9.1 of \cite{Albin_signature} and in section 3 of \cite{albin2013novikov}. 
\end{definition}

For the de Rham complex, the Hilsum-Skandalis replacement commutes with the differential $d$ on smooth compactly supported sections on $\mathring{X}$, since pullback by $p$ and pushforward along $\pi_X$ commute with the differential, and since Thom forms are closed.

We now develop some tools to show that this construction gives a geometric endomorphism on the complex. First we will show that Thom forms restricted to compact sets in $\mathring{X}$ can be realized as fiberwise volume forms. We will now explain what this means.

Since $X^n$ is orientable (i.e., $\mathring{X}$ is a smooth orientable manifold of dimension $n$) then $\Lambda^{n}\prescript{e}{}{TX}$ is trivializable and hence so is its pullback to $\mathbb{D}X$. Recall that a Thom form $\mathcal{T}$ on the bundle $\pi_X: f^*\mathbb{D}_X^n \rightarrow X^n$ is a closed form in $\Lambda^n(T(f^*\mathbb{D}_X))$ which integrates to 1 on each fiber (i.e., $(\pi_X)_*(\mathcal{T})=1$). 

\begin{proposition}
\label{any_Thom_form_is_a_fiberwise_volume}
Given a Thom form $\mathcal{T}$ and a family of metrics $g_{fiber}$ on the fibers of $\pi_X: f^*\mathbb{D}_X^n \rightarrow X^n$, there is a function $h:f^*\mathbb{D}_X^n \rightarrow \mathbb{R}$ such that the volume form of 
\begin{equation}
    g_{f^*\mathbb{D}_X}=\pi_X^*(g_e)+h g_{fiber}
\end{equation}
is $\mathcal{T} \wedge \pi_X^*dV_{g_e}$.
\end{proposition}
\begin{proof}
It is clear that the volume form of the Riemannian metric $g_{f^*\mathbb{D}_X}$ factors into $dvol_{g_e}$ and the volume form of $g_{\text{fiber}}$ up to a power of $h$. Since all Thom forms are sections of a trivializable bundle, we can choose the power of $h$ to be the Thom form divided by the volume form of the fiber metric (since the latter is a non-vanishing quantity).
\end{proof}

In the following section, we will prove that the Hilsum-Skandalis replacement of stratum preserving self maps which are simple b-maps gives a geometric endomorphism for the de Rham differential $d$.
\subsubsection{Mapping properties of the Hilsum-Skandalis geometric endomorphism}


We begin by proving some elementary mapping properties and then showing that the exterior derivative with the minimal domain commutes with the geometric endomorphism $HS(f)$, for a proper self map $f: X \rightarrow X$. We have already shown that they commute on the dense set $C_c^{\infty}(\mathring{X}, {\Lambda^*})$. 
This is analogous to Proposition \ref{Geometric_Endomorphism_Diffeo}. In what follows, we denote ${\Lambda^i}:=\Lambda^i (\prescript{w}{}{T}^*X)$ and $\varphi_i \circ HS(f)$ by $T_i$. We observe that wedge forms with compact support in $\mathring{X}$ are just smooth forms with compact support on $\mathring{X}$, for which we defined the Hilsum-Skandalis replacement in Definition \ref{HS_replacement}.

\begin{proposition}
\label{basic_properties_geometric_endo}

The Hilsum-Skandalis geometric endomorphism defined above satisfies the following properties:
\begin{enumerate}

    \item For each i and for each  ${v \in C^{\infty}_c(\mathring{X}, {\Lambda^i})}$, we have ${T_i v \in C^{\infty}_c(\mathring{X}, {\Lambda^i})}$

    \item For each i, $T_i$ extends to a bounded operator on ${L^2(X, {\Lambda^i})}$ (we denote this extension by $T_i$ as well).
    
    \item For each i, let ${T_i^* : L^2(X, {\Lambda^i}) \rightarrow L^2(X, {\Lambda^i}) }$ be the adjoint of $T_i$.
    Then ${T_i^* :  C^{\infty}_c(\mathring{X}, {\Lambda^i}) \rightarrow  C^{\infty}_c(\mathring{X}, {\Lambda^i}) }$. 

\end{enumerate}
\end{proposition}

\begin{proof}
What we need to show are the mapping properties of $HS(f)$ and $HS(f)^*$.

\textit{proof of property 1}:
It is easy to see that $HS(f)$ takes smooth sections to smooth sections. Since we picked the Thom form 
such that it is smooth and supported on the ball bundle of radius less than $r_L$ on $\mathbb{D}_X$, $supp(HS(f) \psi ) \subseteq \overline{\bigcup_{x \in supp(\psi)} \mathbb{B}_{x,r_L}}$, where $\mathbb{B}_{x,r_L}$ denotes the open ball with radius less than or equal to the lower bound on the injectivity radius around $x \in \mathring{X}$, with respect to the complete metric used in the definition of $HS(f)$. Since the metric is complete, and the singular strata are at the boundary hypersurfaces at infinity, we have that $X^{sing} \cap \mathbb{B}_{x,r_L}=\emptyset$ for all $x \in supp(\psi)$.
The proper condition shows that compactly supported sections are mapped to compactly supported sections.

\textit{proof of property 2}: 
The operator is a composition of the pullback by a submersive b-map, which induces a bounded map,  
exterior multiplication by a Thom form, and integration along the fibers, which are all bounded maps.


\textit{proof of property 3}: 
Let $u, v$ be wedge forms of the same degree, compactly supported in $\mathring{X}$. 
Then

\begin{equation*}
    \langle HS(f)u,v \rangle_X= \int_X HS(f)u \wedge \ast_X v
\end{equation*}
where $\ast_X$ is the Hodge star operator on $\mathring{X}$. By the definition of the Hilsum-Skandalis operator, this is equal to
\begin{equation*}
    \int_X \int_{f^*\mathbb{D}_X/X} \mathcal{T} \wedge p^*u \wedge \ast_X v = \int_{f^*\mathbb{D}_X} \mathcal{T} \wedge p^*u \wedge \pi_X^*(\ast_X v). 
\end{equation*}
Since we can write the Thom form we constructed as the fiberwise volume form of a metric on $\mathbb{D}_X$, the expression above is further equal to 

\begin{equation*}
    \pm \int_{f^*\mathbb{D}_X} p^*u \wedge \ast_{f^*\mathbb{D}_X} \pi_X^*v
\end{equation*}
where $\ast_{f^*\mathbb{D}_X}$ is the Hodge star operator of the metric on the bundle we constructed in Proposition \ref{any_Thom_form_is_a_fiberwise_volume} and where the equality is up to a sign because we commute the forms across the wedge product. But this is equal to

\begin{equation*}
    \langle p^*u, \pi_X^*v \rangle_{f^*\mathbb{D}_X} = \langle u, p_* \pi_X^*v \rangle_X
\end{equation*}
which shows that $HS(f)^* = p_* \pi_X^*$, up to a sign. If we show that $p$ is a proper submersion, then by Ehresmann's theorem, $p_* \pi_X^*v$ is a smooth section. But $p=\exp \circ \mathbb{D}f$ is proper because $\mathbb{D}X$ is a ball bundle where the radii of the balls are bounded and the exponential map is proper. If $v$ has compact support in $\mathring{X}$, so does $p_* \pi_X^*v$ by the same argument as in the proof of the first part of the proposition.
\end{proof}

We now prove the analog of Proposition 10 of \cite{Bei_2012_L2atiyahbottlefschetz}, which finishes the proof that $HS(f)$ is indeed a geometric endomorphism of the de Rham complex on a Witt space.

\begin{proposition}
\label{commutates with operator}

In the same setting, for each $i=0,1,2,...,n$, for each $s \in \mathcal{D}_{\min}(d)_i$, we have $T_i(s)=HS(f)(s) \in \mathcal{D}_{\min}(d)_i$ and $(d)_i \circ T_i = T_{i+1} \circ (d)_i$.
\end{proposition}

\begin{proof}

The proof given in Proposition \ref{Geometric_Endomorphism_Diffeo} goes through here with the modifications for the Hilsum-Skandalis operator given by Proposition \ref{basic_properties_geometric_endo} proven above replacing the mapping properties given by Lemma \ref{lemmata_mapping_properties} which are used in the proof of Proposition \ref{Geometric_Endomorphism_Diffeo}.

\end{proof}

\begin{remark}[Independence of choice of Thom form]
The Hilsum-Skandalis geometric endomorphisms corresponding to different Thom forms which satisfy the conditions that we impose will all induce the same map in $L^2$ cohomology. This 
can be proven by an argument similar to that in Proposition \ref{Homotopy_operator_1}.
\end{remark}

\subsection{Induced maps in cohomology}
\label{subsection_induced_map_in_cohomology}

We now show that if the pullback by $f$, $f^*$ already induces a  self map on the VAPS domain, then the two geometric endomorphisms constructed by using $f^*$ and $HS(f)$ induce the same map in the $L^2$ de Rham cohomology.


\begin{proposition}
\label{Homotopy_operator_1}
Let $f : \mathring{X} \rightarrow \mathring{X}$ be a proper self map. If the pullback by $f$ induces a bounded operator $f^* : \mathcal{D}_{\min}(d_X) \rightarrow \mathcal{D}_{\min}(d_X)$, then $HS(f)$ and $f^*$ induce the same endomorphism on the cohomology of the complex $\mathcal{P}(X)=(L^2\Omega^k(X), d_X)$.
\end{proposition}

\begin{proof}
We begin by setting up a lemma. Given the proper map, $f :\mathring{X} \rightarrow \mathring{X}$, we define,
\begin{equation}
\label{temp_1}
    (e_s)_{0 \leq s \leq 1} : \mathbb{D}_X \rightarrow \mathring{X}, \quad
    (x,\xi) \rightarrow exp(x,(1-s)\xi)
\end{equation}
This is a homotopy of maps which induces a surjective map on wedge vector fields for each $s$. We consider the homotopy constructed by composing with the lift of $f$ to the disc bundle,
\begin{equation}
\label{temp_2}
    p_s=e_s \circ \mathbb{D}(f_s) : f_s^*\mathbb{D}_X \rightarrow \mathring{X}.
\end{equation}

\begin{lemma}
With $f, e$ and $p$ as in \eqref{temp_1} and \eqref{temp_2} above, there exists a bounded (homotopy) operator $\Upsilon: L^2(X;E) \rightarrow L^2(f^*\mathbb{D}_X;E)$ where $E=\Lambda^*(\prescript{w}{}{T}^*X)$ such that

\begin{equation}
\label{Homotopy_operator}
     \mathcal{T} \wedge p_1^* - \mathcal{T} \wedge p_0^* = d \Upsilon + \Upsilon d
\end{equation}
where $d=d_s\otimes 1 + 1 \otimes d_X$ (where $d_s$ is the exterior differential on forms on $[0,1]$) acts on $L^2(f^*\mathbb{D}_X \times [0,1] ; F)$ where $F=\Lambda^*(T^*[0,1]) \times E$ and we have the explicit formula $\Upsilon(\omega) = \mathcal{T} \wedge \int_0^1 i_{\frac{\partial}{\partial_s}}(p_s^*)\omega ds$.
\end{lemma}

\begin{proof}
We follow lemma 9.1 in \cite{Albin_signature}.
Let $\omega \in L^2(X; E)$. Then $p_s^*{\omega} \in L^2(f^*\mathbb{D}_X \times [0,1];F)$ can be decomposed as $p^*_s(\omega) = \omega_t + ds \wedge \omega_n$. Then $\Upsilon(\omega)= \mathcal{T} \wedge \int_0^1 \omega_n ds$.
The proof follows from the computation
\begin{multline}
\label{tarantino_27}
    (d \Upsilon + \Upsilon d) \omega = d ( \mathcal{T} \wedge \int_0^1 \omega_n ds) + \mathcal{T} \wedge \int_0^1 i_{\frac{\partial}{\partial_s}}(p_s^*) d\omega ds \\
    = \mathcal{T} \wedge \int_0^1 d_X \omega_n ds + \mathcal{T} \wedge ds \wedge \int_0^1 \partial_s \omega_n ds + \mathcal{T} \wedge \int_0^1 i_{\frac{\partial}{\partial_s}} dp_s^*\omega ds  
\end{multline}
where in the last step we have used naturality, and that Thom forms are closed. The last term simplifies into the expression
\begin{equation*}
    \mathcal{T} \wedge \int_0^1 i_{\frac{\partial}{\partial_s}} (d_X \omega_t + ds \wedge \frac{\partial}{\partial_s} \omega_t - ds \wedge (d_X\omega_n)) ds = \mathcal{T} \wedge \int_0^1 \frac{\partial}{\partial_s} \omega_t - d_X\omega_n ds.    
\end{equation*}
Substituting this into the original expression in equation \eqref{tarantino_27} and doing the obvious cancellation, we get 
\begin{equation*}
    (d \Upsilon + \Upsilon d) \omega = \mathcal{T} \wedge [\omega_t + ds \wedge \omega_n] \big|_0^1 = \mathcal{T} \wedge (p_1^*- p_0^*) \omega
\end{equation*}
which proves the result.
\end{proof} 

To complete the proof of the theorem, observe the following.

\begin{center}
    $HS(f)(u)-f^{*}(u) = (\pi_X)_{*}(\mathcal{T} \wedge [u(exp(f(x),\zeta)
    - u(exp(f(x),0)])$
\end{center}
where ${\zeta}$ is the variable over the cotangent bundle and the integrand is

\begin{equation}
    \mathcal{T} \wedge (u(exp(f(x),\zeta) - u(exp(f(x),0)) 
     = \mathcal{T} \wedge \int_{0}^{1} {\iota_{\partial t} u(exp(f(x),t\zeta)} dt
\end{equation}

The last equation is exactly the left hand side of \eqref{Homotopy_operator}. What remains to be seen is that $\Upsilon$ composed with the pushforward $(\pi_X)*$ and the  bundle maps $\varphi_i$ preserves the minimal domain. 
But this follows from the same reasons that ensure the Hilsum-Skandalis replacement preserves the minimal domain. 
Since \eqref{Homotopy_operator} gives a homotopy operator $\Upsilon$, the fact that $(d \Upsilon + \Upsilon d)$ takes closed forms to exact forms, along with the fact that $d$ commutes with $(\pi_X)_*$ proves the theorem.
\end{proof}

In a similar vein, we can prove that proper maps that are homotopy equivalent (via proper maps) induce the same endomorphisms in $L^2$ de Rham cohomology. In particular, this proves that the Lefschetz numbers for the de Rham complex on a compact stratified pseudomanifold are well defined in the stratified homotopy class. 

\begin{proposition}
\label{Homotopy_operator_2}
Let $(f_t)_{0 \leq t \leq 1} : \mathring{X} \rightarrow \mathring{X}$ be a homotopy of proper maps.
Then $HS(f_0)$ and $HS(f_1)$ induce the same endomorphism on the $L^2$ cohomology of the de Rham complex with the minimal domain.
\end{proposition}

\begin{proof}
We introduce a homotopy operator following Bott-Tu for forms $u(x)$
\begin{equation}
    K(u)(x)= \int_0^1 \iota_{\frac{\partial}{\partial_t}} \omega(x,t)dt
\end{equation}
where 
\begin{equation}
\label{equation_to_replace_in_next_proof}
    \omega(x,t)=(\pi_X)_*(u(exp(f_t(x), y)) \wedge \mathcal{T}).
\end{equation}
By a calculation similar to that done in the proof of proposition 4.1 of Chapter 1 of Bott-Tu, one sees that
\begin{equation}
\label{Mahinda_1}
    \omega(x,1)-\omega(x,0)= (dK(u)(x)-Kd(u)(x)) 
\end{equation}
where $d$ is the de Rham differential on $\mathring{X}\times [0,1]_t$, $\omega(x,1)=HS(f_1)u(x)$ and $\omega(x,0)=HS(f_0)(u)(x)$. The right hand side of \eqref{Mahinda_1} vanishes in cohomology since it maps closed forms to exact forms, proving the proposition.
\end{proof}

Moreover, one can use a hybrid operator by only modifying the endomorphism where the rank of Jacobian of the self map drops.
\begin{proposition}
\label{Homotopy_operator_3}
Let $f : \mathring{X} \rightarrow \mathring{X}$ be a proper map. Let $(\chi_S)$ be a smooth cutoff of an open neighbourhood of the set where the rank of $Df$ drops. Then the operator $\widetilde{HS}(f)$ defined by
\begin{equation}
    \widetilde{HS}(f)u(x)=(\pi_X)_*(u(exp(f(x), (\chi_S) y)) \wedge \mathcal{T})
\end{equation}
induces the same map in cohomology as $HS(f)$.
\end{proposition}

\begin{proof}
We can repeat the proof of the earlier proposition, replacing $\omega(x,t)$ in \eqref{equation_to_replace_in_next_proof} with
\begin{equation*}
    \omega(x,t)= (\pi_X)_*(u(exp(f(x), ((\chi_S)t+(1-t)) y)) \wedge \mathcal{T}).
\end{equation*}
\end{proof}

\begin{remark}    
Remark \ref{Remark_only_need_morphism_in_neighbourhoods} shows that one only needs the self-map $f$ to be of model form in small enough fundamental neighbourhoods of the fixed points for the localization theorem to hold. Thus, we can apply the result whenever we have a proper map that is a local diffeomorphism near fixed points, where away from a neighbourhood of the fixed point we can use the Hilsum-Skandalis replacement using Proposition \ref{Homotopy_operator_3}.
\end{remark}

\subsection{Analytic derivation of the de Rham local Lefschetz number}
\label{subsection_analytic_derivation_intersection}

In this section we derive the results of Goresky and MacPherson using our analytic definitions of local cohomology.

\begin{theorem}
\label{Shylock}
Let $\widehat{X}$ be a Witt space of dimension $n$. Let $f$ be a placid self map on $\widehat{X}$ with an isolated simple non-expanding fixed point at a singularity $p$, with a fundamental neighbourhood $U_p=\mathbb{D}^j_{y} \times C_{x}(Z_{z})$ of $p$ with no other fixed points. Let $T_f=\phi \circ HS(f)$ be the endomorphism on $\mathcal{P}_N(U_p)=(L^2\Omega(U_p),d)$.
Let the induced map at the tangent cone at $p$ be given by 
\begin{equation}
\label{slam_dunk}
    \mathfrak{T}_pf(y,x,z)=( \widetilde{C}(y), x \widetilde{A}(z),  \widetilde{B}(z))
\end{equation}
as in equation \eqref{equation_map_induced_on_tangent_cone}. Then the local Lefschetz number at $p$ is given by
\begin{equation}
L(\mathcal{P}_N(U_p), T_f, p) = \sum_{k=0}^{\lfloor \frac{l-1}{2} \rfloor} (-1)^k Tr(\tilde{B}^*: \mathcal{H}^k(Z) \rightarrow \mathcal{H}^k(Z))
\end{equation}
where $Z$ is the link of $p$ in $X$, $l$ its dimension, and $\mathcal{H}^{\cdot}(Z)$ denotes its $L^2$ de Rham cohomology groups.

\end{theorem}

\begin{proof}
We first compute the cohomology explicitly for a fundamental neighbourhood with a product type metric.

\begin{lemma}
\label{conic_cohomology_1}
Let $C(Z_z)$ be the truncated cone with metric $g=dx^2+x^2g_Z$, on which we have the de Rham complex, imposing absolute boundary conditions at the link at the boundary $Z=\{x=1\}$. The cohomology of this elliptic complex is
\begin{equation*}
\mathcal{H}^k(\mathcal{P}_N(C(Z)))=
\begin{cases}
      0, & \text{for}\ k \geq [\frac{l+1}{2}] \\
      \mathcal{H}^k(Z), & \text{for}\ k \leq [\frac{l-1}{2}]
\end{cases}
\end{equation*}
where $\mathcal{H}^k(Z)$ is the $L^2$ de Rham cohomology of the link $Z$.
\end{lemma}
 
\begin{proof}

Observe that $r=1-x$ is a boundary defining function for the boundary on the truncated cone. The zeroth and first order conditions (\eqref{equation_generalized_Neumann_zeroth_first_order_conditions}) corresponding to the generalized Neumann boundary conditions given in \eqref{equation_generalized_Neumann_Robin_type} for the Laplace type operator of the complex $\mathcal{P}_N(C(Z))$ are
\begin{align}
\label{Neumann_bc_1}
    \sigma(\delta, dx)u|_{x=1} =0 \\ 
    \label{Neumann_bc_2}
    (\sigma(\delta, dx) du)|_{x=1} =0
\end{align}
where $-\sigma(\delta, dx)=\sigma(\delta, dr)$ is simply the interior product with $\partial_r$.

On page 588 of \cite{cheeger1983spectral} Cheeger shows that the harmonic forms on the infinite cone are of four types, similar to the forms of types $1,2,3,4$ that we used in Step 1.1 of Proposition \ref{Proposition_spectral_properties} (see Remark \ref{Remark_growth_of_harmonic_sections}). These are linearly independent, and it is easy to see that only the first of those types of forms do not have $dx$ factors in them, and thus only these satisfy the $0$-th order condition \eqref{Neumann_bc_1}. These have a basis of forms $x^{a^{\pm}_j(k)}\phi^k_j$, where $\phi^k_j$ are co-closed eigenforms of $\Delta_Z$, and $a^{\pm}_j(k)$ are real numbers made explicit in \cite[\S 3]{cheeger1983spectral}. We observe that only the harmonic forms for which $a^{\pm}_j(k)=0$ satisfy the first order condition \eqref{Neumann_bc_2}. But these are precisely the harmonic forms of $\Delta_Z$.

Thus the null space is given by forms in $\pi_Z^*u$ for $u \in \mathcal{H}^*(Z)$, where $\pi_Z:C(Z) \rightarrow Z=\{x=1\}$. It is easy to check that these forms satisfy the boundary conditions. Since not all $\pi_Z^*u$ will be bounded in $L^2$ with respect to the conic measure we need to restrict to those that are.

It is clear that what we need to prove is that the elements of the form $\pi_Z^*u$ for $u \in \mathcal{H}^k(Z)$ that are bounded in $L^2$ with respect to the conic measure are those of degree $k \leq [\frac{l-1}{2}]$. 
Since there is no $x$ dependence in sections of the form $\pi_Z^*u$, the $L^2$ norm of a section $\pi_Z^*u$ is given by $\int_{C(Z)} x^{-2k+l}||u(z)||^2dvol_{g_z}dx$. 
This can be seen for instance by considering the Hodge star operator for wedge forms $u$ of degree $k$, which takes local wedge forms $xdz_1 \wedge xdz_2...\wedge xdz_k$ to $xdz_{k+1} \wedge xdz_{k+2}...\wedge xdz_l \wedge dx$, where $\{xdz_1, xdz_2,..., xdz_l, dx\}$ is an oriented local basis for the wedge cotangent bundle on the cone. 
This integral is finite if and only if $-2k+l>-1$, or equivalently $k < \frac{l+1}{2}$. Since we are in the Witt case, this is equivalent to $k \leq [\frac{l-1}{2}]$. Since $\mathcal{D}_{\max}(d)=\mathcal{D}_{VAPS}(d)$, we are done.

\end{proof}

\begin{remark}[Generalizations]
\label{Remark_generalizations_perversities_L2boundedness}
This was first worked out by Cheeger \cite{cheeger1980hodge} in a more general setting which includes the case of horn metrics of the form $dx^2+x^{2c}dg_Z$. The condition $-2k+l>-1$ that arises in the proof of the above lemma is replaced by $c(-2k+l)>-1$, equivalently $k< 1/2(l+1/c)$ (see for instance \cite[\S 3]{Bei_2011_Perversities}).
One can also work out the $L^p$ cohomology of these spaces (see \cite{bei2020p}) by similar local computations and relate it to certain perversity functions and the related intersection homology theory.
In \cite[\S 15]{Goresky_1985_Lefschetz}, Goresky and MacPherson indicate how their Theorem can be extended to other perversities.
One could similarly study the Hodge cohomology of other warped product metrics (which appear in many spaces such as real algebraic varieties, symmetric spaces) using these boundary conditions (see for instance Theorem 4.30 of \cite{cruz2020examples}, \cite[\S 7]{Albin_hodge_theory_on_stratified_spaces},  \cite{GoreskyL2IntersectionexpositZucker92},\cite{Jesus2018Wittens,Jesus2018Wittensgeneral}).
In the non-Witt case, different choices of mezzo-perversities and their associated cohomology theories have been studied in \cite{Albin_hodge_theory_cheeger_spaces,Albin_hodge_theory_on_stratified_spaces}.

\end{remark}

In particular, Lemma 6.16 shows that the de Rham complex on this fundamental neighbourhood is a Fredholm complex.
The proof of the theorem now follows from Proposition \ref{Local_Lefschetz_numbers_definition} where the Lefschetz boundary contributions vanish since we use local boundary conditions (see Remark \ref{Remark_on_local_boundary_condition_boundary_vanishing}), and by applying Corollary \ref{Theorem_local_Lefschetz_tangent_cone} to the tangent cone, with factors given by the cone over the link $Z$ and the disk  $\mathbb{D}^j$. 
It is well known that these boundary conditions give an elliptic complex in the sense of Boutet-de-Monvel (see \cite{brenner1990atiyah}), and we can even use the usual McKean-Singer theorem instead of the renormalized version.
\end{proof}

The Laplace type operator of the complex $\mathcal{P}^*_N(C(Z))$ has domain as given in \eqref{equation_generalized_Neumann_Robin_type} for $P=\delta$ with the zeroth and first order conditions \eqref{equation_generalized_Neumann_zeroth_first_order_conditions} given by
\begin{align}
\label{Dirichlet_bc_1}
    \sigma(d)(dr)u|_{x=1} =0\\
    \label{Dirichlet_bc_2}
    \sigma(d)(dr) \delta u|_{x=1} =0.
\end{align}
The Hodge star operator intertwines these boundary conditions with the boundary conditions given in equations \eqref{Neumann_bc_1} and \eqref{Neumann_bc_2}. A good reference for this is Section 9 of Chapter 5 of \cite{taylor1996partial}.

\begin{corollary}
\label{Shylock_2}
In the same setting as Theorem \ref{Shylock}, but for non-attracting fixed points, let $\mathcal{P}^*_N(U_p)$ be the dual complex of the de Rham complex with generalized Neumann boundary conditions. We assume that $f$ is locally invertible. We have
\begin{equation}
\label{never_again_17856}
\begin{split}
L(\mathcal{P}_N(U_p) , T_f, p)= (-1)^{n} L(\mathcal{P}^*_N(U_p) , T^{\delta}_{f^{-1}}, p) &= (-1)^{n} \sum_{k=0}^{n} (-1)^{k} Tr(T^{\delta}_{f^{-1}}|_{\mathcal{H}^k(\mathcal{P}^*_N(U_p))})\\
&= (-1)^{n} \varepsilon \sum_{k=0}^{l+1} (-1)^{k} Tr(T^{\delta}_{f^{-1}}|_{\mathcal{H}^k(\mathcal{P}^*_N(C(Z)))})\\
\end{split}
\end{equation}
where $\varepsilon$ is $1$ if $\widetilde{C}$ given in equation \eqref{slam_dunk} is orientation preserving and $-1$ if not. 
\end{corollary}

This formula is precisely the analog of equation \eqref{intersection_relative_number} in intersection homology. 

\begin{proof}
This follows from the proof of Theorem \ref{Shylock} above and the duality given in Proposition \ref{proposition_local_product_Lefschetz_heat}. The factor of $(-1)^{n}$ comes from setting $b=-1$ in the factor of $b^{n}$ given by Proposition \ref{proposition_Lefschetz_on_adjoint}, and this proves the first equality. The second equality follows from the definition of the Lefschetz numbers.

We can make the duality in the latter proposition explicit via the Hodge star operator. From the setting of Theorem \ref{Shylock} we see that $n=j+l+1$ is the dimension of $X$, where $l$ is the dimension of $Z$, and $j$ the dimension of $\mathbb{D}^j$.
Taking the Hodge star of a form $\omega$ of degree $k$ in the cohomology of $C(Z)$ computed in Lemma \ref{conic_cohomology_1}, we get the form $dx \wedge x^{-2k+l} \star_Z \omega$, which is a form of degree $l+1-k$. The volume form of $\mathbb{D}^j$ is a representative of the cohomology class of the disc for the dual complex. This shows that if $\widetilde{C}$ is orientation preserving, then $\varepsilon=1$, and if not $\varepsilon=-1$. Now we can use the K\"unneth formula of Proposition \ref{Lefschetz_product_formula} to write the Lefschetz numbers of $U_p$ as the product of that on the two factors, and since that of $\mathbb{D}^j$ is $\varepsilon$, we have the result.

\end{proof}

In \cite{brenner1990atiyah}, the Lefschetz fixed point theorem for elliptic complexes on smooth manifolds with boundary was developed, and the de Rham case was studied in detail. It is observed in that article that the sheaf of germs of forms smooth up to the boundary form a fine resolution of the constant sheaf on the manifold with boundary, and the Poincar\'e lemma shows that the higher cohomology vanishes. With the presence of conic singularities, the higher $L^2$ de Rham cohomology groups in a fundamental neighbourhood do not vanish in general. 

For strictly non-expanding or strictly non-attracting fixed points, we can use the de Rham complex or its adjoint complex and compute the Lefschetz numbers in cohomology. If there is a product decomposition (see Definition \ref{definition_product_decompositions_formalized}), we can use Proposition \ref{proposition_local_product_Lefschetz_heat} which shows that the local formula should be considered as the product of two formulas in cohomology.

\subsubsection{Recovering the Atiyah-Bott local formula}
\label{subsubsection_de_rham_smooth}
We shall now briefly explain why the topological formulas in Theorem \ref{Shylock} and Corollary \ref{Shylock_2} as well as the product formula given in Proposition \ref{proposition_local_product_Lefschetz_heat} can be used to recover the linear algebraic formula in Theorem \ref{Nonsingular_Contribution} for smooth fixed points $p$, which for the de Rham complex reduces to
\begin{equation}
\label{Nimal_Perera}
    Tr_p(f)= sign(det(Id-df_p)).
\end{equation}
First choose coordinates so that the matrix $df_p$ is in Jordan form on the tangent space. Let $\mathbb{R}^{k_1}$ be the sum of the generalized eigenspaces of eigenvalues with absolute value less than or equal to $1$, and let $\mathbb{R}^{k_2}$ be the complement.

We can use Proposition \ref{proposition_local_product_Lefschetz_heat} to see that the local Lefschetz number at $p$ can be computed using the local Lefschetz numbers for the induced maps on the two factors in $\mathbb{D}^{k_1}\times \mathbb{D}^{k_2}$ (see Definition \ref{definition_product_decompositions_formalized}), where $\mathbb{D}^{k_i}$ is the unit disc in $\mathbb{R}^{k_i}$ for $i=1,2$. The Lefschetz number for the non-expanding map on the factor $\mathbb{D}^{k_1}$ is obtained by applying Theorem \ref{Shylock} to the unit disc in this factor. Since the cohomology is generated by the constants in degree $0$, and the pullback of any constant is the same constant, the map induced in cohomology is the identity and we see that the Lefschetz number is $1$. Similarly for the factor $\mathbb{R}^{k_2}$, we use Corollary \ref{Shylock_2} and see that the Lefschetz number is $(-1)^{k_2}$, and the Lefschetz number at $p$ is the product, $(-1)^{k_2}$. The integer $k_2$ is the number of eigenvalues which have absolute value greater than $1$, and Proposition \ref{proposition_local_product_Lefschetz_heat} shows that the local Lefschetz number is equal to $sign(det(Id-df_p))$.

\subsection{Witten instanton complex and Lefschetz-Morse inequalities}
\label{subsection_witten_deformation_morse_de_Rham}

In this section we construct the Witten instanton complex or the small eigenvalue complex and prove the polynomial Morse inequalities, and Lefschetz-Morse inequalities for Witt spaces. Let us first consider the polynomial de Rham Morse inequalities in the smooth setting.

A smooth function $h$ on a compact smooth manifold $X$ of dimension $n$ with isolated non-degenerate critical points is called a Morse function, where non-degenerate means that the Hessians at the critical points are invertible.
One consequence of the non-degeneracy is that the number of critical points is finite. The Morse index $m_a$ at a critical point is the number of negative eigenvalues of the Hessian at $a$, and the $k$-th Morse number $M_k$ is the number of critical points of Morse index $k$ for the function $h$. If the $k$-th Betti numbers of the manifold are denoted by $\beta_k$, then the polynomial Morse inequalities can be expressed by saying that the coefficients $Q_k$ in the expression,
\begin{equation}
    \label{equation_polynomial_morse_inequalties}
    \sum_{k=0}^n M_k b^k -\sum_{k=0}^n \beta_k b^k = (1+b) \sum_{k=0}^{n-1} Q_k b^k
\end{equation}
are non-negative integers.

It is well known that the left hand side of \eqref{equation_polynomial_morse_inequalties} vanishes  when $b=-1$ by the Poincar\'e-Hopf theorem, which is a corollary of the Lefschetz fixed point theorem for the self map given by flowing along the gradient flow corresponding to the dual vector field of $dh$ with respect to a Riemannian metric for short time. This is clear by the equality
\begin{equation}
    \label{Chamitha_598}
    \sum_{k=0}^n M_k b^k = \sum_{a \in Crit(h)} b^{m_a}
\end{equation}
where the left hand side is known as the \textbf{\textit{Morse polynomial}}. The right hand side for $b=-1$ is the Euler characteristic of the manifold by the Poincar\'e Hopf theorem, which is equal to the evaluation at $b=-1$ of the \textbf{\textit{Poincar\'e polynomial}}, $\sum_{k=0}^n \beta_k b^k$.

Ludwig uses a small eigenvalue complex (see Section 2.9 of \cite{ludwig2020extension}) to prove the polynomial Morse inequalities for the $L^2$ de Rham complex on a stratified pseudomanifold with conic singularities in \cite{ludwig2017index} for what are called \textit{radial and anti-radial Morse functions} in \cite{ludwig2020extension}.
The following is a generalization of those Morse functions for stratified spaces, also called relatively Morse functions in \cite{Jesus2018Wittens,Jesus2018Wittensgeneral}.

\begin{definition}[stratified Morse function]
\label{definition_stratified_Morse_function}
Let $\widehat{X}$ be a stratified pseudomanifold with a wedge metric $g_w$ and a continuous function $h$ which lifts to a map $h' \in C_{\Phi}^{\infty}(X)$ (see \eqref{equation_smooth_functions_on_stratified_spaces}) on the resolved manifold with corners, i.e., $h'=h \circ \beta$, which is a Morse function when restricted to $\widehat{X}^{reg}$. We demand that the image of the set $|dh'|_{g_w}^{-1}(0)$ under the blow-down map $\beta$ consists of isolated points on the stratified pseudomanifold.

We call such points the \textbf{\textit{critical points of h}}. Moreover, at critical points $a$, we ask that there exist fundamental neighbourhoods $\widehat{U_a}=\widehat{U_1}\times \widehat{U_2}$ where the metric respects the product decomposition, with radial coordinates $x_a$ on $\widehat{U_1}$ and $r_a$ on $\widehat{U_2}$ such that restricted to $\widehat{U_1}$ (respectively $\widehat{U_2}$), $x_a$ (respectively $r_a$) is the geodesic distance to $a$, and such that the function $h$ restricted to $\widehat{U_a}$ can be written as $x_a^2-r_a^2$. 
We shall notate these radial coordinates as $x, r$ for brevity, where it is understood that there are such coordinates corresponding to each critical point $a$. 
Then we say that $h$ is a \textit{\textbf{stratified Morse function}}.
We will denote $dh'$ by $dh$ by abuse of notation.
\end{definition}

It is well known that such coordinates exist around smooth critical points on $\widehat{X}^{reg}$ by the classical Morse Lemma, possibly up to deforming the metric within the quasi-isometry class. In the case of an isolated singularity, there are many examples of such stratified Morse functions. 
Consider for instance a Morse-Bott function $h' \in C_{\Phi}^{\infty}(X)$ on the resolved manifold with boundary corresponding to a pseudomanifold with a single isolated singularity, where the boundary hypersurface corresponding to the resolved singularity is a Morse-Bott singularity. This maps to the singular point under the blow-down map $\beta$, and the composition of $h'$ with $\beta$ is a stratified Morse function. Another source of examples are Hamiltonian Morse functions arising from moment map reductions on smooth spaces. We study such examples arising from K\"ahler reduction in Subsection \ref{subsection_Computations_for_singular_examples}, and \cite{jayasinghe2024holomorphicwitteninstantoncomplexes}.

Consider a fundamental neighbourhood $U_a$ of a critical point $a$ of a stratified Morse function $h$ on a stratified Witt space $\widehat{X}$, with a decomposition $U_a=U_1\times U_2$ as in Definition \ref{definition_stratified_Morse_function}. The gradient flow of $h=x^2-r^2$ for a short time yields a 
self map $f$ on $X$ which is attracting on $U_1$ and expanding on $U_2$.
If the gradient flow is stratum preserving then in order to compute the Lefschetz numbers for the geometric endomorphism induced by $f$ on the de Rham complex we can use the complexes $\mathcal{P}_{N}(U_1)$ and $\mathcal{P}^*_{N}(U_2)$, and compute the local Lefschetz numbers for $\mathcal{P}_{B}(U_a)$ (see Proposition \ref{proposition_local_product_Lefschetz_heat}). This also happens to give the correct decomposition for the Morse inequalities.

Given a product decomposition of a fundamental neighbourhood $U_a$ at a critical point as in Definition \ref{definition_stratified_Morse_function}, we define the \textit{\textbf{product complex at the critical point}} as $\mathcal{P}_B(U_a)=\mathcal{P}_N(U_1) \times \mathcal{P}^*_N(U_2)$, generalizing Definition \ref{definition_product_decompositions_formalized} which is for fixed points of self-maps.

\begin{definition}
\label{definition_product_decomposition_critical_point}
In the setting of the discussion we define the \textit{\textbf{local Morse polynomial}} at a critical point $a$ to be
\begin{equation}
    M(X,h,a)(b):=\sum_{k=0}^n b^k \dim(\mathcal{H}^{k}(\mathcal{P}_B(U_a))),
\end{equation}
with the product complex as in the discussion above, and the \textit{\textbf{global Morse polynomial}} to be
\begin{equation}
    M(X,h)(b)=\sum_{a \in Crit(h)} M(X,h,a)(b),
\end{equation}
the sum of the local Morse polynomials at the critical points. The \textit{\textbf{Poincar\'e polynomial}} is defined to be
\begin{equation}
    N(X)(b)=\sum_{k=0}^n \beta_k(X) b^k
\end{equation}
similar to the smooth case, where $\beta_k(X)=\dim(\mathcal{H}^{k}(\mathcal{P}(X)))$ are the \textit{\textbf{Betti numbers}} of the complex.
\end{definition}

\begin{remark}[Assumptions]
\label{remark_general_domain_choices_Morse}
We can always perturb a metric within its quasi-isometry class in a neighbourhood of the critical points of $h$ and arrange that it is isometric to a neighbourhood in the union of the tangent cones with a model metric such that the Laplace-type operators on the neighbourhoods are the same as the model operators obtained by freezing coefficients at the critical point. Moreover as discussed in the beginning of this section, we can always scale the links on Witt spaces in order to ensure that the geometric Witt condition is satisfied and we will continue to assume that the Hodge-de Rham operator $D=d+d^*$ is essentially self-adjoint for the rest of this article.
\end{remark}

\subsubsection{Witten deformed complexes}
\label{subsubsection_Witten_deformed_elliptic}

Given a stratified pseudomanifold $\widehat{X}$ equipped with a wedge metric 
and a stratified Morse function $h$, we denote the de Rham complex by $\mathcal{P}(X)=(H_k=L^2\Omega^k(X),P_k=d_k)$. We define the Witten deformed differential $P_\varepsilon:=e^{-\varepsilon h}Pe^{\varepsilon h}$, where $\varepsilon \geq 0 $ is a parameter. Since $h$ is smooth (i.e., lifts to $C_{\Phi}^{\infty}(X)$), the map $P_\varepsilon-P=\varepsilon (dh) \wedge$ extends to a bounded map $H_k \rightarrow H_{k+1}$ for any $k$, as does $\varepsilon \iota_{dh^{\#}}$.
We follow \cite[\S 4.3]{Zhanglectures} and introduce the Clifford operators 
\begin{equation}
    cl(u)=u \wedge - \iota_{u^\#}, \quad \widehat{c}(u)=u \wedge + \iota_{u^\#}
\end{equation}
where $cl(u)$ is the Clifford multiplication for the de Rham complex. Zhang's notation is for the Clifford algebra on the tangent space, while we use that on the wedge cotangent bundle by duality as in Definition \ref{wedge_Clifford_module}. Zhang shows that $D_\varepsilon= D+ \varepsilon \widehat{c}(dh)$ and we refer the reader to \cite{Zhanglectures} for more details.  
Let $F=\Lambda^*(\prescript{w}{}{T^*X})$.
Since $h' \in C_{\Phi}^{\infty}(X)$, given $u \in L^2(X;F)$ we have that $\widehat{c}(dh)u \in L^2(X;F)$. Moreover given a sequence $(u_n) \subset C^{\infty}_c(\mathring{X},F)$, we have that ($\widehat{c}(dh)u_n) \subset C^{\infty}_c(\mathring{X},F)$. If $D(u_n)$ is $u_n$ $L^2$-Cauchy, then so is $D_\varepsilon(u_n)$.
It follows that $\mathcal{D}_{\min}(D)=\mathcal{D}_{\min}(D_\varepsilon)$. 

Note also that multiplication by $e^{\varepsilon h}$ is a homeomorphism $H_k \rightarrow H_k$ that takes a section in the domain of any closed extension $\mathcal{D}(P)=\mathcal{D}(P_\varepsilon)$ to another section in the same space 
and conjugates $P$ and $P_\varepsilon$.
We define the \textit{\textbf{Witten deformed complex}} $\mathcal{P}_\varepsilon=(H_{k},P_{\varepsilon,k})$ by changing the differential in $\mathcal{P}(X)$ to $P_\varepsilon$ with the same domain as $P$. Multiplication by $e^{-\varepsilon h}$ induces an isomorphism between the cohomology of $\mathcal{P}_\varepsilon(X)$ and $\mathcal{P}(X)=\mathcal{P}_0(X)$. 
This can be summarized using the commutative diagram 
\[\begin{tikzcd}
	{...} && {H_k} && {H_{k+1}} && {...} \\
	\\
	{...} && {H_{k}} && {H_{k+1}} && {...}
	\arrow["P", from=1-1, to=1-3]
	\arrow["P", from=1-3, to=1-5]
	\arrow["P", from=1-5, to=1-7]
	\arrow["{P_\varepsilon}", from=3-3, to=3-5]
	\arrow["{P_\varepsilon}", from=3-1, to=3-3]
	\arrow["{P_\varepsilon}", from=3-5, to=3-7]
	\arrow["{e^{-\varepsilon h}}", from=1-3, to=3-3]
	\arrow["{e^{-\varepsilon h}}", from=1-5, to=3-5]
\end{tikzcd}\]
where the domains $\mathcal{D}(P_{k})$ and $\mathcal{D}(P_{\varepsilon,k})$ can be canonically identified as described above.

Similar to the complexes on $X$, when we restrict to a fundamental neighbourhood $U_a$ of a critical point $a$ of $h$, we can define the complex $\mathcal{P}_{\varepsilon,N}(U_a)$ by changing the differential in $\mathcal{P}_N(U_a)$ to $P_\varepsilon$ with the same domain as $P_{U_a}$.
There are isomorphisms from the Hilbert complexes $\mathcal{P}_{N}(U_a)$ to $\mathcal{P}_{\varepsilon,N}(U_a)$, and $(\mathcal{P}^*_{N}(U_a))^*=\mathcal{P}_{D}(U_a)$ to $(\mathcal{P}^*_{\varepsilon,N}(U_a))^*=\mathcal{P}_{\varepsilon,D}(U_a)$ given by multiplication by the function $e^{-\varepsilon h}$. Similarly, multiplication by
$e^{+\varepsilon h}$ gives an isomorphism from $\mathcal{P}^*_{N}(U_a)$ to $\mathcal{P}^*_{\varepsilon,N}(U_a)$, and  $(\mathcal{P}_{N}(U_a))^*$ to $(\mathcal{P}_{\varepsilon,N}(U_a))^*$.

Consider the Laplacian $\Delta_{\varepsilon}=P_{\varepsilon}P^*_{\varepsilon}+P^*_{\varepsilon}P_{\varepsilon}$. We show that we have a Bochner type formula
\begin{equation}    
\label{equation_Witten_deformed_Laplace_type_expansion}
\Delta_{\varepsilon}=\Delta+\varepsilon^2||dh||^2+\varepsilon K
\end{equation}
generalizing that of Proposition 4.6 of \cite{Zhanglectures}. 

\begin{proposition}[Bochner type formula]
\label{Proposition_stratified_Morse_Laplace_structure}
Let $\widehat{X}$ be a stratified pseudomanifold with a wedge metric and a stratified Morse function $h$, and let $\mathcal{P}(X)$ be the de Rham complex. For any $\varepsilon \in \mathbb{R}$, given $s \in \mathcal{D}(D^2)$, we have
\begin{equation}
\label{Zhangs_Bochner_formula_1}
    ||D_\varepsilon s||^2_{L^2(X;F)}=\langle [D^2+\varepsilon K  +\varepsilon^2 |dh|^2] s,s \rangle_{L^2(X;F)}.
\end{equation}
where $F=\Lambda^*(\prescript{w}{}{T^*X})$ and $K=(D \widehat{c}(dh) + \widehat{c}(dh) D)$ is a $0$-th order operator.
\end{proposition}

\begin{proof}
We begin by expanding the left hand side of the expression to get 
\begin{equation}
\begin{split}
||D_\varepsilon s||_{L^2(X;F)}^2 & = \langle D_\varepsilon s, D_\varepsilon s \rangle_{L^2(X;F)} \\
 & = \langle Ds, Ds \rangle_{L^2(X;F)} + 2\varepsilon \langle Ds, \widehat{c}(dh)s \rangle_{L^2(X;F)} + \varepsilon^2 \langle \widehat{c}(dh)s, \widehat{c}(dh)s \rangle_{L^2(X;F)}\\
& = \langle [D^2+\varepsilon^2 |dh|^2]s, s \rangle_{L^2(X;F)}  + \varepsilon [\langle Ds, \widehat{c}(dh)s \rangle + \langle \widehat{c}(dh)s, Ds \rangle_{L^2(X;F)} ]
\end{split}
\end{equation}
where we have used the self-adjointess of $D$.
We have that 
\begin{multline}
\label{equation_with_boundary_term_Dirac_deformed}
    \langle Ds, \widehat{c}(dh)s \rangle_{L^2(X;F)}+ \langle \widehat{c}(dh)s, Ds \rangle_{L^2(X;F)} \\
    =\langle s, [D \widehat{c}(dh) + \widehat{c}(dh) D]s \rangle_{L^2(X;F)}+ \int_{\partial X} \langle i cl(d\rho_X)s,\widehat{c}(dh)s \rangle_F dvol_{\partial X}
\end{multline}
where the boundary integral vanishes if $\widehat{c}(dh)s \in \mathcal{D}(D)$. But we have that $D_\varepsilon s- Ds=\varepsilon\widehat{c}(dh)s$, and that $\mathcal{D}(D_\varepsilon)=\mathcal{D}_{\min}(D_\varepsilon)=\mathcal{D}_{\min}(D)=\mathcal{D}(D)$.
Thus given $s \in \mathcal{D}(D^2)$, we have that $\widehat{c}(dh)s \in \mathcal{D}(D)$.

It is easy to verify that $K$ is a zeroth order operator and we refer to \cite{Zhanglectures} (c.f., \cite{witten1982supersymmetry}) for details.
\end{proof}

\begin{remark}
\label{remark_Shubin_vs_us_difference}
In the work of \cite{ProkhorenkovShubin2009Morseboundary}, the Morse inequalities for manifolds with boundary are worked out.
Their definition of \textit{Morse function on manifolds with boundary} corresponds to functions $h$ which are Morse in the interior, have no critical points on the boundary, and the restriction of $h$ to the boundary is also Morse. There the boundary term in equation \eqref{equation_with_boundary_term_Dirac_deformed} plays an important role, and there are contributions at the boundary from the critical points of the Morse function restricted to the boundary. The condition $\mathcal{D}(D_\varepsilon)=\mathcal{D}(D)$ does not hold for the choices of domain in that article, which is why the boundary term does not vanish. 
The gradient flow of such Morse functions would not preserve the manifold with boundary. In particular, the self maps associated to such gradient flows can only be defined in some extension of the manifold with boundary, and there would not be fixed points on the original boundary.

On the other hand there is a Lefschetz fixed point theorem on manifolds with boundary in \cite{brenner1981atiyah,brenner1990atiyah} where there are boundary fixed points. While this is different from the Lefschetz fixed point theorem obtained by setting $b=-1$ in the Morse inequalities of \cite{ProkhorenkovShubin2009Morseboundary}, there are similarities in the boundary contributions.

Our setting is similar to that of Goresky and MacPherson, matching their Morse inequalities and Lefschetz fixed point theorem in intersection homology for Witt spaces.
\end{remark}

Now we can prove the following estimate away from the critical points of a stratified Morse function.

\begin{proposition}
\label{Propostion_growth_estimate_witten_deformed}
In the same setting as Proposition \ref{Proposition_stratified_Morse_Laplace_structure}, there exist constants $C>0$, $\varepsilon_0>0$ such that for any section $s$ in the domain with $\text{supp}(s) \subset (X \setminus \cup_{a \in crit(h)} U_{a})$ and $\varepsilon \geq \varepsilon_0$, one has 
\begin{equation}
    ||D_{\varepsilon}s||_{L^2(X;F)} \geq C \sqrt{\varepsilon}||s||_{L^2(X;F)}.
\end{equation}
\end{proposition}

\begin{proof}
This is similar to Proposition 4.7 of \cite{Zhanglectures}.
Let $C_1$ be the minimum value of $||dh||$ on $X \setminus \cup_{a \in crit(h)} U_{a}$.
Using the Bochner type formula \eqref{equation_Witten_deformed_Laplace_type_expansion},  it is easy to see that there exists a finite constant $K$ such that
\begin{equation}
||D_{\varepsilon} s||_{L^2(X;F)}^2  = \langle D_{\varepsilon} s, D_{\varepsilon} s \rangle_{L^2(X;F)} \geq (\varepsilon^2 C_1^2-\varepsilon |K|)||s||_{L^2(X;F)}
\end{equation}
since $\langle D s, D s \rangle_{L^2(X;F)}$ is positive.
\end{proof}
From the Laplace-type operator in a neighbourhood with a product type metric $U_a$ of a critical point $a$, we get a model operator on the tangent cone (the model harmonic oscillator in the smooth case) for which the null space can be computed explicitly. The elements in the null space of the model operator on the tangent cone restrict to elements on the truncated tangent cone satisfying the boundary conditions corresponding to the choices of complexes on $U_1$ and $U_2$ described above. For the de Rham complex on the infinite cone with the conic metric (the tangent cone at isolated conic singularities), the null space of the model operator has been computed explicitly in \cite{ludwig2017index}. 

The articles mentioned above, and indeed much of the existing literature studies model operators on infinite cones (including the tangent space in the smooth setting) and uses scaling arguments for the Witten deformed Laplace type operators to show that the eigenvalues scale. This can be done in our setting as well on the non-compact tangent cone for a wedge metric following the model computations are in Section 4 of \cite{ludwig2017index}. 
Our next proposition explains this. We begin with a definition.

\begin{definition}
\label{definition_tangent_cone_truncations}
Given the infinite tangent cone $\mathfrak{T}_aX$ of a critical point $a$ of a stratified Morse function $h$ on $\widehat{X}$, we can find a product decomposition $\mathfrak{T}_aX=V_1 \times V_2$, and we can extend the Morse function from the truncated tangent cone $U_a$ to this space as $x_a^2-r_a^2$ where $x_a$ is the geodesic distance on $V_1$ and $r_a$ that on $V_2$, from the critical point corresponding to the induced metric on the tangent cone. 
For any $z >0$, we define $U_{1,z}:=\{x_a \leq z\}$, $U_{2,z}:=\{r_a \leq z\}$, and $U_{a,z}:=U_{1,z} \times U_{2,z}$.
In particular $U_a=U_{a,1}$. Given a complex $\mathcal{P}(d)$, we define the complex $\mathcal{P}_{\varepsilon,B}(U_{a,z})$ similarly to how we defined it for the case where $z=1$.
We denote the non-compact tangent cone by $U_{a,\infty}$.
\end{definition}

We observe that while the local complexes $\mathcal{P}_{B,\varepsilon}(U_{a,z})$ are only defined for $z \in [0,\infty)$ for $\varepsilon \geq 0$, when we restrict to $\varepsilon>0$ they are defined for $z=\infty$ as well, and the cohomology groups of all these complexes are isomorphic as discussed above. In the case of $z=\infty$, the key is that factors of $e^{-\varepsilon (x_a^2+r_a^2)}$ appear in the local cohomology, giving sufficient decay at infinity for the harmonic forms to be $L^2$ bounded on the infinite tangent cone.

\begin{proposition}[Model equations and spectral gap]
\label{Proposition_model_spectral_gap_modified_general}
In the setting of this subsection, consider the complex $\mathcal{P}_{\varepsilon, B}(U_{a,z/\sqrt{\varepsilon}})$ for some fixed $z \in [0,\infty]$. Let $\Delta_{a,\varepsilon}=D_\varepsilon^2$ be the Laplace-type operator on $U_a$ for $\varepsilon>0$. 
If the positive eigenvalues of the Laplace type operator can be written as $\{\lambda^2_i\}_{i \in \mathbb{N}}$ for $\Delta_{a,1}$, then the positive eigenvalues for $\Delta_{a,\varepsilon}$ are $\{\varepsilon \lambda^2_i\}_{i \in \mathbb{N}}$. 
\end{proposition}

In the case of $z=\infty$, we note that $\mathcal{P}_{\varepsilon, B}(U_{a,z/\sqrt{\varepsilon}})=\mathcal{P}_{\varepsilon, B}(U_{a,z})$.

\begin{proof}
Let us first consider the case of $\mathcal{P}_{\varepsilon,N}(U_a)$ where $U_a=C_x(M)$ and $h=x^2$.
For the case where $z=\infty$, it is easy to verify that the cohomology is finite dimensional and the harmonic forms are the forms $\omega_a e^{-\varepsilon(x^2)}$ ($x=x_a$) where $\omega_a \in \mathcal{H}^*(\mathcal{P}_{0,N}(U_a))$, which are $L^2$ bounded for $\varepsilon>0$ since there is enough decay at infinity (see, e.g., \cite[\S 4]{ludwig2017index}).
The complex is Fredholm and there is a spectral gap. Indeed, we can follow the methods in Proposition \ref{Proposition_spectral_properties} to construct an orthonormal basis of eigensections for the Laplace type operator. This corresponds to the singular Sturm-Liouville problem where instead of imposing boundary conditions on an interval as in Step 2 of Proposition \ref{Proposition_spectral_properties}, following  \cite[\S 2.4.3]{al2008sturm} we use the decay at $\infty$ on the infinite cone for the corresponding singular Sturm-Liouville problem on the interval $(0,\infty)$.

The scaling given by $S(x)=\sqrt{\varepsilon}x$ induces an isometry between 
$U_{a,z/\sqrt{\varepsilon}}$ with the metric $dx^2+x^2g_M$ and $U_{a,z}$ with the metric $\varepsilon(dx^2+x^2g_M)$. Thus $S^*\Delta=\varepsilon \Delta$, the eigenforms of $S^*\Delta$ with domain $S^*\mathcal{D}(\Delta)$ are $S^*\psi$ where $\psi$ are the eigenforms of $\Delta$ with domain $\mathcal{D}(\Delta)$, and the eigenvalues scale by $\varepsilon$.

Restricted to the neighbourhood $U_a$, we have that the Witten deformed Laplace-type operator in \eqref{equation_Witten_deformed_Laplace_type_expansion} reduces to $\Delta_{\varepsilon}=\Delta+\varepsilon^2 x^2 +\varepsilon K$, and it is easy to observe that 
\begin{equation}
    \frac{1}{x} [S^* u] = \sqrt{\varepsilon}S^*[\frac{1}{x} u]
\end{equation}
for forms $u$. Thus, the eigenvalues of $\Delta_{\varepsilon}$ for eigensections on the domain of $\mathcal{P}_{\varepsilon,N}(U_{a,z/\sqrt{\varepsilon}})$ scale by $\varepsilon$ under the scaling $S(x)$. This proves the Proposition for $\mathcal{P}_{\varepsilon,N}(C_x(M))$.

It is easy to see that the case for $(\mathcal{P}_{N,\varepsilon}(C_r(M)))^*$ with $h=-r^2$ is similar.
Since the complex on $U_a \cong \mathbb{D}^j \times C(Z)$ is a product complex, separation of variables shows that the eigenforms of the Witten deformed Laplace type operator are power series where the terms are products of eigenforms on the smooth and the conic factors on the model tangent cone. Furthermore, since the Witten deformed complex decomposes as a product (see Definition \ref{definition_product_decompositions_formalized}) for the factors $U_{1,z/\sqrt{\varepsilon}}, U_{2,z/\sqrt{\varepsilon}}$ described in Definition \ref{definition_stratified_Morse_function}. 
This proves the proposition.
\end{proof}

\begin{remark}
We note that the model Witten deformed Laplace type operators can we written as 
\begin{equation}    
\Delta_{\varepsilon}=J_{\varepsilon}+ \varepsilon K, \quad J_{\varepsilon}=D^2+\varepsilon^2||dh||^2
\end{equation}
where $J$ and $K$ commute with each other. This is the analog of equation (17) of \cite{witten1982supersymmetry}, and using the creation and annihilation operator formalism (also known as the ladder operator formalism) used in that article one can give an alternative proof of the growth of eigenvalues.
\end{remark}

Given a geometric endomorphism $T_f$ of $\mathcal{P}$, we define \textit{\textbf{Witten deformed geometric endomorphism}} of $\mathcal{P}_{\varepsilon}$ on $X$ to be $T_{f,\varepsilon}=e^{-\varepsilon h} T_f e^{\varepsilon h}$, and the corresponding adjoint endomorphism $T^*_{f,\varepsilon}=e^{\varepsilon h} T_f^* e^{-\varepsilon h}$, where $T_f=T_{f,0}$ is the geometric endomorphism on the undeformed complex $\mathcal{P}$.
In Example \ref{remark_notation_bundles_vs_density_bundles_2} we saw that $T_f^* \omega = \star^{-1} (f^{-1})^* \star \omega$ for the de Rham complex, for invertible self maps $f$. 
The adjoint endomorphism for the Witten deformed complex can be expressed using the Hodge star as
\begin{equation}
\label{Witten_deformed_geometric_endomorphism}
T_{f,\varepsilon}^*(v)=e^{\varepsilon h}\star^{-1} (f^{-1})^* (\star (v_\lambda e^{-\varepsilon h}))
\end{equation}
which is easy to verify via the computation
\begin{equation*}
\langle T_{f,\varepsilon} v, u \rangle = \int_X e^{-\varepsilon h}f^*(v e^{\varepsilon h}) \wedge \star u = \int_X (v e^{\varepsilon h}) \wedge (f^{-1})^* (\star (u e^{-\varepsilon h}))
\end{equation*}
for forms $u$ and $v$.

\subsubsection{The de Rham Witten instanton complex}

This Subsection \ref{subsection_witten_deformation_morse_de_Rham} is devoted to the proof of Theorem \ref{theorem_small_eig_complex}. The method of Witten deformation has now become a standard proof technique for these types of results and Chapters 4 and 5 of \cite{Zhanglectures} give an excellent exposition  which we follow in closely in this subsection (c.f. Remark 4 of \cite[\S 5.4]{Zhanglectures}). We refer the reader to \cite{jayasinghe2024holomorphicwitteninstantoncomplexes} for a version of this proof for twisted Dolbeault complexes with various choices of domains.

Given $z \in (0, \infty]$, choose $\gamma_{a,z}: \mathbb{R}_s \rightarrow[0,z]$ to be a function in $C^{\infty}(\mathbb{R})$ such that $\gamma_a(s)=1$ when $s<z/2$, and $\gamma_{a,z}(s)=0$ when $s>3z/4$. At each critical point $a$, we define $t=x^2+r^2$ where the functions $x,r$ are the functions in Definition \ref{definition_stratified_Morse_function} corresponding to each critical point $a$.
Given a harmonic form $\omega_{a} \in \mathcal{H}(\mathcal{P}_{B}(U_{a}))$, we define $\omega_{a,\varepsilon}:=\omega_a e^{-t \varepsilon} \in \mathcal{H}(\mathcal{P}_{\varepsilon,B}(U_{a,z}))$ and 
\begin{equation}
\label{equation_modify_forms_cutoff}
    \alpha_{a, z, \varepsilon} :=\left\|\gamma_a(t) \omega_{a,\varepsilon}\right\|_{L^{2}(U_{a,z})}, \hspace{5mm} \eta_{a, z, \varepsilon}:=\frac{\gamma_{a,z}(t) \omega_{a,\varepsilon}}{\alpha_{a, z,  \varepsilon}},
\end{equation}
and we define 
\begin{equation}
\label{definition_perturbed_basis}
    \mathcal{W}(\mathcal{P}_{\varepsilon,B}(U_{a,z})):=\Bigg\{ \eta_{a, z, \varepsilon}=\frac{\gamma_{a,z}(t) \omega_{a,\varepsilon}}{\alpha_{a, z,  \varepsilon}} : \omega_{a,\varepsilon} \in \mathcal{H}(\mathcal{P}_{\varepsilon, B}(U_{a,z})) \Bigg\}
\end{equation}
where the forms $\eta_{a, z, \varepsilon}$ each have unit $L^2$ norm and are supported on $U_{a,z}$ as in Definition \ref{definition_tangent_cone_truncations}. In the case where $z=1$ we drop the subscript $z=1$ and denote the  corresponding forms $\alpha_{a, \varepsilon}, \eta_{a, \varepsilon}$, and the cutoff function $\gamma_a$. Since the fundamental neighbourhood $U_a$ of any critical point $a$ is quasi-isometric to that on the truncated tangent cone (see Remark \ref{remark_general_domain_choices_Morse}), we can extend the forms from a small fundamental neighbourhood $U_{a}$ to $X$ by $0$ away from their supports.
We have the Witten deformed Dirac type operator $D_{\varepsilon}=P_{\varepsilon}+P^*_{\varepsilon}$, whose square is the Witten deformed Laplace type operator $\Delta_{\varepsilon}$.

Let $E_{\varepsilon}$ be the vector space generated by the set $\left\{\mathcal{W}(\mathcal{P}_{\varepsilon,B}(U_{a})) : a \in Cr(h)\right\}$ corresponding to all the harmonic sections $\omega_{a,\varepsilon}$ as above; $E_{\varepsilon}$ is a subspace of $L^2\Omega(X;E)$ since each $\eta_{a, \varepsilon}$ has finite length and compact support. Since $E_{\varepsilon}$ is finite dimensional (in particular closed), there exists an orthogonal splitting
\begin{equation}
    L^2\Omega(X;E)=E_{\varepsilon} \oplus E_{\varepsilon}^{\perp}.
\end{equation}
where $E_{\varepsilon}^{\perp}$ is the orthogonal complement of $E_{\varepsilon}$ in $L^2\Omega(X;E)$. Denote by $\Pi_{\varepsilon}, \Pi_{\varepsilon}^{\perp}$ the orthogonal projection maps from $L^2\Omega(X;E)$ to $E_{\varepsilon}, E_{\varepsilon}^{\perp}$, respectively.

We split the deformed Witten operator by the projections as follows.
\begin{equation}
\label{equation_fourfold_operators}
D_{\varepsilon, 1} =\Pi_{\varepsilon} D_{\varepsilon} \Pi_{\varepsilon}, \hspace{3mm}
D_{\varepsilon, 2} =\Pi_{\varepsilon} D_{\varepsilon} \Pi_{\varepsilon}^{\perp}, \hspace{3mm} D_{\varepsilon, 3} =\Pi_{\varepsilon}^{\perp} D_{\varepsilon} \Pi_{\varepsilon}, \hspace{3mm} D_{\varepsilon, 4} =\Pi_{\varepsilon}^{\perp} D_{\varepsilon} \Pi_{\varepsilon}^{\perp}.
\end{equation}
In the following proposition and its proof, the norms and inner products are those for the $L^2$ forms (unless otherwise specified) corresponding to the Hilbert space of the complex $\mathcal{P}_{\varepsilon}(X)$, and the inner products are the same for all $\varepsilon \geq 0$. We observe that quantities such as $\eta_{a,\epsilon}$ and $\gamma_a(t)$ are only supported in the fundamental neighbourhoods corresponding to the critical points $a$, where the $L^2$ inner products for the local complexes $\mathcal{P}_{\varepsilon,B}(U_{a,z})$ match the global inner product. In the proof, we will only need the complexes for $z=1$ and $z=\infty$, the latter being used only in the proof of the third statement of the proposition.

\begin{proposition}
\label{proposition_Zhangs_Morse_inequalities_intermediate_estimates}
In the setting described above, we have the following estimates.
\begin{enumerate}
    \item There exist a constant $\varepsilon_{0}>0$ such that for any $\varepsilon>\varepsilon_{0}$ and for any $s \in \mathcal{D}(D_{\varepsilon})$, $$||D_{\varepsilon, 1}s|| \leq \frac{\|s\|}{2\varepsilon}.$$ 
    \item There exists a constant $\varepsilon_{1}>0$ such that for any $s \in E_{\varepsilon}^{\perp} \cap \mathcal{D}(D_{\varepsilon}), s^{\prime} \in E_{\varepsilon}$, and $\varepsilon>\varepsilon_{1}$,
$$
\begin{aligned}
\left\|D_{\varepsilon, 2} s\right\| & \leq \frac{\|s\|}{2\varepsilon} \\
\left\|D_{\varepsilon, 3} s^{\prime}\right\| & \leq \frac{\left\|s^{\prime}\right\|}{2\varepsilon}
\end{aligned}
$$
    
    \item  There exist constants $\varepsilon_{2}>0$ and $C>0$ such that for any $s \in E_{\varepsilon}^{\perp} \cap \mathcal{D}(D_{\varepsilon})$ and $\varepsilon>\varepsilon_{2}$,
$$
\left\|D_{\varepsilon} s\right\| \geq C \sqrt{\varepsilon}\|s\|
$$
\end{enumerate}
\end{proposition}

\begin{remark}
We remark that the first Sobolev space appearing in the statement of Proposition 5.6 of \cite{Zhanglectures} in the smooth setting is replaced by the domain $\mathcal{D}(D_\varepsilon)$ in this generalization to our singular setting. We observe that on non-Witt spaces different choices of domains for $P_\varepsilon$ have different local cohomology groups, and the estimates hold only for sections in the domain $\mathcal{D}(D_\varepsilon)$.
\end{remark}

\begin{proof}
\textbf{\textit{Outline:}} We adapt the proof of Proposition 5.6 of \cite{Zhanglectures}. The notation is slightly different (in particular the Witten deformation parameter is taken to be $T$ in that article) but it is easy to follow that proof.
The first statement in this version is different from that in the smooth setting and we prove it in detail. The proof of the second and third statements are essentially the same as in the smooth setting. The same proofs in the smooth case go through for these two results after replacing the model solutions in the smooth proofs with the forms in $\mathcal{W}(\mathcal{P}_{\varepsilon,B}(U_{a}))$. We can pick an orthonormal basis $\widehat{W_{a,\varepsilon}}$ for the vector space generated by those forms. 

\textbf{\textit{proof of 1:}} For any $s \in L^2\Omega(X)$ the projection $\Pi_{\varepsilon} s$ can be written
\begin{equation}
\label{Projection_first_witten_deform}    
\Pi_{\varepsilon} s=\sum_{a \in Cr(h)} \sum_{\eta \in \widehat{W_{a,\varepsilon}}} \left\langle \eta, s\right\rangle_{L^2} \eta
\end{equation}
where each $\eta$ can be written as a linear combination of forms $\eta_{a,\varepsilon}$ as defined in equation \eqref{definition_perturbed_basis}. We have that
\begin{equation}
    ||\Pi_{\varepsilon} D_\varepsilon \eta_{a, \varepsilon}|| \leq e^{-C_0 \varepsilon}
\end{equation}
for some $C_0>0$ and $\varepsilon$ large enough using the following argument. Since $D_{\varepsilon}=D+ \varepsilon \widehat{c}(dh)$, observe that for a smooth function $v$ we have 
\begin{equation}
\label{equation_Leibniz_for_Witten_deformation}
    D_{\varepsilon}( v \omega_{a,\varepsilon}) = D(v \omega_{a,\varepsilon})+ \varepsilon \widehat{c}(dh)( v \omega_{a,\varepsilon}) = cl(dv) \omega_{a,\varepsilon} + v \big ( (D \omega_{a,\varepsilon}) + \varepsilon \widehat{c}(dh)(\omega_{a,\varepsilon}) \big ) = cl(dv) \omega_{a,\varepsilon}
\end{equation}
since $D_{\varepsilon} \omega_{a,\varepsilon}=0$. Since $\text{supp}( d\gamma_a(t)) \subseteq \{1/2 \leq t=x^2+r^2 \leq 3/4\}$ (recall that $\gamma_a=\gamma_{a,z=1}$), 
each $D_{\varepsilon}\eta_{a,\varepsilon}$ is compactly supported in $\{1/2 \leq t=x^2+r^2 \leq 3/4\}$. Then for $\varepsilon$ large enough
\begin{equation} 
\label{inequality_useful_for_Morse_proof_234}
\left\langle D_{\varepsilon} \eta_{a, \varepsilon}, \eta_{a, \varepsilon} \right\rangle_{L^2} =  \left\langle cl(d\gamma_a(t)) \omega_{a, \varepsilon}/ {\alpha_{a,\varepsilon}}, \eta_{a, \varepsilon} \right\rangle_{L^2} \leq e^{-C_0 \varepsilon}
\end{equation}
for some large enough positive constant $C_0$ 
since $cl(d\gamma_a(t)) \omega_{a, \varepsilon}/ {\alpha_{a,\varepsilon}}$ is supported away from $t \leq 1/2$ and $\omega_{a, \varepsilon}=\omega_a e^{-t\varepsilon}$. 

Since each $\eta$ has unit norm, the Cauchy Schwartz inequality shows that $|\left\langle \eta, s\right\rangle_{L^2}| \leq ||s||$.
Since the forms $\eta$ in the basis $\widehat{W_{a,\varepsilon}}$ for a given critical point $a$ are orthogonal,
and since the supports of the forms in $\widehat{W_{a,\varepsilon}}$ for different critical points have no intersection, using equation \eqref{Projection_first_witten_deform} we see that
\begin{equation}
    ||\Pi_{\varepsilon} D_\varepsilon \Pi_{\varepsilon}s|| \leq e^{-C_0 \varepsilon}||s||
\end{equation}
for large enough $\varepsilon$ in order to compensate for the finite sum of terms as well as ensuring \eqref{inequality_useful_for_Morse_proof_234}. The estimate of the Proposition statement follows since the exponential decays faster than the required decay.

\textbf{\textit{proof of 2:}} In the smooth setting, this follows from Proposition 4.11 of \cite{Zhanglectures} which we will redo in this case, changing the details where necessary.
Since $D_\varepsilon$ is self-adjoint, it is easy to see that $D_{\varepsilon,2}$ is the adjoint of $D_{\varepsilon,3}$, and it suffices to prove the first estimate of the two.

Since each $\eta_{a, \varepsilon}$ has support in $U_a$, one deduces that for any $s \in E_{\varepsilon}^{\perp} \cap \mathcal{D}(D_{\varepsilon})$,

$$
\begin{aligned}
D_{\varepsilon, 2} s & =\Pi_{\varepsilon} D_{\varepsilon} \Pi_{\varepsilon}^{\perp} s=\Pi_{\varepsilon} D_{\varepsilon} s \\
& =\sum_{a \in Cr(h)} \sum_{\eta \in \widehat{W_{a,\varepsilon}}} \left\langle\eta, D_{\varepsilon} s\right\rangle_{L^2} \eta \\
& =\sum_{a \in Cr(h)} \sum_{\eta \in \widehat{W_{a,\varepsilon}}} \eta \int_{U_a} \left\langle\eta, D_{\varepsilon} s\right\rangle_{\Lambda^{\cdot}} dv_{U_a} \\
& =\sum_{a \in Cr(h)} \sum_{\eta \in \widehat{W_{a,\varepsilon}}} \eta \int_{U_a} \left\langle D_{\varepsilon} \eta,  s\right\rangle_{\Lambda^{\cdot}}  dv_{U_a}
\end{aligned}
$$
where we have denoted the inner product on the wedge exterior bundle by the angle brackets with subscript $\Lambda^{\cdot}$. Since each $\eta$ can be written as a linear combination of forms $\eta_{a,\varepsilon}$ as defined in equation \eqref{definition_perturbed_basis}, it suffices to estimate the integral over $U_a$ with integrand $\left\langle D_{\varepsilon} \eta_{a,\varepsilon},  s\right\rangle_{\Lambda^{\cdot}}$. This can be expanded as
\begin{equation}
    \left\langle D_{\varepsilon} \frac{\gamma_a(t) \omega_{a} e^{-(x^2+r^2) \varepsilon}}{{\alpha_{a, \varepsilon}}},  s\right\rangle_{L^2} =\left\langle \frac{cl(d\gamma_a(t)) \omega_{a} e^{-(x^2+r^2) \varepsilon}}{{\alpha_{a, \varepsilon}}},  s\right\rangle_{L^2} 
\end{equation}
restricted to each $U_a$, where we have used the argument in \eqref{equation_Leibniz_for_Witten_deformation}, similar to equation (4.40) of \cite{Zhanglectures}. We know that $d\gamma_a$ is only supported on the set $\{1/2 \leq t=(x^2+r^2) \leq 3/4\}$ in each $U_a$, we see that the desired inequality follows from 
Cauchy-Schwarz.

\textbf{\textit{proof of 3:}} This is the third statement in Proposition 5.6 of in \cite{Zhanglectures}, where the proof is that of Proposition 4.12 of that article.
This is proven in three steps:
\begin{enumerate}
    \item  Assume $\operatorname{supp}(s) \subset \bigcup_{a \in Cr(h)} U_{a}$.
    \item  Assume $\operatorname{supp}(s) \subset X \backslash \bigcup_{a \in Cr(h)} V_{a}$, where $V_{a}=\{t=(x^2+r^2) \leq 3/4\} \subset U_a$ where the functions $x,r$ are those in the neighbourhood given in Definition \ref{definition_stratified_Morse_function}.
    \item General Case.
\end{enumerate}
Note that the forms in $\mathcal{W}(\mathcal{P}_{\varepsilon,N}(U_{a}))$ are only supported on $V_a$. Thus the second step above is the analog of that of the second step of Proposition 4.12 of \cite{Zhanglectures}.

\textit{Step 1:} 
In this step, Zhang uses the forms $\mathcal{W}(\mathcal{P}_{\varepsilon,B}(U_{a,\infty}))$ in the smooth case, where he explicitly computes the normalization factors $\alpha_{a, \varepsilon}$ in \eqref{equation_modify_forms_cutoff}. We follow his approach without attempting to making the factors explicit. Observe that since $\gamma_{a,\infty}$ is identically $1$ on the infinite cone, the forms in $\mathcal{W}(\mathcal{P}_{\varepsilon,B}(U_{a,\infty}))$ are elements in $\mathcal{H}(\mathcal{P}_{\varepsilon,B}(U_{a,\infty}))$ on the infinite cone with unit norm.
For any section $s$ verifying $\operatorname{supp}(s) \in \bigcup_{a \in {cr}(h)} U_{a}$, the projection $\Pi'_{\varepsilon} s$ is defined by
\begin{equation}
\label{Projection_first_witten_deform_non_compact}    
\Pi'_{\varepsilon} s:=\sum_{a \in Cr(h)} \sum_{\omega_{a,\varepsilon} \in \widehat{W'_{\varepsilon,a}}} \left\langle \omega_{a,\varepsilon}, s\right\rangle_{L^2(U_{a,\infty})} \omega_{a,\varepsilon}
\end{equation}
where $\widehat{W'_{\varepsilon,a}}$ is an orthonormal basis for the vector space $\mathcal{H}(\mathcal{P}_{\varepsilon,B}(U_{a,\infty}))$.
Then $\Pi'_{\varepsilon}$ is an orthogonal projection
\begin{equation}
    \Pi'_{\varepsilon} :\bigoplus_{a \in {cr}(h)} L^2\Omega(U_{a,\infty}) \rightarrow \mathcal{H}(\mathcal{P}_{\varepsilon,B}(U_{a,\infty})).
\end{equation}
Note that if $\omega_{a,\varepsilon} \in \mathcal{H}(\mathcal{P}_{\varepsilon,B}(U_{a,\infty}))$, then $\gamma_a(t) \omega_{a,\varepsilon} \in E_{\varepsilon}$ and so if $s \in E^{\perp}_{\varepsilon}$ we have
\begin{equation}
\label{equation_for_projection_modified_21}
\Pi'_{\varepsilon} s=\sum_{a \in Cr(h)} \sum_{\omega_{a,\varepsilon} \in \widehat{W'_{\varepsilon,a}}}
\omega_{a,\varepsilon} 
\left\langle(1-\gamma_a(t)) \omega_{a,\varepsilon}, s \right\rangle_{L^2\Omega(U_{a,\infty};E)}.
\end{equation}
Using an argument similar to that in the first point of the proposition, we can show that
\begin{equation}
    ||(1-\gamma_a(t))\omega_{a,\varepsilon}||^2_{L^2\Omega(U_{a,\infty};E)} \leq \frac{C}{\sqrt{\varepsilon}}
\end{equation}
for some constant $C$ since $\gamma_a$ equals to 1 near each $a$. Then equation \eqref{equation_for_projection_modified_21} shows that there exists $C_{5}>0$ such that when $\varepsilon \geq 1$,
\begin{equation}
\label{equation_4.46_of_Zhang}
\left\|\Pi'_{\varepsilon} s \right\|^{2} \leq \frac{C_{5}}{\sqrt{\varepsilon}}\|s\|^{2}.
\end{equation}
Since $s-\Pi'_{\varepsilon} s$ is in $E_{\varepsilon}^{\perp}$, by equation \eqref{equation_4.46_of_Zhang} and the spectral gap given by Proposition \ref{Proposition_model_spectral_gap_modified_general},
there exist constants $C_{6}>0$, $C_{7}>0$ such that
\begin{equation}
\begin{aligned}
\left\|D_{\varepsilon} s\right\|^{2}= & \left\|D_{\varepsilon}\left(s-\Pi'_{\varepsilon} s\right)\right\|^{2} \geq C_{6} \varepsilon\left\|s-\Pi'_{\varepsilon} s\right\|^{2} \\
& \geq \frac{C_{6} \varepsilon}{2}\|s\|^{2}-C_{7} \sqrt{\varepsilon}\|s\|^{2},
\end{aligned}
\end{equation}
from which one sees directly that there exists $\varepsilon_{1}>0$ such that for any $\varepsilon \geq \varepsilon_{1}$
\begin{equation}
\left\|D_{\varepsilon} s\right\| \geq \frac{\sqrt{C_{6} \varepsilon}}{2}\|s\|,
\end{equation}
proving step 1. 

\textit{Step 2:} 
Since $\operatorname{supp}(s) \subset X \backslash \bigcup_{a \in Cr(h)} V_{a}$, one can proceed as in the proof of Proposition \ref{Propostion_growth_estimate_witten_deformed} to find constants $\varepsilon_{2}>0$ and $C_{8}>0$, such that for any $\varepsilon \geq \varepsilon_{2}$,
$$
\left\|D_{\varepsilon} s\right\| \geq C_{8} \sqrt{\varepsilon}\|s\|,
$$
proving step 2.

\textit{Step 3:} 
Let $\widetilde{\gamma} \in C_{\Phi}^{\infty}(X)$ be defined such that restricted to each $U_{a}$ for critical points $a$, $\widetilde{\gamma}(t)=\gamma_a(t)$, and that $\left.\widetilde{\gamma}\right|_{X \backslash \bigcup_{a \in Cr(h)} U_{a}}=0$.
For any  $s \in E_{\varepsilon}^{\perp} \cap \mathcal{D}(D_{\varepsilon})$ we see that $\widetilde{\gamma} s \in E_{\varepsilon}^{\perp} \cap \mathcal{D}(D_{\varepsilon})$.
Then the results of steps 1 and 2 shows that there exists $C_{9}>0$ such that for any $\varepsilon \geq \varepsilon_{0}+\varepsilon_{1}+\varepsilon_{2}$,
$$
\begin{gathered}
\left\|D_{\varepsilon} s\right\| \geq \frac{1}{2}\left(\left\|(1-\widetilde{\gamma}) D_{\varepsilon} s\right\|+\left\|\widetilde{\gamma} D_{\varepsilon} s\right\|\right)
=\frac{1}{2}\left(\left\|D_{\varepsilon}((1-\widetilde{\gamma}) s)+[D, \widetilde{\gamma}] s\right\|+\left\|D_{\varepsilon}(\widetilde{\gamma} s)+[\widetilde{\gamma}, D] s\right\|\right) \\
\geq \frac{\sqrt{\varepsilon}}{2}\left(C_{8}\|(1-\widetilde{\gamma}) s\|+\sqrt{C_{6}}\|\widetilde{\gamma} s\|_{0}\right)-C_{9}\|s\| 
\geq C_{10} \sqrt{\varepsilon}\|s\|_{0}-C_{9}\|s\|,
\end{gathered}
$$
where $C_{10}=\min \left\{\sqrt{C_{6}} / 2, C_{8} / 2\right\}$, which completes the proof of the proposition.
\end{proof}

For any $c>0$, denote by {$E_{\varepsilon,c}$} the direct sum of the eigenspaces of $D_{\varepsilon}$ with eigenvalues lying in $[-c, c]$, which is a finite dimensional subspace of $L^2\Omega(x)$. Let $\Pi_{\varepsilon,c}$ be the orthogonal projection from $L^2\Omega(x)$ to $E_{\varepsilon,c}$. The following is a generalization of Lemma 5.8 of \cite{Zhanglectures}. Again we closely follow Zhang's proof, modifying as needed for the singular setting.

\begin{lemma}
\label{Lemma_inequality_spectral_for_Witten_deformation}
There exist $C_1>0$, $\varepsilon_3>0$ such that for any $\varepsilon>\varepsilon_3$, and any $\sigma \in E_{\varepsilon}$, 
\begin{equation}
    \left\|\Pi_{\varepsilon,c} \sigma-\sigma\right\| \leq \frac{C_1}{\varepsilon}\|\sigma\|.
\end{equation}
\end{lemma}

\begin{proof}
Let $\delta=\{\lambda \in \mathbf{C}:|\lambda|=c\}$ be the counter-clockwise oriented circle. By Proposition \ref{proposition_Zhangs_Morse_inequalities_intermediate_estimates}, one deduces that for any $\lambda \in \delta, \varepsilon \geq \varepsilon_{0}+\varepsilon_{1}+\varepsilon_{2}$ and $s \in \mathcal{D_\varepsilon}$, there exists positive constants $c_0, C_0$ such that
\begin{equation}   
\begin{gathered}
\left\|\left(\lambda-D_{\varepsilon}\right) s\right\| \geq \frac{1}{2}\left\|\lambda \Pi_{\varepsilon} s -D_{\varepsilon, 1}\Pi_{\varepsilon}s   -D_{\varepsilon, 2} \Pi_{\varepsilon} s\right\| +\frac{1}{2}\left\|\lambda \Pi_{\varepsilon}^{\perp} s-D_{\varepsilon, 3} \Pi_{\varepsilon}^{\perp} s-D_{\varepsilon, 4} \Pi_{\varepsilon}^{\perp} s\right\| \\
\geq \frac{1}{2}\left(\left(c_0-\frac{1}{\varepsilon}\right)\left\|\Pi_{\varepsilon} s\right\|+\left(C_0 \sqrt{\varepsilon}-c_0-\frac{1}{\varepsilon}\right)\left\|\Pi_{\varepsilon}^{\perp} s\right\|\right) .
\end{gathered}
\end{equation}
This shows that there exist $\varepsilon_{4}>\varepsilon_{0}+\varepsilon_{1}+\varepsilon_{2}$ and $C_{2}>0$ such that for any $\varepsilon \geq \varepsilon_{4}$ and $s \in \mathcal{D_\varepsilon}$,
\begin{equation}  
\label{equation_(5.27)_Zhang}
\left\|\left(\lambda-D_{\varepsilon}\right) s\right\| \geq C_{2}\|s\|.
\end{equation}
Thus, for any $\varepsilon  \geq \varepsilon _{4}$ and $\lambda \in \delta$,
\begin{equation}  
\lambda-D_{\varepsilon }:  \mathcal{D_\varepsilon} \rightarrow L^2\Omega(X)
\end{equation}
is invertible and the resolvent $\left(\lambda-D_{\varepsilon}\right)^{-1}$ is well-defined.
By the spectral theorem 
one has
\begin{equation}
\label{equation_spectral_projector_countour_integral}
\Pi_{\varepsilon,c} \sigma-\sigma=\frac{1}{2 \pi \sqrt{-1}} \int_{\delta}\left(\left(\lambda-D_{\varepsilon}\right)^{-1}-\lambda^{-1}\right) \sigma d \lambda.
\end{equation}
Now one verifies directly
that for any $\sigma \in E_\varepsilon$
\begin{equation}
\left(\left(\lambda-D_{\varepsilon}\right)^{-1}-\lambda^{-1}\right) \sigma=\lambda^{-1}\left(\lambda-D_{\varepsilon}\right)^{-1} (D_{\varepsilon, 1}+D_{\varepsilon, 3}) \sigma .
\end{equation}
From Proposition \ref{proposition_Zhangs_Morse_inequalities_intermediate_estimates} and \eqref{equation_(5.27)_Zhang}, one then deduces that for any $\varepsilon \geq \varepsilon_{4}$ and $\sigma \in E_{\varepsilon}$,
\begin{equation}
\label{equation_estimate_finale}
\left\|\left(\lambda-D_{\varepsilon}\right)^{-1} (D_{\varepsilon, 1}+D_{\varepsilon, 3}) \sigma\right\| \leq C_{2}^{-1}\left\|(D_{\varepsilon, 1}+D_{\varepsilon, 3}) \sigma\right\| \leq \frac{1}{C_{2} \varepsilon}\|\sigma\|
\end{equation}
From \eqref{equation_spectral_projector_countour_integral}-\eqref{equation_estimate_finale}, we get the estimate in the statement of the Lemma, finishing the proof.
\end{proof}

In Remark 5.9 of \cite{Zhanglectures}, Zhang explains that one can work out an analog of the proof with real coefficients, whereas the proof above implicitly uses the fact that we work in the category of complex coefficients.
The following is a generalization of Proposition 5.5 of \cite{Zhanglectures}.
\begin{proposition}
\label{Proposition_small_eig_estimate_and_dimension}
For any $c>0$, there exists $\varepsilon_0>0$ such that when $\varepsilon>\varepsilon_0$, the number of eigenvalues in $[0,c]$ of $\Delta_{\varepsilon,k}$, the Laplace-type operator acting on forms of degree $k$ of the complex $\mathcal{P}_{\varepsilon}(X)$, is equal to $\sum_{a \in Cr(h)} \dim \mathcal{H}^k(\mathcal{P}_{\varepsilon,B}(U_{a}))$.
\end{proposition}

\begin{proof}
By applying Lemma \ref{Lemma_inequality_spectral_for_Witten_deformation} to the elements of $\eta_{a,\varepsilon} \in \left\{\mathcal{W}(\mathcal{P}_{\varepsilon,B}(U_{a})) : a \in Cr(h)\right\}$, one sees easily that for any $c>0$ when $\varepsilon$ is large enough, the elements of the set $\{\Pi_{\varepsilon,c} \eta_{a,\varepsilon}: \eta_{a,\varepsilon} \in  \cup_{a \in Cr(h)} \widehat{W}_{a,\varepsilon}\}$, where $\widehat{W}_{a,\varepsilon}$ is the basis chosen in the proof of Proposition \ref{proposition_Zhangs_Morse_inequalities_intermediate_estimates}, are linearly independent. Thus, there exists $\varepsilon_{5}>0$ such that when $\varepsilon \geq \varepsilon_{5}$,
\begin{equation}
    \dim E_{\varepsilon,c} \geq \operatorname{dim} E_{\varepsilon}.
\end{equation}
Now if $\operatorname{dim} E_{\varepsilon,c}>\operatorname{dim} E_{\varepsilon}$, then there should exist a nonzero element $s \in E_{\varepsilon,c}$ such that $s$ is perpendicular to $\Pi_{\varepsilon,c} E_{\varepsilon}$. That is, $\left\langle s,\Pi_{\varepsilon,c} \eta_{a,\varepsilon} \right\rangle_{L^2(X;E)}=0$ for any $\eta_{a,\varepsilon}$ as above.
Then as in \eqref{Projection_first_witten_deform}, we can write the projection
\begin{equation}
\begin{gathered}
\Pi_{\varepsilon} s=\sum_{a \in Cr(h)} \sum_{\eta \in \widehat{W_{a,\varepsilon}}} \Big( \left\langle s, \eta\right\rangle_{L^2} \eta
-\left\langle s, \Pi_{\varepsilon,c} \eta 
\right\rangle_{L^2}\Pi_{\varepsilon,c} \eta
\Big)\\
=\sum_{a \in Cr(h)} \sum_{\eta \in \widehat{W_{a,\varepsilon}}} \left\langle s, \eta\right\rangle_{L^2} (\eta-\Pi_{\varepsilon,c} \eta)
+\sum_{a \in Cr(h)} \sum_{\eta \in \widehat{W_{a,\varepsilon}}} \left\langle s, (\eta-\Pi_{\varepsilon,c} \eta)\right\rangle_{L^2} \Pi_{\varepsilon,c} \eta.
\end{gathered}
\end{equation}
This together with Lemma \ref{Lemma_inequality_spectral_for_Witten_deformation} shows that there exists $C_{3}>0$ such that when $\varepsilon \geq \varepsilon_{5}$,
\begin{equation}
    \left\|\Pi_\varepsilon s\right\| \leq \frac{C_{3}}{\varepsilon}\|s\|.
\end{equation}
Thus, there exists a constant $C_{4}>0$ such that when $\varepsilon>0$ is large enough,
\begin{equation}
\left\|\Pi_\varepsilon^{\perp} s\right\| \geq\|s\|-\left\|\Pi_\varepsilon s\right\| \geq C_{4}\|s\| .
\end{equation}

Using this and Proposition \ref{proposition_Zhangs_Morse_inequalities_intermediate_estimates} one sees that when $\varepsilon>0$ is large enough,
\begin{equation}
\begin{gathered}
C C_{4} \sqrt{\varepsilon}\|s\| \leq\left\|D_{\varepsilon} \Pi_\varepsilon^{\perp} s\right\|=\left\|D_{\varepsilon} s-D_{\varepsilon} \Pi_\varepsilon s\right\|
=\left\|D_{\varepsilon} s-D_{\varepsilon, 1} s - D_{\varepsilon, 3} s\right\|\\
\leq\left\|D_{\varepsilon} s\right\|+\left\|D_{\varepsilon, 1} s\right\|+\left\|D_{\varepsilon, 3} s\right\|
\leq\left\|D_{\varepsilon} s\right\|+\frac{1}{\varepsilon}\|s\|
\end{gathered}
\end{equation}
from which one gets
\begin{equation}
    \left\|D_{\varepsilon} s\right\| \geq C C_{4} \sqrt{\varepsilon}\|s\|-\frac{1}{\varepsilon}\|s\|.
\end{equation}
Clearly, when $\varepsilon>0$ is large enough, this contradicts the assumption that $s$ is a nonzero element in $E_{\varepsilon,c}$.
Thus, one has
\begin{equation}
    \dim E_{\varepsilon,c}=\dim E_\varepsilon=\sum_{a \in Cr(h)} \dim \mathcal{H}^k(\mathcal{P}_{\varepsilon,B}(U_{a,z}))
\end{equation}
proving the result.
\end{proof}

These results show that the small eigenvalue complex of the Witten deformed Laplacian encodes information about the critical point structures. This is also known as the Witten instanton complex.

\begin{theorem}[Witten instanton complex]
\label{theorem_small_eig_complex}
For any integer $0 \leq k \leq n$, let $
\mathrm{F}_{\varepsilon, k}^{[0, c]} \subset L^2\Omega^k(X)$ denote the vector space generated by the eigenspaces of $\Delta_{\varepsilon,k}$ associated with eigenvalues in $[0, c]$. For any $c>0$, there exists $\varepsilon_0>0$ such that when $\varepsilon>\varepsilon_0$ this has the same dimension as $\sum_{a \in Cr(h)} \mathcal{H}^k(\mathcal{P}_{\varepsilon,B}(U_a))$, and together form a finite dimensional subcomplex of $\mathcal{P}_{\varepsilon}(X)$ :
\begin{equation}
    \label{small_eigenvalue_complex}
\left(\mathrm{F}_{\varepsilon, k}^{[0, c]}, P_{\varepsilon}\right): 0 \longrightarrow \mathrm{F}_{\varepsilon, 0}^{[0, c]} \stackrel{P_{\varepsilon}}{\longrightarrow} \mathrm{F}_{\varepsilon, 1}^{[0, c]} \stackrel{P_{\varepsilon}}{\longrightarrow} \cdots \stackrel{P_{\varepsilon}}{\longrightarrow} \mathrm{F}_{\varepsilon, n}^{[0, c]} \longrightarrow 0.
\end{equation}
\end{theorem}

\begin{proof}    
This follows from Proposition \ref{Proposition_small_eig_estimate_and_dimension} once one shows that the small eigenvalue eigensections form a complex.
Since
$$
P_{\varepsilon} \Delta_{\varepsilon}=\Delta_{\varepsilon} P_{\varepsilon} =P_{\varepsilon} P_{\varepsilon}^* P_{\varepsilon} \text{   and    }
P_{\varepsilon}^* \Delta_{\varepsilon}=\Delta_{\varepsilon} P_{\varepsilon}^* =P_{\varepsilon}^* P_{\varepsilon} P_{\varepsilon}^*
$$
one sees that $P_{\varepsilon}$ (resp. $P_{\varepsilon}^*$ ) maps each $\mathrm{F}_{\varepsilon, k}^{[0, c]}$ to $\mathrm{F}_{\varepsilon, k+1}^{[0, c]}$ (resp. $\mathrm{F}_{\varepsilon, k-1}^{[0, c]}$ ). 
The Kodaira decomposition of $\mathcal{P}_{\varepsilon}(X)$ restricts to this finite dimensional complex $\left(\mathrm{F}_{\varepsilon, k}^{[0, c]}, P_{\varepsilon}\right)$. In particular for an integer $0 \leq k \leq n$,
$$
\beta_{\varepsilon, k}^{[0, c]}:=\operatorname{dim}\left(\frac{\left.\operatorname{ker} P_{\varepsilon}\right|_{\mathrm{F}_{\varepsilon, k}^{[0, c]}}}{\left.\operatorname{Im} P_{\varepsilon}\right|_{\mathrm{F}_{\varepsilon, k}^{[0, c]}}}\right)
$$
equals $\operatorname{dim}\left(\left.\operatorname{ker} \Delta_{\varepsilon}\right|_{L^2\Omega^{k}(X)}\right)$, which in turn equals $\beta_{k}$ by the discussion in Subsection \ref{subsubsection_Witten_deformed_elliptic}.
\end{proof}

\subsubsection{Morse and Lefschetz Morse inequalities}

We prove the following polynomial form of the Strong Morse inequalities.

\begin{theorem}
\label{Theorem_strong_Morse_de_Rham}
[Strong polynomial Morse inequalities]
Let $\widehat{X}$ be a Witt space of dimension $n$ with a wedge metric and a stratified Morse function $h$. Let $\mathcal{P}(X)=(L^2\Omega(X),d)$ be the de Rham complex.
Then there exist non-negative integers $Q_0,..., Q_{n-1}$ such that
\begin{equation}
\label{Morse_inequality_de_Rham_dimension_cohomology}
    M(X,h)(b)-N(X)(b) = (1+b) \sum_{k=0}^{n-1} Q_k b^k.
\end{equation}
\end{theorem}

We observe that for Morse functions for which the gradient flow of $h$ is stratum preserving, equation \eqref{Morse_inequality_de_Rham_dimension_cohomology} evaluated at $b=-1$ corresponds to the Lefschetz fixed point theorem for the geometric endomorphism $T_f$ obtained from the short time flow by the gradient flow of $h$.
Indeed the evaluation at $b=-1$ is preserved by the duality given by Proposition \ref{proposition_Lefschetz_on_adjoint}, which corresponds to the well known manifestation of Poincar\'e duality in Morse theory in the smooth case (see, e.g., the introduction of \cite{witten1984holomorphic}).
In the smooth case we see that the local cohomology $\mathcal{H}^{*}(\mathcal{P}_B(U_a))$ at a critical point $a$ of $h$, where the Hessian of $h$ has $j$ negative eigenvalues, will be a one dimensional vector space in degree $j$ and all other cohomology groups will be $0$. This can also be seen from the results in Subsection \ref{subsection_analytic_derivation_intersection} where we have worked out the cohomology groups explicitly, using the fact that $U_1$ and $U_2$ both are discs with spherical links which only have cohomology in degree $0$, and the formulas of Proposition \ref{proposition_local_product_Lefschetz_heat}.
Similarly, it is easy to see that the Morse inequalities in \cite{ludwig2017index} for isolated conic singularities can be recovered using the isomorphism of $L^2$ de Rham cohomology and intersection homology with middle perversity (both for the local and global complexes). 

\begin{proof}[Proof of Theorem \ref{Theorem_strong_Morse_de_Rham}]
To prove the polynomial Morse inequalities, we apply equation \eqref{equation_with_the_b} of Theorem \ref{Lefschetz_supertrace} with the endomorphism $T=Id$
\begin{equation}
\label{interesting_step_small_eig_complex}
    \mathcal{L}(\mathcal{P},T)(b,t)=L(\mathcal{P},T)(b)+ (1+b) \sum_{k=0}^{n-1} b^k S_k(t)
\end{equation}
to the complex \eqref{small_eigenvalue_complex}. Since this is a complex of finite dimensional Hilbert spaces, we can take $t$ to $0$ 
to see that the left hand side is exactly the expression
\begin{equation}
    \Big( \sum_{a \in Crit(h)}  \sum_{k=0}^n b^k \dim(\mathcal{H}^{k}(\mathcal{P}_{\varepsilon,B}(U_{a,z}))) \Big)
\end{equation}
and the right hand side is of the form
\begin{equation}
    \sum_{k=0}^n b^k \dim(\mathcal{H}^{k}(\mathcal{P}_{\varepsilon}(X))) + (1+b) \sum_{k=0}^{n-1} Q_k b^k
\end{equation}
where $Q_k$ are non-negative integers, indeed the dimension of the co-exact small eigenvalue eigensections of the deformed Laplace-type operator. Since the cohomology groups of the Witten deformed and the undeformed complexes have the same dimensions for both the local and global complexes (see Subsection \ref{subsubsection_Witten_deformed_elliptic}), we see that this proves Theorem \ref{Theorem_strong_Morse_de_Rham}.
\end{proof}

\begin{definition}[Perfect stratified Morse function]
\label{definition_perfect_morse_functions}
    We say that a stratified Morse function $h$ is a \textbf{perfect stratified Morse function at degree $k$} if the cohomology of the corresponding instanton complex at degree $k$ has the same dimension as the cohomology at degree $k$ of the $L^2$ de Rham complex.
    
    We say \textbf{$h$ is perfect} if it is perfect at each degree $k$, equivalently if the Morse polynomial is equal to the Poincar\'e polynomial.
\end{definition}

Observe that if $h$ perfect at degree $k$, then it is also perfect at either degree $k+1$ or $k-1$, or both. This is due to the supersymmetry of the instanton complex.

In equation \eqref{interesting_step_small_eig_complex} of the above proof, we took the geometric endomorphism to be that induced by the identity map, which is always a geometric endomorphism on any elliptic complex. However, we can consider other natural geometric endomorphisms on the Witten deformed small eigenvalue complex to extract more information.

\begin{theorem}[Lefschetz-Morse inequalities]
\label{theorem_Lefschetz_Morse_inequalities_de_Rham}
Let $X$ be a resolution of a stratified pseudomanifold equipped with a resolution of a stratified Morse function $h$. Let $f$ be a self map of $X$ such that $f^*h=h$ and the fixed points of $f$ are a subset of the critical points of $h$ that lifts to a geometric endomorphism of the de Rham complex on $X$. Then we have the Lefschetz inequalities
\begin{equation}
\label{Lefschetz_inequalities}
    \mathcal{L}(\mathcal{P},T_f)(b)=L(\mathcal{P},T_f)(b)+ (1+b) \sum_{k=0}^{n-1} b^k Q_k
\end{equation}
where $Q_k=0$ if $h$ is perfect at degrees $k, k+1$. 
\end{theorem}

\begin{proof}
If $f^*h=h$, then $T_f$ is a geometric endomorphism not only of the de Rham complex but also of the Witten deformed de Rham complex for any $\varepsilon>0$, and thus for the Witten instanton complex. Then equation \eqref{Lefschetz_inequalities} follows from applying Theorem \ref{Lefschetz_supertrace}. When $h$ is perfect at degrees $k,k+1$, the cohomology of the de Rham complex is isomorphic to the Witten instanton complex as a graded vector space in degree $k$ and $Q_k=0$.
\end{proof}

We provide the following simple example in the smooth setting for intuition.

\begin{example}
\label{Example_Lefschetz_Morse_1}

Consider the two torus $\mathbb{T}_{\alpha_1\alpha_2}$, that can be identified with the square $[0,2\pi]\times [0,2\pi]$ modulo the relations $\alpha_1=0 \sim \alpha_1=2\pi$ and $\alpha_2=0 \sim \alpha_2=2\pi$. There is a Morse function $h=\cos(\alpha_1)+\cos(\alpha_2)$ which has four critical points $(0,0), (\pi,0), (0,\pi), (\pi,\pi)$, which have Morse indices $+2, +1, +1, 0$ respectively. We consider the torus with the flat metric, given by the product of circles of length $2\pi$. Consider the fundamental domain for the torus on $\mathbb{R}^2$ centered at the origin. This is a square of length $2\pi$. If we rotate this square around it's center by $\pi/2$ radians, we get the generator $g$ of a $\mathbb{Z}_4$ action, which is a self map on the torus. Explicitly we can write $g(\alpha_1,\alpha_2) \mapsto (\alpha_2,-\alpha_1)$.
 
\begin{equation}
    g(\alpha_1,\alpha_2) \mapsto (\alpha_2,-\alpha_1).
\end{equation}

It is easy to see that $g$ preserves the Morse function and fixes $(0,0)$ and $(-\pi,-\pi)$ (up to identifications of the fundamental domain), which are also critical points of the Morse function. The same holds for $g^2:=g\cdot g$, which fixes all the critical points.

The map $g$ gives rise to the natural geometric endomorphism $T_g$ on the small eigenvalue complex. Applying this to equation \eqref{interesting_step_small_eig_complex} and taking the limit as $\varepsilon$ goes to $\infty$, we get $b^2+b^0=b^2+b^1(0)+b^0$, 
 
\begin{equation}
    b^2+b^0=b^2+b^1(0)+b^0
\end{equation}

where for the computation of the global polynomial Lefschetz supertrace we have used that $1, d\alpha_1, d\alpha_2$ and $d\alpha_1 \wedge d\alpha_2$ generate the cohomology. Similarly for $g^2$, we get $b^2-2b^1+b^0=b^2-2b^1+b^0$.

\begin{equation}
    b^2-2b^1+b^0=b^2-2b^1+b^0.
\end{equation}

If we consider the map sending $(\alpha_1,\alpha_2) \mapsto (-\alpha_1,\alpha_2)$, we see that the global Lefschetz polynomial is $b^2-1b^0$, while the Lefschetz number is $0$, which shows that we must have at least 2 fixed points, even though the Lefschetz number (setting $b=-1)$ is $0$. We can compute these at the fixed points as well (if orientation reversing on the $u_2$ factor, the contribution is $-1$, and if not it is $+1$), and computing the local Lefschetz formulas tells us the sum of the self intersection numbers in each degree.
It is clear that this yields more information than the global Lefschetz number. This and similar examples show the potential of such Lefschetz polynomials being more useful than the Lefschetz fixed point theorem, even when the maps are not homotopic to the identity.
\end{example}

The example above shows that there are geometric endomorphisms on the Witten instanton complex corresponding to self-maps $f$ which are not homotopic to the identity. In the smooth setting, Witten deformation is used to study Hamiltonian group actions on symplectic manifolds (see for instance \cite{ZhanganalyticQR98}), and while the operator studied there is the spin$^{\mathbb{C}}$ Dirac operator, such group actions also induce geometric endomorphisms on the Witten deformed de Rham instanton complex constructed above. 

We study geometric endomorphisms on the de Rham Witten instanton complex corresponding to K\"ahler Hamiltonian group actions in \cite{jayasinghe2024holomorphicwitteninstantoncomplexes}.
When there are geometric endomorphisms on the de Rham Witten instanton complex which preserves K\"ahler structures, we derive formulas for other invariants including the Poincar\"e Hodge polynomials, Signature, self dual and anti-self dual complexes. This is related to the question of rigidity which we also address in \cite{jayasinghe2024holomorphicwitteninstantoncomplexes}.

We will study some applications in Subsection \ref{subsection_Morse_Hirzebruch_cohomological}, including relations to Hirzebruch $\chi_y$ genera and generalizations of the Arnold-Floer theorem on spaces with cohomological conditions.
We also compute de Rham Morse polynomials for some examples in Subsection \ref{subsection_Computations_for_singular_examples}.

\section{Dolbeault and spin$^{\mathbb{C}}$ Dirac complexes}
\label{Dolbeault_section}

In this section we study local Lefschetz numbers corresponding to the Dolbeault complex on stratified pseudomanifolds with a complex structure. Since the groundbreaking work of Atiyah and Bott, various holomorphic Lefschetz fixed point formulas have been studied using analytical techniques (see \cite{patodi1973holomorphic,toledo1975holomorphic,donnelly1986fixed,kytmanov2004holomorphic}) as well as algebraic techniques \cite{baum1979lefschetz,Baumformula81,equivariant_intersection_edidin,localization_algebraic_Graham_Edidin,MaximSaitoJorgHodgemodules2011,MaximJorgCharacteristicsingulartoric2015}.

Let us consider the simplest example here, which is the case of the disc of real dimension 2 with the standard complex structure inherited from $\mathbb{C}$. Indeed the work in this section began from the observation that Atiyah and Bott's local formula for the local Holomorphic Lefschetz number of a holomorphic map $f$ with an attracting fixed point $a$ can be written as
\begin{equation}
\label{equation_simple_example}
    \frac{1}{1-df'|_a}=\sum_{k=0}^{\infty} df'|_a^k = Tr\Big( f^* : \mathcal{O}(\mathbb{D}^2) \rightarrow \mathcal{O}(\mathbb{D}^2) \Big)
\end{equation}
where $df'$ is the holomorphic differential on the holomorphic tangent bundle, and where the last expression is the trace of the geometric endomorphism induced by the map $f$ on the Hilbert space of holomorphic functions on the disc with respect to the $L^2$ norm, which is the Hardy space on the unit disc. Here since the map is attracting, $|df'|<1$ and we do not need to renormalize the trace, but in general we can use the framework developed in Subsection \ref{subsubsection_renormalized_Lefschetz}. The holomorphic functions $\{ \frac{2\pi}{2k+1} z^k\}_{k \in \mathbb{Z}_{\geq 0}}$ form an orthonormal Schauder basis for the local Dolbeault cohomology of the unit disc with the flat metric as a manifold with boundary. This suggests that we can use Proposition \ref{Local_Lefschetz_numbers_definition} with a proper choice of boundary conditions to find the holomorphic Lefschetz numbers. It is easy to check that the generalized Neumann conditions in the case of the Dolbeault complex are the well known $\overline{\partial}$-Neumann boundary conditions which will result in a non-Fredholm Hilbert complex.

Our perspective is that these are (possibly renormalized) traces on local cohomology groups which are Hilbert spaces. Holomorphic Lefschetz numbers on coherent sheaves on projective varieties have been interpreted as renormalized traces in \cite{Baumformula81} and we will see that in certain cases our Lefschetz numbers will match theirs, while they are different in others.
Witten also used this perspective to interpret Lefschetz numbers as traces over holomorphic sections in \cite{witten1984holomorphic}.

We will begin by studying the geometry of cones with complex structures, with a focus on understanding the null space of the Dolbeault Dirac operator in fundamental neighbourhoods. In the singular setting there are non-trivial elements in local cohomology in higher degrees at fixed points. We illustrate how local cohomology changes as the domains change in an explicit example, for which we compute Lefschetz numbers later in the section. We also discuss dualities for local cohomology groups, which enable us to make powerful statements about Lefschetz numbers for various complexes later on in the section.

In Subsection \ref{subsection_Local_Lefschetz_Dolbeault_main}
we will discuss the Lefschetz-Riemann-Roch numbers of Baum-Fulton-Quart, before introducing $L^2$-Lefschetz-Riemann-Roch numbers and explaning their similarities and differences on spaces with various classes of singularities.
After briefly explaining how one can recover the Atiyah-Bott formulae
we introduce techniques to understand Lefschetz numbers for spin$^\mathbb{C}$ Dirac and spin Dirac operators on stratified pseudomanifolds. 
The Lefschetz Hirzebruch $\chi_y$ invariant encodes information about many Lefschetz numbers including those associated to the signature, self-dual and anti-self-dual complexes, and indeed for K-theoretic reasons those of any elliptic complex (see Section 5 of \cite{bott1967vector}). We study some of these invariants for singular spaces, discussing relationships to instanton counting problems in mathematical physics.

In Subsection \ref{subsection_Computations_for_singular_examples} we will compute these invariants on certain singular spaces, highlighting how the invariants see different aspects of the spaces, and contrasting against the algebraic invariants of Baum-Fulton-MacPherson-Quart.

\subsection{Complex cones and $L^2$ Dolbeault cohomology}
\label{subsection_L2_Dolbeault_cohomology}

Complex structures on cones and their CR structures as well as Sasaki structures of K\"ahler cones have been widely studied, and we draw from the sources \cite{boyer&galicki,dragomir2007differential,blair2010riemannian}.
It is well known that a complex structure on a cone $(C(M), dx^2 +x^2 g_M)$ induces a CR structure on the link $(M,g_M)$, which in particular yields an almost contact metric structure on the link if it is smooth (see for instance Section 1.1.4 of \cite{dragomir2007differential}). In the case where the CR structure is pseudoconvex, this is an honest contact structure. 
A Sasaki structure on a link $M$ is denoted as $(M^{2n+1}, \mathcal{S})$ (see chapter 7 of \cite{boyer&galicki}), and this is equivalent to a K\"ahler structure on the metric cone.
In this case, the CR structure is pseudoconvex.
In particular it has a contact structure and taking the quotient by the action of the Reeb vector field $\xi$, one gets (in general) an orbifold $\Sigma = M / \xi$ which has a K\"ahler structure, usually known as the transversal K\"ahler structure on the Sasaki manifold. The Reeb foliation $\mathcal{F}_\xi$ on $M$ happens to be a taut Riemannian foliation, and the contact distribution has a splitting into holomorphic and anti holomorphic components. This is studied broadly in complex and CR geometry. 

We now revert to the more general case of a truncated cone with $(C(M), dx^2 +x^2 g_M)$  a wedge metric with just a complex structure as opposed to a K\"ahler structure. The boundary is the link at $\{x=1\}$. We will pick the domains that we studied in Subsection \ref{subsubsubsection_Neumann_boundary_condition} for this space. 
For instance, the $\overline{\partial}$-Neumann boundary conditions for the Laplacian associated to the complex $\mathcal{P}_N(C(M))=(L^2\Omega^{0,q}(C(M);E), \overline{\partial}_E)$, using the boundary defining function $r=1-x$, 
require that a section $u$ in the domain of the Laplacian satisfy
\begin{align}
\label{del_bar_Neumann_bc_1}
    \sigma(\overline{\partial}_E^*, dr)u|_{x=1} =0, \\ 
    \label{del_bar_Neumann_bc_2}
    \sigma(\overline{\partial}_E^*, dr) \overline{\partial}u|_{x=1} =0
\end{align}
where $\sigma$ denotes the principal symbol of the operator.

\subsubsection{The local Dolbeault cohomology}
\label{subsubsection_higher_degree_local_cohomology}

The cohomology of the complex in degree $0$ is precisely the null space of the operator $P$ in degree $0$.  
For instance, for the complex $\mathcal{P}^*_N(\mathbb{D}^{2n})=(L^2 \Omega^{0,n-q}(\mathbb{D}^{2n};E), \overline{\partial_E}^*)$ the cohomology in degree $0$ is given by the antiholomorphic $n$ forms. The Lefschetz numbers for this complex on a pseudoconvex fundamental neighbourhood are related to the Lefschetz numbers in \cite{donnelly1986fixed}.

In the case where the boundary CR structure is pseudoconvex, it is known that the cohomology of the complex on a smooth manifold with boundary is finite dimensional in all degrees greater than $0$ (see Theorem 5.3.8 \cite{chen2001partial}). While local cohomology vanishes for positive degrees for fundamental neighbourhoods of the tangent cone in the smooth case, we will see that this is not so in the singular space. This is similar to the de Rham case where the cohomology is isomorphic to intersection cohomology.
We will show that the local Dolbeault cohomology at singular points can be non-trivial in higher degrees. We begin with the following proposition.

\begin{proposition}
Let $C(M)^{2n+2}$ be a truncated cone with a K\"ahler wedge metric of product type (in particular $M$ is Sasaki). If $\eta$ is the harmonic representative of a cohomology class of the de Rham complex $\mathcal{R}_N(C(M))=(L^2\Omega^k(C(M)),d)$ (that is $\eta \in \mathcal{H}^k(\mathcal{R}_N(C(M)))$, then $\eta \in \oplus_{p+q=k} \mathcal{H}^q(\mathcal{P}^p_N(C(M)))$ where $\mathcal{P}^p_N(C(M))$ is the Dolbeault complex $(L^2\Omega^{0,q}(C(M);\Lambda^{p,0}), \overline{\partial}_{\Lambda^{p,0}})$.
\end{proposition}

\begin{proof}
Since the Dolbeault Laplacian only differs from the Hodge Laplacian by a factor of $1/2$ for K\"ahler metrics, they have the same null space before imposing boundary conditions. We need only show that the harmonic forms in $\mathcal{H}^k(\mathcal{R}_N(C(M)))$ satisfy the boundary conditions for the Dolbeault complex.

We begin by summarizing a few facts on K\"ahler cones from Chapter 7 of \cite{boyer&galicki}.
The basic cohomology $H^*_B(\mathcal{F})$ of the Sasaki manifold with the Reeb foliation $(M^{2n+1}, \mathcal{F}_\xi)$, has a transverse Hodge decomposition and a transverse complex structure on the contact distribution, induced by the one on the cone. 
The unique harmonic representative $\omega$ of a class $[\omega] \in H^{p,q}_B{(\mathcal{F})}$ lifts to a harmonic form (which we denote by $\omega$ with some abuse of notation) on $M$ for the Laplace type operator of the de Rham complex on $M$, and thus on the cone $C(M)$ as well if they have degree $0 \leq k \leq n$ (see Lemma \ref{conic_cohomology_1}). 
Moreover, for $0 \leq k \leq n$, a $k$-form on a Sasaki manifold is harmonic only if it is basic harmonic (see Proposition 7.4.13 of \cite{boyer&galicki}), so all the harmonic forms in $\mathcal{H}^k(\mathcal{R}_N(C(M)))$ are basic harmonic forms.

On the truncated cone we have the decomposition of the forms given in  \eqref{Dolbeault_form_decomposition_1},
\begin{equation}
\label{Dolbeault_form_decomposition_2}
    \Omega^{p,q}(C(M))=\Omega^{p,q}(M) \bigoplus (dx+ixJ(dx)) \wedge \Omega^{p,q-1} (M)
\end{equation}
and the harmonic forms $\omega$ on $M$ lift to the first summand. Note that since there is no $dx+ixJ(dx)$ parts in the form in these sections and they are harmonic, they satisfy the $\overline{\partial}$ boundary conditions. But these are precisely the basic harmonic forms since $J(dx)$ is a contact form corresponding to the Reeb vector field of the Sasaki structure on $M$ and we have shown that the harmonic forms in $\mathcal{H}^k(\mathcal{R}_N(C(M)))$ are basic harmonic forms, which proves the result.

\end{proof}

There are many examples of spaces which have such higher cohomology groups, and some broadly studied examples are conifolds, which are essentially spaces with K\"ahler structures and conic singularities. Conifolds are widely studied in the context of Mirror Symmetry and its applications in both mathematics and physics. The most famous of these is probably the Fermat quintic (see \cite{candelas1991pair,klebanov2000supergravity,davies2010quotients}), which has isolated conic singularities and can be equipped with a natural Calabi-Yau metric with an isolated conic (wedge) singularity. This implies that the tangent cone is a cone over a link with a Sasaki-Einstein metric. This is known as the $T^{1,1}$ Sasaki structure and the link is homeomorphic to $S^2 \times S^3$. Sasaki structures on this space are also studied in \cite{boyer2001new}. 

The $L^2$ de Rham cohomology on $C(T^{1,1})$ (with the generalized Neumann boundary conditions) has an element in degree 2 arising from the volume form of ${S}^2$ (see Lemma \ref{conic_cohomology_1}). Since complex conjugation on the K\"ahler cone generates a symmetry of the Hodge diamond and there is only one two form in the local cohomology, the corresponding harmonic two form is in $\mathcal{H}^{1,1}$ (for the de Rham complex with the natural Hodge decomposition) of $C(S^2 \times S^3)$, and by the proposition above, it must show up in the local Dolbeault cohomology as well. We will use this in computations in Subsection \ref{example_conifold}.

In order to better motivate some of the ideas on symmetries that we introduce next, we will first study an example of a singular space with different choices of domains for the Dolbeault complex. 

\subsubsection{Domains and cohomology: an illustrative example}
\label{example_cusp_singularity_preamble}

Consider the singular algebraic variety $\widehat{V}$ given by $p=ZY^2-X^3=0$ in $\mathbb{CP}^2$. 
It has an isolated singularity at $[X:Y:Z]=[0:0:1]$.
On the affine chart $Z=1$, we have the variety given by the equation $y^2=x^3$, where $x=X/Z,y=Y/Z$. Observe that this space is a topologically normal pseudomanifold as in Definition \ref{definition_normal_pseudomanifold} (see page 214 of \cite{Yaubookfoundersindex} for an exposition by Brieskorn describing why the link at the singularity is a trefoil knot which is topologically simply a circle). However it is well known that the singularity is not normal in the sense of algebraic geometry, and we will study this aspect of the space in Subsection \ref{subsection_nonnormal_examples}.

Since the Fubini-Study metric on the projective algebraic variety is conformal to the pullback of the affine metric on the chart $Z=1$ of $\mathbb{CP}^2$, let us consider the affine metric which is a wedge metric on the resolved manifold with boundary (where the pre-image of the singular point corresponding to the real blow-up is a circle).
We consider the pullback of the affine metric on this chart with the singularity using the map  
\begin{equation}
    t \rightarrow (t^2,t^3)=(x,y)
\end{equation}
where $t=re^{i\alpha}$ and $\alpha \in [0, 2\pi]$. Let us compactify the regular part of this space with the choice of boundary defining function $\rho=|x|=|t|^2$ for the boundary corresponding to the resolved manifold with boundary, where the preimage of the boundary under the blowup-map is a circle at $\rho=0$. Let us denote $x=t^2=\rho e^{i\theta}$ where $\theta=2\alpha \in [0, 4\pi]$.
The pulled back metric is 
\begin{equation}
    (dx \otimes d\overline{x})(1+(9/4)|x|)=(d\rho^2 +\rho^2 d\theta^2) (1+(9/4)|\rho|)=(d\rho^2+\rho^2 4d\alpha^2)(1+(9/4)|\rho|)
\end{equation}
which can also be written as 
\begin{equation}
    4|t|^2 (dt \otimes d\overline{t}) (1+(9/4)|t|^2)=4|r|^2 (dr^2+r^2 4d\alpha^2)(1+(9/4)|r|^2).
\end{equation}
 
\begin{remark}[Different Thom-Mather structures]
\label{Remark_different_compactifications_for_cusp_curve}

We observe that since the metric on the pseudomanifold is determined by the metric on $\widehat{V}^{reg}$, the $L^2$ cohomology of complexes are determined by the full measure set $\widehat{V}^{reg}$. However, how we choose to compactify the regular set as a stratified space determines whether a given Riemannian metric is a wedge metric or some other form of metric. These different compactifications correspond to different Thom-Mather structures. 

The defining function $\rho=|t|^2$ corresponds to a wedge structure, while if we express the metric in terms of $|t|$, we obtain 
an example of a degenerate horn metric.
It is well known that any horn metric is conformal to a wedge metric (see page 656 of \cite{LeschPeyerimhoffHornindex1998} for the conformal change that shows this), albeit with a singular conformal factor.

\end{remark}

We will continue to use the singular holomorphic coordinate $t$ in the discussion below. 
It is easy to see that $t^{-1}=\rho^{-1/2}e^{-i\alpha}$ is $L^2$ bounded with respect to the volume form of this metric on a fundamental neighbourhood $\widehat{U}_a$ of the isolated singularity $a$. Thus it is in the domain of the $(L^2\Omega^{0,q}_{\max}(U_a), \overline{\partial}_{\max})$ complex (which corresponds to choosing the maximal domain for the differential of the complex) in degree $0$, since $(\overline{\partial}+\overline{\partial}^*)t=0$. The cohomology of the complex on the truncated tangent cone with the induced product type wedge metric $d\rho^2 +\rho^2 d\theta^2$ is a Hilbert space with a Schauder basis given by $t^{-1},1,t,t^2,...$ in degree $0$ (it is easy to check that these are orthogonal on the tangent cone), and vanishes (is $\{0\}$) in higher degrees.
However the meromorphic function $t^{-1}$ is not in the VAPS domain since $t^{-1}$ is not in $\rho^{1/2}L^2$ (which is equivalent to the statement that $|t|^{-2}$ is not in $L^2$).
If we choose the VAPS domain, the cohomology of the complex $(L^2\Omega^{0,q}_N(U_a), \overline{\partial}_{VAPS})$ is the Hilbert space generated by the Schauder basis $1,t,t^2,...$ in degree $0$ (and vanishes in higher degrees).

The adjoint complex of $(L^2\Omega^{0,q}_{\max}(U_a), \overline{\partial}_{\max})$ is $(L^2\Omega^{0,n-q}_{\min}(U_a), \overline{\partial}^*_{\min})$. We observe that $\overline{dx}=2\overline{t}\overline{dt}$, while $\overline{dt}$ cannot be realized as the restriction of a regular one form on $\mathbb{C}^2$ to the variety, but can be realized as the restriction of a meromorphic one form that is $L^2$ bounded on the singular space (for instance realized as $\overline{dt}=\overline{dy}/3\overline{x}$).

In fact since the wedge cotangent bundle is spanned near the singularity by $ \{ d\rho, \rho d\theta \}$, the anti-holomorphic part of which is spanned by $\overline{dx}=2\overline{t}\overline{dt}$ (since $x=\rho e^{i\theta}$).
This shows that the local cohomology in degree $0$ for $(L^2\Omega^{0,n-q}_{\min}(U_a), \overline{\partial}^*_{\min})$ has a Schauder basis $\{\overline{t}^k\overline{dt}\}_{k \geq 1}$ for integer $k$ while it vanishes in higher degrees.
We also see that the adjoint complex of $(L^2\Omega^{0,n-q}_{N}(U_a), \overline{\partial}^*_{VAPS})$ has local cohomology with a Schauder basis $\{\overline{t}^k\overline{dt}\}_{k \geq 0}$ in degree $0$.

Consider the Hodge star operator for the metric $4|t|^2 (dt \otimes \overline{dt})$ in the neighbourhood $U_a$ (since the higher order terms in $|t|$ do not effect the integrability and the choices of domain that we discuss here). We see that $\star 2t dt= 2\overline{t} \overline{dt}$, and that $\overline{\star} {f(t)} 2{t} {dt}= \overline{f(t)} 2 \overline{t} \overline{dt}$ where by $\overline{\star}$ we denote complex conjugation of forms, composed with the Hodge star operator.

Let us now consider the complex $(L^2\Omega^{1,q}_{N}(U_a), \overline{\partial}_{VAPS})$. The discussion in the next subsection will show that the cohomology of this complex in degree $0$ has a Schauder basis $\{t^k dt\}_{k \geq 0}$ for integer $k$, while the cohomology of the complex $(L^2\Omega^{1,q}_{\min}(U_a), \overline{\partial}_{\min})$ in degree $0$ is has a Schauder basis $\{t^k dt\}_{k \geq 1}$.

\subsubsection{Serre duality for the Dolbeault complex}
\label{subsubsection_Serre_duality}

Consider the complex $\mathcal{P}(X):=(L^2\Omega^{0,q}(X;E), \overline{\partial}_{E})$ for a Hermitian bundle $E$, which includes the case of $\Omega^{0,q}(X;\Lambda^{p,0} \otimes F)=\Omega^{p,q}(X;F)$ for a fixed value of $p$ using the standard Hodge decomposition on $X$. We can identify the dual complex $\mathcal{P}^*(X)$ with $\mathcal{Q}(X):=(L^2\Omega^{n,n-q}(X;E^*), \overline{\partial}^*_{{K}_X \otimes E^*})$ using the Hodge star operator, where $X$ has real dimension $2n$. This is an instance of Serre duality and is well known on smooth complex manifolds (see Remark 5.11 of \cite{voisin2002hodge}), and we will see how it extends to the singular setting we study. We follow the exposition in Section 5.1.3 of \cite{voisin2002hodge}.

Let $E$ be a holomorphic vector bundle over a resolved complex pseudomanifold $X^{2n}$, equipped with a  Hermitian metric $h$. 
The Hermitian metric gives a conjugate-linear isomorphism $E \cong E^{\ast}$ between $E$ and its dual bundle. This induces a duality of forms valued in $E$ and $E^*$ and when $E=\Lambda^{p,0}X \otimes F$ (where $\Lambda^{p,q}X:= \Lambda^p (\prescript{w}{}{T^*X)^{1,0}} \otimes \Lambda^q (\prescript{w}{}{T^*X)^{0,1}}$)
\begin{equation}
    {\star}_E: \Omega^{p,q}(X;F) \rightarrow \Omega^{n-p,n-q}(X;F^*)
\end{equation}
which we can use to write the Hermitian $L^2$ inner product on $E$ valued forms as
\begin{equation}
    \langle \alpha, \beta \rangle_{L^2(X;E)} =\int_X \langle \alpha \wedge \overline{\star_E \beta} \rangle_{F}.
\end{equation}

The canonical bundle on the regular part $\widehat{X}^{reg}$ of the pseudomanifold $\widehat{X}$ with a complex structure is $K_X=\Lambda^{n,0}(\widehat{X}^{reg})$, which is a holomorphic line bundle. This is a very important object in complex algebraic geometry and on smooth projective varieties is called the dualizing sheaf, since $\Lambda^{n,n-q}X \otimes E$ is isomorphic to $\Lambda^{0,n-q}X \otimes K_X \otimes E$. For general singular spaces the dualizing sheaf becomes more complicated (see Part III, Section 7 of \cite{Hartshornebook}).

For sections supported on $\widehat{X}^{reg}$, we can write the adjoint operator of $\overline{\partial}_E$ as
\begin{equation}
\label{equation_dual_of_del_bar_serre}
    \overline{\partial}_E^*= (-1)^q \star_E^{-1} \circ \overline{\partial}_{K_X \otimes E^{\star}} \circ \star_E
\end{equation}
and the Laplace type operator is $\Delta_E=\overline{\partial}_E \overline{\partial}_E^*+\overline{\partial}_E^* \overline{\partial}_E$. 

For a fundamental neighbourhood $\widehat{U_a}$, we have the complex $\mathcal{P}_N(U_a):=(L^2\Omega^{0,q}(U_a;E), \overline{\partial}_{E})$ 
given by the maximal domain of $\overline{\partial}_E$ after imposing VAPS conditions at singularities (see Remark \ref{remark_boundary_conditions_for_singular_links}). The induced domain of the Laplace type operator has the boundary conditions
\begin{align*}
    \sigma(\overline{\partial}_E^*, dr)u|_{\partial {U_a}} =0 \\ 
    \sigma(\overline{\partial}_E^*, dr) \overline{\partial}_Eu|_{\partial {U_a}} =0.
\end{align*}
For the dual complex, the roles of the operator $\overline{\partial}_E$ and its dual are switched in the boundary conditions. The boundary conditions of these two complexes are intertwined by the Hodge star operator $\star_E$ which is easy to see using equation \eqref{equation_dual_of_del_bar_serre}. This shows that we have an isomorphism between the complexes $\mathcal{P}^*_N(U_a)$ and
\begin{equation}
\label{identify_dual_complex}
    \mathcal{Q}_N(U_a):=(L^2\Omega^{n,n-q}(U_a;E^*), \overline{\partial}^*_{K \otimes E^*}).
\end{equation}
induced by the Hodge star operator. In particular we have the identification of cohomology groups 
\begin{equation}
\label{identify_dual_complex_2}
    \mathcal{H}^{q}(\mathcal{P}_N(U_a)) \cong \mathcal{H}^{n-q}(\mathcal{Q}_N(U_a)).
\end{equation}
We have a similar statement for the cohomology for the complex on $\mathcal{P}(X)$ where the choices of domain are simpler (just the VAPS conditions),  
\begin{equation}
\label{equation_cohomological_Serre_duality}
    \mathcal{H}^{p,q}(X;E) \cong \mathcal{H}^{n-p,n-q}(X;E^*)^*,
\end{equation}
which complements the duality given by Proposition \ref{Kernel_equals_cohomology}.
The example in Subsection \ref{example_cusp_singularity_preamble} shows how these dualities are more complicated on singular spaces with different choices of domains. For instance, the dual of the complex 
$(L^2\Omega^{0,n-q}_{N}(U_a), \overline{\partial}^*_{VAPS})$ is the complex $(L^2\Omega^{n,q}_{N}(U_a), \overline{\partial}_{VAPS})$ and the duality in $L^2$ cohomology together with the remarks on the Hodge star operator in the subsection above show why the cohomology of the latter complex in degree $0$ is given by $\{t^k dt\}_{k \geq 0}$ for integer $k$.
In light of Proposition \ref{proposition_local_product_Lefschetz_heat}, we can anticipate how these dualities play an important role in local Lefschetz numbers, as we shall demonstrate later.

\subsection{Local Lefschetz formulas for the Dolbeault complex}
\label{subsection_Local_Lefschetz_Dolbeault_main}

We begin by studying the Lefschetz-Riemann-Roch formulas of Baum-Fulton-Quart for automorphisms of coherent sheaves on projective varieties, which are probably the closest formulas in the literature to what we develop in $L^2$ Dolbeault cohomology for stratified spaces with wedge metrics and complex structures. We then introduce our $L^2$ versions. We discuss how local cohomology groups capture different properties of various classes of singularities, which are reflected in local Lefschetz numbers. We then show how to compute Lefschetz numbers for various other complexes as discussed in the introduction of this section, extending many known results in the smooth setting to singular spaces, studying new phenomena in the process. We emphasize that the dualities we discussed in the previous subsection play a key role in understanding many of these.

\subsubsection{Lefschetz-Riemann-Roch numbers of Baum-Fulton-MacPherson}
\label{subsubsection_Baum_Fulton_Quart_method_compute}

In \cite[\S 3.3]{baum1979lefschetz}, a formula for computing Lefschetz numbers for group actions on coherent sheaves over singular varieties embedded in smooth spaces was introduced by Baum-Fulton-Quart. These are Lefschetz versions of the Riemann-Roch formulas developed by Baum-Fulton-MacPherson in \cite{baumfultonmacKtheoryRiemannRoch,baumfultonmacpheresonRiemannRoch}. These Lefchetz formulas have since been generalized for Chow groups of stacks in \cite{equivariant_intersection_edidin,localization_algebraic_Graham_Edidin}.
We will see that for certain algebraic varieties the Lefschetz numbers for the Dolbeault complex will match those of Baum-Fulton-MacPherson, but that they are different on others.
In section 3 of the expository article \cite{Baumformula81}, Baum explains how these Lefschetz numbers for the structure sheaf can be computed as the evaluation of what are essentially renormalized traces in the local ring at isolated fixed points.
We briefly present Baum's methods.

Consider $\widehat{X}$ to be a quasi-projective (reduced) variety over $\mathbb{C}$ that admits a complex-analytic embedding into a smooth projective variety $Y$. Its structure sheaf is denoted by $\mathcal{O}_{\widehat{X}}$. 
Let $T_f$ be a geometric endomorphism corresponding to a holomorphic automorphism $f$ of $\widehat{X}$, that extends to one on $Y$. We demand that $f$ is of finite order and has only a finite number of isolated fixed points. The \textit{\textbf{local ring of $\widehat{X}$ at $a$}}, $\mathcal{O}_a$, is the ring of germs of holomorphic functions at $a$. 
Given an isolated fixed point $a$, there is an induced automorphism 
\begin{equation}
\label{induced_auto_on_local_ring}
    T_f=f^*: \mathcal{O}_a \rightarrow \mathcal{O}_a
\end{equation} 
of $\mathcal{O}_a$ where Baum uses the assumption that it is a finite order automorphism on the structure sheaf to simplify the action of the geometric endomorphism. In $\mathcal{O}_a$, let $I_a$ be the ideal generated by all elements of the form $u-f^*u$. Then for $r$ a non-negative integer $I_a^r / I_a^{r+1}$ is a finite dimensional vector space over $\mathbb{C}$. The automorphism \eqref{induced_auto_on_local_ring} maps $I_a^r$ to itself, and so induces a linear transformation 
\begin{equation}
\label{induced_auto_on_local_ring_each_filtration}
    T_{f,r}: I_a^r / I_a^{r+1} \rightarrow I_a^r / I_a^{r+1}.
\end{equation} 
With $t$ an indeterminate, Baum points out that the power series
\begin{equation}
\label{Baum_power_series}
    Q_a(t)=\sum_{r=0}^{\infty} \text{Trace}(T_{f,r})t^r,
\end{equation}
converges to a rational function in $t$ with no pole at $t=1$ and the local Lefschetz number is defined to be $Q_a(1)$.
Consider the simple case where at a fixed point, the variety is locally described by the vanishing of a reduced homogeneous polynomial $g(x_1,...,x_n)$ on $\mathbb{C}^n$ with coordinates $x_1,...,x_n$. 
Following Baum we take $R$, the local ring of the fixed point on $Y$, that is the ring of germs of holomorphic functions.
The local ring of the singular space is the quotient ring $R/(g)$ where $(g)$ is the ideal generated by $g$. Baum observes (see case 2 in \cite[\S 2]{Baumformula81}) that this corresponds to the ring of germs of holomorphic functions at the fixed point on $X$.
If there is an automorphism $f$, such that $f^*g=bg$, and $f^*x_i=a_i$ then the regularized trace of the induced action of $f$ over $R/(g)$ is the difference of $\Pi_{i=1}^n (1-a_i)^{-1}$, and $(b) (\Pi_{i=1}^n (1-a_i))^{-1}$.
This gives the formula for the local Lefschetz number in \cite[\S 3.3]{baum1979lefschetz} for the case discussed. If an affine neighbourhood of the fixed point of a reduced singular variety is defined by the vanishing of several homogeneous polynomials (a local complete intersection), a similar argument gives the general formula.

The results in \cite{baum1979lefschetz} are for coherent sheaves of $\mathcal{O}_{\widehat{X}}$ modules on quasi-projective schemes. Given a locally free coherent sheaf $\mathcal{E}$ on $X$, the corollary on page 196 of that article shows that the local Lefschetz number at an isolated fixed point $a$ for an automorphism $f:\widehat{X} \rightarrow \widehat{X}$ such that $f^n=Id$ for some integer $n$ (inducing a geometric endomorphism $T_f$), factors as 
\begin{equation}
\label{equation_locally_free_sheaf_factorization}
    [Tr f^*(\mathcal{E}_a)] L_a \widehat{X}
\end{equation}
where $Tr f^*(\mathcal{E}_a)$ is the trace over the fiber of the locally free sheaf at $a$, and $L_a \widehat{X}$ is the local Lefschetz number of the structure sheaf $\mathcal{O}_{\widehat{X}}$ at $a$.

\subsubsection{Lefschetz numbers in $L^2$ cohomology}
\label{subsection_Lefschetz_L2_cohomology_development}

Let us consider the local Lefschetz numbers for the twisted complex $\mathcal{P}(X)=(L^2\Omega^{p,q}(X;F), \overline{\partial}_F)$ for fixed values of $p$, which is the same as $(L^2\Omega^{0,q}(X;E=\Omega^{p,0} \otimes F), \overline{\partial}_E)$. We do this restricted to fundamental neighbourhoods of the fixed points but, in the case of simple fixed points $a$, this can be taken to be the tangent cone as we have discussed earlier. The formula for local Lefschetz numbers given in \eqref{Local_Lefschetz_numbers_definition}, applied to the Dolbeault complex on a fundamental neighbourhood with $\overline{\partial}_E$-Neumann boundary conditions reduces to the following.
\begin{theorem}
\label{theorem_renormalized_McKean_Singer_Dolbeault}
Let $\widehat{X}^{2n}$ be a stratified pseudomanifold and $\widehat{f}$ a self map of this space. Let $a$ be an isolated simple fixed point of $\widehat{f}$ with a fundamental neighbourhood $U_a \cong \mathbb{D}^k \times C(Z)$ with a wedge metric and a compatible complex structure. Given a Hermitian bundle $E$ on $U_a$, we require that the map $f$ has a lift to a bundle map so that there is an associated geometric endomorphism $T_f$ for the complex $\mathcal{P}_B(U_a)$.
Then the local $L^2$ Dolbeault Lefschetz number can be expressed as 
\begin{equation}
L(\mathcal{P}_B(U_a), f, a) = L(\mathcal{P}_N(U_1), T_{f_1},p) (-1)^m L(\mathcal{P}^*_N(U_2), T^{P^*}_{f_2^{-1}},p)
\end{equation}
as in Proposition \ref{proposition_local_product_Lefschetz_heat}. When the fixed points are strictly attracting on $U_1$ and strictly expanding restricted to $U_2$, the renormalization is not needed.
\end{theorem}

\begin{proof}
This follows from Theorem \ref{Corollary_Polynomial_Lefschetz_fixed_point_theorem}. Theorem \ref{theorem_renormalized_McKean_Singer} shows that the renormalization is not necessary for the strictly attracting case on $U_1$, and the duality used in Proposition \ref{proposition_local_product_Lefschetz_heat} shows that it is not needed for strictly expanding maps on $U_2$.
\end{proof}

\begin{remark}[Non-simple fixed points]
\label{remark_non_simple_generalize_Dolbeault}
We make the observation that we do not need that the fixed point is simple if there is a neighbourhood of the fixed point where there are no other fixed points, where the self map is either strictly attracting or strictly expanding, at least under the assumption that there is a neighbourhood which is isometric to a metric cone. This is because we can then use the localization result in Proposition \ref{Gluing_heat_kernels} which does not need the simplicity assumption, in place of Theorem \ref{theorem_localization_model_metric_cone}.
In \cite{toledo1975holomorphic}, the holomorphic Lefschetz fixed point formula on smooth manifolds was extended for non-simple isolated fixed points in the smooth case. Since our computations are now in cohomology, we can simply take a fundamental neighbourhood and compute the Lefschetz number, similar to de Rham cohomology.

\end{remark}

In \cite{kytmanov2004holomorphic}, the $\overline{\partial}$ complex is studied on smooth complex manifolds with boundary and a renormalized Lefschetz fixed point theorem is proven.
They introduce a principal value regularization/ renormalization for the heat trace which gives the Atiyah-Bott formula at interior fixed points and they also present a boundary formula for fixed points on the boundary. In the work of \cite{donnelly1986fixed}, the authors present what is in their words a \textit{holomorphic Lefschetz formula for the Bergman metric}. It is easy to see that their formula in Theorem 1.1 of that article is precisely the sum of the local holomorphic Lefschetz numbers at isolated fixed points for the complex $(L^2\Omega^{0,n-q}, \overline{\partial}^*)$. Proposition 2.9 of that article can be modified to see this, using the methods we have introduced in this article.

With our methods, when the self maps are strictly attracting or strictly expanding, we do not require renormalization. For example, for self maps arising from algebraic torus actions, this is the case for any self map corresponding to an element $(\lambda_1,...,\lambda_N) \in (\mathbb{C}^*)^N$ such that $|\lambda_i| \neq 1$ and the fixed points are isolated. This yields a self map where the fixed points are strictly attracting/expanding yielding trace class geometric endomorphisms at fundamental neighbourhoods of fixed points.

We prove the following duality result for local Lefschetz numbers.
\begin{proposition}
\label{Proposition_duality_complex_conjugation_for_local_Lefschetz_numbers}
Let $\widehat{X}$ be a stratified pseudomanifold of dimension $2n$ with a wedge metric and complex structure and let $E$ be a Hermitian bundle. Let $f$ be a holomorphic isometry with isolated fixed points, including one at $a \in X$ with a fundamental neighbourhood $U_a$. Let $\mathcal{P}(X)=(L^2\Omega^{0,q}_{VAPS}(X;E),\overline{\partial}_E)$, so that $\mathcal{P}^*(X)=(L^2\Omega^{0,n-q}_{VAPS}(X;E),\overline{\partial}^*_E)$ and let $\mathcal{Q}(X)=(L^2\Omega^{n,q}_{VAPS}(X;E^*), \overline{\partial}_{E^*})$. 
We assume that there are geometric endomorphisms $T_f$ associated to $f$ on $\mathcal{P}_B(U_a)$, and $T^Q_{f^{-1}}$ associated to $f^{-1}$ on $\mathcal{Q}_B(U_a)$. Then
\begin{equation}
    L(\mathcal{P}_B(U_a),T_f)=(-1)^n L((\mathcal{P}_B(U_a))^*,T^*_f) = (-1)^n L(\mathcal{Q}_B(U_a),T^Q_{f^{-1}}).
\end{equation}
\end{proposition}

\begin{proof}
The first equality follows from Proposition \ref{proposition_Lefschetz_on_adjoint}. 
Recall from equation \eqref{equation_dual_complex_adjoint_geometric_endo_by_inverse} that the adjoint endomorphism $T^*_f$ on $(\mathcal{P}_B(U_a))^*$ gives the same Lefschetz number as that associated to $f^{-1}$ on $\mathcal{P}^*_{B'}(U_a)$ where $B'$ are the boundary conditions specified by 
Proposition \ref{proposition_local_product_Lefschetz_heat} for $f^{-1}$. Moreover since we consider $f$ to be a holomorphic isometry, the pullback of sections $f^*$ commutes with the Hodge star operator. 
The local cohomology groups of $\mathcal{P}^*_{B'}(U_a)$ and $\mathcal{Q}_B(U_a)$ are identified by the dualities described in subsection \ref{subsubsection_Serre_duality} by the Hodge star operator and the duality of $E$ and $E^*$. Thus the second equality follows since the trace of $T^Q_{f^{-1}}$ on the local cohomology group of $\mathcal{Q}_B(U_a)$ is the same as the trace of $T^*_f$ on $(\mathcal{P}_B(U_a))^*$.
\end{proof}

Consider the twisted $L^2$ Dolbeault complex $\mathcal{P}^p_B(U_a;E)=(L^2\Omega^{p,q}(E), \overline{\partial}_E)$. The local $L^2$ cohomology is a module over the local holomorphic functions by the Leibniz rule. 
A majority of the examples we study in this article will be Hilbert complexes over projective varities and we will be able to express the local $L^2$ cohomology as a module over an infinite dimensional Hilbert space obtained as the completion of a ring of polynomials with respect to an inner product. We illustrate with the examples in Subsection \ref{example_cusp_singularity_preamble}.
We saw that the $L^2$ cohomology of the complex $(L^2\Omega^{0,q}_N(U_a), \overline{\partial}_{VAPS})$ in that subsection has a Schauder basis $1,t,t^2,...$ in degree $0$ (and vanishes in higher degrees). The corresponding Hilbert space is the Hilbert completion of the ring of polynomials $\mathbb{C}[t]$, which we shall denote as $\overline{\mathbb{C}[t]}$. We also saw that the adjoint complex of $(L^2\Omega^{0,q}_N(U_a), \overline{\partial}_{VAPS})$ has local cohomology corresponding to the Schauder basis $\{\overline{t}^k\overline{dt}\}_{k \geq 0}$ in degree $0$. We shall denote the corresponding Hilbert space as $\overline{R[\overline{dt}]}$, which is the Hilbert space completion of the module generated by $\overline{dt}$ over the ring $R=\mathbb{C}[\overline{t}]$.

\subsubsection{Classes of singularities and local cohomology}
\label{subsection_classes_of_singularities}

A dominant theme in modern geometry is to study spaces by studying the functions on them. On stratified pseudomanifolds with wedge metrics, there are different choices of domains for the complexes that we study, which coalesce in the case of smooth manifolds and more generally on spaces where the complexes satisfy the geometric Witt condition. Canonical definitions such as the maximal, minimal or VAPS domains for operators can be used to study cohomological invariants and local cohomology of these complexes are geometric invariants which detect topological as well as geometric differences in singularities.
In algebraic geometry, the fact that there is an embedding into $\mathbb{CP}^N$ for some $N$ is used to define sheafs on varieties by restriction, and the Lefschetz formulas of \cite{baum1979lefschetz} for such sheafs allows one to compute such invariants easily.

In the smooth setting it is well known that the $L^2$ analytic and algebraic versions of these invariants coalesce, and it is natural to compare them in different singular settings. Since localization has been developed more broadly in the algebraic setting, one can use those techniques to understand analytic versions.
There are many classifications of singularities considered in algebraic geometry. In this section we discuss some relations between the analytic and algebraic local cohomology groups in relation to some classes of singularities. We begin with the following definitions from \cite[\S 5]{KollarBull87} to which we refer the reader to more details about topics relevant to our discussion here.

\begin{definition}
\label{definition_normal_algebraic_variety}
Let $V \subset \mathbb{C}^n$ be an algebraic variety and $v \in V$. 
A rational function $f$ is called \textit{regular at $v \in V$} if there are polynomials $g$ and $h$ such that $f=g/h$ and $h(v) \neq 0$. A rational function is called \textit{regular} on $V$ if it is regular at each point.
The variety $V$ is said to be \textit{normal at $v \in V$} if every rational function on $\mathbb{C}^n$ bounded in some neighborhood of $v$ is regular at $V$, and $V$ is called \textit{normal} if it is normal at every point. In
particular if $V$ is normal at $v$, then a rational function is regular at $v$ iff it is continuous at $v$.
\end{definition}

It is observed in \cite[\S 4]{goresky1980intersection} that normal complex algebraic varieties are normal pseudomanifolds (see Definition \ref{definition_normal_pseudomanifold}), though the converse is not true in general.

An equivalent definition of normal that is more widely used in algebraic geometry (and for fields other than $\mathbb{C}$) is that the local ring at the point is an integrally closed domain. We refer the reader to Exercise 3.17 in Chapter 1 of \cite{Hartshornebook} for a definition. We will explore Lefschetz numbers for group actions on the non-normal examples in parts $b$ ($z^2=xy$) and $c$ ($y^2=x^3$) of exercise 3.17 in Subsection \ref{subsection_Computations_for_singular_examples}. Part $e$ of the said exercise shows that for a given affine variety $V$, there is an \textit{\textbf{algebraic normalization}} $W$, which is an affine variety that is birational to $V$. We refer the reader to Section 4 of part I of \cite{Hartshornebook} for more details on the following definitions.

\begin{definition}
\label{definition_birational_et_al}
A \textit{rational map} $\phi: V_1 \dashrightarrow V_2$ between algebraic varieties is a morphism from a nonempty Zariski open subset $U \subset V_1$ (in particular dense in $V_1$) to $V_2$. A \textit{birational map} is a rational map $\phi$ which has an inverse which is a rational map. We call $V$ a \textit{rational variety} if it is birational to $\mathbb{C}^{N}$ (or equivalently to $\mathbb{CP}^{N}$) for some $N$.
\end{definition}

Another notion that is relevant in comparing the analytic and algebraic invariants is that of a rational singularity. We refer to \cite[\S 6]{AnalyticToddBeiPiazza} and references therein for the following. We say that a complex space $\widehat{X}$ has \textit{rational singularities} if $\widehat{X}$ is normal, and there exists a resolution (in the sense of Hironaka) $\pi:M \rightarrow \widehat{X}$ such that $R^k\pi_*\mathcal{O}_M =0$ for all $k>0$, i.e., the pushforward of the structure sheaf of $M$ to $\widehat{X}$ under $\pi$ does not have higher (global) cohomology. We note that the normal condition implies that $\mathcal{O}_{\widehat{X}}$ can be identified with $\pi_*\mathcal{O}_M$. 

In the case of complex projective varieties, normal toric singularities, log-terminal, canonical singularities are rational. For surfaces, Du Val singularities are rational. The quadric surface that we study in Subsection \ref{subsection_example_nodal_surface} is a normal toric surface with a Du Val singularity. 

There are many results proven in \cite{AnalyticToddBeiPiazza} at the level of Todd classes in K theory that show similarities and differences of the Baum-Fulton-MacPherson Todd classes \cite{baumfultonmacKtheoryRiemannRoch} and analytic Todd classes corresponding to various choices of domains and under various assumptions on singularities and their resolutions. Naturally these imply similarities and differences for the Baum-Fulton-Quart Lefschetz numbers and our $L^2$ holomorphic Lefschetz numbers.

Theorem 5.4 of \cite{AnalyticToddBeiPiazza} shows that a complex projective surface with isolated singularities which does not have rational singularities will not have a closed extension of the Dolbeault complex such that the Euler characteristic of the complex is equal to the arithmetic genus of Baum-Fulton-MacPherson. If in the same setting there is a finite group acting by automorphisms, then the Baum-Fulton-Quart Lefschetz numbers must be different from our $L^2$ versions.

In Proposition 6.1 of \cite{AnalyticToddBeiPiazza}, Piazza and Bei show that the complex of fine sheaves corresponding to the sheafification of the Dolbeault complex with maximal domain on $\widehat{X}^{reg}$ is a resolution of the structure sheaf of a compact irreducible complex variety $\widehat{X}$ iff $\widehat{X}$ has rational singularities. They also noted that this is equivalent to the existence of a resolution of the structure sheaf by a complex built from the minimal $L^2$ Dolbeault complex \cite{RuppenthalSerre2018}. They go on to show identifications of analytic Todd classes with those of Baum-Fulton-MacPherson in \cite{baumfultonmacKtheoryRiemannRoch} under the additional assumption of isolated singularities.

We prove the following proposition to demonstrate that our results can be used to prove equivariant versions of such results, in a restricted setting for simplicity. Note that we do not assume that our singularities are isolated.

\begin{proposition}
\label{Proposition_BFM_equals_L2}
Let $\widehat{X}$ be a projective variety where the Fubini-Study metric can be realized as a wedge metric on $X$ (see Remark \ref{Remark_different_compactifications_for_cusp_curve}), equipped with an algebraic torus action with generically isolated fixed points, with an element of the torus inducing a self map $f$ with isolated fixed points and such that $f^n=Id$. Let there be an embedding of $\widehat{X}$ into a smooth variety $Y$ such that $f$ extends to an automorphism of $Y$ with isolated fixed points. We denote the Dolbeault complex with VAPS domain by $\mathcal{P}(X)=(L^2\Omega^{0,\cdot}(X), \overline{\partial}_{VAPS})$, and the geometric endomorphism associated to $f$ as $T_f$.
If $\widehat{X}$ has normal singularities, and the local cohomology of $\mathcal{P}_N(U_a)$ for any fixed point $a$ of $f$ is equal to the holomorphic functions on $U_a$, with no higher degree cohomology, then 
the local Lefschetz numbers $L(\mathcal{P}(U_a), T_f)$ are equal to the local Lefschetz numbers of Baum-Fulton-Quart (see Subsection \ref{subsubsection_Baum_Fulton_Quart_method_compute}).    
\end{proposition}

\begin{proof}
We showed in the discussion below equation \eqref{Baum_power_series} that the Baum-Fulton-Quart local holomorphic Lefschetz numbers can be derived by taking renormalized traces over the local ring in this setting. 

Since we have an algebraic torus action on $X$, we can pick elements $(\lambda_1,...,\lambda_N) \in (\mathbb{C}^*)^N$ such that $|\lambda_i| \neq 1$ acting on $X$ so that we do not have to renormalize the local Lefschetz numbers (as per the discussion above Proposition \ref{Proposition_duality_complex_conjugation_for_local_Lefschetz_numbers}). For an attracting fixed point, since we assume that the local cohomology is exactly the holomorphic functions, we can pick a basis as indicated in the discussion below equation \eqref{Baum_power_series} and compute the trace without renormalization. That discussion shows that the local Lefschetz number at each fixed point is a rational function in the variables $\lambda_i$ of the torus. The sum of these local Lefschetz numbers can be written as a finite sum of products of powers of $\lambda_i$ by the Lefschetz fixed point theorem since the global cohomology is finite. Thus the global Lefschetz number is a polynomial in the $\lambda_i$ variables that can be evaluated on any element of the torus.

For local complete intersections, it is easy to see that this recovers the local Lefschetz-Riemann-Roch number in \cite[\S 3.3]{baum1979lefschetz}.

\end{proof}

We observe that since $\widehat{X}$ is a normal toric variety, its singularities are rational and our result is compatible with Propositions 5.4 and 6.1 of \cite{AnalyticToddBeiPiazza}. Results such as the above proposition can provide a useful tool to probe similarities and differences in algebraic and analytic invariants in different settings (see Remark \ref{remark_Bott_residue_quadric}).

It follows from Definition \ref{definition_normal_algebraic_variety} that the regular functions on an affine chart are the local holomorphic functions for normal algebraic varieties. The set of regular functions on an affine chart is equal to the \textit{local ring} at the origin of the chart. 
Key differences in the formulas of Baum-Fulton-Quart in singular cohomology as opposed to ours stem from the fact that the local $L^2$ cohomology at fixed points of non-normal algebraic varieties include holomorphic functions in the integral closure of the local ring, that are not in the local ring.

The notion of a \textit{good categorical quotient} (see for instance Definition 5.05 of \cite{cox2011toric}) also relates holomorphic sections and regular sections on singular spaces. For a good categorical quotient corresponding to the morphism $\pi: X \rightarrow Y$, for any $U \subset Y$ that is open (in the Zariski topology), the natural map $\mathcal{O}_Y(U) \rightarrow \mathcal{O}_X(\pi^{-1}(U))$ induces an isomorphism $\mathcal{O}_Y(U) \simeq \mathcal{O}_X(\pi^{-1}(U))^G$. The non-normal algebraic varieties that we study can be realized as symplectic quotients which do not satisfy this condition, hence the discrepancies of the holomorphic sections on the embedded space and those on the singular space of interest. In most moduli space problems, the singularities are nice enough that there are no such discrepancies. For instance in section 3.2.1 of \cite{NekrasovABCDinsta}, resolutions of spaces are used to compute certain Lefschetz numbers which appear in the guise of character formulas, and the independence of the resolution of the K\"ahler moduli is used to see that the characters are well defined on the singular space (also see Remark \ref{remark_yau_nekrasov_conifold}).

In the case of non-normal algebraic varieties, there are instances where for all choices of domains for the Dolbeault complex the analytic Todd classes of \cite{AnalyticToddBeiPiazza} are different from those of Baum-Fulton-MacPherson, and we will study such examples in Subsection \ref{subsection_nonnormal_examples}.

In \cite{LottHilbertconplex}, a different Hilbert complex is constructed to match the Todd classes of Baum-Fulton-MacPherson, which accounts for elements in the higher global cohomology groups of the complex which are missing in the global $L^2$ cohomology. Indeed, in equation (1.1) of that paper, the algebraic Riemann-Roch theorem for complex curves is reviewed. It is pointed out that the minimal domain for the Dolbeault complex with the Fubini Study metric on curves gives the minimal index (by \cite{Bruningalgebraiccurves90}) while the index of Baum-Fulton-MacPherson-Quart increases as the domain of the $\overline{\partial}$ operator becomes larger. 
That is, if there are domains $D_1 \subset D_2$ for the Dolbeault complex on an algebraic curve, then the Euler characteristic of the complex with the domain $D_2$ is greater than or equal to that with the domain $D_1$.
We will see this in the case of the cusp curve in Subsection \ref{subsection_nonnormal_examples} as we study Lefschetz numbers for different domains. 

\subsubsection{Smooth holomorphic Lefschetz formulae}

At an isolated fixed point $p$ of a map $f$ in the regular part of $\widehat{X}$, we know that the local Lefschetz number for the Dolbeault complex $\mathcal{P}(X)$, $L(\mathcal{P}(X),T_f,p)$, is given by the Atiyah-Bott (Woods Hole) formula (see Section 4 of \cite{AtiyahBott2}). Now we verify that this is equal to $L(\mathcal{P}(U_p),T_f)$ for a fundamental neighbourhood $U_p$ of the fixed point. In the smooth case, the tangent cone is just the tangent space $\mathbb{C}^{n}$, and the truncated tangent cone is a ball for which the local $L^2$ cohomology is trivial for $q>0$.

We will first study the case where $n=1$ for simplicity which is the case of the disc in the complex plane, where the fixed point is the origin. For $p=0$, the $\overline{\partial}$ operator has an infinite dimensional null space which is precisely the space of holomorphic functions. This has a Schauder basis $\{z^k\}_{k \in \mathbb{N} \cup \{0\}}$. We denote the complex for holomorphic $p$ forms as $\mathcal{P}^p=(L^2\Omega^{p,\cdot},\overline{\partial})$, obtained by twisting $\Omega^{0,q}$ by the bundle $\Lambda^pT^{(1,0)}_{\mathbb{D}}$.
For $p=1$, the cohomology is generated by $dz$ (as a module over the algebra of holomorphic functions).
For a holomorphic map $f$ which is expanding, one obtains the meromorphic weighted trace function (see  Definition \ref{definition_renormalized_trace})
\begin{equation}
    L(\mathcal{P}^{p=0}_N(\mathbb{D}), T_f, 0)(s)= \sum_{k=0}^{\infty} e^{-sk}  \frac{\langle f^*({z^k}), z^k \rangle}{\langle z^k, z^k \rangle}  =\sum_{k=0}^{\infty} e^{-sk} {df_0'}^k
\end{equation}
where $df_0'$ is the differential of the map $f$ restricted to the holomorphic tangent space at the origin. If $s$ is large enough so that $e^{-s}df_0'$ has norm less than $1$, then this series is summable and results in the function
\begin{equation*}
     L(\mathcal{P}^{p=0}_N(\mathbb{D}), T_f, 0)(s)= \frac{1}{1-e^{-s}df_0'}
\end{equation*}
which admits a meromorphic continuation to $\mathbb{C}$. It is clear that the poles of the meromorphic function are given by $Log(df_0')+ i2\pi m$ for $m \in \mathbb{Z}$ and they are simple poles (where by $Log$ we mean any branch of the logarithm). The presence of a pole at $s=0$ is a degenerate condition corresponding to a non-simple fixed point and it occurs precisely when the linearized map induced on the tangent cone has a non-isolated fixed point. In particular, for attracting fixed points, for which $Log(|df_0'|)$ is negative, the function is holomorphic on a right half plane including the origin.

In the case where $p=1$, a Schauder basis for the null space is obtained by multiplying $dz$ by a basis for the holomorphic functions. We get
\begin{equation}
    L(\mathcal{P}^{p=1}_N(\mathbb{D}), T_f, 0)(s)= \sum_{k=0}^{\infty} e^{-sk} \frac{\langle f^*({z^k}dz), z^k dz \rangle}{\langle z^k dz, z^k dz \rangle} =\sum_{k=0}^{\infty} (e^{-sk} {df_0'}^k) df_0'
\end{equation}
where the additional factor of $df_0'$ is a result of $f^*(dz)=d(f^*z)=\partial(f^*z)$. The meromorphic function is now given by 
\begin{equation}
\label{equation_Sumith_1}
     L(\mathcal{P}^{1}_N(\mathbb{D}), T_f, 0)(s)= \frac{df_0'}{1-e^{-s}df_0'}
\end{equation}
which satisfies the same properties as discussed above. Such formulas actually hold for expanding fixed points as well, for which one could take the dual complex and get formulas without renormalization. Observe that this formula has the same form as that in equation \eqref{equation_locally_free_sheaf_factorization}. 
The case for general $n$ and general $p$ is similar and can be recovered by considering polydisc neighbourhoods of the fixed point and that the fiber over the fixed point for the sheaf of $p$ forms is generated by the linearly independent $p$-fold wedge products of $\{dz_1,..., dz_n\}$.

This recovers the theorem of Atiyah and Bott for the local Lefschetz numbers for the Dolbeault complex in the smooth case, given by
\begin{equation}
\label{woods_hole_formula}
L(\mathcal{P}^{p}_N(U_a), T_f, a) = \frac{tr(\lambda^{p}df_a')}{det(Id-df_a')} 
\end{equation}
where $\lambda^{p}df_a'$ is the map induced on the vector space $\Lambda^pT^{(1,0)}_a{U_a}$ by the geometric endomorphism, the trace of which yields the $p$-th term in the characteristic polynomial for $df_a'$.
This is also known as the Woods Hole formula (see \cite{tu2015genesis} for a brief history).

\subsubsection{Spin$^{\mathbb{C}}$ Dirac complexes}
\label{subsection_almost_complex_spin_c_Dirac}

Spin$^{\mathbb{C}}$ structures play an important role in many aspects of geometry and topology, and the spin$^{\mathbb{C}}$ Dirac operator gives a link to study some of these aspects. We refer the reader to \cite{duistermaat2013heat} for a comprehensive introduction to the operator on smooth manifolds and orbifolds.

It is well known that for an almost complex manifold of dimension $2n$ there is a corresponding spin$^{\mathbb{C}}$ structure and spin$^{\mathbb{C}}$ Dirac operator, and the complex spinor bundle $\mathcal{S}$ is canonically identified with $\oplus_q \Lambda^{0,q}X$, where $\oplus_{p,q} \Lambda^{p,q}X$ is the splitting of the exterior algebra of $T^*X$ in $(p,q)$ types (see, e.g., \cite{epstein2006subelliptic}).
Unless the almost complex structure is integrable, the square of the spin$^{\mathbb{C}}$ Dirac operator does not preserve the form degree. 
Using the identification of $\mathcal{S}$ with $\oplus_q \Lambda^{0,q}X$ (or more generally with $\oplus_q \Lambda^{p,q}X$ after twisting with $E=\Lambda^{p,0}$), we can try and generalize the Lefschetz fixed point theorem when there is some integrability of the almost complex structure near fixed points.

Suppose $X$ is a resolution of a pseudomanifold with a spin$^{\mathbb{C}}$ structure.
Given a compatible almost complex structure in a neighbourhood of a fixed point, there is an induced almost complex structure $J_p$ at the tangent cone at the fixed point (realized by freezing coefficients at $p$ and extending homogeneously) which is homogeneous with respect to dilations on the infinite cone.
If $J_p$ is integrable then we are in the setting of the holomorphic Lefschetz fixed point theorem locally.

Let $\mathcal{R}(X)=(L^2(X;\mathcal{S} \otimes E), D)$ be the Dirac complex (see Definition \ref{equation_two_term_complex}) for a twisted spin$^{\mathbb{C}}$ Dirac operator which satisfies the Witt condition. Let $f$ be a self map preserving the spin$^{\mathbb{C}}$ structure, with isolated fixed points, and inducing a geometric endomorphism $T_f$ on the complex.

Let $U=U_a$ be the truncated tangent cone of a non-expanding fixed point $a$, where the spin$^{\mathbb{C}}$ structure is defined by an almost complex structure which is integrable. Then restricted to $U_a$, we can identify the complex with the Dolbeault complex with the $\mathbb{Z}_2$ grading corresponding to the even and odd degree forms, that is $\mathcal{R}_N(U_a)=(L^2\Omega^{0,\cdot}(U;E), D_U)$.

For more general fixed points with a product decomposition as in Proposition \ref{proposition_local_product_Lefschetz_heat}, we can use the product complex $\mathcal{R}_B(U_a)$. However most geometric endomorphisms of spin$^{\mathbb{C}}$ Dirac complexes that are studied widely arise from isometric actions of Lie groups, in which case 
$\mathcal{R}_B(U_a)=\mathcal{R}_N(U_a)$.
Since we have suitable domains for the local Hilbert complex, we are able to define local cohomology groups, and thus define local Lefschetz numbers.
\begin{theorem}
In the setting described above, we have that
\begin{equation}
    L(\mathcal{R}(X),T_f)=\sum_{a \in Fix(f)} L(\mathcal{R}_B(U_a),T_f).
\end{equation}
\end{theorem}

\begin{proof}
This follows from the fact that the local complex can be identified with a twisted Dolbeault complex, whence the result follows from Theorem \ref{Corollary_Polynomial_Lefschetz_fixed_point_theorem}.
\end{proof}

\begin{remark}
\label{Remark_can_generalize_to_almost_complex}
On 4 manifolds, Proposition 6.3.1 of \cite{SurgerycontactOzbagciStipsicz04} shows that a spin$^{\mathbb{C}}$ structure determines an almost complex structure away from finitely many isolated points (c.f. \cite{LeBrunTwistorsSDSPIN2021}).
Some of the other invariants that we compute in this section can be generalized to the setting where there are spin$^{\mathbb{C}}$ structures. In particular, formulas for the signature 
for almost complex manifolds are well known. In this article we will stick to the setting of complex structures for simplicity.
\end{remark}

\subsubsection{Spin Dirac complex with complex structures.}
\label{subsubsection_spin_dirac_kahler}

In the case where we have a pseudomanifold with a spin structure, and when fundamental neighbourhoods carry a complex structure, we can use the following identification of Hilbert complexes. It is well known that a spin structure on $\widehat{X}^{reg}$ with a complex structure corresponds to a square root of the canonical bundle $K$, if it exists. More precisely we have a line bundle $L$ on $\widehat{X}^{reg}$ such that $L^{\otimes 2}=K$, and a spin Dirac operator with coefficients twisted by a Hermitian bundle $F$,
\begin{equation}
\label{identified_complex_spin}
    D=\sqrt{2}(\overline{\partial}_E+\overline{\partial}_E^*) : \bigoplus_{q \text{  even}} (\Omega^{0,q}(X) \otimes F \otimes L) \rightarrow \bigoplus_{q\text{  odd}} (\Omega^{0,q}(X) \otimes F \otimes L).
\end{equation}
where $E=F \otimes L$ (see for instance Theorem 2.2 of \cite{hitchin1974harmonic}, Section 3.4 of \cite{friedrich2000dirac}).
We denote the associated elliptic complexes as $\mathcal{P'}=(L^2\Omega^{0,q}(X)\otimes F, \overline{\partial}_F)$, $\mathcal{P}=(L^2\Omega^{0,q}(E)(X), \overline{\partial}_E)$, and the associated Dirac complexes as $\mathcal{R'}$ and $\mathcal{R}$.

\begin{remark}[Self duality of spin complex]
\label{remark_self_dual_spin_serre}
In particular, for $E$ a self dual bundle ($E=E^*$), we see that the complex $\mathcal{P}$ is a \textit{\textbf{self dual complex under Serre duality}} since $L^{-1} \otimes L^{\otimes 2}= L$. In particular, this holds for the spin Dirac operator when $E$ is trivial.
\end{remark}

Let $a$ be an isolated fixed point of a map $f: U_a \rightarrow U_a$ of a fundamental neighbourhood $U_a$ of a stratified pseudomanifold $X$, such that the resolved space has a spin structure and a wedge metric. Let us assume that $U_a$ has a K\"ahler structure and a Hermitian bundle $F$ such that we can write the associated Dirac complex as $\mathcal{R}(U_a)$ on $U_a$, in the notation of the discussion above. 
We require that the map $f$ has a lift to a bundle map on the line bundle $L(=\sqrt{K})$ on $X^{reg}$ preserving the Hermitian structure, and denote the associated geometric endomorphism by $T_f$.
Then the local spin Lefschetz number is given by the holomorphic Lefschetz number
$L(\mathcal{P}_B(U_a), f, a)$. 

\begin{example}[Spin number on a disk]
\label{Example_Spin_number}
Atiyah and Bott proved (in Theorem 8.35 of \cite{AtiyahBott2}) that the local Lefschetz number (which they called the spin number) for the spin Dirac operator at the smooth fixed point at the origin for a rotation by an angle $\alpha$
of the disc $\mathbb{D}^2=U_0$, with the trivial spin structure, is
\begin{equation}
\label{Baseema}
    L(\mathcal{P}_N(U_0), T_{\widehat{f}}, 0)= \pm \frac{i}{2}cosec(\alpha /2).
\end{equation}
The sign of this number can be determined using the structure of the spin bundle and the choice of lift $\widehat{f}$ of the self map $f$ to an automorphism of the spin bundle.
\end{example}

The derivation of this formula in \cite{AtiyahBott2} is involved, and we provide an easier derivation for the smooth case now.
First we note that the spin number above (up to the sign) can be written as the evaluation at $s=0$ of the weighted trace functions
\begin{equation}
\label{equation_spin_trace}
    \sum_{k=0}^\infty e^{(2k+1)(-s+i\alpha/2)} = \frac{e^{-s+i\alpha/2}}{(1-e^{-s+i\alpha})}
\end{equation}
as in Definition \ref{definition_renormalized_trace}. This is simply the holomorphic Lefschetz number up to the factor of $e^{i\alpha/2}$ in the numerator.
Since the tangent space of a fixed point on any even dimensional manifold of dimension $2n$ can be identified with $\mathbb{C}^n$, we can use the identification of the Spin Dirac complex with the twisted (by a square-root of the canonical bundle) Dolbeault complex from Subsection \ref{subsubsection_spin_dirac_kahler}. A square root of the canonical bundle is simply a choice of square root for $dz_1 \wedge dz_2 \wedge... \wedge dz_n$. 
For the case of the disk, taking $\sqrt{dz}$, it is easy to see that the trace of the geometric endomorphism in Example \ref{Example_Spin_number} acting on the vector space generated by the section $\sqrt{dz}$ at the origin is given by $e^{i\alpha/2}$.
The local cohomology has a Schauder basis $\{z^k\sqrt{dz}\}$ where $k$ runs over the non-negative integers. Now the formula follows from equation \eqref{equation_locally_free_sheaf_factorization}. The general case for $\mathbb{C}^n$ follows by taking a polydisk decomposition and using the product of Lefschetz numbers using Proposition \ref{Lefschetz_product_formula}.

Observe that this number satisfies the \textit{\textbf{duality identity}}
\begin{equation}
\label{Namalge_redda}
    \frac{e^{i\alpha/2}}{(1-e^{i\alpha})}=-\frac{e^{-i\alpha/2}}{(1-e^{-i\alpha})}.
\end{equation}
This is an example of a more general duality identity for the spin numbers where there is an underlying complex structure, locally. Note that this is an instance of the duality in Proposition \ref{proposition_Lefschetz_on_adjoint}, and is a consequence of the \textit{self duality} that we explained in Remark \ref{remark_self_dual_spin_serre}.

In this case of the disc, the local Lefschetz number was simply $e^{i\alpha/2}$ times the holomorphic Lefschetz number. This is an instance of the $\widehat{A}$-genus of a spin manifold with a complex structure being equal to the Todd class of the manifold up to the exponential of half the first Chern class, 
now at the local level at the fixed point as in \cite{bott1967vector}.
In the singular setting this is trickier, as we shall see in Example \ref{example_spin_number_for_nodal_surface}.

\subsubsection{Hirzebruch's $\chi_y$-genera and the Signature invariant}
\label{subsection_Hirzebruch_invariants}

Given a stratified pseudomanifold $\widehat{X}$ with a wedge complex structure and a Hermitian bundle $E$, we can define $h^{p,q}(X,E)$ to be the dimension of the cohomology group of degree $q$ of $\mathcal{P}(X)=(L^2\Omega^{p,\cdot}(X;E),\overline{\partial}_E)$. We define the $\chi_y$ genus for this complex as
\begin{equation}
    \chi_y(X;E)=\sum_{p=0}^n y^p \chi^p(X;E) \quad \text{ where } \chi^p(X;E)=\sum_{q=0}^n (-1)^q h^{p,q}(X;E).
\end{equation}
In the case of the trivial bundle, we will drop $E$ from the notation.
Serre duality, which in particular implies $h^{p,q}(X;E)=h^{n-p,n-q}(X;E^*)$, shows that we have a duality
\begin{equation}
\label{equation_Hirzebruch_invariant_duality_rigid}
    \chi_y(X;E)=\sum_{p=0}^n y^p \sum_{q=0}^n (-1)^q h^{p,q}(X;E)= (-y)^n \sum_{p=0}^n (y^{-1})^{p} \sum_{q=0}^n (-1)^q h^{p,q}(X;E^*) = (-y)^n\chi_{y^{-1}}(X;E^*).
\end{equation}
In \cite[\S 4]{hirzebruch1966topological}, Hirzebruch introduced his $\chi_y$ genera for smooth complex manifolds, and this is a simple extension. 
While we will work with the Dolbeault complex with VAPS domain, one can define pick other extensions and define corresponding $\chi_y$ invariants as well.
For instance in the case of the cusp curve that we study in Subsection \ref{example_cusp_singularity_preamble}, we see that the Hodge numbers are different for the local cohomology corresponding to the two choices of domains studied there. We will see in Subsection \ref{subsection_nonnormal_examples} that the global Hodge numbers are different as well for these choices.

\begin{remark}   
While the $\chi_y$ genus may prove to be interesting on other domains, for the rest of our discussion of this invariant we will restrict to the case where the global cohomology of the $L^2$ Dolbeault complexes (for all $p$) of a space $X$ with VAPS conditions matches the global cohomology of the $L^2$ de Rham complex.

\end{remark}

Given a resolution of a pseudomanifold $X$ with a wedge complex structure and a Hermitian bundle $E$, let $f$ be a holomorphic self map with isolated fixed points that lifts to a linear self map on $E$.
We define the \textbf{\textit{local $L^2$ equivariant $\chi_y$-genus}} for a fixed point $a$ with fundamental neighbourhood $U_a$ with a K\"ahler structure as
\begin{equation}
    \chi_y(U_a,E,f):=\sum_{p=0}^n y^p L(\mathcal{P}_B^p(U_a), T_{f})
\end{equation}
where $T_f$ is the geometric endomorphism associated to $f$ on the Dolbeault complex $\mathcal{P}_B^p(U_a)$ for each $p$. We will drop $E$ from the notation when we consider the trivial bundle.
The \textbf{\textit{global $L^2$ equivariant $\chi_y$-genus}} is defined to be 
\begin{equation}
    \chi_y(X,E,f):=\sum_{p=0}^n y^p L(\mathcal{P}^p(X), T_{f})
\end{equation}
and by the Lefschetz fixed point theorem it follows easily that this is the sum of the local genera at the isolated fixed points.

\begin{proposition}
\label{Proposition_duality_Hirzebruch_genera}
In the setting of the discussion above, we have the duality of local equivariant $\chi_y$ genus
\begin{equation}
\label{equation_duality_Hirzebruch_chi}
    \chi_y(U_a,E,f)=\chi_{y^{-1}}(U_a,E^*,f^{-1})(-y)^n.
\end{equation}
\end{proposition}

\begin{proof}
This follows from applying Proposition \ref{Proposition_duality_complex_conjugation_for_local_Lefschetz_numbers} for the duality of Lefschetz numbers and using Serre duality, similar to equation \eqref{equation_Hirzebruch_invariant_duality_rigid}.
\end{proof}

\begin{remark}
\label{Remark_our_dualities_are_better}
The above proposition is the analog in $L^2$ cohomology of Proposition 3 of \cite{donten2018equivariant} for the Hirzebruch $\chi_y$ genus corresponding to the Baum-Fulton-Quart Lefschetz-Riemann-Roch theorem on quotient singularities. The proof there is using an explicit functional equation for the invariant for orbifolds, as opposed to ours which is more conceptual since it uses the dualities of local cohomology groups. In the remarks following the proof of Proposition 3 of that article, it is observed that for affine cones over curves of degree $4$ in $\mathbb{P}^2$, the duality does not hold in the setting of that article when $E$ is the trivial bundle.
We will see an explicit computation of our $L^2$ version in Example \ref{Singular_quadric_equivariant_Hirzebruch}. It is easy to verify that the duality in Proposition \ref{Proposition_duality_Hirzebruch_genera} holds in that example.
\end{remark}

The Hodge index theorem on a smooth K\"ahler manifold $\widehat{X}$ implies that the signature $\tau(X)=\chi_1(X)$. This is Theorem 15.8.2 of \cite{hirzebruch1966topological}, and the proof uses properties of the \textit{K\"ahler package} (Hirzebruch also indicates that it holds more generally for smooth \textit{complex} manifolds with other methods). It is easy to see that the same proof goes through for the singular setting, since it was shown that the K\"ahler package extends to the $L^2$ cohomology of K\"ahler wedge metrics in \cite{cheeger1982l2}. 
Thus the local Lefschetz numbers for the signature complex, also known as Lefschetz signature numbers, are given by evaluating the Lefschetz $\chi_y$ genera at $y=1$.
We formalize this as follows, generalizing Theorem 6.27 of \cite{AtiyahBott2}, in the wedge K\"ahler setting.

\begin{theorem}[Lefschetz Hodge index theorem]
\label{Theorem_L2_Lefschetz_Hodge_index}
In the same setting as that of Proposition \ref{Proposition_duality_Hirzebruch_genera}, the local Lefschetz Signature number at the fixed point $a$ is given by $\chi_{1}(U_a,f)$.
\end{theorem}

This result is in contrast to equivariant $\chi_y$ genera in singular cohomology for singular spaces, where it is well known that the K\"ahler package fails to hold for singular cohomology on stratified pseudomanifolds since Poincar\'e duality fails. Indeed, standard proofs of the result in the smooth setting uses explicit functional equations while ours now follow from arguments at the level of cohomology (c.f. Remark \ref{Remark_our_dualities_are_better}).

It is easy to observe that $\chi_{0}(X)$ is the Riemann Roch number for the trivial bundle on $X$, and that $\chi_{-1}(X)$ is the Euler characteristic of the de Rham complex on $X$.
Just as we computed the Spin number of Atiyah and Bott for a group action at a smooth point by using the holomorphic Lefschetz numbers, we can derive Atiyah and Bott's formula in equation (7.3) of \cite{AtiyahBott2} for the signature number of an isolated smooth fixed point of an isometry using the holomorphic Lefschetz numbers.

While our setting is more general than the algebraic one, there are different extensions of Hirzebruch's $\chi_y$ genus to singular algebraic varieties. For instance, in \cite[\S 3.6]{MaximSaitoJorgHodgemodules2011}, a Hodge index theorem is formulated for intersection cohomology with middle perversity with the pure Hodge structure corresponding to the K\"ahler structure. Ours is an equivariant version of this, as is clear by the equivalence between intersection cohomology and $L^2$ cohomology for wedge metrics. 

Our Lefschetz signature invariant is the Lefschetz version of the $L^2$ Signature (see \cite{Albin_signature}) for wedge metrics which is therefore related to the Goresky MacPherson $L$ classes (see \cite[\S 6.3]{Banaglbook07}).
For normal toric algebraic varieties (where the types of singularities that occur are even more restrictive, in particular being rational as we discussed in Subsection \ref{subsection_classes_of_singularities}), 
a $\chi_y$ invariant is defined in \cite{MaximJorgCharacteristicsingulartoric2015}, which matches the Baum-Fulton-MacPherson Todd classes for $y=0$, MacPherson Chern classes for $y=-1$, and Thom-Milnor $L$-classes for $y=1$ (in particular see equation 1.12 of \cite{MaximJorgCharacteristicsingulartoric2015}). There exists a lot more work studying the different relationships between these characteristic classes in the algebraic setting, and we refer the reader to \cite{cappell2023equivariant,GeneralatticeCappellShaneson94,Brasselet2010}.

Hirzebruch's $\chi_y$ invariant and its equivariant version plays a key role in the study of the elliptic genus (see Sections 5 and 6 of \cite{hirzebruch1992manifolds} and Section 6 of \cite{donten2018equivariant}) and the signature invariant. The duality in \eqref{equation_duality_Hirzebruch_chi} plays a key role in modular invariance (see Theorem 4.3 of \cite{borisov2003elliptic}) and other important properties of such invariants. 

In \cite{bott1989rigidity}, Bott and Taubes show that there are \textit{``twisting relations"} for the spin Dirac operator, the signature operator and their local Lefschetz numbers (see equations (2.9), (3.4)-(3.6) of that article) in the case of smooth manifolds which play a key role in their study of elliptic genera.
We can now use our methods to study $L^2$ versions of elliptic genera on singular spaces, at least formally. There is precedent for studying this on singular spaces, as in \cite{donten2018equivariant,borisov2003elliptic}.
We will see how the equivariant $\chi_y$ invariants can be used to compute anti-self-dual and self-dual indices of four dimensional spaces in Subsection \ref{subsection_asd_sd_complexes_localization}.

\begin{remark} 
Relationships between Hirzebruch $\chi_y$ genera and Molien series have been studied in \cite{donten2018equivariant} on singular spaces. Molien series on certain orbifolds have been computed in \cite{Degeratuthesis}, relating them to generalized eta invariants.
The work of \cite{donnelly1978eta} seems to be where Lefschetz numbers were first interpreted as generalized eta invariants and that perspective has since been used in the computation of Lefschetz numbers as can be seen for instance in \cite{weiping1990note,goette2012computations,goodman2022moduli,dessai2021moduli}. 
\end{remark}

\begin{remark}
As noted in Remark \ref{Remark_can_generalize_to_almost_complex}, much of the content in this subsection can be generalized to the almost complex setting where the Hirzebruch $\chi_y$ genus is used in many applications 
(see, e.g., \cite{jang2022complex}, 
\cite[\S 2]{borisov2003elliptic}), with some constraints near fixed points.     
\end{remark}

\subsubsection{Self-Dual and Anti-Self-Dual complexes}
\label{subsection_asd_sd_complexes_localization}

The anti self-dual (ASD) and self-dual (SD) complexes of spaces are of great interest, especially in the study of 4 dimensional spaces. The indices of these complexes appear in formulas for the dimension of moduli spaces of ASD and SD instantons and Seiberg-Witten moduli spaces. In \cite{WittenTopologicalQFT88}, Witten constructed a topologically twisted supersymmetric gauge theory which localizes on instantons and in the presence of torus actions equivariant versions of this Donaldson-Witten theory have been defined, and generalized in various directions following the work of Nekrasov (see \cite{NekrasovABCDinsta}).

\begin{remark}[Convention for self dual/ anti self dual complexes vs instantons]
\label{remark_anti_convention}
There is a rift in the notion of what is a self-dual/ anti-self dual instanton. In \cite{festuccia2020transversally,festuccia2020twisting}, the indices of the self dual/ anti-self dual complexes are used to compute instanton partition functions of  anti-self dual/self dual instantons respectively, which can be very confusing. This probably goes back to \cite{WittenTopologicalQFT88} where Witten says in the paragraph following equation 3.12 of that article that ``it would be tiresome to call them anti-instantons; solutions of the opposite equation will play no role in this section".
While that convention is justified in that article, we have found that it causes overall confusion in the literature, especially to the uninitiated.

We will thus call the instantons corresponding to the self-dual complex as the self-dual instantons, or just instantons. Those corresponding to the anti-self-dual complex will be refered to as anti-self dual instantons or anti-instantons.
\end{remark}

We build on the articles
\cite{festuccia2020transversally,festuccia2020twisting} written in expository style, where one main idea is to boil down the equivariant computations to those for local equivariant indices of the SD or ASD complexes.
The work there includes methods to compute partition functions of theories corresponding to \textit{``flipped instantons"}, which can roughly be described as \textit{placing self-dual instantons at some zeros of the Killing vector field and anti-self-dual instantons at the others}. The name \textit{Pestunization} is used in these articles since these ideas were introduced by Pestun in \cite[\S 4]{pestun2012localization}.

We focus on the work in article \cite{mauch2022index}, which gives rigorous justifications for computations done in \cite{pestun2012localization,festuccia2020twisting,festuccia2020transversally}, and refer the reader to the introduction of \cite{festuccia2020twisting} for the broader context.
The SD and ASD complexes have been labelled differently in certain places in \cite{festuccia2020transversally,festuccia2020twisting} as we noted in Remark \ref{remark_anti_convention} and we refer to Propositions 4.1 and 4.2 in Appendix A of \cite{mauch2022index} for the formulas.

We will restrict our study to stratified Witt spaces with complex structures and Hermitian wedge metrics. We denote the $L^2$ Betti-numbers of degree $k$ for a connected stratified pseudomanifold with a wedge metric by $\beta_k$. These numbers are independent of the choice of domain for the exterior derivative. We denote the corresponding Hodge numbers of Hodge bi-degree $(p,q)$ as $h^{p,q}$, the number of SD 2 forms as $\beta_2^+$, and the number of ASD 2 forms as $\beta_2^-$.

The $L^2$ Euler characteristic (of the de Rham complex) is $\chi=\beta_0-\beta_1+\beta_2-\beta_3+\beta_4$, and since Poincar\'e duality implies that $\beta_0=\beta_4$ and $\beta_1=\beta_3$, we have $\chi=2\beta_0-2\beta_1+\beta_2$. The signature of the Witt space is $\tau=\beta_2^+ - \beta_2^-$.
The $L^2$ SD complex is 
\begin{equation}
    0 \rightarrow L^2\Omega^0 \xrightharpoonup{d} L^2\Omega^1 \xrightharpoonup{d^+} L^2\Omega^2_+ \rightarrow 0
\end{equation}
and this has index $\text{ind}_{SD}=\beta_0-\beta_1+\beta_2^+=\frac{1}{2}(\chi+\tau)$. 
The $L^2$ ASD complex is isomorphic to
\begin{equation}
    0 \rightarrow L^2\Omega^0 \xrightharpoonup{d} L^2\Omega^1 \xrightharpoonup{d^-} L^2\Omega^2_- \rightarrow 0
\end{equation}
and this has index $\text{ind}_{ASD}=\beta_0-\beta_1+\beta_2^-=\frac{1}{2}(\chi-\tau)$. We will drop the $L^2$ in the notation for brevity in the rest of this subsection.
In page 4 of \cite{donaldson1985anti}, Donaldson indicates that 
\begin{equation}
    \Omega^2_+ = \Omega^{0,2}\oplus \overline{\Omega^{0,2}} \oplus \Omega^{1,1}_{||}
\end{equation}
where $\Omega^{1,1}_{||}$ denotes the span of the \textit{metric} (or rather the antisymmetrized form $\omega(\cdot, \cdot) =g(J \cdot, \cdot) $ where $J$ is the complex structure) for a Hermitian surface. He also indicates that the ASD 2 forms are the primitive two forms with respect to the Lefschetz map, in bi-degree $(1,1)$, which we denote as $\Omega^{1,1}_{\perp}$ following \cite{mauch2022index}. That is we have
\begin{equation}
\label{asd_useful_equation}
    \Omega^{1,1} = \Omega^{1,1}_{||} \oplus \Omega^{1,1}_{\perp}.
\end{equation}
This shows that $\beta_2^+=1+2h^{0,2}$ and $\beta_2^-=h^{1,1}-1$, and $\text{ind}_{SD}=\beta_0-2h^{0,1}+1+h^{0,2}$ where we have used the fact that $\beta_1=h^{1,0}+h^{0,1}=2h^{0,1}$. Since $\beta_0=1$ for a connected space, we have that $\text{ind}_{SD}=2\beta_0-2h^{0,1}+h^{0,2}=(h^{0,0}-h^{0,1}+h^{0,2})+(h^{0,0}-h^{1,0}+h^{2,0})$. This equality of Hodge numbers ascends to the level of $\mathbb{Z}_2$ graded complexes as outlined in Proposition 4.1 of \cite{mauch2022index}.

Similarly, $\text{ind}_{ASD}=\beta_0-2h^{0,1}+(h^{1,1}-1)$, and for a connected space since $b_0=1$ we have $\text{ind}_{ASD}=-h^{1,0}+h^{1,1}-h^{1,2}$ where we have also used Serre duality. It is easy to see from this that $\text{ind}_{ASD}$ is equal to minus of the Euler characteristic of the complex $(L^2\Omega^{1,\cdot},\overline{\partial})$.
We formalize this as follows.

\begin{proposition}
\label{Proposition_ASD_index_new_first_method}
Let $\widehat{X}^4$ be a connected Witt space with a wedge Hermitian structure and an isometric action of a Killing vector field. The index of the ASD complex on $X$ is equal to minus of that of $(L^2\Omega^{1,\cdot}(X),\overline{\partial})$.
\end{proposition}

Similar to Proposition 4.1 of \cite{mauch2022index}, we can consider lifts of this equality to an equality of complexes, but this is beyond the scope of this article.
Proposition \ref{Proposition_ASD_index_new_first_method} suggests a different method to compute the equivariant index of the ASD complex than that given in Proposition 4.2 of \cite{mauch2022index}. The method of that article is to use the isomorphism of SD and ASD complexes on $\mathbb{C}^2$ given by the diffeomorphism $(z_1,z_2) \mapsto (\overline{z_1},z_2)$; complex conjugating the first coordinate function.
The algorithm outlined in \cite{festuccia2020twisting,mauch2022index} to compute the local equivariant index for the ASD complex does not work for singularities such as those we will study in Subsection \ref{subsection_non_isolated_singularity}, and our methods give the necessary modifications.

In fact we observe that since $\chi=\chi_{-1}$ and $\tau=\chi_{1}$, we can write the indices of the SD and ASD complexes as
\begin{equation}
    \text{ind}_{SD}=\frac{1}{2}(\chi_{-1}+\chi_1), \quad \text{ind}_{ASD}=\frac{1}{2}(\chi_{-1}-\chi_1)
\end{equation}
and the local equivariant indices can be written in terms of the local equivariant $\chi_y$ invariants, which by our previous work can be computed in terms of the local holomorphic Lefschetz numbers.

The methods in \cite{festuccia2020twisting,mauch2022index} are mainly used to compute the local equivariant indices in order to compute partition functions of flipped instantons. In certain cases, this corresponds to theories on five dimensional Sasaki manifolds reduced along circles studied in \cite{qiu20155d,baraglia2016moduli,hekmati2023equivariant}. 
In the discussion in Subsection \ref{subsection_intro_to_compute} we will see that taking affine cones over either of the projective varieties which we will study in subsections \ref{subsection_example_nodal_surface} and \ref{subsection_non_isolated_singularity} gives a singular \textit{K\"ahler sandwich}, in particular containing a compact singular 5 dimensional Sasaki space where the reduction along the Reeb vector field yields the projective variety over which the affine cone was taken.

Given a 4 dimensional Witt space with a K\"ahler structure and a K\"ahler action of a complex torus $(\mathbb{C}^*)^2$ with variables $(\lambda,\mu)$, we denote the equivariant indices $ind_{SD}(U_a,\lambda,\mu), ind_{ASD}(U_a,\lambda,\mu)$.
On $U_a$ smooth, one can easily verify the relation 
\begin{equation}
\label{equation_duality_SD_ASD}
    ind_{ASD}(U_a,\lambda,\mu)=ind_{SD}(U_a,\lambda,\overline{\mu})=ind_{SD}(U_a,\overline{\lambda},\mu),
\end{equation}
and it has been observed that this is the source of the relation 
\begin{equation}
\label{equation_duality_partition_function_SD_ASD}
    Z^{\text{anti-inst}}_{\epsilon_1,\epsilon_2}(a,\overline{q})=Z^{\text{inst}}_{\epsilon_1,-\epsilon_2}(a,\overline{q})=Z^{\text{inst}}_{-\epsilon_1,\epsilon_2}(a,\overline{q})
\end{equation}
given in equation (109) in \cite{festuccia2020twisting} between the instanton and anti-instanton partition functions of certain supersymmetric Yang Mills theories. This can be seen from how certain superdeterminants are obtained by a change of variables of the form $\lambda=e^{-i\epsilon_1},\mu=e^{-i\epsilon_2}$ in the equivariant indices, expressing them in terms of the Lie algebra of the action. We refer to \cite{festuccia2020twisting} for more details and definitions.
This is a generalization of the relation
\begin{equation}
\label{equation_duality_partition_function_SD_ASD_simpler}
    Z^{\text{anti-inst}}_{\epsilon_1,\epsilon_2}(a,\overline{q})=Z^{\text{inst}}_{\epsilon_1,\epsilon_2}(a,q)
\end{equation}
used in Pestun's original work (see the beginning of page 72 of \cite{Pestun:2008ora})
that only holds for real values of $(\epsilon_1,\epsilon_2)$ (see equation (110) of \cite{festuccia2020twisting}). For such real valued $(\epsilon_1,\epsilon_2)$, Pestun notes that the full partition function has the factor $|Z^{\text{inst}}_{1/r,1/r}(ia,q)|^2$ in the integrand in equation (1.4) of \cite{pestun2012localization} (c.f., equation (112) of \cite{festuccia2020twisting}) and observes that 
this appearance is similar to the Ooguri-Strominger-Vafa relation between the black hole entropy and the topological string partition function 
(see equation (1.6) \cite{pestun2012localization}.
The third research direction suggested at the end of the introduction of \cite{pestun2012localization} is the following. 
\textit{Is there a more precise relation between the gauge theory formula of this
paper and Ooguri-Strominger-Vafa conjecture, or is there a four dimensional analogue of the $tt^*$-fusion}.
By analogy with \cite[\S 6]{OoguriStromVafa04} one should interpret $Z^{\text{inst}}_{1/r,1/r}(ia,q)$ as a quantum state, and the square of its magnitude as being proportional to the amplitude function.

In this light it is important to understand the relation in \eqref{equation_duality_SD_ASD}. For orbifolds, using explicit orbifold versions of Lefschetz-Riemann-Roch formulas (see, e.g. \cite[\S 3.3]{donten2018equivariant}) one can show that the relation \ref{equation_duality_SD_ASD} continues to hold.
Indeed, one can verify this for the example of the singular quadric we study in Subsection \ref{subsection_example_nodal_surface} where the singularity is an orbifold singularity, using the formulas for the equivariant $\chi_y$ invariant in Example \ref{Singular_quadric_equivariant_Hirzebruch}.
However it fails in the case of the surface with the non-normal depth 2 singularity in Subsection \ref{subsection_non_isolated_singularity}. This can be readily checked by our explicit formulae for the Lefschetz $\chi_y$ invariants at singularities given in those subsections.
That computation shows that the analog of the OSV conjecture described by Pestun fails for the space given by the suspension of the link at the depth 2 singularity in Subsection \ref{subsection_non_isolated_singularity} with the metric $d\phi^2+sin^2(\phi)g_Z$ where $g_Z$ is the Sasaki metric on the link. It is easy to see that this space has a torus action that preserves an almost complex structure.

It would be interesting to know if there are any \textit{physical} insights to be drawn from anomalous behaviour such as the failure of such symmetries in the partition functions in the singular setting.

\subsubsection{Morse inequalities, Hirzebruch's $\chi_y$-genera and cohomology.}
\label{subsection_Morse_Hirzebruch_cohomological}

The Hirzebruch's $\chi_y$ genera are related to Poincar\'e polynomials (see Definition \ref{definition_product_decomposition_critical_point}) on spaces with cohomological conditions, as has been studied in the smooth case (see for instance \cite{PingLigaptheorembettiKahler2013}). Some early applications of the holomorphic Lefschetz fixed point theorem include proving positivity results for certain invariants associated with the $\chi_y$ genera (see \cite{KosniowskiLefschetzApplication1970,CarellKosniowski73}).
We generalize these in a restricted setting to briefly demonstrate considerations of local cohomology in the singular setting.

The following notions are natural generalizations of those in the smooth setting (see for instance Definition 1.1 of \cite{PingLigaptheorembettiKahler2013}).

\begin{definition}
A Witt space $\widehat{X}$ ($X$) with a wedge metric is called \textit{\textbf{odd vanishing}} if the $L^2$ Betti numbers satisfy $\beta_{2k+1}=0$ for all $k$. If the local $L^2$ de Rham cohomology vanishes in odd degree for a small enough fundamental neighbourhood of a point $a$, we say the point $a$ is \textit{\textbf{odd vanishing}}.
If the wedge metric is K\"ahler, $\widehat{X}$ ($X$) is called \textit{\textbf{pure type}} if $h^{p,q}=0$ whenever $p \neq q$.
\end{definition}
There are many classes of smooth spaces for which these conditions hold (see Remark 1.2 of \cite{PingLigaptheorembettiKahler2013}), including K\"ahler manifolds that admit K\"ahler Hamiltonian Morse functions. This generalizes as follows.

\begin{proposition}
Let $\widehat{X}$ be a Witt space with a K\"ahler wedge metric. If there is a K\"ahler Hamiltonian Morse function $h$ associated to a K\"ahler circle action, and if the critical points of $h$ are odd vanishing, then the de Rham Morse polynomial $M(X,h)(b)$ has only terms with even powers of $b$, $M(X,h)=N(X)$ (see Definition \ref{definition_product_decomposition_critical_point}) and $X$ is odd vanishing.
\end{proposition}

\begin{proof}
Any isolated fixed point of a K\"ahler Hamiltonian Morse function has fundamental neighbourhood with a product decomposition as in Definition \ref{definition_product_decomposition_critical_point} where the attracting and expanding directions are each even dimensional. Together with the condition that the local cohomology vanishes in odd degree, this implies that the local Morse polynomial only has terms with even powers of $b$. Then the Morse inequalities (Theorem \ref{Theorem_strong_Morse_de_Rham}) and the Morse lacunary principle (Theorem 3.39 of \cite{banyaga2004lectures}) imply that $M(X,h)=N(X)$, and therefore that the $L^2$-Betti numbers vanish in odd degree.
\end{proof}

\begin{remark}
In the case where $\widehat{X}$ is smooth in the setting of the above proposition, the local cohomology of the critical points have local Hodge numbers $h^{p,q}=0$ for all $p,q$ except for $h^{0,0}$ which is $1$. In this case, not only is $X$ odd vanishing, but it is also pure. The pure condition follows from a corollary of Witten's equivariant holomorphic Morse inequalities, proved in \cite{witten1984holomorphic} (it is shown to hold in greater generality by more complicated methods in \cite{carrell1973holomorphic}).
\end{remark}

When there is a Hamiltonian stratified Morse function $h$ on a stratified pseudomanifold $\widehat{X}$ 
such that the gradient flow is stratum preserving, the de Rham Morse inequalities show that $M(X,h)(-1)=N(X)(-1)$. This is equivalent to the information obtained by applying the Lefschetz fixed point theorem to the map induced by a short time gradient flow of the Morse function. If $X$ is odd vanishing, then $N(X)(-1)=N(X)(1)$ and the Morse inequalities further show that $M(X,h)(1) \geq N(X)(1)$ (since $M(X,h)(1) \geq M(X,h)(-1)$), giving an upper bound for the number of critical points for the Morse function.
The Arnold-Floer theorem is a generalization of this for \textit{non-degenerate} Hamiltonian flows (see \cite[\S 1.1]{salamon1999lectures}) on smooth symplectic manifolds. The non-degeneracy of the Hamiltonian Morse function roughly corresponds to the simpleness of isolated fixed points of map induced by the Hamiltonian flow. 

It is well known that the Lefschetz fixed point theorem gives a lower bound for the number of fixed points for \textit{any} (even without the Hamiltonian condition) smooth self-map on a smooth manifold, which is homotopic to the identity map and has only isolated simple fixed points. This gives a trivial lower bound for the number of critical points for the Arnold conjecture (see the discussion below conjecture 1.1 of \cite[\S 1.1]{salamon1999lectures}). This lower bound corresponds to the bound for the Arnold conjecture (the sum of global Betti numbers) when it is an odd vanishing space. This admits an obvious generalization for any singular space with a cohomology theory admitting a suitable generalization of the Lefschetz fixed point theorem.

For smooth manifolds, in $L^2$ cohomology, the lower bound for the number of fixed points is $N(1)$ because local Lefschetz numbers must be $\pm 1$, and if there are $C_1$ many $-1$ simple fixed points with local Lefschetz number $-1$, there should be $N(1)+C_1$ many fixed points with local Lefschetz number $+1$. 

\begin{remark}[Towards an $L^2$-Arnold-Floer theorem]
On a Witt space $\widehat{X}$, consider a self map $f$ associated to a geometric endomorphism on the $L^2$ de Rham complex. When the fixed points of $f$ are isolated, simple, and odd vanishing, we define $D_a$ to be the (alternating) sum of local $L^2$ Betti numbers of each fixed point $a$.
Then an argument similar to that in the smooth case shows that $$N(X)-\sum_{a \in fix(f)} (D_a-1)$$ gives a lower bound for the number of fixed points of the flow.
Another way to package this information is to say, \textit{\textbf{the number of critical points counted with multiplicity $D_a$ is bounded below by the sum of $L^2$ Betti numbers on $X$}}.
\end{remark}

In Subsection \ref{example_conifold}, we work out the de Rham Morse inequalities for a stratified wedge Morse function on the conifold where there is a fixed point with $D_a=2$, and our count with multiplicity is sharp.
There is interest in extending the Arnold-Floer theorem to singular spaces (see \cite{brugués2023arnold}).
Since our Morse inequalities give the result for autonomous flows, it is interesting that we can prove it for Morse functions where the gradient flow need not necessarily preserve the strata.

\begin{proposition}
\label{Proposition_generalize_two_minor_positivity_etc}
Let $X^{2n}$ be a resolved Witt space with a K\"ahler wedge metric.
\begin{enumerate}
    \item If $X$ is odd vanishing then $(-1)^p \chi^p(X) \geq 0$ for all $p$.
    \item If $X$ is pure type then the de Rham Poincar\'e polynomial satisfies $N(X)(b)=\chi_{-b^2}(X)$. In particular $N(X)(i)=\tau(X)$.
\end{enumerate}
\end{proposition}

\begin{proof}
The proof of the first part follows from the fact that odd vanishing implies the vanishing of $h^{p,q}$ if $p+q$ is odd, which implies that 
\begin{equation}
    (-1)^p \sum_{q=0}^n (-1)^q h^{p,q}=\sum_{q=0}^n (-1)^{p+q} h^{p,q}=\sum_{q=0}^n h^{p,q} \geq 0.
\end{equation}
The second part follows from the fact that pure types implies $\beta_{2p}=h^{p,p}$ and odd vanishing, using which we can see that
\begin{equation}
    \chi_{-b^2}(X)=\sum_{p=0}^n (-b^2)^{p} \sum_{q=0}^n (-1)^q h^{p,q}=\sum_{p=0}^n b^{2p} (-1)^{2p} h^{p,p}=\sum_{k=0}^{2n} b^{k} \beta^{k}
\end{equation}
where we have used that $X$ being pure type implies it is odd vanishing in the last equality. Now since $\chi_1(X)=\tau(X)$ by Theorem \ref{Theorem_L2_Lefschetz_Hodge_index}, and using the odd vanishing, we see that $N(X)(i)=\tau(X)$.
\end{proof}

\begin{remark}
We remark that the method outlined in Remark 4 of page 304 of \cite{carrell1973holomorphic} is precisely the method used in \cite{KosniowskiLefschetzApplication1970} to prove the first part of Proposition \ref{Proposition_generalize_two_minor_positivity_etc} in the smooth setting. It is easy to observe that the said method can be replaced by the de Rham Morse inequalities.
\end{remark}

\subsection{Example computations}
\label{subsection_Computations_for_singular_examples}

We will compute Lefschetz numbers in several important examples where our results are applicable and where there is a sufficiently good enough understanding of the cohomology of local complexes. We develop a set of tools which we believe will provide a framework for computations on a range of spaces.

One way to compute global Lefschetz numbers is to compute all the local Lefschetz numbers, which we call the \textbf{\textit{direct method}}. Alternately, one could compute the global Lefschetz numbers by other means, as well as the local Lefschetz numbers at all fixed points other than the one where we need to compute it, and use the Lefschetz fixed point theorem to get the formula for the fixed point where we want to compute the local Lefschetz number. We call this the \textbf{\textit{indirect method}}. This second method was used heavily to gather insight, as well as to verify computations which we believe will be useful to the reader as well. In most cases, we will do computations of the global Lefschetz numbers using several means which can be used to cross-check computations.

We now give an outline of these examples, and describe certain aspects to keep an eye out for which we hope improves the readers insight on these spaces, group actions and the information given by the invariants.

\subsubsection{Outline of computations}
\label{subsection_intro_to_compute}

We will begin with the case of an orbifold singularity, in example \ref{example_teardrop} where we will compute the Lefschetz numbers in an example worked out by Meinrenken using a localization formula for orbifolds by Vergne. While the method is different, we obtain the same Lefschetz number.

The other examples that we study here are toric varieties, the first of which is a quadric surface with an orbifold singularity arising as a complex algebraic variety in Subsection \ref{subsection_example_nodal_surface}. Here we shall introduce computational techniques which we will then use in the rest of the section to study other spaces. This includes the use of the Poincar\'e residue to understand harmonic sections of the canonical bundle and the spin bundle, which we use to compute spin Lefschetz numbers. We also introduce moment maps corresponding to algebraic $\mathbb{C}^*$ actions, and compute the de Rham Morse Polynomials that we introduced in Subsection \ref{subsection_witten_deformation_morse_de_Rham}. In particular we show how this quadric surface can be realized as the reduced space of a canonical moment map on $\mathbb{CP}^3$, and we will see that this method of construction generalizes to the other toric varieties that we study in this section as well. We also compute equivariant $\chi_y$ invariants.

We compute local Lefschetz numbers and the de Rham Morse Polynomial for a conifold in Subsection \ref{example_conifold}. This is an example where the holomorphic Lefschetz numbers have been investigated in both mathematics and physics applications using orbifold resolutions of the space and we will explore the appearance of this example in the literature. 

In Subsection \ref{subsection_Lefschetz_L2_cohomology_development} we discussed how for certain algebraically normal varieties the local Lefschetz-Riemann-Roch numbers of \cite{baum1979lefschetz} match our $L^2$ local Lefschetz numbers. In particular, all normal toric varieties are rational, and we will not see examples of spaces with normal singularities that are not rational.
For non-normal varieties, the integral completion of the local ring contains holomorphic functions which are not regular and these appear in the local cohomology group of the Dolbeault complex for Hodge degree $p=0$. The computations in other Hodge degrees too are different from those of Baum-Fulton-MacPherson and we will see this when we study non-normal algebraic varieties.

We begin with the case of the cusp curve in Example \ref{example_cusp_singularity}. This is a non-normal algebraic variety that is a normal pseudomanifold.
We will also study a space which is not a normal pseudomanifold (see Definition \ref{definition_normal_pseudomanifold} and the discussion that follows it) in Example \ref{example_two_spheres}.
The toric algebraic varieties that we study all arise as symplectic reductions of $\mathbb{CP}^N$ for some $N>0$ corresponding to the action of some real torus.
In particular, when the toric varieties $Y$ we study are not normal, the quotient corresponding to the reduction where $X=\mathbb{CP}^N$ is not a good categorical quotient (see the discussion in Subsection \ref{subsection_classes_of_singularities}). 
This plays an important role in questions of geometric quantization.

Finally we will present a complex projective surface with a non-isolated singularity in Subsection \ref{subsection_non_isolated_singularity}. This space is similar to the cusp curve in that it is an algebraically non-normal topological pseudomanifold and, in both of these non-normal spaces, the Lefschetz-Riemann-Roch number of Baum-Fulton-Quart can be used to see that there is non-trivial global cohomology for the coherent sheafs that they study, in degrees $(0,q)$ for $q>0$. It is well known that in the smooth setting, the appearance of such cohomology classes are an obstruction to having K\"ahler $S^1$ actions with fixed points, and such actions can only exist on Fano varieties. An easy way of proving this is via index theoretic methods for smooth manifolds. For instance, Witten gave a short outline of a proof of a Theorem of Carell and Liebermann \cite{carrell1973holomorphic} using the holomorphic Morse inequalities in the K\"ahler setting (see page 330 of \cite{witten1984holomorphic}), from which this follows.
But both of the spaces actually have algebraic $\mathbb{C}^*$ actions with fixed points. However we see that the $L^2$ cohomology does not admit higher cohomology in $(0,q)$ for these examples, indicating that it might be a better candidate for generalizing such vanishing theorems in higher degree cohomology given group actions with fixed points.

We make some observations on the applicability of the techniques we introduce here to projective varieties with the Fubini-Study metric. It is well known that the Fubini-Study metric on arbitrary algebraic varieties need not be a wedge metric (see for instance Section 6 of \cite{cruz2020examples}). All one needs for our results to hold is some wedge metric with a complex structure.

Let us consider the case of complex dimension $1$. Here, a pseudomanifold must have isolated singularities and the links must be disjoint unions of circles (since circles are the only closed manifolds of real dimension $1$). In this case, the $\overline{\partial}$ boundary conditions for the complex $(\Omega^{0,q}_N(U_a), \overline{\partial})$, with $U_a$ a fundamental neighbourhood of a singular point $a$, preclude the possibility of there being non-trivial harmonic forms in the local complexes in degree $1$, and the harmonic forms are the $L^2$ bounded meromorphic forms, which are in the VAPS domain (or other choices of domain if specified for the complex). It is well known for instance that horn metrics are conformal to conic metrics as has been studied in \cite{LeschPeyerimhoffHornindex1998} (see page 656 of that article), along with other issues of change of domain as metrics are rescaled for various operators. As we indicated in Remark \ref{Remark_different_compactifications_for_cusp_curve}, we can choose different compactifications of the regular part of the singular space with different choices of boundary defining functions to obtain wedge metrics for which our results are applicable. This is how we apply our theory to the cusp curve. The other singular algebraic curve that we study here is that of two spheres connected at a point (Example \ref{example_two_spheres}) and the Fubini-Study metric is a wedge metric, so our theorem will apply.

In the case of complex dimension $2$ and higher, there are a large class of examples of Witt spaces with wedge metrics and complex structures arising from the compactification of affine cones over smooth projective algebraic varieties. The affine cone over a projective space is a smooth cone. Indeed, the projective space $\mathbb{CP}^N$ can be realized as the quotient of $\mathbb{C}^{N+1}$ by a $\mathbb{C}^*$ action (say this algebraic torus is parametrized by $\lambda$). The action of the unit circle of $\mathbb{C}^*$ describes a circle fibration over the base $\mathbb{CP}^N$, the entire space being $\mathbb{S}^{2N-1}$. This is called a K\"ahler sandwich; a Sasaki manifold arises as a circle fibration over a K\"ahler manifold, and the cone over the Sasaki manifold is K\"ahler (see the introduction of \cite{boyer&galicki}). 
Given a projective variety $Y \subset \mathbb{CP}^N$, the affine cone over it is the cone in $\mathbb{C}^{N+1}$ which when modded out by the $\mathbb{C}^*$ on $\mathbb{C}^{N+1}$ recovers $Y$, and the pullback of the affine metric on $\mathbb{C}^{N+1}$ to the affine cone is a wedge metric.

We will study such examples in subsections \ref{subsection_example_nodal_surface} and \ref{example_conifold} where there are isolated singularities with affine charts that can be realized as affine cones.
In Subsection \ref{subsection_non_isolated_singularity} we will study an example where there is an non-isolated singular set that can be realized by compactifying an affine cone over a singular curve. However the regular set can be compactified to get a wedge metric and a Thom-Mather structure, even over the non-isolated singularity, by methods similar to those used in Subsection \ref{example_cusp_singularity_preamble}.

As discussed in Section \ref{section_de_Rham}, the choices of domain do not affect the cohomology groups of the de Rham complex with the Witt condition. However, as we saw in Example \ref{example_cusp_singularity}, the choices of domain change the local (and hence global) cohomology of the Dolbeault complex. Indeed, it is the complex with the maximal domain that is commonly referred to as the $L^2$ Dolbeault cohomology, and clearly the complex and its cohomology depends on the metric.
In \cite{Bei_2012_L2atiyahbottlefschetz}, the Lefschetz fixed point theorem has been proven for arbitrary elliptic complexes with the maximal and minimal domains in the case of isolated fixed points. In Proposition \ref{Proposition_spectral_properties}, we studied the spectrum on truncated model cones with the VAPS conditions, but it is easy to observe that the result can be generalized for the maximal and minimal domains for spaces with isolated singularities. 
Moreover as we observed in Remark \ref{remark_non_simple_generalize_Dolbeault}, we can prove the Lefschetz fixed point theorem for metric cones using the localization results of Proposition \ref{Gluing_heat_kernels}, which can be proven for spaces with isolated conic singularities with the methods in \cite{Bei_2012_L2atiyahbottlefschetz} for the maximal and minimal domains. In Subsection \ref{example_cusp_singularity_preamble} we saw how for the cusp curve the minimal domain (which has the same cohomology as the VAPS domain for that space) and the maximal domain have different local cohomology, and we will see how the Lefschetz fixed point theorem can be verified for the maximal domain in that example explicitly in the computations following Example \ref{example_cusp_singularity}.

In all the computations in $L^2$ Dolbeault cohomology below for projective algebraic varieties, unless otherwise specified, we shall use the Hilbert space structure induced by the pull back of the Fubini-Study metric to \textit{\textbf{a}} Thom-Mather stratified space where it is a wedge metric (see Remark \ref{Remark_different_compactifications_for_cusp_curve}). For the local complexes at fundamental neighbourhoods of fixed points, we shall use the Hilbert space structure on the $L^2$ forms on the tangent cone induced by the pullback of the affine metric which is the induced metric on the tangent cone, similar to our computations in Subsection \ref{example_cusp_singularity_preamble}. 
We will use the notation explained in Subsection \ref{subsection_Lefschetz_L2_cohomology_development} to indicate the Hilbert spaces corresponding to local cohomology groups, where the notation implicitly gives a Schauder basis for computations. With these considerations, we begin computing.

\subsubsection{A computation for a teardrop orbifold}

We rework the computation of the equivariant index of a rotation of a teardrop given in the example following Theorem 3.5 in \cite{meinrenken1998symplectic}. Said theorem, which was proven by Michele Vergne in \cite{vergne1994quantification}, is used to compute the local contributions of a \textit{tear drop orbifold} in that example in order to compute the global Riemann-Roch number. We shall now compute the $L^2$ Dolbeault index also known as the Riemann-Roch number in $L^2$ cohomology using our formula.

\begin{example}[Teardrop as a stratified space]
\label{example_teardrop}
For simplicity we shall consider the teardrop orbifold $M$ of obtained by gluing a copy of $\mathbb{C}/ \mathbb{Z}_2$, with coordinate $w^2$, with a copy of $\mathbb{C}$ with coordinate $z$, via the mapping $z \mapsto w^{-2}$ for $z \neq 0$. We consider the action on the trivial bundle. The group $G=S^1$ acts by rotation $z \mapsto e^{i\phi}\cdot z$, and $w \mapsto e^{\frac{-i}{2}\phi}\cdot w$. There are two fixed points, one at the singular point at $w=0$ and the other at $z=0$. It is clear that the Lefschetz number on the disc neighbourhood of the smooth fixed point is $(1-e^{i\phi})^{-1}$ by the formula of Atiyah and Bott for the disc that we discussed earlier. 

In the neighbourhood of the singular point which is fixed, the holomorphic coordinate is $w^2$, which generates a Schauder basis $\{1, w^2, w^4,...\}$ for the local cohomology in degree $0$.
Since the map induced by the rotation is an isometry, the radial distance is preserved by the pullback meaning that the geometric endomorphism at the tangent cone is given by the action of $w^2 \mapsto e^{-i \phi}\cdot w^2$ on the forms of the complex $(L^2\Omega^{0,q},\overline{\partial})$, and we obtain the local Lefschetz number, $(1-e^{-i\phi})^{-1}$, by taking the trace over the local cohomology in degree $0$ giving the same value as the computation using the equivariant localization formula in \cite{meinrenken1998symplectic}.

\end{example}

In this example both the orbifold Riemann-Roch numbers computed in \cite{meinrenken1998symplectic} and the $L^2$ Riemann-Roch numbers match, essentially because the global Lefschetz number must be $1$ since the global cohomology of the space is generated by the constants for both the orbifold cohomology as well as the $L^2$ cohomology.
The instances where the stratified pseudomanifolds we study are orbifolds and stacks merit further specialized study since they appear in many geometric applications.

\subsubsection{Example of a singular quadric with an $A_1$ singularity.} 
\label{subsection_example_nodal_surface}

We shall do an extensive study of the singular complex surface $\widehat{M}$ cut out by $G(W,X,Y,Z)=Z^2-X Y=0$ in $\mathbb{C P}^3$ with coordinates $[W:X:Y:Z]$, including performing equivariant localization computations.
It is well known that this space has a conic singularity, and is an orbifold. 
The algebraic tangent cone of this space is described on page 157 of \cite{mumford1999red} while the tangent cone in our sense is a cone over a quotient of a three dimensional sphere by a $\mathbb{Z}_2$ action. Indeed, this space is the projective compactification of the affine cone over the quadric cone $Z^2-XY=0$ in $\mathbb{CP}^2$.

We begin by computing the local Lefschetz numbers using direct and indirect methods. In what follows, we study the complexes $\mathcal{P}^p=(L^2\Omega^{p, \cdot},\overline{\partial})$.

\begin{example}[Toric action on a complex quadric surface]
\label{example_nodal_singularity}

Consider $\widehat{M}$, the zero set of the polynomial $G(W,X,Y,Z)=Z^2-XY$ in $\mathbb{CP}^3$, which has an isolated singularity at $[W:X:Y:Z]=[1:0:0:0]$. 
This has an action of the complex torus $(\mathbb{C}^*)^2$ given by 
\begin{equation}
\label{equation_group_action_for_quadric}
    (\lambda,\mu) \cdot [W:X:Y:Z]=[W:\lambda^2 X:\mu^2 Y:\lambda \mu Z].
\end{equation}
A simple computation shows that this action, for generic values of $\lambda$ and $\mu$, will fix three points. The singular point is fixed, as are $a_1=[0:1:0:0]$ and $a_2=[0:0:1:0]$.
For the complex $\mathcal{P}^{p=0}$, the contribution of the smooth fixed point at $a_1=[0:1:0:0]$ corresponds to $\frac{1}{(1-\mu/\lambda)(1-1/\lambda^2)}$. This can be computed by considering the affine chart at $a_1$ where $X=1$, where the action can be written as $(\lambda,\mu) \cdot (w, y, z)=(w/\lambda^{2}, \mu^2y/\lambda^2, \mu z /\lambda )$ where $w=W/X, y=Y/X, z=Z/X$ are coordinates on the affine chart, and the local cohomology is $\overline{\mathbb{C}[w, z]}$ where we use the notation in Subsection \ref{subsection_Lefschetz_L2_cohomology_development}.
Similarly the contribution of $a_2=[0:0:1:0]$ is $\frac{1}{(1-\lambda/\mu)(1-1/\mu^2)}$. The contributions for $p>0$ can be computed for these fixed points using the formulas in the previous section for these smooth fixed points.

We can add the local Lefschetz numbers for these fixed points to obtain
\begin{equation}
\label{equation_Wimalee}
    1-\frac{1+\lambda\mu}{(1-\lambda^2)(1-\mu^2)}, -\frac{(1+\lambda\mu)^2}{(1-\lambda^2)(1-\mu^2)}, 1-\frac{(1+\lambda\mu)\lambda\mu}{(1-\lambda^2)(1-\mu^2)}
\end{equation}
respectively for the complexes with Hodge degree $p=0,1,2$. 
If we know the global cohomology, we can compute the local Lefschetz numbers at the singular fixed points using the indirect method.
\end{example}

The third fixed point is at the singularity $a_3=[1:0:0:0]$, and on the chart $W=1$, we denote the holomorphic functions $X/W,Y/W,Z/W$ by $x,y,z$ respectively, for the rest of the discussion for this example.
Since $z^2=xy$, the local cohomology corresponds to the Hilbert space $Q=\overline{R[1, z]}$ where $R=\mathbb{C}[x,y]$. 
The trace of the geometric endomorphism over the completion of $R$ can be computed similarly to the smooth case, and yields $(1-\lambda^2)(1-\mu^2)$, and the trace of the geometric endomorphism over the holomorphic functions is then given by $(1+\lambda\mu)/(1-\lambda^2)(1-\mu^2)$ since the trace of the map acting on the one dimensional vector space generated by $z$ is $\lambda \mu$, which matches up with 1 minus the first number in \eqref{equation_Wimalee}. 
Summarizing, we have
\begin{equation}
    \frac{(1+\lambda\mu)}{(1-\lambda^2)(1-\mu^2)}+\frac{1}{(1-\mu/\lambda)(1-1/\lambda^2)}+\frac{1}{(1-\lambda/\mu)(1-1/\mu^2)} =1
\end{equation}
where the three terms on the left hand side are the Lefschetz-Riemann-Roch numbers at the fixed points $a_1, a_2, a_3$ in that order and the right hand side is the global Lefschetz-Riemann-Roch number.
It is easy to see that this is equal to the Lefschetz-Riemann-Roch number of Baum-Fulton-MacPherson for this action on the singular surface by Proposition \ref{Proposition_BFM_equals_L2}.

The singularity of the quadric is an $A_1$ Du Val singularity, and the adjunction formula (see Section 2.3 of \cite{reid2012val}) holds for these spaces since it is a canonical singularity. Moreover there is a crepant resolution which in particular implies that the canonical bundle is trivial in a neighbourhood of the singularity, thus corresponding to a locally free sheaf. Since this is a quadric in $\mathbb{CP}^3$, the adjunction formula suffices to show that the first Chern class is even ($\mathcal{O}(n+1-d)$ is the canonical bundle by the adjunction formula and $n+1-d=2$ for this complete intersection, where $n=3$ and $d=2$ is the degree of the curve) and that there is a square root of the canonical bundle. 

In fact, a representative for the local holomorphic two form generating the dualizing section corresponding to the canonical class on the chart $W=1$ is given by the Poincar\'e residue
\begin{equation}
\label{equation_poincare_residue}
    \omega=Res_{\mathbb{C}^3|M} \frac{dx \wedge dy \wedge dz}{G} = \frac{dx \wedge dy}{\partial G/\partial z}= \frac{dx \wedge dy}{2z}
\end{equation}
(see Section 2.3 of \cite{reid2012val} for more details).
It is easy to see check that this local two form is $L^2$ bounded in a fundamental neighbourhood of the fixed point (heuristically  $\omega \sim dx\wedge dy / 2 \sqrt{xy}$), and that the trace of the geometric endomorphism of the Dolbeault complex with $p=2$ over the one dimensional vector space generated by this form yields $\lambda\mu$. In fact, the local cohomology for the complex with $p=2$ is given by the Hilbert space $\overline{Q[\omega]}$ where $Q$ is the Hilbert space for the local cohomology for $p=0$, introduced above.
This explains how the local Lefschetz number for $p=2$ is that for $p=0$ times $\lambda\mu$, which is $(1+\lambda\mu)\lambda \mu /[(1-\lambda^2)(1-\mu^2)]$. That is, we have
\begin{equation}
\label{equation_nodal_computation_p=2}
    \frac{(1+\lambda\mu)(\lambda \mu)}{(1-\lambda^2)(1-\mu^2)}+\frac{(\mu/\lambda)(1/\lambda^2)}{(1-\mu/\lambda)(1-1/\lambda^2)}+\frac{(\lambda/\mu) (1/\mu^2)}{(1-\lambda/\mu)(1-1/\mu^2)} =1
\end{equation}
where the three terms on the left hand side are the Lefschetz numbers for the Dolbeault complex twisted by the canonical bundle at the fixed points $a_1, a_2, a_3$ in that order and the right hand side is the global Lefschetz number. For the contributions at the smooth fixed points, we have used the Woods Hole formula \eqref{woods_hole_formula} that we derived earlier.

For $p=1$, the local $L^2$ cohomology at the singularity in the chart where $W=1$ is $\overline{R_1}$ where $R_1=R[dx,dy, xdy/z, ydx/z]$ is the module over the ring $R=\mathbb{C}[x,y]$ as in the case for $p=0$. We note that $2zdz=xdy+ydx$ so that $dz$ is included in this space.
Taking the trace of the geometric endomorphism over this Hilbert space, we get the local Lefschetz number 
\begin{equation}
\label{quadric_Lefschetz_p1}
    \frac{\lambda^2+\mu^2+\lambda\mu+\lambda\mu}{(1-\lambda^2)(1-\mu^2)}=-1+\frac{(1+\lambda\mu)^2}{(1-\lambda^2)(1-\mu^2)}
\end{equation}
which summed with the second expression in \eqref{equation_Wimalee} yields $-1$, since the only element in the global cohomology group for this complex is given by the K\"ahler form in bi-degree $(1,1)$. We observe that the local Lefschetz number for this complex at the singular fixed point does not factorize as in equation \eqref{equation_locally_free_sheaf_factorization}. This is because the module $R_1$ is not a free module over $Q=R[1,z]$.

\begin{example}[$L^2$ equivariant $\chi_y$ and signature invariants]
\label{Singular_quadric_equivariant_Hirzebruch}
We follow the notation introduced in Subsection \ref{subsection_Hirzebruch_invariants}.
For the group action (which for $|\lambda|=|\mu|=1$ is a family of isometries) on the space we study, the $L^2$ equivariant $\chi_y$ invariant at the singular point is
\begin{equation}
    \chi_y(U_{a_1},f)=\sum_{p=0}^2 y^p L(\mathcal{P}_B^p(U_a), T_{f})=\frac{(1+\lambda\mu)+y^1(\lambda^2+\mu^2+\lambda\mu+\lambda\mu)+y^2(1+\lambda\mu)(\lambda\mu)}{(1-\lambda^2)(1-\mu^2)}.
\end{equation}    
We compare this with Example 7.4 of \cite{weber2016equivariant} where the equivariant $\chi_y$ invariants are computed in singular cohomology for the same space as we study here. In particular, the numerator of the right hand side expression above coincides with the expression in equation (17) of \cite{weber2016equivariant} (up to differences in notation).
In Section 5 of \cite{donten2018equivariant} the $\chi_y$ genus is used to compute Molien series and the elliptic genus.
The equivariant $L^2$ signature Lefschetz number is given by $\chi_1(U_{a_1},f)$
and the equivariant invariants can be computed at the smooth points using the Atiyah-Bott formula given in Subsection \ref{subsection_Hirzebruch_invariants}. In particular, the signature of this space is $1$.   

\end{example}

\begin{example}[Spin number for the complex surface]
\label{example_spin_number_for_nodal_surface}
As explained in Subsection \ref{subsubsection_spin_dirac_kahler}, the group action lifts to an action on the canonical bundle and we know that there is a square root of that bundle, $L(=\sqrt{K})$, corresponding to a spin structure. The group action we consider on the space lifts to one on the square root bundle. We can compute the local Lefschetz numbers for the twisted Dolbeault complex at smooth points using the Atiyah-Bott formula. Since the canonical bundle is locally trivial, so is the square root bundle and we can use equation \ref{equation_locally_free_sheaf_factorization} to compute the local Lefschetz number at the singular point, which is
\begin{equation}
    \frac{(1+\lambda\mu)(\lambda\mu)^{1/2}}{(1-\lambda^2)(1-\mu^2)}=\frac{(1+\lambda^{-1}\mu^{-1})(\lambda\mu)^{-1/2}}{(1-\lambda^{-2})(1-\mu^{-2})}
\end{equation}
where the two ways of writing it are an instance of the duality identity that we discussed in Subsection \ref{subsubsection_spin_dirac_kahler}.
We can therefore write 
\begin{equation}
    \frac{(1+\lambda\mu)(\sqrt{\lambda \mu})}{(1-\lambda^2)(1-\mu^2)}+\frac{\sqrt{(\mu/\lambda)(1/\lambda^2)}}{(1-\mu/\lambda)(1-1/\lambda^2)}+\frac{\sqrt{(\lambda/\mu) (1/\mu^2)}}{(1-\lambda/\mu)(1-1/\mu^2)} = 0
\end{equation}
where the three terms on the left hand side are the local Lefschetz numbers for the Dolbeault complex twisted by \textit{the} square root of the canonical bundle at the fixed points $a_1, a_2, a_3$ in that order and the right hand side is the global Lefschetz number. This is equal to the global Lefschetz number of the spin Dirac complex. In general there can be many spin structures on a space but in this case there is only one up to isomorphism since the abelianization of the fundamental group is trivial for the space $\widehat{M}$.
\end{example}

It is well known that given an effective action of a compact Lie group on a spin manifold, either it has a lift to the spin bundle, or the action admits a covering linear action by a Lie group $G_1$, such that there exists a homomorphism of Lie groups $h:G_1 \rightarrow G$ (see for instance Proposition 2.1 of \cite{atiyah1970spin}).
Group actions which do not lift to the spin bundle are known as odd actions and their existence is related to interesting questions in topology (see \cite{landweber1988circle}).
More generally, there exist linear covering actions for any rational power $k/N$ of the canonical bundle when such a bundle exists (see for instance \cite[\S 6.4]{DessaiToric16}, \cite{hirzebruch1988elliptic}).

In general, it is not necessarily true that group actions lift to linear actions on integer powers of line bundles and there are important questions in geometric invariant theory related to this (see \cite{MundetiRieralifts01}). In this case, the Lefschetz fixed point theorem can sometimes identify the failure of lifting. For instance, the Lefschetz numbers can be computed for the line bundles $K^2$, and $K^{-1}$, similar to our computations above, and one can verify by a formal computation that if the action lifts, then the corresponding global Lefschetz numbers are both $1$. However, in either of the cases of $K^4$, and $K^{-3}$, the sum of the local Lefschetz numbers gives a rational function in $\lambda, \mu$ which has poles at $\lambda=\mu=1$, which shows that the action does not lift as a linear action to these bundles. This is because the global Lefschetz number of such actions evaluated at the identity corresponds to the index of the corresponding twisted Dolbeault-Dirac operator, which is a finite integer by the index theorem of \cite{Albin_2017_index}. This is an easy computation to do following what we have done above and we leave it to the interested reader.

We briefly explain how the de Rham Morse inequalities can be worked out for some choices of Hamiltonian Morse functions corresponding to the $(\mathbb{C}^*)^2$ action on $\widehat{M}$.
Consider the $T^3$ toric action on $\mathbb{CP}^3$ given by $(\theta_1,\theta_2,\theta_3)\cdot [w:x:y:z] \rightarrow [w:e^{i\theta_1}x:e^{i\theta_2}y:e^{i\theta_3}z]$. This corresponds to the symplectic moment map 
\begin{equation}
\label{Moment_map_for_quadric}
    [w:x:y:z] \rightarrow [H_1,H_2,H_3]=\Bigg[\frac{|x|^2}{|\rho|^2},\frac{|y|^2}{|\rho|^2},\frac{|z|^2}{|\rho|^2} \Bigg]
\end{equation}
where $|\rho|^2=|w|^2+|x|^2+|y|^2+|z|^2$. Consider the Hamiltonian function $F_1=H_3^2-H_2H_1$, and let us write $\lambda^2=e^{i\theta_1}$, $\mu^2=e^{i\theta_2}$ and $\alpha=e^{i\theta_3}/(\lambda\mu)=e^{i(\theta_3-(\theta_1+\theta_2)/2)}=e^{i\theta_4}$. We see that $F_1^{-1}(0)=\{|z|^2-|xy|=0\}$, which is the space given by the points satisfying $(z^2 \alpha^2 -xy)=0$ as $\theta_4$ varies, for all $w,x,y,z$ satisfying $z^2-xy=0$. Quotienting out by this circle action, we get the symplectic reduction of $F_1^{-1}(0)$, which is precisely the space $M$ and the K\"ahler action we studied is the restriction of the residual effective action on the reduced space. 

A K\"ahler action with at least one fixed point is Hamiltonian (see \cite[\S 3.2]{Mazzeo_2015}) and the Morse inequalities corresponding to such Hamiltonian Morse functions arising from K\"ahler actions were already studied by Goresky and MacPherson in intersection homology. The de Rham Morse inequalities we worked out in Subsection \ref{subsection_witten_deformation_morse_de_Rham} correspond to the Morse inequalities in intersection homology given in \cite[\S 6.12]{goresky1988stratified}. 

If at a smooth fixed point of the K\"ahler action, it acts on the local holomorphic coordinate function $z_i$ of an affine chart by $z_i \rightarrow e^{i \ell \theta}z_i$, then the local normal form of the Hamiltonian function on the chart is given by $\sum_{i=1}^n \ell_i |z_i|^2$, and the de Rham Morse index corresponds to twice the number, $k$, of negative $\ell_i$ (see for instance page 322 of \cite{witten1984holomorphic}). This is because when these constant \textit{weights} $l_i$ are negative, the corresponding gradient flow on a small polydisc neighbourhood of the critical point has a decomposition as in Definition \ref{definition_product_decomposition_critical_point} where $U_1 \cong \mathbb{D}^{2n-2k}$ and $U_2 \cong \mathbb{D}^{2k}$. This can be seen by the correspondence of the symplectic gradient and the metric gradient for a Hamiltonian vector field $V$, given by a compatible complex structure $J$. If the corresponding Hamiltonian function is $H$, then we have
\begin{equation}
    dH =\omega(V,\cdot)=g(JV, \cdot)=g(J \nabla H, \cdot), \quad dH^{\#} = J(\nabla H)
\end{equation}
Then we can use Theorem \ref{Theorem_strong_Morse_de_Rham} to see that the local Morse polynomial at a smooth fixed point of Morse index $2k$ is $b^{2k}$.

If we pick $\theta_1>\theta_2>0$, then the above argument 
shows that $a_1$ has Morse index $0$ and $a_2$ has Morse index $2$. 
Moreover the singular fixed point is a critical point of the Morse function where the gradient flow is attracting. This follows again from the fact that the weights on the chart $\mathbb{C}^3$ given by $W=1$ in $\mathbb{CP}^3$ are positive.
The link at this point is a quotient of $S^3$ by a $\mathbb{Z}^2$ action, and is a lens space with trivial cohomology in degree $1$. By Lemma \ref{conic_cohomology_1}, the local Morse polynomial is $1$ at this point, and the global Morse polynomial is $1+b^2+b^4$, and by the Morse Lacunary principle (see Theorem 3.39 of \cite{banyaga2004lectures}), it is equal to the global Poincar\'e polynomial.

\begin{remark}
\label{remark_reduced_spaces_moment_map_later_use}
We consider another reduced space that will be relevant to Example \ref{example_two_spheres}.
The symplectic reduction corresponding to the Hamiltonian function $F_2=H_2-H_1$ corresponds to the copy of $\mathbb{CP}^2$ in $\mathbb{CP}^3$ given by the vanishing set of $x-y$.

If we do both reductions (in either order), we see that the resulting space is $\{z^2-xy=0\} \cap \{x-y=0\}$. It is easy to see that $z^2-x^2=0$ corresponds to the union of the two copies of $\mathbb{CP}^1$ ($z=x$ and $z=-x$) in the variety $z^2-xy=0$. These two spheres meet at a point and together form a non-normal variety arising from a quotient that is not a good categorical quotient as we discussed in Subsection \ref{subsection_intro_to_compute}.
This variety has the $\mathbb{C}^*$ action given by the subgroup of the $(\mathbb{C}^*)^2$ action in Example \ref{example_nodal_singularity}, where $\lambda=\mu$.
\end{remark}

\begin{remark}[Bott residue formula and more general localization]
\label{remark_Bott_residue_quadric}
In \cite{localization_algebraic_Graham_Edidin}, Edidin and Graham prove more general localization results on certain singular algebraic schemes including complete intersections, and in section 5.3 of that article study the example of the quadric with the $\mathbb{C}^*$ action given by fixing $\lambda=\mu^{-1}$ for the action given in equation \eqref{equation_group_action_for_quadric}. They show how to extend the Bott residue formula to the algebraic setting when the singular variety is regularly embedded in some smooth algebraic variety. 
\end{remark}

\subsubsection{Example of a conifold}
\label{example_conifold}

Conifolds are objects of great interest in both physics and mathematics, showing up in the study of mirror symmetry and moduli spaces. These have non-orbifold, isolated conic singularities with non-spherical links, similar to the example of the Fermat quintic we introduced in Subsection \ref{subsubsection_higher_degree_local_cohomology}. As we discussed there, these are singular spaces with conic metrics.

Consider the quadric $Y_1 Y_4 - Y_2 Y_3 =0$ in $\mathbb{CP}^4$ with coordinates $[W:Y_1:Y_2:Y_3:Y_4]$. The zero set of this polynomial is a conifold which we shall call $X$. Restricted to the affine chart $W=1$, this is the local normal form of any quintic conifold in $\mathbb{C}^4$ (see footnote 2 on page 969 of \cite{davies2010quotients}). This is also sometimes called the three dimensional $A_1$ singularity.

This space has a $(\mathbb{C}^*)^3$ algebraic toric action, $(\lambda, \mu, \gamma) \cdot [W: Y_1: Y_2: Y_3: Y_4] = [W: \lambda Y_1:\mu\gamma Y_2: \mu^{-1}\lambda Y_3: \gamma Y_4]$.
It is easy to verify that generically there are 5 fixed points on $X$, given by $[1:0:0:0:0]$, $[0:1:0:0:0]$, $[0:0:1:0:0]$, $[0:0:0:1:0]$, $[0:0:0:0:1]$. The last four fixed points are smooth and the first is singular. We compute the holomorphic Lefschetz numbers for the associated family of geometric endomorphisms of the Dolbeault complex for $p=0$.

Consider the point $[0:1:0:0:0]$ on the affine chart where $Y_1=1$. In the standard coordinates, the action is given by $(\lambda, \mu, \gamma) \cdot (W, Y_2, Y_3, Y_4) = (1/\lambda W, \mu\gamma/\lambda Y_2, 1/\mu Y_3, \gamma/\lambda Y_4)$. On $Y_1=1$, since the space we are looking at is the zero set of $g=Y_4-Y_2 Y_3$, $dg=dY_4-Y_2dY_3-Y_3dY_2$, and at the origin, $dg=dY_4$. This shows that $dY_4$ is the conormal direction to the zero set of $g$ at the origin and that the cotangent space is spanned by $dW, dY_2, dY_3$, and we can use the Atiyah Bott formula (see Section 4 of \cite{AtiyahBott2}).
Alternately, we see that $\overline{\mathbb{C}[W, Y_2, Y_3]}$ is the Hilbert space corresponding to the local cohomology for a fundamental neighbourhood of the origin. Then the Lefschetz number for the fixed point is $1/((1-1/\lambda)(1-\mu\gamma/\lambda)(1-1/\mu))$ by the methods developed in Subsection \ref{subsection_Lefschetz_L2_cohomology_development}.
We can similarly compute the local Lefschetz numbers at the other smooth fixed points.

For the singular fixed point at $[1:0:0:0:0]$, we consider the affine chart $W=1$ with coordinates $y_j=Y_j/W$ for $j=1,2,3,4$. The action $(\lambda, \mu, \gamma) \cdot (y_1, y_2, y_3, y_4)$ is given by $(\lambda y_1, \mu\gamma y_2, \lambda y_3/\mu , \gamma y_4)$. Since the variety is locally defined by the equation $y_1 y_4-y_2 y_3=0$, we can compute the formula of Baum-Fulton-Quart given in \cite[\S 3.3]{baum1979lefschetz}, which yields the local Lefschetz number
\begin{equation}
    \frac{(1-\lambda \gamma)}{(1-\lambda)(1-\mu\gamma)(1-\lambda/\mu)(1-\gamma)}.
\end{equation}
We see that the local $L^2$ cohomology corresponds to $\overline{Q}$ where $Q=R/R_1$ is the quotient ring where $R=\mathbb{C}[y_1,y_2,y_3,y_4]$ and $R_1$ is the ideal $(y_1, y_4)$. Using this we can compute the local Lefschetz number as before and it is easy to see that this is equal to the Baum-Fulton-Quart Lefschetz number in this example.
It is an easy exercise to check that the sum of the local Lefschetz numbers at the fixed points is $1$.

\begin{remark}[Appearance in the literature]
\label{remark_yau_nekrasov_conifold}
The Lefschetz number for a toric action at the singular fixed point on the chart $W=1$ has been worked out in \cite{Martellisparksyau08}, specifically in Subsection 7.4.1 of that article. In Section 6 of that article, the equivariant index of the Dolbeault complex on this non-compact conic space is expressed as a character.
The authors justify computations using equivariant resolutions (with respect to torus actions) with at worst orbifold singularities, and using the orbifold equivariant index theorem.

The authors point out that such resolution based methods are also used in the computation of partition functions such as the Nekrasov partition function in \cite{NekrasovABCDinsta}, and \cite{GrassiPureSpinors,grassi2006curved} (see equation (3.9) of the last paper). Indeed in equation (22) of \cite{NekrasovABCDinsta}, we can see the holomorphic Lefschetz number for the conifold computed using resolutions, with the understanding that it is independent of the K\"ahler moduli of the resolution. We also observe that the computed character formulas match that of \cite{baum1979lefschetz} on this normal variety.
The higher cohomology of conifolds become important in related questions, as studied in sections 5 and 6 of \cite{shiftsprepotentialnekrasovzabzine}.
\end{remark}

In this example, we can find the structure of the Hodge diamond using the de Rham Morse inequalities and the symmetries of the Hodge diamond. Again the K\"ahler action has a moment map that can be written explicitly, similarly to that in the earlier example.
Let $C(Z)$ be a fundamental neighbourhood of the singular fixed point, where the link is $Z\cong S^2 \times S^3$. For the de Rham complex, if we impose generalized Neumann conditions at the link at the boundary $Z=\{x=1\}\cong S^2 \times S^3$, the harmonic representatives of the cohomology of the local complex consist of the constants in degree $0$ and the constant multiples of the volume form of the $S^2$ factor in degree $2$ (see Lemma \ref{conic_cohomology_1}).

We follow the methods we explained in Subsection \ref{subsection_example_nodal_surface}.
If we pick $\gamma=e^{i\theta}$, $\mu=\gamma^2$ and $\lambda=\gamma^4$, where $\theta>0$, the action on the chart $W=1$ corresponds to a Hamiltonian function where the singular fixed point is attracting, and the local Morse polynomial is $1+b^2$, where we use that the link at this point is $S^2\times S^3$, and Lemma \ref{conic_cohomology_1}. 

Using the argument for smooth fixed points of K\"ahler actions, we have the local Morse polynomial $b^6$ at $[0:1:0:0:0]$, $b^4$ at $[0:0:1:0:0]$, $b^4$ at $[0:0:0:1:0]$ and $b^2$ at $[0:0:0:0:1]$.
Theorem \ref{conic_cohomology_1} shows that the Morse polynomial is $1+2b^2+2b^4+b^6$, and by the Morse Lacunary principle (see Theorem 3.39 of \cite{banyaga2004lectures}), it is equal to the global Poincar\'e polynomial.

The K\"ahler form generates four harmonic forms in bi-degree $(k,k)$ for $k=0,1,2,3$, and it is a straightforward argument using Poincar\'e duality, the symmetries of the Hodge diamond and the Poincar\'e polynomial to see that the other two cohomology classes are in bi-degrees $(1,1)$ and $(2,2)$.

\subsubsection{Non-normal examples}
\label{subsection_nonnormal_examples}

The two complex algebraic varieties that we studied earlier in subsections \ref{subsection_example_nodal_surface} and \ref{example_conifold} are examples of normal toric algebraic varieties, where the local Lefschetz formulas corresponded to those of Baum-Fulton-MacPherson-Quart. In the non-normal case, we get different formulas as we anticipated in Subsection \ref{subsection_intro_to_compute}.
We begin by studying the example from Subsection \ref{example_cusp_singularity_preamble}. 

\begin{example}[Cusp singularity in $\mathbb{CP}^2$]
\label{example_cusp_singularity}
Consider the cusp curve given by $ZY^2-X^3=0$ in $\mathbb{CP}^2$, where we have the $\mathbb{C}^*$ action $(\lambda)\cdot [X:Y:Z]=[\lambda^2X:\lambda^3Y:Z]$. We consider the associated family of geometric endomorphisms on the Dolbeault complex for the trivial bundle. The action has one smooth fixed point at $[0:1:0]$ with holomorphic Lefschetz number $1/(1-\lambda^{-1})$. The other fixed point is at the singularity $[0:0:1]$.
\end{example}

If we compute the local Lefschetz number using the Lefschetz-Riemann-Roch formula of Baum-Fulton-Quart \cite[\S 3.3]{baum1979lefschetz}, we get the local Lefschetz number $(1-\lambda^6)/((1-\lambda^2)(1-\lambda^3))$. The corresponding global Lefschetz number is $1-\lambda$. Here $1$ corresponds to the trace over the constants in degree $0$ of the global cohomology while $\lambda$ is the trace over a global section in degree $1$. Restricted to the chart with the singularity, the Poincar\'e residue gives the section $dx/2y=dy/3x^2$. Here we use the notation $x=X/Z,y=Y/Z$. It is easy to check that there are similar sections on the other charts ($dX$ on $Y=1$ and $dY/Y^2$ on $X=1$) which glue to a global section.
The trace of the geometric endomorphism over the complex conjugate of this global section section is $\lambda$. Having a trivializing global section of the canonical sheaf is characteristic of Calabi-Yau spaces in the smooth setting. However it is also known that smooth Calabi-Yau manifolds do not admit Hamiltonian K\"ahler actions, which in the smooth case is obstructed by cohomology classes not in bi-degrees $(k,k)$.

We observe that the section $dx/2y=dt/t^2$ is not in $L^2$ with respect to the volume form of the wedge metric we worked out in Subsection \ref{example_cusp_singularity_preamble}, in a neighbourhood of the singular point. In that subsection we computed the cohomology of several local complexes at the singular point and we now consider the Lefschetz numbers for those complexes for the action in Example \ref{example_cusp_singularity}.

Let us first consider the case where $|\lambda|<1$ and the singular fixed point is attracting.
The cohomology of the complex $(L^2\Omega^{0,q}_N(U_a), \overline{\partial}_{VAPS})$ is given by the Hilbert space generated by the Schauder basis $1,t,t^2,...$ in degree $0$ (and vanishes in higher degrees) and since the action maps $t \mapsto \lambda t$, we get the local Lefschetz number $1/(1-\lambda)$. The corresponding global Lefschetz number is $1$. The cohomology of the complex with the maximal domain $(L^2\Omega^{0,q}_{\max}(U_a), \overline{\partial}_{\max})$ has a Schauder basis $t^{-1},1,t,t^2,...$ in degree $0$ and the local Lefschetz number is
\begin{equation}
\label{equation_expansion_character_cusp_1}
    - \sum_{k=-1}^{\infty} \lambda^{k} = (1-\lambda)^{-1} +(\lambda)^{-1}.
\end{equation}
The corresponding global Lefschetz number is $1+\lambda^{-1}$, where the last term reflects that the local meromorphic function $t^{-1}$ near the fixed point glues to a global meromorphic function that is $L^2$ bounded.

If instead we consider the case where $|\lambda|>1$, then the singular fixed point is expanding, and Proposition \ref{proposition_local_product_Lefschetz_heat} tells us that we can use adjoint complexes for the computations. The adjoint complex of $(L^2\Omega^{0,q}_{\max}(U_a), \overline{\partial}_{\max})$ is $(L^2\Omega^{0,n-q}_{\min}(U_a), \overline{\partial}^*_{\min})$, the local cohomology of which in degree $0$ has a basis $\{\overline{t}^k\overline{dt}\}$ where $k$ are positive integers, and vanishes in higher degrees. Since $\overline{t}^k\overline{dt} \mapsto \lambda^{-(k+1)} \overline{t}^k\overline{dt}$ via the adjoint geometric endomorphism, we have that the local Lefschetz number is 
\begin{equation}
\label{equation_expansion_character_cusp_2}
    - \sum_{k=2}^{\infty} \lambda^{-k} = -\lambda^{-2}/(1-(\lambda)^{-1})= (1-\lambda)^{-1} +(\lambda)^{-1}
\end{equation}
which matches the local Lefschetz number in the attracting case. The local cohomology of the adjoint complex of $(L^2\Omega^{0,q}_N(U_a), \overline{\partial}_{VAPS})$ has a Schauder basis given by $\{\overline{t}^k\overline{dt}\}_{k \geq 0}$ in degree $0$ and a computation similar to the one above shows that the local Lefschetz number is $1/(1-\lambda)$, which matches that in the attracting case.
Similarly, the Lefschetz numbers for the complexes $(L^2\Omega^{1,q}_{\max}(U_a), \overline{\partial}_{\max})$ and $(L^2\Omega^{1,q}_N(U_a), \overline{\partial}_{VAPS})$ can be computed directly.

\begin{remark}
    The space in this example generalizes to the family of spaces $zy^m=x^{m+1}$ in $\mathbb{CP}^2$ where the angle at the isolated singularity in each space for the corresponding wedge metric is $2\pi m$,
    by an analysis similar to that in Subsection \ref{example_cusp_singularity_preamble} for the case where $m=2$. One can check that $m$ is the Euler characteristic of the Dolbeault complex with the maximal domain on these spaces.
    
    It is easy to check that $x/y$ on the chart where $z=1$, will be in the local cohomology of the complex $(L^2\Omega^{0,q}_N(U_a), \overline{\partial}_{VAPS})$ for $m>3$, with higher powers of $x/y$ in the VAPS domain for higher values of $m$. Computing the index for the complex with the VAPS domain is as easy as checking the integrability of positive integer powers of $x/y$.
\end{remark}

\begin{remark}[Dual complexes and Laurent expansions]
\label{remark_dual_complexes_and_laurent_expansions}
    In the case of the Dolbeault complex in the smooth setting, the idea of using the dual complex is already in Witten's article where he introduced the equivariant holomorphic Morse inequalities (see equations 26 and 27 of \cite{witten1984holomorphic}) with an interpretation of local Lefschetz numbers as traces over the null space of a model harmonic oscillator, that is closely related to holomorphic functions and anti-holomorphic top forms on neighbourhoods of fixed points. 
    In the smooth setting, isometric K\"ahler actions are considered and the correspondence of the attracting/expanding to choices of positive and negative \textit{weights} of the action is as we discussed in the discussion above Remark \ref{remark_reduced_spaces_moment_map_later_use} (very roughly the difference between the symplectic and metric gradients of the Hamiltonian Morse function). Our work gives this a more rigorous footing that generalizes to singular spaces. 

    In many applications the holomorphic Lefschetz number is interpreted as a character formula by taking Laurent series expansions. The choice of Laurent expansion is very important and shows up in character formulas such as the Kirillov and Weyl character formulas in representation theory (see \cite[\S 8]{berline2003heat}) and in supersymmetric localization computations of quantum field theories of interest to both mathematicians and physicists (see Section 1.2 of \cite{pestun2017localization}). In the latter, the equivariant characters are notated by $[ ]_{\pm}$ following the notation in \cite[\S 5]{Atiyahellipticopsandcompactgroups74}.
    For the case of the complex $(L^2\Omega^{0,q}_{\max}(U_a), \overline{\partial}_{\max})$ with the action on the space above, the character for $[(1-\lambda)^{-1} +(\lambda)^{-1}]_+$ is precisely the Laurent expansion in equation \eqref{equation_expansion_character_cusp_1}, while that for $[(1-\lambda)^{-1} +(\lambda)^{-1}]_-$ is the Laurent expansion in equation \eqref{equation_expansion_character_cusp_2}. In general the characters corresponding to the plus and minus expansions correspond to the traces over the complex and the adjoint complex.
\end{remark}

In this example, the function $t=y/x$ is in the integral closure of the local ring at the singular fixed point, but not in the local ring itself. This follows from the observation that $y/x$ is a rational function which is $L^{\infty}$ bounded in a deleted neighbourhood of the fixed point (since $y \sim x^{3/2}$), and the discussion in Subsection \ref{subsection_intro_to_compute}. What this shows is that when we consider the local $L^2$ Dolbeault cohomology of the resolution, we get rational functions which are $L^2$ bounded that are not in the ring of regular functions at the singularity. If we consider a domain with sufficient decay at the singularity (for instance the VAPS domain in the case of $y^2=x^3$), then the local cohomology contains only local holomorphic functions which are $L^2$ bounded, the set of which is isomorphic to the local ring of the algebraic normalization.

The regular functions on the cusp are defined via the inclusion of the space in $\mathbb{CP}^2$, and carry the information that the cusp curve is a deformation of an elliptic curve (a torus). This is reflected in the Riemann-Roch number of Baum-Fulton-MacPherson for the cusp being equal to that of the torus.
In Remark \ref{remark_normal_boundary_hypersurfaces}, we introduced the example of a pinched torus and the cusp curve above arises as a degeneration of the family of curves $zy^2=x^3+ax^2z+bz^3$, by taking $a,b$ to $0$. 

We next consider the example of two spheres joined at a point (see page 55 of \cite{Kirwan&woolf_2006_book}). This features another difference between the Riemann-Roch numbers of \cite{baumfultonmacpheresonRiemannRoch,baum1979lefschetz} and the $L^2$ Riemann-Roch numbers, in the case of non-reduced, non-normal topological pseudomanifolds with K\"ahler structures (hence algebraically non-normal as well). Namely, a connected pseudomanifold that is only connected along a set of codimension $\geq 2$ will have $h^{0,0} \geq 1$.

\begin{example}[Two spheres connected at a point]
\label{example_two_spheres}
Consider the complex variety given by $x^2+y^2=0$ in $\mathbb{CP}^2$, with homogeneous coordinates $[x:y:z]$. This is topologically a union of two spheres $x-iy=0$ and $x+iy=0$, connected at the point $[0:0:1]$, where the link is the union of two circles. We have the algebraic $\mathbb{C}^*$ action $(\lambda)[x:y:z]=[\lambda x: \lambda y : z]$ This action has three fixed points, of which $[i:1:0]$ and $[-i:1:0]$ are smooth while $[0:0:1]$ is singular. The holomorphic Lefschetz numbers at each of the smooth fixed points is $1/(1-\lambda^{-1})$.
\end{example}

Computing the local Lefschetz numbers of \cite{baum1979lefschetz} as in the previous examples, we get the local Lefschetz number $(1-\lambda^2)/(1-\lambda)^2$ which is equal to $(1+\lambda)/(1-\lambda)$ at the singular fixed point, and the global holomorphic Lefschetz number is $1$.

However, let us consider the $L^2$ holomorphic sections at the fixed point with the VAPS condition at the singularity. 
Here, the local cohomology in a fundamental neighbourhood of the singular fixed point is the Hilbert space $\overline{R'}$ where $R'=R[1,x/y]$ with $R=\mathbb{C}[x/z]$. One can check that this has all of the rational functions on $\mathbb{C}^2$ that are $L^2$ bounded in some neighbourhood of the singular point, as in the case of the example of the cusp curve.
The functions $x/y$ and $1$ on the variety are global $L^2$ functions in the null space of the Dolbeault operator with the VAPS conditions. Since the space is $x^2+y^2=0$, the function $x/y$ takes the constant values $i$ and $-i$ restricted to the branches corresponding to the two spheres $x-iy=0$ and $x+iy=0$. This shows that the function $x/y$ is is linearly independent from $1$ in the fundamental neighbourhood.
Now it is easy to compute that the local $L^2$ holomorphic Lefschetz number at this fixed point is $2/(1-\lambda)$ by taking renormalized traces over a Schauder basis for $R'$, and therefore the global holomorphic Lefschetz number for the trivial bundle is $2$. 

Let us now consider the computation on the topological normalization. This corresponds to disconnecting the two spheres, whence there are two new fixed points replacing the single singular fixed point. On each of these, we can take the holomorphic functions to be generated by $R$, described above. Notice that restricted to each sphere, $1$ and $x/y$ are not linearly independent since they differ by an element in $\mathbb{C}$. It is easy to compute that the local Lefschetz numbers at both of these are simply $1/(1-\lambda)$, whence the global Lefschetz number is again $2$.

Similarly to the moment map construction given by \ref{Moment_map_for_quadric} and the computation of the de Rham Morse polynomial in Subsection \ref{subsection_example_nodal_surface}, we can compute the de Rham Morse Polynomial for the Hamiltonian Morse function corresponding to the K\"ahler action in this example (both in the topologically normal and non-normal cases) to see that the Morse polynomial is $2b^0+2b^2$, which corresponds to the Poincar\'e polynomial (see page 56 of \cite{Kirwan&woolf_2006_book} for the intersection homology groups of this space).   
Again this example shows a difference between algebraically normal and non-normal spaces and how this is seen by the $L^2$ cohomology groups. 
In particular, while the Lefschetz-Riemann-Roch formulas of \cite{baum1979lefschetz} differ between singular spaces and their normalizations, the $L^2$ holomorphic Lefschetz numbers are the same for the Dolbeault complex with the VAPS domain.

In Remark \ref{remark_reduced_spaces_moment_map_later_use}, we saw how we can obtain moment maps on an example of a union of two spheres, and the variety (the two spheres) in that example is biholomorphic to the variety here.
There, we studied a $\mathbb{C}^*$ action on $\widehat{M}$
such that $z^2=x^2$ is the corresponding symplectic reduction, and the index of the Dolbeault Dirac operators on $\widehat{M}$ and the reduced space are different. Quantization does not commute with reduction (see \cite{meinrenken1999singular}) since it is not a good categorical quotient as we discussed in Subsection \ref{subsection_intro_to_compute}.

\subsubsection{A depth two singularity}
\label{subsection_non_isolated_singularity}

We study an algebraic variety with a non-isolated singular $\mathbb{CP}^1$.

\begin{example}
\label{example_singular_calabi_yau}
Consider the example of $Z^4-X^3Y=0$ in $\mathbb{CP}^3$ with coordinates $[W:X:Y:Z]$, which admits the algebraic torus action $(\lambda,\mu)\cdot [W:X:Y:Z]=[W:\lambda^4X:\mu^4Y:\lambda^3\mu Z]$. We compute the equivariant index for the untwisted Dolbeault complex $\mathcal{P}=(\Omega^{0,q}(X), \overline{\partial})$.
There are three fixed points at $a_1=[1:0:0:0]$, $a_2=[0:1:0:0]$ and $a_3=[0:0:1:0]$, where $a_2$ is a smooth fixed point while $a_1$ and $a_3$ are singular. The singular set is $[W:0:Y:0]$.
\end{example}

We observe that this space can be obtained as the compactification of the affine cone over the singular algebraic curve $Z^4-X^3Y=0$ in $\mathbb{CP}^2$. Thus the point $a_1$ is a stratum of depth $2$ within the singular set.

If we compute the Baum-Fulton-Quart Lefschetz Riemann-Roch number following \cite[\S 3.3]{baum1979lefschetz} summing up the local Lefschetz numbers for $a_1, a_2$ and $a_3$ in order, we get
\begin{equation}
\label{equation_Singular_Calabi_Yau_p=0}
    \frac{1+(\lambda^3\mu)+(\lambda^3\mu)^2+(\lambda^3\mu)^3}{(1-\lambda^4)(1-\mu^4)}+ \frac{1}{(1-1/\lambda^4)(1-\mu/\lambda)}
    +\frac{1+(\lambda^3/\mu^3)
    +(\lambda^3/\mu^3)^2+(\lambda^3/\mu^3)^3}{(1-1/\mu^4)(1-\lambda^4/\mu^4)}= 1+\frac{\lambda^5}{\mu}.
\end{equation}
Let us understand the terms appearing on the right hand side of this equation. Computing the Poincar\'e residue as in equation \eqref{equation_poincare_residue}, we see for instance that on the chart where $W=1$, a section of the canonical bundle is given by 
\begin{equation}
    \frac{dz \wedge dx}{x^3}
\end{equation}
and it is easy to check that this glues to a global section seen by the Baum-Fulton-Quart formula, just as in the case of the cusp curve in example \ref{example_cusp_singularity}.
The adjunction formula for a smooth degree $4$ surface in $\mathbb{CP}^3$ shows that the canonical bundle must be trivial and the space must be a Calabi-Yau, and the Lefschetz-Riemann-Roch number of Baum-Fulton-Quart gives a computation which matches the adjunction formula. This is similar to the case of the cusp curve we discussed above.

Let us now compute the equivariant Riemann-Roch number in $L^2$ cohomology for $p=0$.
The smooth point gives the contribution $1/[(1-1/\lambda^4)(1-\mu/\lambda)]$. The normalized local ring at the singular fixed point $a_3$ is $Q=\mathbb{C}[X/Z,W/Y]$.
The local cohomology in degree $0$ at this point is $\overline{Q}$ and the local Lefschetz number is $1/[(1-\lambda/\mu)(1-1/\mu^4)]$.
On the chart where $W=1$, we consider coordinates $x=X/W$, $y=Y/W$ and $z=Z/W$. The local cohomology group at the singular fixed point $a_1$ is $\overline{R_1}$, where $R_1=R[1,z,xy/z,x^2y/z^2]$, where $R=\mathbb{C}[x,y]$, which can be used to compute the local Lefschetz number $[1+(\lambda^3\mu)+(\lambda^2\mu^2)+(\lambda\mu^3)]/[(1-\lambda^4)(1-\mu^4)]$ as in previous examples. Summing up, the global $L^2$ Lefschetz Riemann-Roch number is
\begin{equation}
    \frac{1+(\lambda^3\mu)+(\lambda^2\mu^2)+(\lambda\mu^3)}{(1-\lambda^4)(1-\mu^4)}+\frac{1}{(1-1/\lambda^4)(1-\mu/\lambda)}+\frac{1}{(1-\lambda/\mu)(1-1/\mu^4)}=1.
\end{equation}

Let us consider the $L^2$ Dolbeault complex with Hodge degree $p=1$. For the space with the group action in Example \ref{example_singular_calabi_yau}, we observe that the Lefschetz number is given by
\begin{equation}
\label{equation_Singular_Calabi_Yau_p=1_L^2}
    \frac{(\lambda+\mu)^2 (\lambda^2+\mu^2)}{(1-\lambda^4)(1-\mu^4)}+ \frac{(1/\lambda^4)+(\mu/\lambda)}{(1-1/\lambda^4)(1-\mu/\lambda)}
    +\frac{(\lambda/\mu)
    +(1/\mu^4)}{(1-\lambda/\mu)(1-1/\mu^4)}= -1
\end{equation}
where the local cohomology at $a_1$ is given by $\overline{Q_1}$ where
\begin{equation}
    Q_1=R\Big[ dx, \frac{x^3dy}{z^3}, \frac{yx^2dy}{z^3}, \frac{ydx}{z}, \frac{xdy}{z}, \frac{ydx}{z}, \frac{xydx}{z^2}, \frac{x^2dy}{z^2}, dy \Big]
\end{equation}
where $R=\mathbb{C}[x,y]$ as for the case of $p=0$ and the local cohomology at the other fixed points are easier to find.
Similarly, for $p=2$, the Lefschetz number is given by
\begin{equation}
\label{equation_Singular_Calabi_Yau_p=2_L^2}
    \frac{(\lambda^4\mu^4+ \lambda\mu^3 + \lambda^2\mu^2 + \lambda^3\mu )}{(1-\lambda^4)(1-\mu^4)}+ \frac{(1/\lambda^4)(\mu/\lambda)}{(1-1/\lambda^4)(1-\mu/\lambda)}
    +\frac{(\lambda/\mu)
    (1/\mu^4)}{(1-\lambda/\mu)(1-1/\mu^4)}= 1
\end{equation}
where the local Lefschetz number at $a_1$ corresponds to the trace of the geometric endomorphism over the Hilbert space $\overline{Q_2}$ where 
\begin{equation}
    Q_2=R\Big[ dx \wedge dy, \frac{dx \wedge dy}{z}, \frac{xdx \wedge dy}{z^2}, \frac{x^2dx \wedge dy}{z^3} \Big].
\end{equation}
With these results, we can easily compute the 
$L^2$ equivariant $\chi_y$ and signature invariants for the space with this action, following the definitions in Subsection \ref{subsection_Hirzebruch_invariants}.

\begin{example}[$L^2$ equivariant $\chi_y$ and signature invariants]
\label{Singular_non_isolated_equivariant_Hirzebruch}
For the group action (which for $|\lambda|=|\mu|=1$ is a family of isometries) on this space, the $L^2$ equivariant $\chi_y$ invariant at $a_1$ is
\begin{equation}
    \chi_y(U_{a_1},f)=\frac{1+(\lambda^3\mu)+(\lambda^2\mu^2)+(\lambda\mu^3)+y^1(\lambda+\mu)^2 (\lambda^2+\mu^2)+y^2(\lambda^4\mu^4+ \lambda\mu^3 + \lambda^2\mu^2 + \lambda^3\mu )}{(1-\lambda^4)(1-\mu^4)}.
\end{equation}    
It is easy to verify that this expression satisfies the duality of equation \eqref{equation_duality_Hirzebruch_chi}. This is in contrast to the version of Baum-Fulton-Quart as is discussed in \cite{donten2018equivariant} below Proposition 3 of that article.
\end{example}


\bibliographystyle{alpha}
\bibliography{reference}

\end{document}